\documentclass[aop]{imsart}

\RequirePackage{amsthm,amsmath,amsfonts,amssymb}
\RequirePackage[numbers]{natbib}
 
\usepackage{accents}
\usepackage{algorithm}
\usepackage{algpseudocode}
\usepackage{amsfonts}
\usepackage{amsmath}
\usepackage{amssymb}
\usepackage{amsthm}
\usepackage[english]{babel}
\usepackage{bm}
\usepackage{bbm}
\usepackage{caption}
\usepackage[usenames,dvipsnames]{color}
\usepackage{dsfont}
\usepackage{enumitem}
\usepackage{esvect}
\usepackage{float}
\usepackage[T1]{fontenc}
\usepackage{graphicx}
\usepackage{hyperref}
\usepackage[utf8]{inputenc}
\usepackage{mathtools}
\usepackage{nccmath}
\usepackage{relsize}
\usepackage{subcaption}
\usepackage{verbatim}

\startlocaldefs
\theoremstyle{plain}
\newtheorem{theorem}{Theorem}[section]
\newtheorem{assumption}[theorem]{Assumption}
\newtheorem{corollary}[theorem]{Corollary}
\newtheorem{claim}[theorem]{Claim}
\newtheorem{proposition}[theorem]{Proposition}

\newtheorem{lemma}[theorem]{Lemma}

\newtheorem{observation}[theorem]{Observation}
\newtheorem{conjecture}[theorem]{Conjecture}
\newtheorem{definition}[theorem]{Definition}

\theoremstyle{definition}
\newtheorem{remark}[theorem]{Remark}

\numberwithin{equation}{section}

\DeclareMathOperator*{\argmin}{arg\,min}

\newcommand{\R}{{\mathbb R}}
\newcommand{\Z}{{\mathbb Z}}
\newcommand{\N}{{\mathbb N}}
\newcommand{\E}{\mathbb E}
\newcommand{\Prob}{\mathbb{P}}
\newcommand\eps{\varepsilon}
\newcommand{\CA}{{\mathcal A}}
\newcommand{\CB}{{\mathcal B}}
\newcommand{\CC}{{\mathcal C}}

\newcommand{\CE}{{\mathcal E}}
\newcommand{\CG}{{\mathcal G}}
\newcommand{\CH}{{\mathcal H}}
\newcommand{\CI}{{\mathcal I}}
\newcommand{\CJ}{{\mathcal J}}
\newcommand{\CK}{{\mathcal K}}
\newcommand{\CL}{{\mathcal L}}
\newcommand{\CM}{{\mathcal M}}

\newcommand{\CQ}{{\mathcal Q}}
\newcommand{\CR}{{\mathcal R}}
\newcommand{\CS}{{\mathcal S}}
\newcommand{\CU}{{\mathcal U}}
\newcommand{\CV}{{\mathcal V}}
\newcommand{\CY}{{\mathcal Y}}
\newcommand{\CZ}{{\mathcal Z}}
\newcommand{\re}{{\mathrm e}}
\newcommand{\Unif}{{\mathrm{Unif}}}
\newcommand{\Poi}{{\mathrm{Poi}}}
\newcommand{\Bin}{{\mathrm{Bin}}}
\newcommand{\ind}[1]{\mathbbm{1}_{\{#1\}}}
\newcommand{\Ind}[1]{\mathbbm{1}{\{#1\}}}
\newcommand{\rd}{\mathrm{d}}
\newcommand{\sss}[1]{{\scriptscriptstyle #1}}
\endlocaldefs

\begin{document}

\begin{frontmatter}
\title{Cluster-size decay in supercritical\\ kernel-based spatial random graphs}
\runtitle{Cluster-size decay in supercritical KSRGs}

\begin{aug}
\author[A]{\fnms{Joost}~\snm{Jorritsma}\ead[label=e1]{Department of Statistics, University of Oxford.}},
\author[B]{\fnms{J\'ulia}~\snm{Komj\'athy}\ead[label=e2]{Delft Institute of Applied Mathematics, Delft University of Technology.}}
\and
\author[C]{\fnms{Dieter}~\snm{Mitsche}\ead[label=e3]{Institut Camille Jordan, Univ. Jean Monnet, Univ. de Lyon and IMC,  Pontif\'icia Universidad Cat\'olica de Chile.}}
\address[A]{joost.jorritsma@stats.ox.ac.uk\printead[presep={,\ }]{e1}}
\address[B]{j.komjathy@tudelft.nl\printead[presep={,\ }]{e2}}
\address[C]{dmitsche@gmail.com\printead[presep={,\ }]{e3}}
\end{aug}

\begin{abstract}
We consider a large class of spatially-embedded random graphs that includes among others long-range percolation, continuum scale-free percolation and the age-dependent random connection model. We assume that the model is supercritical: there is an infinite component. We identify the stretch-exponent $\zeta\in(0,1)$ of the decay of the cluster-size distribution. That is, with $|\CC(0)|$ denoting the number of vertices in the component of the vertex at $0\in \R^d$, we prove
 \[
  \Prob(k< |\CC(0)|<\infty)=\exp\big(-\Theta(k^{\zeta})\big), \qquad \text{as }k\to\infty.
 \]
 The value of $\zeta$ undergoes several phase transitions with respect to three main model parameters: the Euclidean dimension $d$, the power-law tail exponent $\tau$ of the degree distribution and a long-range parameter $\alpha$ governing the presence of long edges in Euclidean space. 
 
 In this paper we present the proof for the region in the phase diagram where the model is a generalization of continuum scale-free percolation and/or hyperbolic random graphs: $\zeta$ in this regime depends both on $\tau,\alpha$. We also prove that the second-largest component in a box of volume $n$ is of size $\Theta((\log n)^{1/\zeta})$ with high probability. We develop a deterministic algorithm, the cover expansion, as new methodology. 
 This algorithm 
 enables us to prevent too large components that may be de-localized or locally dense in space. 
\end{abstract}


\end{frontmatter}
\setcounter{tocdepth}{1}
\tableofcontents
\section{Introduction}\label{sec:intro}
Consider nearest-neighbor Bernoulli percolation on $\Z^d$ \cite{broadbent1957percolation} (NNP), and write  $|\CC(0)|$ for the number of vertices in the connected component containing the origin. Assume that the model is supercritical, i.e., let $p>p_c(\Z^d)$ -- the critical percolation probability on $\Z^d$. It is a result of a sequence of works \cite{aizenman1980lower, alexander1990wulff,cerf2000large, grimmett1990supercritical, kesten1990bondpercolation, kunz1978essential} that  
\begin{equation}\label{eq:intro-basic}
 \log \Prob(k< |\CC(0)|< \infty)      =
        -\Theta(k^{\zeta}),\qquad\text{with }\zeta=(d-1)/d.
\end{equation}
Thus, the \emph{cluster-size decay} in this model is stretched exponential with stretch-exponent $(d-1)/d$.
This decay rate emanates from \emph{surface tension}: all the $\Omega(k^{(d-1)/d})$ edges on the outer boundary of a cluster $\CC$ with $|\CC|> k$  need to be absent. More recently, these results have been extended to Bernoulli percolation on general classes of transitive graphs \cite{contreras2021supercritical, hutchcroft2022transience}.

The present paper and our related works \cite{clusterII, clusterIII} consider $\Prob(k< |\CC(0)|< \infty)$ for a large class of supercritical inhomogeneous percolation models where the degree distribution and/or the edge-length distribution have heavy tails. Our goal is  to 
\begin{equation}\tag{{\bf Goal}}\label{metaquestion}
     \parbox{\dimexpr\linewidth-10em}{%
    \strut
    \centering
    {\emph{ Determine how high-degree vertices and long-range edges \\change the surface-tension driven behavior of cluster-size decay.}}%
    \strut
  }
\end{equation}
We show that the cluster-size decay in \eqref{eq:intro-basic} remains stretched exponential in inhomogeneous models, but with a new exponent $\zeta$ that depends both on the decay of the edge-length and the decay of the degree distribution. The new value of $\zeta$ reflects the structure of the infinite/largest component in the graph induced by a volume-$n$ box: it describes the most likely way that a box is isolated, and represents the scale and structure of a ``backbone'', i.e., a skeleton holding the largest component $\CC_n^{\sss{(1)}}$ together.
These topological descriptions uncover an intimate connection between the cluster-size decay, the size of the second-largest component $\CC_n^{\sss{(2)}}$, and the lower tail of large deviations for the size of $\CC_n^{\sss{(1)}}$. We develop general methods to move between these quantities.  This paper and~\cite{clusterIII} focus on the cluster-size decay and $|\CC_n^\sss{(2)}|$, while \cite{clusterII} treats large deviations of $|\CC_n^\sss{(1)}|$ in more detail.
\smallskip 

\noindent\emph{Results for a special case.} 
We identify the formula for $\zeta$, and prove matching lower and upper bounds for $\log\Prob(k<|\CC(0)|<\infty)$ and $|\CC_n^\sss{(2)}|$ for supercritical continuum scale-free percolation (CSFP) \cite{DeiHofHoo13, DepWut18},  (in)finite geometric inhomogeneous random graphs (GIRG)~\cite{BriKeuLen19}, and hyperbolic random graphs (HRG)~\cite{krioukov2010hyperbolic}. We focus on a region of the parameter space where these models are robust under percolation. These three models can all be parametrized so that  the vertex set is generated by a Poisson point process on $\R^d$, and each vertex $u$ (with spatial location $x_u$) has an independent and identically distributed (iid) random \emph{vertex mark} $w_u$ from a Pareto distribution $\Prob(W\ge x)\propto x^{-(\tau-1)}$. With $a\wedge b:=\min(a,b)$, each pair of vertices $u,v$ is conditionally independently connected by an edge with probability
\begin{equation}\label{eq:product-connectivity}
\Prob\big(u\sim v \mid (x_u, w_u), (x_v, w_v)\big) =p\cdot\bigg(1\wedge\frac{w_u w_v}{\|x_u-x_v\|^d}\bigg)^\alpha.
\end{equation} 
Here $\alpha>1$ is called the \emph{long-range parameter} and $p\in(0,1]$. When $\tau\in(2,3)$, the models are supercritical for all $p\in(0,1]$ \cite{DeiHofHoo13, DepWut18, krioukov2010hyperbolic}. We state our main result applied to these models. Let $\Prob^{\sss{0}}$ be the Palm measure of having a vertex at location $0\in \R^d$ with a random vertex mark.
\begin{theorem}[Special case of main result]\label{thm:informal}
 Consider continuum scale-free percolation, (finite and infinite) geometric inhomogeneous random graphs, and hyperbolic random graphs with parametrization as in \eqref{eq:product-connectivity} and $\tau\in(2,3)$. When $\zeta_{\sss{\mathrm{GIRG}}}=(3-\tau)/(2-(\tau-1)/\alpha)> \max(2-\alpha, (d-1)/d)$, then
 \begin{equation}\label{eq:product-kernel}
 \begin{aligned}
 \log \Prob^\sss{0}( k < |\CC(0)|<\infty) &=  - \Theta(k^{\zeta_{\sss{\mathrm{GIRG}}}}), & &
|\CC_n^{\sss{(2)}}|\,\big/\,(\log n)^{1/\zeta_{\sss{\mathrm{GIRG}}}} \ \mbox{ is tight}; \\
|\CC_n^{\sss{(1)}}|\,\big/\,n\ &{\buildrel \Prob \over \longrightarrow} \  \Prob^\sss{0}(0\leftrightarrow\infty),
& &\log \Prob\big(|\CC_n^\sss{(1)}|<\rho n\big)\overset{(\star)}= -\Omega(n^{\zeta_{\sss{\mathrm{GIRG}}}})
\end{aligned}
 \end{equation}
  for all $\rho>0$. 
 A matching upper bound on $(\star)$  is proven in~\cite{clusterII} for any $\rho<\Prob^\sss{0}(0\leftrightarrow \infty)$.
\end{theorem}
Theorem~\ref{thm:informal} exemplifies that sufficiently many high-degree vertices ($\tau\in(2,3)$) can change the surface-tension driven behavior of the cluster-size decay compared to~\eqref{eq:intro-basic}.
            It has been folklore in the community that the models CSFP, GIRG, HRG in their robust phase $\tau\in(2,3)$ qualitatively behave like their `non-spatial' analogues, namely rank-one inhomogeneous random graphs such as the Chung--Lu or Norros--Reittu model \cite{ChungLu02.1, NorRei06}. This is true with respect to graph distances, first-passage percolation, and the metastable density of the contact process \cite{bringmann2016average, KomLod20, linker2021contact}. 
            In contrast, the underlying geometry affects cluster-size decay, as  $\zeta_{\sss{\mathrm{GIRG}}}\in(\tfrac{d-1}{d},1)$ depends on the long-range parameter $\alpha$ and  dimension $d$;             
            while the distribution of non-giant components in non-spatial models decays exponentially, i.e., $\zeta=1$.  

Instead of treating CSFP, GIRG, and HRG only, we work with a general model that we call \emph{kernel-based spatial random graph} (KSRG), which is a hidden-variable model that incorporates the three models of Theorem~\ref{thm:informal}, and also includes other models: \emph{long-range percolation} (LRP) \cite{schulman_1983}, the (soft) \emph{Poisson--Boolean model}  \cite{gouere2008subcritical, hall1985continuum} (SPBM), the \emph{age- and weight-dependent random connection models} (ARCM) \cite{gracar2019age, GraHeyMonMor19}, and the \emph{scale-free Gilbert graph} \cite{hirsch2017gilbertgraph}. The KSRG model allows for interpolation between these models, which gives rise to a rich phase diagram for the exponent $\zeta$. We obtain partial proofs of \eqref{eq:intro-basic} for these other models and for parameter settings complementary to Theorem~\ref{thm:informal} with $\zeta_{\sss{\mathrm{GIRG}}}\le\max(2-\alpha, (d-1)/d)$. The techniques we develop here form the main technical tools for proving \eqref{eq:intro-basic} for these models for other values $\zeta>\tfrac{d-1}d$ in~\cite{clusterII,clusterIII}, constructing the backbone in \cite{clusterII} with renormalization techniques,  and using combinatorial methods in~\cite{clusterIII}.

\smallskip

\noindent\emph{New methodology.} 
The setting in Theorem~\ref{thm:informal} presents the greatest challenge when it comes to controlling  the size of finite or non-largest clusters in KSRGs.  
In SPBM and ARCM  high-mark vertices tend to be connected by an edge to vertices of lower mark only, 
while in CSFP, GIRG, HRG high-mark vertices tend to have edges to even higher-mark vertices. So, if a partially-revealed finite cluster contains some `fairly' high-mark vertices then the probability of the cluster being isolated is small. However, a finite cluster may be present on vertices of only low marks and be spatially spread out as well, i.e., we cannot guarantee a typical mark distribution. To still obtain the stretched-exponential decay with exponent $\zeta_{\sss{\mathrm{GIRG}}}$, we need to show that any partially-revealed finite cluster has many `backbone' vertices relatively close, where `relatively close' depends on the particular model in question. In CSFP, GIRG, and HRG, sufficiently many backbone vertices need to be `essentially' $O(1)$ distance  away from the set $\CC$ for all partially-revealed clusters $\CC$. For atypically dense clusters, we cannot guarantee the 
$O(1)$-distance bound. In LRP, SPBM, and ARCM, the much weaker distance-estimate $O(k^{1/d})$ suffices to obtain \eqref{eq:intro-basic}~\cite{clusterII}.

\phantomsection
\label{sec:new-method-cover}{\emph{The cover expansion}} is our main novel methodology that overcomes this problem. 
The cover expansion algorithm takes as input a cluster $\CC$ in a partially revealed graph. Making use of `dense areas' of vertices in $\CC$, it allocates a sufficiently large spatial area $\CK(\CC)$ to $\CC$, with the property that backbone vertices located in $\CK(\CC)$ connect by an edge to the set of vertices $\CC$ with constant probability, regardless of the mark distribution in $\CC$. As a result, any partially-revealed cluster $\CC$ stays isolated with probability $\exp(-\Omega(k^{\zeta_\mathrm{GIRG}}))$.
The cover-expansion algorithm is robustly applicable and adaptable to other spatial models. 

For all supercritical models in the KSRG class, we unfold the general relation between the cluster-size decay and the size of the second-largest component. This is an elaborate truncation and sequential boxing argument, and is our main tool in proving upper bounds for \eqref{eq:intro-basic} for other KSRG models in \cite{clusterII, clusterIII}. The present paper  proves lower bounds on $\Prob^\sss{0}(k<|\CC(0)|<\infty)$ and $|\CC_n^\sss{(2)}|$ for all supercritical KSRG models up to the the existence of a linear-sized component in a typical box, which is generally not known for supercritical KSRGs as SPBM and ARCM.
These lower bounds give the formula for $\zeta$ for all KSRGs at once as the solution of a variational problem that describes the most likely way that a box is isolated from its complement.
Before the explanation of this variational problem, we give the definition of the general model encompassing the above  inhomogeneous percolation models.

\begin{definition}[Kernel-based spatial random graphs (KSRG)]\label{def:ksrg} Fix a dimension $d\ge 1$. Let the vertex set $V$ be either $\Z^d$ or a homogeneous Poisson point process (PPP) on $\R^d$. Given $V$, we equip each vertex $u\in V$ with an independent positive mark following distribution $F_W$. 
Let $\kappa: \R_+^2 \to \R_+$ be a symmetric function, called the kernel function. Let $\varrho: \R_+\to [0,1]$ be a non-decreasing function, called the profile function, let $\beta>0$ be the edge-density parameter, and let $p\in(0,1]$ be the edge-percolation parameter. 
Conditionally on the marked vertex set $\CV := \{(x_u, w_u)\}_{u\in V}\subset \R^d\times \R_+$, 
each pair $\{u,v\}$ is independently present in the edge-set $\CE$ with probability
\begin{equation}\label{eq:connection-prob-gen-intro}
\mathrm{p}(u,v):=\Prob\big(u\text{ connected by an edge to }v \mid \CV\,\big) = p\cdot \varrho\Big(\beta\cdot\frac{\kappa(w_u, w_v)}{\|x_u-x_v\|^d}\Big).
\end{equation}
We denote the obtained infinite graph by $\CG=(\CV,\CE)$. We write $\Lambda_n:=[-n^{1/d}/2, n^{1/d}/2]^d$ for a box of volume $n$ centered at the origin, and denote by $\CG_n=(\CV_n,\CE_n)$ the graph induced by vertices with spatial location in $\Lambda_n$. 
We write $\CC_n^{\sss{(i)}}$ for the $i$th largest component of $\CG_n$, and $\CC_n(0)$  for the component containing a vertex at the origin in $\CG_n$, and $\CC(0)$ or $\CC_\infty(0)$ for the component containing this vertex in $\CG$. We write $\Prob^\sss{x}$ for the Palm-measure when the vertex set of a homogeneous Poisson point process is conditioned to contain a vertex at location $x\in\R^d$ with unknown mark.\end{definition}

Definition~\ref{def:ksrg} allows for general kernels, profile functions, and mark distributions, and generalizes the setup above Theorem~\ref{thm:informal}. 
In the rest of the paper we restrict to settings that are commonly used, and which cover the specific models in the introduction \cite{BriKeuLen19, DeiHofHoo13, gracar2019age, hall1985continuum, hirsch2017gilbertgraph, krioukov2010hyperbolic, schulman_1983}. For any  $a, b \in \R$ we write $a\wedge b$ for $\min(a,b)$, and $a\vee b$ for $\max(a,b)$.
 \begin{assumption}\label{assumption:main}
The mark distribution $F_W$ is either constant, i.e., $W_v\equiv 1$ for all $v$, or follows a Pareto distribution with parameter $\tau>2$, i.e., 
\begin{equation}\label{eq:power-law}
1-F_W(w):=\Prob\big(W_v\ge w\big)=w^{-(\tau-1)}, \qquad w\ge 1.
\end{equation}
The profile function $\varrho$ is either threshold or polynomial: for a constant $\alpha>1$, $\varrho$ is either
\begin{equation}\label{eq:profile}
\varrho_{\alpha}(s):=(1\wedge s)^{\alpha} \qquad \text{or }\qquad\varrho_{\mathrm{thr}}(s):=\ind{s\ge1}.
\end{equation}
We assume that the kernel $\kappa$ is one of the following for some parameter $\sigma\ge 0$:
\begin{equation}\label{eq:kernels}
\kappa_\sigma(w_1, w_2)=(w_1\vee w_2)(w_1\wedge w_2)^\sigma, \quad \text{or}\quad\kappa_\mathrm{sum}(w_1, w_2)=\big(w_1^{1/d}+w_2^{1/d}\big)^d.
\end{equation}
When the vertex set is a homogeneous Poisson point process, w.l.o.g.\ we assume   unit intensity. 
When the vertex set is $\Z^d$, we assume that $p\wedge\beta<1$ so that the graph is not connected a.s.
\end{assumption}
When $W_v\equiv 1$ for all $v\in V$ we say 
that $\tau=\infty$; when $\varrho=\varrho_\mathrm{thres}$ we say that $\alpha=\infty$. As $\kappa_0\le\kappa_\mathrm{sum}\le 2^d\kappa_0$, the qualitative behavior of models with $\kappa_0$ and $\kappa_\mathrm{sum}$ is the same. Therefore, when $\kappa=\kappa_{\mathrm{sum}}$ we say that $\sigma=0$.
Assumption \ref{assumption:main} ensures that the model is parametrized so that the expected degree of a vertex is proportional to its mark iff 
$\tau> \sigma+1$~\cite{luchtrathThesis2022}. 
The restrictions $\tau>2$ and $\alpha>1$ ensure that the graph is locally finite. Increasing $\tau$ and/or $\alpha$ leads to less inhomogeneity, that is, lighter-tailed degrees and fewer long edges, respectively. The parameter $\sigma$ allows us to continuously interpolate between well-known models that are special cases. Therefore, we call $\kappa_\sigma$ the \emph{interpolation kernel}. Independently of our work, $\kappa_\sigma$ appeared recently in \cite{luchtrathThesis2022} and was used in \cite{maitraClustering2022}.
This kernel generalizes commonly used kernels in the literature: \emph{trivial}, \emph{strong},  
\emph{product} and \emph{preferential attachment} (PA) kernels, the last one mimicking the spatial preferential attachment model \cite{AieBonCooJanss08, JacMor15}. With $\tau>2$ as in~\eqref{eq:power-law}, 
\begin{equation}\label{eq:kernel-examples}
 \begin{aligned}
  \kappa_{\text{triv}}(x,y) & =1, &\kappa_{\mathrm{strong}}(x,y)&=x\vee y, 
   \\\kappa_{\mathrm{prod}}(x,y)&=xy,              & \kappa_{\mathrm{pa}}(x,y)&=(x\vee y)(x\wedge y)^{\tau-2}.&
 \end{aligned}
\end{equation}
These kernel parametrizations 
all ensure that the degree distribution decays as a power law with exponent $\tau$~\cite{GraHeyMonMor19}.  Any KSRG model with kernel $\kappa_{\mathrm{triv}}$ has the same connection probability as models with $\kappa_{0}$ and marks identical to $1$. Thus, in this case we set $\kappa=\kappa_0$ and $\tau:=\infty$.
A slightly more general version of $\kappa_\sigma$ is the following: let $\sigma_1\ge 0$ and $\sigma_2\in\R$, and define
\begin{equation}\kappa_{\sigma_1, \sigma_2}(x,y):=(x\vee y)^{\sigma_1}(x\wedge y)^{\sigma_2}.\label{eq:kernels-gen}
\end{equation} 
Contrary to $\kappa_\sigma$, the kernel $\kappa_{\sigma_1,\sigma_2}$  includes $\kappa_{\mathrm{weak}}(x,y):=(x\wedge y)^\tau$ by setting $\sigma_1=0, \sigma_2=\tau$. However, models with $\sigma_1=0$ can still be approximated with  $\kappa_{\sigma}$ \cite{jorritsmaThesis}. Moreover, any KSRG with kernel $\kappa_{\sigma_1, \sigma_2}$ and $\sigma_1>0$  can be re-parametrized to have $\sigma_1=1$ by  changing $\tau$ in \eqref{eq:power-law}. 

The parameter $\sigma$ can also be interpreted as an \emph{assortativity} parameter: in a natural coupling of these models using common edge-variables,  edges incident to at least one low-mark vertex are barely  affected by changing $\sigma$. However, edges between two high-mark vertices are created rapidly if $\sigma$ increases. In the next section we explain how the parameters affect the stretch exponent $\zeta$ of the cluster-size decay, inspired by the proof of the lower bound.

\subsection{Downward vertex-boundary and the phase diagram of \texorpdfstring{$\zeta$}{Z}}\label{sec:four-regimes}One possible way for the event $\{k<|\CC(0)|<\infty\}$ to occur is the following: in $\CG_K$, the induced subgraph in the box $\Lambda_K$ of volume $K=\Theta(k)$, the origin is in a (localized) component $\CC_{\mathrm{local}}(0)$ of size larger than $k$, 
and there are also no edges from $\CC_{\mathrm{local}}(0)$ to  $\Lambda_K^{\sss{\complement}}:=\R^d\setminus \Lambda_K$ in $\CG$. The probability that this event occurs is of the same order as the probability that there are no crossing edges  from inside $\Lambda_K$ to outside $\Lambda_K$, provided that we show that $\{|\CC_\mathrm{local}(0)|>k\}$ occurs with constant probability given this \emph{isolation event}. This event $\{\Lambda_K \not \sim \Lambda_K^{\sss{\complement}}\}$ is rare, and the likeliest way it occurs is when there are no `high-mark' vertices in $\Lambda_K$, no high-mark vertices close to $\Lambda_K$, and no crossing edges between lower-mark vertices. The threshold for being of high-mark must balance the expected number of high-mark vertices and that of crossing edges between lower-mark vertices so that they are both of order $\Theta(k^\zeta)$\footnote{On phase boundaries of $\zeta_\star$, polylogarithmic correction factors are required here and in~\eqref{zeta-1}, see Remark~\ref{remark:polylog}.}. 
The isolation event then occurs with probability $\exp(-\Theta(k^\zeta))$. By symmetry, it suffices to only count lower-mark vertices inside $\Lambda_K$ with \emph{downward edges} to $\Lambda_K^\sss{\complement}$: we say that the edge $\{u,v\}=\{(x_u, w_u), (x_v, w_v)\}$ is a `downward edge' from $u$ if $w_u\ge w_v$. We write $u\searrow\Lambda_K^\sss{\complement}$ if $u$ has a downward edge to a vertex in $\Lambda_K^\sss{\complement}$. 
In our proof we show that for all KSRGs 
\begin{equation}\label{zeta-1}
\log\Prob\big(\Lambda_K\not\sim \Lambda_K^\complement\big)=-\Omega\big(\E\big[ \big|\big\{u\in \Lambda_K: u\searrow\Lambda_K^\complement\big\}\big|\big]\big)=-\Omega(k^{\zeta_\star}),
\end{equation} where we define $\zeta_\star$ as
\begin{equation}\label{eq:zeta-star}
\zeta_\star:=\lim_{k\to\infty}\frac{\log \E\Big[ \big|\big\{u\in \Lambda_k: u\searrow\Lambda_k^\complement\big\}\big|\Big]}{\log k}.
\end{equation}
The absence of a mark restriction on the vertices $u$ in~\eqref{zeta-1} indicates that the expected number of high-mark vertices in $\Lambda_K$ is of smaller order than the total expected size of the downward vertex boundary. The restriction to downward edges (in place of just `edges') avoids counting upward edges to a few high-mark vertices outside $\Lambda_K$ that are not present on the isolation event. This restriction is necessary for kernel and profile pairs when ``high-low connections'' dominate the expectation in~\eqref{eq:zeta-long} below, but is unnecessary otherwise.

 In nearest-neighbor percolation on $\Z^d$ all edges are downward edges and short, giving the surface-tension exponent $\zeta_\star=(d-1)/d$. 
When the profile is long-range and/or $\kappa_\sigma$ is non-trivial, there are long edges, and we will show that $\zeta_\star=\max(\zeta_\mathrm{long},(d-1)/d)$, where
\begin{equation}\label{eq:zeta-long}
\zeta_\mathrm{long}:=\lim_{k\to\infty}\Bigg(0\vee\frac{\log \E\Big[ \big|\big\{u\in \Lambda_{k/2}: u\searrow\Lambda_k^\complement\big\}\big|\Big]}{\log k}\Bigg)
 \end{equation}
  describes the number of vertices incident to long downwards edges, that is, of length $\Omega(k^{1/d})$. We will never use $\zeta_\mathrm{long}$ when it equals $0$. The maximum with $0$ avoids unnecessary computations when the the second term is negative. 
  Both $\zeta_{\star}$ and $\zeta_\mathrm{long}$ are explicitly computable given the profile,  kernel, and  vertex-mark distribution, see Claim \ref{lemma:lower-vertex-boundary} below.
 We now give their potential values based on  back-of-the-envelope calculations for KSRGs satisfying Assumption~\ref{assumption:main}.
 We distinguish four types of connections in the downward vertex boundary, and call the type producing the largest contribution to \eqref{eq:zeta-star} \emph{dominant}. \smallskip

\noindent\emph{Nearest-neighbor edges} are dominant if the main contribution to \eqref{eq:zeta-star} is coming from edges of constant length: there are roughly $\Theta(k^{(d-1)/d})$ vertices incident to such edges in $\Lambda_k$, giving the `surface-tension' exponent 
\begin{equation}\label{eq:zeta-nn}
    \zeta_\mathrm{short}:=(d-1)/d.
\end{equation}
Next, we count vertices with edges of length $\Theta(k^{1/d})$ crossing the boundary of $\Lambda_k$, and thus also contributing to $\zeta_{\mathrm{long}}$ in \eqref{eq:zeta-long}.
\smallskip 

\noindent\emph{Low-low edges} are dominant if the main contribution to \eqref{eq:zeta-star} is coming from constant (low-mark) vertices in $\Lambda_{k/2}$ connected to low-mark vertices $\Lambda_k^\sss{\complement}$. The expected number of such connected pairs is $\Theta(k \cdot k \cdot k^{-\alpha})$. Abbreviating `\emph{low}-mark to \emph{low}-mark' by ll, we obtain
\begin{equation}\label{eq:zeta-ll}
\zeta_{\mathrm{ll}}:=2-\alpha.
\end{equation}  
Models with dominantly low-low type connectivity behave similar to long-range percolation.\smallskip 

The remaining connectivity types describe `high-mark' vertices in $\Lambda_{k/2}$ incident to long-edges.  Model-dependently, we call a vertex high-mark if its mark is at least $k^{\gamma_{\mathrm{high}}}$, where 
\begin{equation}\label{eq:gamma-long}
\gamma_{\mathrm{high}}:=\min\Big\{\gamma\ge0: \liminf_{k\to\infty}\E^\sss{0}\big[ |\{ \mbox{edges between }0\mbox{ and }\Lambda_k^\complement\}| \,\big|\, (0, k^{\gamma})\in\CV\big]>0\Big\}.
\end{equation}  
Then, a constant proportion of vertices of mark at least $k^{\gamma_{\mathrm{high}}}$ inside $\Lambda_{k/2}$ contributes to the vertex boundary. By the Pareto mark-distribution in~\eqref{eq:power-law}, there are $\Theta(k^{1-\gamma_{\mathrm{high}}(\tau-1)})$ many high-mark vertices inside $\Lambda_{k/2}$. 
The values $\tau, \sigma,\alpha$ in~\eqref{eq:power-law}--\eqref{eq:kernels} determine the value of $\gamma_{\mathrm{high}}$.
\smallskip 

\noindent\emph{High-low edges} are dominant if a high-mark vertex in $\Lambda_{k/2}$ is typically connected to low (constant) mark vertices outside $\Lambda_k$. There are $\Theta(k)$ constant-mark vertices at distance $\Theta(k^{1/d})$. Using the connection probability \eqref{eq:connection-prob-gen-intro} with $\kappa_\sigma$ or $\kappa_\mathrm{sum}$ from~\eqref{eq:kernels}, for $\gamma \ge 0$, the expected number of edges between vertex $(0, k^\gamma)$ and constant-mark vertices outside $\Lambda_k$ is roughly $k (1\wedge (k^\gamma / k))^\alpha$. As required in \eqref{eq:gamma-long}, this expression is of constant order when
\begin{equation}\label{eq:gamma-lh}
\gamma=\gamma_\mathrm{hl}:=1-1/\alpha, \quad\text{and}\quad     \zeta_\mathrm{hl}:=1-\gamma_\mathrm{hl}(\tau-1)=(\tau-1)/\alpha - (\tau-2).
\end{equation}
High-low connectivity is dominant in (regions of parameters of) models with small $\sigma$,  for example in the age-dependent random connection model and the soft Poisson--Boolean model. Since the value of $\sigma$ barely affects the presence of edges incident to at least one constant-mark vertex, $\zeta_\mathrm{hl}$ does not depend on $\sigma$, as opposed to the next type.
\smallskip 

\noindent\emph{High-high edges} are dominant if a high-mark vertex in $\Lambda_{k/2}$ is typically connected to another high-mark vertex outside $\Lambda_k$.   
There are $\Theta(k^{1-\gamma(\tau-1)})$ vertices of mark $\Omega(k^\gamma)$ at distance $\Theta(k^{1/d})$ from $0$. Using the connection probability \eqref{eq:connection-prob-gen-intro}, the expected number of edges between $(0, k^\gamma)$ and these vertices is roughly $k^{1-\gamma (\tau-1)} (1\wedge (k^{\gamma(\sigma+1)}/ k))^\alpha$. This expression tends to zero for all $\gamma\ge 0$ when $\tau>\sigma+2$, but satisfies \eqref{eq:gamma-long} when $\tau\le \sigma+2$ and 
\begin{equation}\label{eq:gamma-hh}
\gamma=\gamma_\mathrm{hh}:=
    \begin{dcases}
    \frac{1-1/\alpha}{\sigma+1-(\tau-1)/\alpha},&\text{if }\tau\le \sigma+2\text{ and }\alpha<\infty,\\
    \frac{1}{\sigma+1},&\text{if }\tau>\sigma+2\text{ or }\alpha=\infty,
    \end{dcases}
\end{equation}
which in turn gives 
\begin{equation}\label{eq:zeta-hh}
    \zeta_\mathrm{hh}:=1-\gamma_\mathrm{hh}(\tau-1)=
    \begin{dcases}
        \frac{\sigma+2-\tau}{\sigma+1-(\tau-1)/\alpha},&\text{if }\tau\le \sigma+2\text{ and }\alpha<\infty,\\ 
        \frac{\sigma+2-\tau}{\sigma+1},&\text{if }\tau>\sigma+2\text{ or }\alpha=\infty.
    \end{dcases}
\end{equation}
When $\tau>\sigma+2$, $\zeta_\mathrm{hh}$ is negative and some other connectivity type is dominant. The definition of $\gamma_\mathrm{hh}$ when $\tau>\sigma+2$ is purely technical, giving continuity and monotonicity in the parameters. 
The high-high type connectivity is the only type that depends on $\sigma$, and is dominant (for some parameters) in models with large $\sigma$: the product-kernel models in Theorem~\ref{thm:informal} have $\sigma=1$, and $\zeta_{\sss{\mathrm{GIRG}}}=\zeta_{\mathrm{hh}}$ when $\tau<3$. 
The next claim  shows that these are the only connectivity types. The proof follows directly from Lemma~\ref{lemma:lower-vertex-boundary2} below.
\begin{claim}[Dominant connections]\label{lemma:lower-vertex-boundary}
Consider a KSRG model satisfying Assumption \ref{assumption:main} with parameters $\alpha\in(1,\infty]$, $\tau\in(2,\infty]$, $\sigma\ge 0$, and $d\in\N$. 
With $\zeta_\star, \zeta_\mathrm{long}$ from \eqref{eq:zeta-star}, \eqref{eq:zeta-long},
\begin{equation} \label{eq:zeta-star-long-short}
\zeta_\star=\max(\zeta_\mathrm{long}, \zeta_\mathrm{short}) \quad \text{and}\quad \zeta_\mathrm{long}=\max(\zeta_{\mathrm{ll}},\zeta_{\mathrm{hl}}, \zeta_{\mathrm{hh}},0), 
\end{equation}
where for models with threshold profiles ($\alpha=\infty$) and/or lighter-tailed vertex-marks ($\tau=\infty$) one has to take the corresponding limit in the formulas \eqref{eq:zeta-ll}, \eqref{eq:gamma-lh}, and \eqref{eq:zeta-hh}.
\end{claim}
We visualize the changes of the dominant type of $\zeta_\star$ as a function of the parameter space in Figure \ref{fig:phase-diagram-girg-arcm} for models using $\kappa_\mathrm{prod},\kappa_\mathrm{pa},\kappa_\mathrm{max}, \kappa_\mathrm{sum}$. For these kernels, at most one of the regimes ``high-low'' and ``high-high'' appears on the diagrams, see also Table~\ref{table:overview}. In Figure~\ref{fig:phase-diagram-alpha} we vary $\sigma$ and $\tau$ while keeping $\alpha$  and $d$ fixed.  
\smallskip 

\noindent\emph{A general conjecture.} The connection to the downward vertex boundary gives the method to prove lower bounds. However, upper bounds do not follow from this intuition, and the challenge there is to handle components that are delocalized in space. Relating back to~\eqref{metaquestion}, we state our conjecture for KSRGs in Definition~\ref{def:ksrg} in general.
\begin{conjecture}\label{conj:intro-version2} Consider a supercritical KSRG. Let $\zeta_\star$ be as in~\eqref{eq:zeta-star} and assume that the parameters are such that $\zeta_\star>0$. Then, 
\[
-\log \Prob^\sss{0}(k<|\CC(0)|<\infty)=k^{\zeta_\star\pm o(1)}, \quad  
\Prob\big((\log n)^{1/\zeta_\star - o(1)}\le |\CC_n^\sss{(2)}|\le (\log n)^{1/\zeta_\star + o(1)}\big)\to 1.
\]
Moreover, $-\log \Prob^\sss{0}(|\CC_n^\sss{(1)}|<\rho n)=n^{\zeta_\star\pm o(1)}$ for any $\rho<\Prob^\sss{0}(0\leftrightarrow\infty)$.
\end{conjecture}
Proving this conjecture would achieve \eqref{metaquestion}: since $\zeta_\star=\max((d-1)/d, \zeta_{\mathrm{long}})$, high-degree vertices and long-range edges change the surface-tension behavior of the cluster-size decay only when the downward vertex boundary is dominated by vertices incident to long edges. This paper and~\cite{clusterII} study this region of the parameter space to prove the conjecture for KSRGs on Poisson-point processes satisfying Assumption~\ref{assumption:main} whenever  $\zeta_\star>(d-1)/d$ and $\tau\ge\sigma+1$. Here and in~\cite{clusterII, clusterIII}, we obtain partial results when $(d-1)/d\ge \max(\zeta_\mathrm{ll}, \zeta_\mathrm{hl}, \zeta_\mathrm{hh})$ and for KSRGs on $\Z^d$.
The next section presents the detailed results of this paper that prove the conjecture for the red regions in Figure~\ref{fig:phase-diagram-combined}, of which Theorem~\ref{thm:informal} is a special case. 

\section{Main results}\label{sec:main-results}
Recall $\zeta_\mathrm{ll}$, $\zeta_\mathrm{hl}$, $\zeta_\mathrm{hh}$, and $\zeta_\mathrm{short}$ from \eqref{eq:zeta-ll}, \eqref{eq:gamma-lh}, \eqref{eq:zeta-hh}, and \eqref{eq:zeta-nn}, and that $\zeta_\star=\max(\zeta_\mathrm{ll}, \zeta_\mathrm{hl}, \zeta_\mathrm{hh},\zeta_\mathrm{short})$ by Claim~\ref{lemma:lower-vertex-boundary}. Our following results assume parameters where high-high connections are present, i.e., $\zeta_\mathrm{hh}>0$.  This is equivalent to $\tau\in (2,2+\sigma)$, and includes $\kappa_\mathrm{prod}\equiv\kappa_{1}$ when $\tau\in(2,3)$, as in  Theorem~\ref{thm:informal}. 
 Whenever $\zeta_\mathrm{hh}>0$, the model is supercritical for all $p, \beta>0$ in \eqref{eq:connection-prob-gen-intro} and $\alpha>1$, i.e., there exists a unique  infinite component $\CC_\infty^\sss{(1)}$, see Proposition~\ref{proposition:existence-large} below.  We denote the number of dominant connectivity types by 
\begin{equation}\label{eq:m-star}
\mathfrak{m}_\star:=\ind{\zeta_\star=\zeta_\mathrm{ll}} + \ind{\zeta_\star=\zeta_\mathrm{hl}} +\ind{\zeta_\star=\zeta_\mathrm{hh}} +\ind{\zeta_\star=\zeta_\mathrm{short}}.
\end{equation} 

\begin{figure}[p]
\begin{subfigure}{406pt}
  \includegraphics[width=406pt]{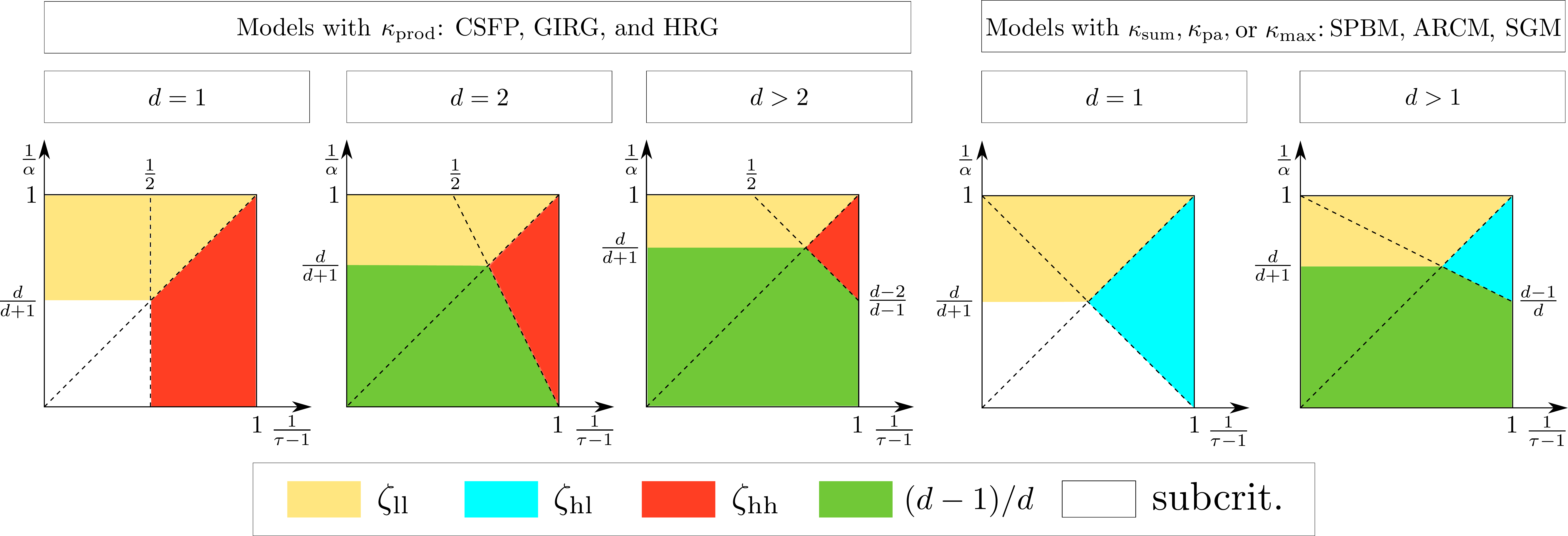}
 \caption{
  Phase-diagrams of $\zeta=\zeta(\tau,\alpha)$ for models with kernels $\kappa_\mathrm{prod}$, and $\kappa_\mathrm{sum}$, $\kappa_\mathrm{pa}$  or $\kappa_\mathrm{strong}$, plotted as a function of $1/(\tau-1)$ and $1/\alpha$.
  The $y$-axis (i.e., $1/(\tau-1)=0$) also describes the phase diagram of  (continuum) long-range percolation that has kernel $\kappa_\mathrm{triv}$, while the models on the $x$-axis ($1/\alpha=0$) coincide with models using a threshold profile function in \eqref{eq:profile}.  When $1/\alpha> 1$ or $1/(\tau-1)> 1$, then $\CG_\infty$ is connected and each vertex has infinite degree almost surely \cite{HeyHulJor17}. A white color within the square means that  the model is subcritical for each value $p,\beta$ in \eqref{eq:connection-prob-gen-intro} \cite{gracar2022finiteness}.\vspace{10pt} 
 }\label{fig:phase-diagram-girg-arcm}
 \end{subfigure}
 \begin{subfigure}{406pt}
  \begin{center}
 \includegraphics[width=406pt]{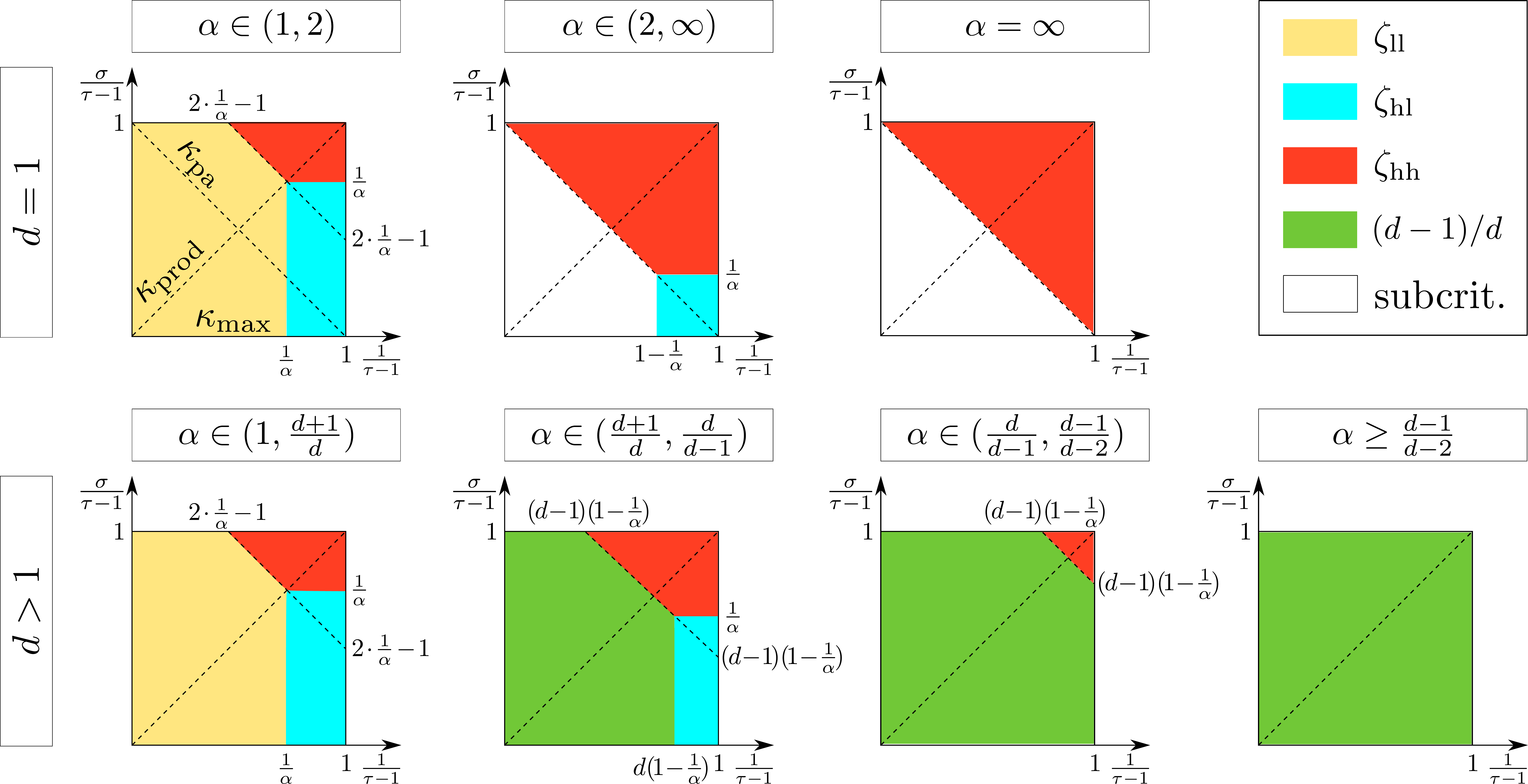}
 \end{center}
 \caption{
 Phase-diagrams of $\zeta=\zeta(\sigma,\tau)$ for fixed values of $\alpha$ in \eqref{eq:profile}, plotted as a function of $1/(\tau-1)$ on the $x$-axis and $\sigma/(\tau-1)$ on the $y$-axis.
  The identity line $y\!=\!x$   corresponds to models using kernel $\kappa_{\mathrm{prod}}\equiv\kappa_{1}$, the $x$-axis to models using $\kappa_{\mathrm{strong}}\!\equiv\!\kappa_{0}$ and the cross-diagonal $x\!+\!y\!=\!1$ to models using $\kappa_{\mathrm{pa}}\!\equiv\!\kappa_{\tau-2}$. The origin  captures models with $\kappa_{\mathrm{triv}}\!\equiv\!\kappa_{0}$.
  Observe that $\zeta_{\mathrm{hl}}$ (blue) is never dominant above the diagonal $y\ge x$ (equivalently, $\sigma\ge 1$), while $\zeta_{\mathrm{hh}}$ (red) is never dominant below the cross-diagonal $x\!+\!y\!=\!1$  (equivalently, $\sigma\!\le\! \tau\!-\!2$). In the quadrant $x\!+\!y\!\ge\! 1, y\!\le\! x$ all four exponents `compete' for dominance. 
  }\label{fig:phase-diagram-alpha}
 \end{subfigure}
\caption{Phase diagrams of the (conjectured) 
cluster-size decay for kernel-based spatial random graphs. Theorem~\ref{thm:subexponential-decay} proves the upper bound in the red regions, and the lower bounds above the $x+y\ge 1$ line on Figure \ref{fig:phase-diagram-alpha}, for all four colors simultaneously, with logarithmic correction terms on phase boundary lines.}
\label{fig:phase-diagram-combined}
\end{figure}

 \begin{table}[t]
 \caption{
  Models belonging to the KSRG framework, their vertex sets, kernels, profiles, and their  value $\zeta_\star$. \\Horizontal lines separate models with different kernels.}\label{table:overview}
\footnotesize
 \begin{tabular}{lcccc}
 \hline
  \textbf{Model}  & \textbf{$\bm{\CV}$} &
  \textbf{Kernel}                           & \textbf{Profile}                       & \textbf{$\zeta_\star$}                                                                                                                                 \\\hline
  Bond-percolation on $\Z^d$                         
  \cite{HammWels65}& $\mathbb{Z}^d$      & $\kappa_{\mathrm{triv}}$    & $\varrho_\mathrm{thr}$                 & $\zeta_\mathrm{short}$                                                                                                                                \\
  Random geometric graph                           
  \cite{Pen}
  &PPP                 &                                           & $\varrho_\mathrm{thr}$                 & $\zeta_\mathrm{short}$                                                                                                                   \\
  Long-range percolation \cite{schulman_1983}                           & $\mathbb{Z}^d$      &                                           & $\varrho_\alpha$                       & $\max(\zeta_\mathrm{ll}, \zeta_\mathrm{short})$                                       \\
  Continuum long-range percolation \cite{longrangeUnique1987}                          & PPP                 &                                           & $\varrho_\alpha$                       & $\max(\zeta_\mathrm{ll}, \zeta_\mathrm{short})$                                                      \\ \hline
  Scale-free percolation  \cite{DeiHofHoo13}                         & $\mathbb{Z}^d$      & $\kappa_{\mathrm{prod}}, \kappa_{1}$    & $\varrho_\alpha$                       & $\max(\zeta_\mathrm{hh}, \zeta_\mathrm{ll}, \zeta_\mathrm{short})$                                                                     \\
  Continuum scale-free percolation           \cite{DepWut18}               & PPP                 &                                           & $\varrho_\alpha$                       & $\max(\zeta_\mathrm{hh}, \zeta_\mathrm{ll}, \zeta_\mathrm{short})$                                          \\
  Geometric inhomogeneous random graph               \cite{BriKeuLen19}           & PPP                 &                                           & $\varrho_\alpha, \varrho_\mathrm{thr}$ & $\max(\zeta_\mathrm{hh}, \zeta_\mathrm{ll}, \zeta_\mathrm{short})$                    \\
  Hyperbolic random graph                           
  \cite{krioukov2010hyperbolic}&PPP      &           &                                            $\varrho_\mathrm{thr}$                 & $\zeta_\mathrm{hh}$                                                                  \\ \hline
  Age-dependent random connection model              \cite{gracar2019age}            & PPP                 & $\kappa_{\mathrm{pa}}, \kappa_{\tau-2}$ & $\varrho_\alpha, \varrho_\mathrm{thr}$ & $\max(\zeta_\mathrm{hl}, \zeta_\mathrm{ll}, \zeta_\mathrm{short})$           \\ \hline
  Poisson--Boolean model \cite{gouere2008subcritical}                           & PPP                 & $\kappa_\mathrm{sum}$        & $\varrho_\mathrm{thr}$                 & $\zeta_\mathrm{short}$                                                                   \\
  Soft Poisson--Boolean model \cite{GraHeyMonMor19}                           & PPP                 &        & $\varrho_\alpha$                 & $\max(\zeta_\mathrm{hl}, \zeta_\mathrm{ll}, \zeta_\mathrm{short})$                   \\ \hline
  Scale-free Gilbert graph \cite{hirsch2017gilbertgraph}                           & PPP                 & $\kappa_\mathrm{strong},\kappa_{1}$        & $\varrho_\mathrm{thr}$                 & $\zeta_\mathrm{short}$                                                                   \\
                 &                                    &                                           & $\varrho_\alpha$                       & $\max(\zeta_\mathrm{hl}, \zeta_\mathrm{ll}, \zeta_\mathrm{short})$                    \\ \hline
  Ultra-small scale-free geometric network           \cite{yukich_2016}                 & $\mathbb{Z}^d$      & $\kappa_\mathrm{weak},\kappa_{0,\tau}$        & $\varrho_\mathrm{thr}$                 & $\max(\zeta_\mathrm{hh}, \zeta_\mathrm{short})$                                                                                                          \\
  \hline \hline
  \textbf{Interpolating KSRG}                   & \textbf{PPP}        & $\kappa_{\sigma}$                       & $\varrho_\alpha, \varrho_\mathrm{thr}$ & $\max(\zeta_\mathrm{ll},\zeta_\mathrm{hl}, \zeta_\mathrm{hh}, \zeta_\mathrm{short})$ \\ \hline \hline
 \end{tabular}
\end{table}

\begin{theorem}[Cluster-size decay]\label{thm:subexponential-decay}
Consider a KSRG in dimension $d\ge 1$ satisfying Assumption \ref{assumption:main} with parameters such that $\zeta_\mathrm{hh}>0$, i.e., $\alpha\in(1,\infty]$, $\sigma> 0$, and $\tau\in(2, 2+\sigma)$.
 There exists a constant $A>0$ such that for all $k\ge1$ the following hold.
 \begin{itemize}
\setlength\itemsep{-0.2em}
  \item[(i)] For all $n\in (Ak, \infty]$,
        \begin{equation}
\Prob^\sss{0}\big(|\CC_n(0)|> k, \; 0\notin\CC_n^\sss{(1)} \big)\,\ge\,  \exp\big(-Ak^{\zeta_\star}(\log k)^{\mathfrak{m}_\star-1}\big).\label{eq:subexponential-main-unlower}
        \end{equation}
  \item[(ii)]   If additionally $\tau\ge \sigma+1$ and the vertex set is formed by a homogeneous Poisson point process, then for all   $n \in (k, \infty]$, 
        \begin{align}
         \Prob^\sss{0}\big(|\CC_n(0)|> k, \; 0\notin\CC_n^\sss{(1)}\big)
         \, \le\,
         \exp\big(-(1/A)k^{\zeta_\mathrm{hh}}\big),\label{eq:subexponential-main-upper}
        \end{align}
        \item[(iii)]  while if  $\tau<\sigma+1$ and the vertex set is formed by a homogeneous Poisson point process, then for all $n \in (k, \infty]$,
        \begin{align}
         \Prob^\sss{0}\big(|\CC_n(0)|> k, \; 0\notin\CC_n^\sss{(1)}\big)
         \, \le\,
         \exp\big(-(1/A)k^{1/(\sigma+1-(\tau-1)/\alpha)}\big).\label{eq:subexponential-main-upper-2}
        \end{align}
 \end{itemize}
\end{theorem}
The  next theorem is the analogue of Theorem~\ref{thm:subexponential-decay} for the size of the second-largest component. The following intuition applies: the maximum value of $n$ \emph{iid} random variables $X_i$ with $\Prob(X_i\ge x)=\exp(-\Theta(x^\zeta))$ is of order $\Theta((\log n)^{1/\zeta})$. Although the non-largest cluster sizes $(|\CC_n(v)|)_{v\notin \CC_n^\sss{(1)}}$ are \emph{not iid}, Theorem~\ref{thm:subexponential-decay} suggests a cluster in $\CG_n$ of this order.
\begin{theorem}[Second-largest component]\label{thm:second-largest}
Consider a KSRG under the same assumptions as in Theorem~\ref{thm:subexponential-decay}. 
 \begin{itemize}
  \setlength\itemsep{-0.3em}
  \item[(i)]
        There exist  constants $A, \delta, n_0>0$, such that for all $n\in[n_0,\infty)$,
        \begin{align}
         \Prob\Big(|\CC^\sss{(2)}_n| \ge (1/A)(\log n)^{1/\zeta_\star}/(\log\log n)^{(\mathfrak{m}_\star-1)/\zeta_\star}\Big)
         \ge 1-n^{-\delta}.\label{eq:second-largest-main-unlower}
        \end{align}

  \item[(ii)] If $\tau\ge \sigma+1$  and the vertex set is formed by a homogeneous
  Poisson point process, then for all $\delta>0$ there exists $A>0$ such that for all $n\in[1,\infty)$,
        \begin{equation}
         \Prob\big(|\CC^\sss{(2)}_n| \le A(\log n)^{1/\zeta_\mathrm{hh}}\big)
         \ge 1-n^{-\delta}.\label{eq:second-largest-main-upper}
        \end{equation}
          \item[(iii)] If $\tau< \sigma+1$  and the vertex set is formed by a homogeneous
          Poisson point process, then for all $\delta>0$, there exists $A>0$ such that for all $n\in[1,\infty)$,
        \begin{equation}
         \Prob\big(|\CC^\sss{(2)}_n| \le A(\log n)^{\sigma+1-(\tau-1)/\alpha}\big)
         \ge 1-n^{-\delta}.\label{eq:second-largest-main-upper-2}
        \end{equation}
 \end{itemize}
\end{theorem}
Let us make a few remarks. We believe that the lower bounds in part~(i) of both theorems are sharp. They give rise to Conjecture \ref{conj:intro-version2}. Part (ii) matches part (i) when $\zeta_\mathrm{hh}$ is the unique maximum (this case includes $\sigma\le 1$ such as $\kappa_\mathrm{prod}$, since we assume $\tau>2$). When the maximum is non-unique, we conjecture the lower bound to be sharp. Part (iii) never matches the lower bound of part (i), 
which is due to (non-negligible) technicalities in our proofs, relating to the degree distributions having a heavier tail exponent than $\tau-1$ \cite{luchtrathThesis2022}. We expect that parts (ii) and (iii) extend to KSRGs with $\Z^d$ as a vertex set, but we leave the technicalities out of this paper to benefit from independence properties of Poisson point processes.

 The upper bound of Theorem~\ref{thm:subexponential-decay} leads to the weak law of large numbers for the size of the largest component, which was already known for hyperbolic random graphs~\cite{FouMul18}, but not for geometric inhomogeneous random graphs and continuum scale-free percolation.\vspace*{-4pt}
\begin{corollary}[Law of large numbers for the giant]\label{cor:lln}
Consider a KSRG under the same assumptions as in Theorem~\ref{thm:subexponential-decay}, with vertex set formed by a homogeneous
Poisson point process.
 Then, 
 \vspace{-2pt}
\begin{equation}\nonumber
    |\CC_n^\sss{(1)}|\,\big/\, n\overset{\Prob}\longrightarrow \Prob^\sss{0}(0\leftrightarrow\infty), \qquad\mbox{as } n\to\infty.
 \end{equation}
\end{corollary}
\vspace{-5pt}
The next theorem shows that $\zeta_\star$ also governs the lower tail of large deviations of $|\CC_n^\sss{(1)}|$. It also holds  when $\zeta_\mathrm{hh}\le0$, contrary to Theorems \ref{thm:subexponential-decay}--\ref{thm:second-largest}.
\vspace*{-6pt}
\begin{theorem}[Speed of the lower tail of large deviations of the giant]\label{prop:lower-dev-giant}
Consider a KSRG in dimension $d\ge 1$ satisfying Assumption \ref{assumption:main}, i.e., $\alpha\in(1,\infty]$, $\sigma\ge 0$, and $\tau\in(2,\infty]$. There is a constant $A>0$ such that for all $\rho>0$ and $n\ge 1$,
 \begin{equation}
\Prob\big(|\CC_n^\sss{(1)}|<\rho n\big)\ge \exp\big(-(A/\rho)\cdot n^{\zeta_\star}(\log n)^{\mathfrak{m}_\star-1}\big).\label{eq:lower-dev-giant}
 \end{equation}
\end{theorem}
  \vspace{-10pt}

\subsection{Discussion and related literature}
The event $\{k<|\CC(0)|<\infty\}$ is non-monotone under edge-addition, which makes it challenging to control the geometry of the ``outer boundary'' of small clusters and the infinite component. Peierls' argument and Grimmett--Marstrand dynamic renormalization are popular tools to control the outer boundary in models where surface tension governs cluster-size decay, such as Bernoulli percolation on $\Z^d$~\cite{grimmett1990supercritical}, and the Poisson--Boolean model~\cite{dembin2022almost}. 
The recent work \cite{contreras2021supercritical} combines a static renormalization method with hypercontractive inequalities to prove surface-tension driven behavior for Bernoulli percolation on transitive graphs of polynomial ball growth (when the number of vertices at distance $r$ from a  vertex grows polynomially in $r$). 

We consider here inhomogeneous percolation models on the complete graph of the vertex set, with correlated edge probabilities dependent on vertex-marks and spatial distance. In this setting, the vertex boundary of finite boxes is either governed by \emph{short} edges (surface tension), or by \emph{long} edges. Long edges can be one of three `types'  depending on the typical degrees or marks of the end-vertices. The decay exponent $\zeta_\star$ is determined by the dominant edge type. We focus on parameter settings where the long edges dominate the vertex boundary of finite boxes, and the ball growth is \emph{superpolynomial} even after percolation~\cite{biskup2004scaling, bringmann2016average, DeiHofHoo13,gracar2022chemical, LapLenKom20}. Due to these long edges, surface tension is no longer the relevant quantity  and the methods above for graphs with polynomial growth do not give sufficiently strong bounds. 
Instead, long edges make connections to a `backbone' of the giant component possible, so the relevant quantity to control is the `effective' distance from this backbone. We guarantee good distance bounds by a new method, the cover expansion, see page \pageref{sec:new-method-cover} above.

Another example where competing phenomena in the boundary lead to phase transitions, is the growth of long-range first passage percolation on $\Z^d$~\cite{chatterjee2016multiple} and that of first-passage percolation on SFP, GIRG, and HRG \cite{KomLod20}. In the former, phase transitions occur at $\alpha=2$ and $\alpha=2+1/d$, and in the latter, phase transitions occur at $\tau=3$ and $\alpha=2$. For the cluster-size decay, the phase transition 
in long-range percolation occurs at $\alpha=1+1/d$; in SFP, GIRG, and HRG transitions occur at $\alpha=\tau-1$, $\alpha=1+1/d$, and when $d(\alpha-1)(\tau-1)=2\alpha-(\tau-1)$. Whereas our exponent $\zeta_\star$ is determined by the ``bulk'' of the vertex boundary, the transitions in~\cite{chatterjee2016multiple, KomLod20} are determined by the presence of ``exceptional'' edges on the edge boundary, causing different transition points. Analogously, the phase transitions for graph distances in KSRGs also differ from those of $\zeta_\star$ of the cluster-size decay~\cite{biskup2004scaling, bringmann2016average, DeiHofHoo13,gracar2022chemical, LapLenKom20}. \smallskip 

\emph{The second-largest component.\ \ }
The study of the second-largest component $\CC_n^\sss{(2)}$ ties in with the percolation duality for non-spatial random graphs (Erd{\H o}s-R\'enyi random graphs, inhomogeneous random graphs \cite{bollobas1984evolution,BolSvaRio07}), for which $\Prob(k<|\CC(0)|<\infty)$ decays exponentially in $k$ and $|\CC_n^\sss{(2)}|$ is logarithmic in $n$. For models with underlying geometry, $|\CC_n^\sss{(2)}|$ was studied for random geometric graphs, long-range percolation, and hyperbolic random graphs~\cite{crawford2012simple, kiwimitsche2ndlargest, lichev2022bernoulli, Pen, penrose2022giant}. By introducing the interpolation kernel $\kappa_\sigma$, see also~\cite{luchtrathThesis2022, maitraClustering2022}, this paper uncovers the intricate connection between $|\CC_n^\sss{(2)}|$ and the cluster-size decay in inhomogeneous percolation models in the KSRG class in general, and enables us to prove analogues of Theorems~\ref{thm:subexponential-decay}--\ref{thm:second-largest} for other parameters in the follow-up works~\cite{clusterII, clusterIII}. 

Both threshold and soft hyperbolic random graphs (HRG) in \cite{krioukov2010hyperbolic} are a special case of Theorems~\ref{thm:subexponential-decay}--\ref{thm:second-largest}: there is an isomorphism between an HRG and a $1$-dimensional  KSRG with a product kernel i.e., $\sigma=1$, $\tau\in(2,3)$, with threshold HRGs having $\alpha=\infty$ and soft HRGs having $\alpha<\infty$, see  \cite{BriKeuLen19} or \cite[Section~9]{KomLod20}. So, for threshold HRGs the exponent equals $\zeta_{\mathrm{HRG}}:=(3-\tau)/2$. Theorem~\ref{thm:second-largest} thus includes the known bound $|\CC_n^\sss{(2)}|=\Theta_\Prob((\log n)^{2/(3-\tau)})$ in threshold hyperbolic random graphs from~\cite{kiwimitsche2ndlargest}. Due to the threshold profile and the underlying  one-dimensional space, in these graphs all small components are localized.   
In contrast, Theorem~\ref{thm:second-largest} of this paper allows for any $\alpha\in(1,\infty]$ and any dimension  $d\in\N$. When $\alpha<\infty$ or $d>1$, de-localized small components may be present, and different proof methods are required for both the lower bound (variational problem, see Section~\ref{sec:four-regimes}) and the upper bound (cover expansion; preventing small-to-large merging, see page~\pageref{small-to-large} below). 
\smallskip 

\emph{Large deviations for the giant.\ \ }
The lower tail of large deviations for the size of the largest cluster in supercritical Bernoulli percolation on $\Z^d$ and random geometric graphs has been studied in~\cite{penrose1996large,pisztora1996surface}, proving  $\Prob(|\CC_n^\sss{(1)}|/n<\rho)=\exp(-\Theta(k^{(d-1)/d}))$ for any $\rho<\theta:=\Prob(0\leftrightarrow\infty)$. For models with long edges, the works~\cite{biskup2004scaling,girgGiantBlasius} prove ---for sufficiently small $\rho>0$--- the upper bounds $\Prob(|\CC_n^\sss{(1)}|/n<\rho)\le \exp(-\Theta(k^{ \zeta}))$ with $\zeta=2-\alpha-o(1)$ for long-range percolation \cite{biskup2004scaling} and $\zeta=\zeta_{\mathrm{HRG}}=(3-\tau)/2$ for hyperbolic random graphs \cite{girgGiantBlasius} using renormalization techniques.   
 Theorem~\ref{prop:lower-dev-giant} here gives the \emph{lower bound} for the same event for models in the kernel-based spatial random graph class in general,  complementing these previous results, and making use of the connection to the cluster-size decay.
In the follow-up paper~\cite{clusterII}, we combine the methods here with renormalization techniques to prove the upper bound $\Prob(|\CC_n^\sss{(1)}|/n<\rho )\le\exp(-\Theta(n^{\zeta_\star}))$ for any $\rho<\theta$ for KSRGs with $\zeta_\mathrm{long}>0$. This gives matching upper and lower bounds outside the phase transition boundaries of $\zeta_\star$, i.e., whenever $\mathfrak{m}_\star=1$ in \eqref{eq:lower-dev-giant}. The upper tail of large deviations behaves differently: for $\rho\in(\theta,1)$, \cite{clusterII}  proves that $\Prob(|\CC_n^\sss{(1)}|>\rho n)$ decays polynomially when $\tau<\infty$ and exponentially when $\tau=\infty$. 

\subsection{Organization of the paper}
In Section~\ref{sec:methodology} we give an elaborate overview of the proofs. Section~\ref{sec:cover-expansion} introduces the cover expansion, our main novel technical contribution, required for  the upper bound on $|\CC_n^\sss{(2)}|$ in Section~\ref{sec:upper-2nd}. Only Section~\ref{sec:upper-2nd} restricts to models with $\zeta_{\mathrm{hh}}>0$. Section~\ref{sec:upper-sub} connects the finite-volume bounds ($|\CC_n^{\sss{(2)}}|$) with the cluster-size decay in the infinite model, leading to the LLN of $|\CC_n^\sss{(1)}|$ (Corollary \ref{cor:lln}). Section~\ref{sec:lower} proves the lower bounds on $|\CC_n^\sss{(2)}|$ and on the cluster-size decay, and Theorem~\ref{prop:lower-dev-giant}.
Sections \ref{sec:upper-2nd}--\ref{sec:lower} start with a proposition each. Together, these imply Theorems \ref{thm:subexponential-decay} and \ref{thm:second-largest}, as verified in Section~\ref{sec:main-proofs}.

We only give proofs for KSRGs using the kernel $\kappa_\sigma$ in~\eqref{eq:kernels} with $\sigma\ge 0$: we restrict to
 \begin{equation}
  \mathrm{p}\big((x_u, w_u), (x_v, w_v)\big):=
  \begin{dcases}
  p\bigg(1\wedge \beta\frac{(w_1\vee w_2)(w_1\wedge w_2)^\sigma}{\|x_v-x_u\|^d}\bigg)^\alpha, & \text{if }\alpha<\infty, \\
  p\mathbbm 1\left\{\beta\frac{(w_1\vee w_2)(w_1\wedge w_2)^\sigma}{\|x_v-x_u\|^d} \ge 1\right\},          & \text{if }\alpha=\infty.
  \end{dcases}
  \label{eq:connection-prob-gen}
 \end{equation}
The proofs for $\kappa_\mathrm{sum}$ can be directly obtained by using the bound $\kappa_0\le \kappa_\mathrm{sum}$ in lower bound estimates and $\kappa_{\mathrm{sum}}\le 2^d\kappa_0$ in upper bound estimates. Further, we only give proofs for models with a Poisson point process as vertex set. The extension to $\Z^d$, when applicable, follows generally from replacing concentration bounds for Poisson random variables by concentration for sums of independent Bernoulli random variables, and by replacing integrals by summations. When more adaptations are required, we comment on those.\smallskip 

\emph{Notation.\ \ }
We write write $|\CS|$ for the size of a \emph{discrete} set $\CS$. We write $\mathrm{Vol}(\CK)$ for the Lebesgue measure of a set $\CK\subseteq \R^d$, $\partial\CK$ for its boundary and $\CK^\complement:=\R^d\setminus \CK$ for its complement. We denote the complement of an event $\mathcal A$ by $\neg \mathcal A$.
Formally we define a vertex $v$ by a pair of location and mark, i.e., $v:=(x_v, w_v)$, but we will sometimes write $v\in \CK$ if $x_v\in \CK$.
For two vertices $u,v$, we write $u\sim v$ if $u$ is connected by an edge to $v$ in the graph under consideration (typically $\CG_\infty$), and $u\not\sim v$ otherwise.  We also write $\{u,v\}$ for the same (undirected) edge. For a set of vertices $\CS$, we write $u\sim \CS$ if there exists $v\in\CS$ such that $u\sim v$.  We write $X\succcurlyeq Y$ if the random variable $X$ stochastically dominates the random variable $Y$, i.e., $\Prob(X\ge x)\ge \Prob(Y\ge x)$ for all $x\in \R$. A random graph $\CG_1=(\CV_1,\CE_1)$ stochastically dominates a random graph $\CG_2=(\CV_2, \CE_2)$ if there exists a coupling such that $\Prob\big(\CV_2\subseteq\CV_1, \CE_2\subseteq\CE_1\big)=1$.
For $x\in\R^d$ and $s\ge 0$, and $\CQ\subseteq\R^d$, $a\le b$,  we introduce notations for boxes of volume $s$ centered at $x$, and vertex sets restricted to locations in $\CQ$ with mark in $[a,b)$:
\begin{equation}\label{eq:xi-q-ab}
\begin{aligned}
\Lambda(x,s):=\Lambda_s(x)&:=x+[-s/2^{1/d}, s/2^{1/d}]^d,&\Lambda_s&:=\Lambda_s(0), \\
 \CV_\CQ[a,b)&:= \CV\cap \big(\CQ\times[a,b)\big), & \CV_s[a,b)&:= \CV_{\Lambda_s}[a,b).
 \end{aligned}
\end{equation}
Lastly, we define 
\begin{equation}\label{eq:gn-ab}
    \CG_n[a,b):=\text{the subgraph of $\CG$ induced by vertices in }\CV_n[a,b).
\end{equation}

\section{Methodology}\label{sec:methodology}
We first sketch the strategy for upper bound on the size of the second-largest component, then we explain how to obtain the cluster-size decay from it, and lastly we sketch the lower bound.
Throughout the outline we assume that $(n/k)^{1/d}\in\mathbb{N}$.
\subsection{Second-largest component}\label{sec:outline-hh}
We aim to show an upper bound of the form
\begin{equation}\label{eq:outline-error-prob}
 \Prob\big(|\CC_n^\sss{(2)}|> k\big) \le (n/k)\exp(-ck^{\zeta_\mathrm{hh}})=:\mathrm{err}_{n,k}.
\end{equation}
for arbitrary values of $n\ge k$ and some constant $c>0$. Such a bound yields \eqref{eq:second-largest-main-upper} when substituting $k=A(\log n)^{1/\zeta_\mathrm{hh}}$ for a sufficiently large constant $A=A(\delta)>0$. 
The proof consists of four revealment stages,  illustrated in Figure~\ref{fig:upper-bound}.\smallskip 

 \noindent {\bf Step 1. Building a backbone.}
We set $w_\mathrm{hh}:=\Theta(k^{\gamma_\mathrm{hh}})$ with $\gamma_\mathrm{hh}>0$ from \eqref{eq:gamma-hh}.
We partition the volume-$n$ box $\Lambda_{n}$ into $n/k$ smaller sub-boxes of volume $k$.
In this first revealment step we only reveal the location and edges between vertices in $\CV_{n}[w_\mathrm{hh},2w_\mathrm{hh})$, obtaining the graph $\CG_{n,1}:=\CG_n[w_\mathrm{hh},2w_\mathrm{hh})$. We show that $\CG_{n,1}$ contains a connected component $\CC_{\mathrm{bb}}$ that contains $\Theta(k^{\zeta_{\mathrm{hh}}})$ many vertices in each subbox, that we call \emph{backbone} vertices. We show that this event -- say $\CA_{\mathrm{bb}}$ -- has probability at least $1-\mathrm{err}_{n,k}.$
We do this by ordering the subboxes so that subboxes with consecutive indices share a $(d-1)$-dimensional face, and by iteratively connecting $\Theta(k^{\zeta_{\mathrm{hh}}})$ many vertices in the next subbox to the component we already built, combined with a union bound over all subboxes.  The event $\CA_{\mathrm{bb}}$ ensures us to show that independently for all $v\in \CV_n[2w_\mathrm{hh},\infty)$, regardless of their locations,
\begin{equation}
 \Prob\big(v\sim \CC_{\mathrm{bb}}\mid \CA_{\mathrm{bb}}, v\in\CV_{n}[2w_\mathrm{hh},\infty)\big)\ge1/2.\label{eq:q-bound}
\end{equation}
 We call vertices in $ \CV_n[2w_\mathrm{hh},\infty)$ \emph{connector} vertices. If $\alpha<\infty$, not all connector vertices will connect to the backbone, i.e., the $1/2$ in \eqref{eq:q-bound} cannot be improved to $1$. \smallskip 

 \noindent{\bf Step 2: Revealing low-mark vertices.}
We now also reveal all vertices with mark in $[1,w_\mathrm{hh})$, and all their incident edges to $\CG_{n,1}$ and towards each other, i.e., the graph $\CG_{n,2}:=\CG_{n}[1,2w_\mathrm{hh})\supseteq \CG_{n,1}$. \smallskip 

\begin{figure}[t]
 \includegraphics[width=406pt]{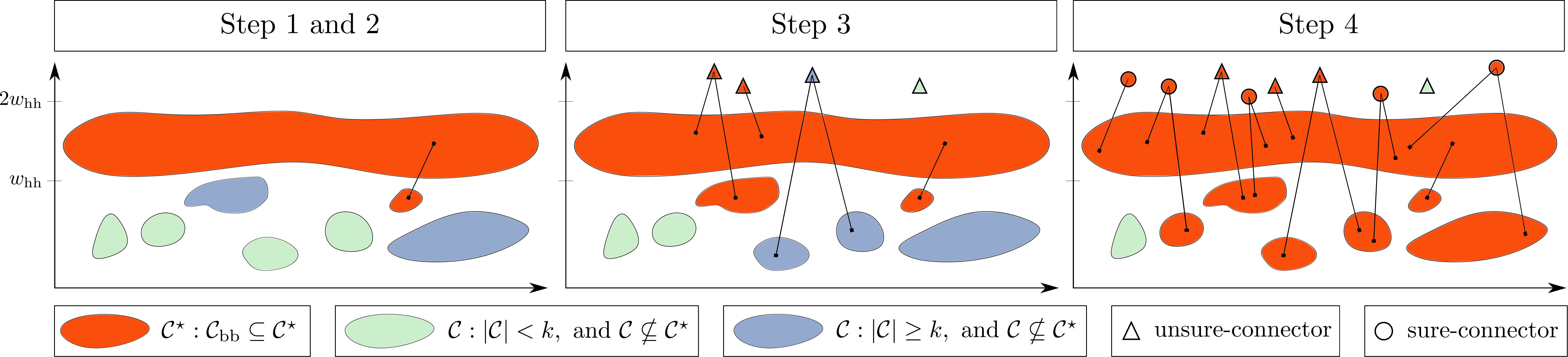}
 \captionsetup{width=406pt}
 \captionsetup{font={footnotesize}} 
 \caption{Upper bound. The $y$-axis represents marks, the $x$-axis represents space. After Steps 1 and 2 there is a component $\CC^\star$ containing the backbone that is connected to some small components from $\CG_n[1,w_\mathrm{hh})$. After Step 3, the unsure-connectors are revealed: there is small-to-large merging; some unsure-connectors connect to the backbone. After Step 4, each component of size at least $k$ merged with the largest component via a sure-connector; unmerged small components and unsure-connectors outside $\CC_n^\sss{(1)}$ remain all of size at most $k$.}\label{fig:upper-bound}
\end{figure}

 \noindent{\bf Step 3: Pre-sampling randomness to avoid merging of smaller components.}
 \phantomsection
 \label{small-to-large}
To show \eqref{eq:outline-error-prob}, in the fourth revealment stage below we must \emph{avoid} \emph{small-to-large merging}: when the edges to/from some $v\in \CV_n[2w_{\mathrm{hh}},\infty)$ are revealed,
a set of small components, each of size at most $k$, could merge into a component of size at least $k$ without connecting to the giant component.
If we simply revealed $\CV_n[2w_{\mathrm{hh}}, \infty)$ after Step 2, \eqref{eq:q-bound} would not be sufficient to show that small-to-large merging occurs with probability at most $1-\mathrm{err}_{n,k}$. So, we pre-sample randomness: we split $\CV_n[2w_{\mathrm{hh}}, \infty)$ into two PPPs:
\begin{equation}
 \CV_{n}[2w_\mathrm{hh},\infty)=\CV_{n}^\sss{(\mathrm{sure})}[2w_\mathrm{hh},\infty) \cup\CV_{n}^\sss{(\mathrm{unsure})}[2w_\mathrm{hh},\infty),\nonumber
\end{equation} where $\CV_{n}^\sss{(\mathrm{sure})}[2w_\mathrm{hh},\infty), \CV_{n}^\sss{(\mathrm{unsure})}[2w_\mathrm{hh},\infty) $ are \emph{independent} PPPs with equal intensity: using \eqref{eq:q-bound} and helping random variables that encode the presence of edges, we pre-sample whether a connector vertex connects for sure to $\CC_{\mathrm{bb}}$ by at least one edge;
forming $\CV_{n}^\sss{(\mathrm{sure})}[2w_\mathrm{hh},\infty)$. Vertices in $\CV_{n}^\sss{(\mathrm{unsure})}[2w_\mathrm{hh},\infty)$
might still connect to $\CC_{\mathrm{bb}}$ since $1/2$ is only a lower bound in \eqref{eq:q-bound}, but we ignore that information. We crucially use the property that thinning a PPP yields two independent PPPs. The adaptation of our technique to lattices as vertex set seems non-trivial due to this step.
We reveal now $\CV_{n}^\sss{(\mathrm{unsure})}[2w_\mathrm{hh},\infty)$. Let $\CG_{n,3}\supseteq \CG_{n,2}$ be the graph induced on the vertex set
\begin{equation}
 \CV_{n,3}:=\CV_{n}[1,2w_\mathrm{hh})\cup\CV_{n}^\sss{(\mathrm{unsure})}[2w_\mathrm{hh},\infty). \label{eq:outline-v3}
\end{equation} 

\noindent{\bf Step 4: Cover expansion, a volume-based argument.}
We now reveal $\CV_n^{\sss{(\mathrm{sure})}}[2w_\mathrm{hh},\infty)$ and merge all components of size at least $k$ with the largest component in $\CG_{n,3}$ with probability at least $1-\mathrm{err}_{n,k}$.
Small-to-large merging cannot happen since vertices in  $\CV_n^\sss{(\mathrm{sure})}[2w_\mathrm{hh},\infty)$ all connect to $\CC_{\mathrm{bb}}$. We argue how to obtain \eqref{eq:outline-error-prob}.\smallskip

\noindent{\bf Step 4a: Not too dense components via proper cover.}
For a component $\CC\subseteq\CG_{n,3}$, the \emph{proper cover} $\CK_n(\CC)\subseteq \Lambda_n$ is the union of volume-$1$ boxes centered at the vertices of $\CC$ (the formal definition below is slightly different).
Fixing a constant $\delta>0$, we call $\CC$ not too dense if
\begin{equation}
 \mathrm{Vol}(\CK_n(\CC))\ge \delta|\CC|.\label{eq:linear-volume}
\end{equation}
Using the connectivity function $\mathrm{p}$ in \eqref{eq:connection-prob-gen}, each pair of vertices within constant distance is connected by an edge with constant probability. Since $w_\mathrm{hh}=\Theta(k^{\gamma_\mathrm{hh}})$, there exists $k_0$ such that for any $k\ge k_0$  and any pair of vertices $u\in\CV_{n,3}$ in \eqref{eq:outline-v3} and $v\in\CV_n^\sss{(\mathrm{sure)}}[2w_\mathrm{hh},\infty)$ within the same volume-1 box,
\begin{equation}
 \mathrm{p}(u, v )\ge p/2.\label{eq:low-mark-bound}
\end{equation}
Using this bound and that $\CV_n^{\sss{\mathrm{(sure)}}}[2w_\mathrm{hh},\infty)$ is a PPP, when $|\CC|> k$, with  probability at least $1-\mathrm{err}_{n,k}$, at least $\Theta(k^{\zeta_{\mathrm{hh}}})$ many vertices of $\CV_n^{\sss{\mathrm{(sure)}}}[2w_\mathrm{hh},\infty)$ fall inside $\CK_n(\CC)$ and at least one of them connects to $\CC$ by an edge. Since these vertices belong to $\CV_n^{\sss{\mathrm{(sure)}}}[2w_\mathrm{hh},\infty)$, they connect to $\CC_{\mathrm{bb}}$ by construction, merging $\CC$ with the component containing $\CC_{\mathrm{bb}}$.\smallskip

 \noindent{\bf Step 4b: Too dense components via cover expansion.}
We still need to handle components $\CC\subseteq \CG_{n,3}$ with $|\CC|> k$ that do not satisfy \eqref{eq:linear-volume}. These may exist (outside the component of $\CC_{\mathrm{bb}}$) since the PPP $\CV_n$ contains dense areas, e.g.,\ volume-one balls with $\Theta((\log n)/\log\log n)$ vertices.
We introduce a deterministic algorithm which works for any vertex set provided that there are no `large' areas containing atypically many vertices. The definition of `large' depends on $w_\mathrm{hh}=w_\mathrm{hh}(\sigma, \tau)$; a homogeneous Poisson point process satisfies this property with probability at least $1-\mathrm{err}_{n,k}$ as long as $\tau\ge\sigma+1$. When $\tau<\sigma+1$, it is at this step that we obtain a slightly worse error bound.

\emph{The cover-expansion algorithm} outputs for any (deterministic) set $\CL$ of at least $k$ vertices a set $\CK^{\mathrm{exp}}(\CL)\subset\R^d$, called the \emph{cover-expansion} of $\CL$, that satisfies bounds similar to
\eqref{eq:linear-volume} and \eqref{eq:low-mark-bound}.
In the design of the set $\CK^{\mathrm{exp}}(\CL)$ we quantify how far a connector vertex may fall from a too dense subset $\CL'\subseteq \CL$, while still ensuring connection probability at least $p/2$ to the \emph{set} $\CL'$. 
 We apply this algorithm with $\CL=\CC$ for components of size at least $k$ of $\CG_{n,3}$ that do not satisfy \eqref{eq:linear-volume} and do not contain $\CC_{\mathrm{bb}}$. The remainder of the proof is identical to Step 4a.
Steps 4a, 4b, and a union bound over all components of size at least $k$ in $\CG_{n,3}$ yield \eqref{eq:outline-error-prob}.

\subsection{Subexponential decay, upper bound} \label{sec:second-sub}
Consider $k$ fixed. We obtain the cluster-size decay \eqref{eq:subexponential-main-upper} for any $n\in [k, n_k]$ with $n_k=\exp(\Theta(k^{\zeta_{\mathrm{hh}}}))$ by substituting $n_k$ into \eqref{eq:outline-error-prob}. To extend it to larger $n$, we first identify the lowest mark $\overline w(n)$ such that \emph{all} vertices with mark at least $\overline w(n)$ belong to the giant component $\CC_{n}^\sss{(1)}\subseteq \CG_n$ with sufficiently high probability (in $n$).
Then we embed  $\Lambda_{n_k}$ in $\Lambda_n$ and show that
\begin{equation}
 \begin{aligned}
  \Prob^\sss{0}\big(&|\CC_{n}(0)|> k, \;  0\not\in \CC_{n}^\sss{(1)}\big)                                                                                  \\
                                  & \le  \Prob\big(|\CC^\sss{(2)}_{n_k}|> k\big) + \Prob\big(\CC^\sss{(1)}_{n_k}\nsubseteq\CC^\sss{(1)}_{n}\big) +
  \Prob^\sss{0}\big(|\CC_{n}(0)|> k, \; 0\not\in \CC_{n}^\sss{(1)}, |\CC_{n_k}(0)|< k\big).
 \end{aligned}    \label{eq:outline-upper-second-sub-last}
\end{equation}
The first term on the right-hand side has the right error bound by \eqref{eq:outline-error-prob}.
We relate the second term to the event that for some $\tilde{n}\in(n_k, n]$ there is no polynomially-sized largest component or the  second-largest component is too large.
The event in the third term implies that one of the at most $k$ vertices in $\CC_{n_k}(0)$  has an edge of length $\Omega(n_k^{1/d}/k)$, which will have probability at most $\mathrm{err}_{n_k,k}$, since these vertices have mark at most $\overline w(n_k)$.
\subsection{Lower bound}\label{sec:lower-method}
For the subexponential decay, we compute the probability of a specific event satisfying $k\le |\CC(0)|<\infty$. We draw a ball $\CB$ of volume $\Theta(k)$ around the origin, and compute an \emph{optimally suppressed mark-profile}: the PPP $\CV$ must fall below a $(d+1)$-dimensional mark-surface $\CM:=\{(x,f(x)), x\in \R^d\}$, i.e., $w_v\le f(x_v)$ must hold for all $(x_v, w_v)\in \CV$. We write $\{\CV\le \CM\}$ for this event. The value of $f(x)$ is increasing in $\|x - \partial\CB\|$ since high-mark vertices close to $\partial\CB$ are most likely to have edges crossing $\partial \CB$. $\CM$ is optimized so that $\Prob(\CV \le \CM)\sim\Prob(\CB\not\sim \CB^\complement \mid \CV \le \CM)$, where $\{\CB\not\sim \CB^\complement\}$ is the event that there is no edge present between vertices in $\CB$ and those in its complement. Both events occur with probability $\exp(-\Theta(k^{\zeta_\star}))$, (up to logarithmic correction factors in the exponent on phase boundaries of $\zeta_\star$).
We then find an isolated component of size at least $k$ inside $\CB$ using a technique that works when $\zeta_\mathrm{hh}>0$.
We use a boxing scheme to extend this argument to the lower bound on $|\CC_n^{\sss{(2)}}|$, similar to \cite{kiwimitsche2ndlargest}. We use another boxing argument to bound $\Prob(|\CC_n^\sss{(1)}|<\rho n)$ from below. 

\section{The cover-expansion algorithm}\label{sec:cover-expansion}
The goal of this section is to develop the cover-expansion technique in Step 4b of Section~\ref{sec:outline-hh}.  The statements apply also to KSRGs on vertex sets other than a PPP.
First we define a desired property for a set of vertices based on their spatial locations.
Recall the definition $\Lambda_{s}(x)=\Lambda(x,s)$ from \eqref{eq:xi-q-ab}. Throughout this section, we often identify vertices with their \emph{locations} and ignore their marks. Slightly abusing notation, when $\CL$ is a set of location of vertices we write $v\sim \CL$ for $v$ having an edge to the corresponding set of vertices.
\begin{definition}[$s$-expandable point-set]\label{def:expandable}
 Let $\CS\subset\R^d$ be a discrete set of points in $\R^d$, and $s>0$. We call $\CS$ $s$-expandable if for all $x\in \Z^d$ and all  $s'\ge s$,
 \[
  |\CS\cap \Lambda_{s'}(x)|/s'\le \re.
 \]
\end{definition}
A discrete set $\CS\subset \R^d$ is $s$-expandable if there are no large boxes with too high ratio of number of vertices in $\CS$ in the box compared to its volume. In particular, the definition enforces $|\CS\cap\Lambda_n|\le \re n$. Moreover, if $\CS$ is $s$-expandable, then any subset of $\CS$ is $s$-expandable; lastly, if $\CS$ is $s$-expandable, then $\CS$ is also $\tilde s$-expandable for any $\tilde s\ge s$.
The next proposition solves the problem of too dense components in space, cf.\ \eqref{eq:linear-volume}. 
\begin{proposition}[Covers and expansions for $s$-expandable sets]\label{prop:cover-expansion-deterministic} Consider a KSRG in dimension $d\ge 1$ satisfying Assumption \ref{assumption:main} on a (arbitrary) marked vertex set $\CV=\{(x_v,w_v)\}_{v\in \CV}$. For a given $\underline w>(2^dd^{d/2}/\beta\vee 1)$, define  $s(\underline{w})>0$ as
 \begin{equation}
  s(\underline w):= (2^d \beta \underline w)^{1/(1-1/\alpha)}.
  \label{eq:prop-cover-expansion-min-weight}
 \end{equation}
 Given $n$, assume that $s(\underline w)\le n$, and $\CL\subseteq \Lambda_n$ is the set of locations of any $s(\underline w)$-expandable set of vertices. Then there is a set $\CK_n(\CL) \subseteq\Lambda_n$ with
 \begin{equation}
  \mathrm{Vol}(\CK_n(\CL))\ge\frac{1}{2^{4d+1}\re d^{d/2}}|\CL|,\label{eq:prop-cover-expansion-min-volume}
\end{equation}
such that any $v\in \CK_n(\CL) \times[\underline w,\infty)$ connects by an edge to $\{(x_u, w_u): x_u\in\CL\}$ independently with probability at least $p/2$, i.e.,
\begin{equation}\label{eq:p-connection}
 \Prob\Big(v\sim \CL \, \big| \, \{(x_u, w_u): x_u\in\CL\}\cup \{v\}\subseteq \CV\Big) \ge p/2.
 \end{equation}
\end{proposition}
The independence here means that the connection to $\CL$ of any set of vertices in $ \CK_n(\CL)\times [\underline w, \infty)$ dominates independent Bernoulli random variables with success probability $p/2$, regardless of the marks of vertices in $\CL$, and the exact location and mark of $v$, as long as it belongs to the set $\CK_n(\CL) \times[\underline w,\infty)$.
We use two  constructions for the set $\CK_n(\CL)$. If $\CL$ is not too dense (see  \eqref{eq:linear-volume}), we will use a \emph{proper cover} (see Definition~\ref{def:cover-proper} below). If, however, the points of $\CL$ are densely concentrated in small areas, we will use a new (deterministic) algorithm, the \emph{cover-expansion algorithm},  producing an \emph{expanded cover} (see Definition~\ref{def:cover-expansion} below) that still satisfies the connection probability in \eqref{eq:p-connection}. This will prove Proposition~\ref{prop:cover-expansion-deterministic}.
We start with some preliminaries.

\begin{definition}[Cells in a volume-$n$ box]\label{def:cells-general} Let $\widetilde B_z$ be a box of volume $1$ centered around $z\in \Z^d$.  For any two neighboring boxes $\widetilde B_z, \widetilde B_{z'}$, allocate the shared boundary $\partial \widetilde B_z\cap \partial \widetilde  B_{z'}$ to precisely one of the boxes (in an arbitrary but fixed way).
 For each $u\in \Z^d$ such that $u\notin \Lambda_n$ but  $\widetilde B_u\cap \Lambda_n \neq \emptyset$, let  $z(u):=\argmin\{\|u-z\|:z\in \Lambda_n\cap \Z^d\}$, and then define for each $z\in \Z^d\cap \Lambda_n$ the \emph{cell of} $z$ as
 \[ B_{z}:=\Big(\widetilde B_z\cap \Lambda_n \Big) \cup \Big(\bigcup_{u\in \Z^d: z(u)=z} \big(\widetilde B_u\cap \Lambda_n\big)\Big). \]
\end{definition}
In words, boxes that have their center inside $\Lambda_n$ but are not fully contained in $\Lambda_n$ are truncated, while boxes that have their centers outside $\Lambda_n$ but intersect $\Lambda_n$ are merged with the closest box with center inside $\Lambda_n$. At every point of $\Lambda_n$ at most $2^d$ cells are merged together, and only $1/2$ of the radius in each coordinate can be truncated. Thus, for each cell~$B_z$,
\begin{equation}
 \sup\{\|x-y\|: x, y \in B_z\}\le2\sqrt{d}; \quad \mbox{and} \quad \quad 2^{-d}\le \mathrm{Vol}(B_z)\le 2^d \label{eq:max-cell-dist}
 .\end{equation}

\begin{definition}[Notation for cells containing vertices]\label{def:cell-zi}
Let $\CL\subseteq\Lambda_n$ be (a subset of) the locations of the vertex realization $\CV$.
 Let $\{B_{z_i}\}_{i=1}^{m'}$ be the cells with $\CL\cap B_{z_i}\neq \emptyset$.
 Let $\CL_i:=\CL \cap B_{z_i}$, $\ell_i:=|\CL_i|$ and $L:=|\CL|=\sum_{i=1}^{m'}\ell_i$.
\end{definition}

We will distinguish two cases for the arrangement of the vertices among the cells: either the number of cells is linear in the number of vertices, or there is a positive fraction of all cells  that all contain `many' vertices.
The next combinatorial claim makes this precise.
\begin{claim}[Pigeon-hole principle for cells]\label{lemma:pigeonhole}
 Let $\delta\in(0,1)$, $\nu\ge1$, and $\ell_1,\dots,\ell_{m'}\ge 1$ integers such that $\sum_{i\le m'}\ell_i=L$.
 If $m'< L(1-\delta)/\nu$
 then
 \begin{equation}
  \exists\   \CI\subseteq[m']: \quad  \forall i\in\CI: \ell_i\ge \nu, \qquad \text{ and }\qquad \sum_{i\in\CI}\ell_i\ge \delta L.\label{eq:pigeonhole-1}
 \end{equation}
\end{claim}
\begin{proof}
 Assume by contradiction that $\delta, \nu, \ell_1,\dots,\ell_{m'}$ are such that $m'< L(1-\delta)/\nu$ holds but \eqref{eq:pigeonhole-1} does not hold.
 Let $\CJ:=\{j: \ell_j<\nu\} \subseteq[m']$ and let $\CJ^\complement:=[m']\backslash\CJ$. Then $\forall i\in \CJ^\complement: \ell_i\ge \nu$ and hence $\sum_{j \in \CJ^\complement} \ell_j < \delta L$, as we assumed the opposite of \eqref{eq:pigeonhole-1}. Since the total sum is $L$, this  implies that
 $\sum_{j\in\CJ}\ell_j\ge  (1-\delta)L$.
 Moreover, since $\ell_j< \nu$ for $j\in \CJ$, it must hold that $m'\ge|\CJ|\ge (1-\delta)L/\nu$, which then contradicts that $m'< L(1-\delta)/\nu$.
\end{proof}

We define the first possibility for the set $\CK_n(\CL)$, which is inspired by Claim \ref{lemma:pigeonhole} with $\nu=\re d^{d/2}2^{3d}$ and $\delta=1/2$.
\begin{definition}[Proper cover]\label{def:cover-proper}
 We say that $\CL$ admits a \emph{proper cover} if
 $
  m'\ge |\CL|/(2\re d^{d/2}2^{3d})
 $ in Definition~\ref{def:cell-zi},
 and  we define the cover of $\CL$ as
 \begin{equation}\nonumber
  \CK_n^{(\mathrm{prop.})}(\CL) := \bigcup_{i\in[m'  ]}B_{z_i},\quad  \mbox{ satisfying }\quad \mathrm{Vol}(\CK_n^\sss{(\mathrm{prop.})})\ge \frac{1}{d^{d/2}\re 2^{4d+1}} |\CL|.
 \end{equation}
\end{definition}
By \eqref{eq:max-cell-dist}, $\nu=\re d^{d/2}2^{3d}$, and $\delta=1/2$, hence, we obtain the desired volume bound on the right-hand side above,
establishing \eqref{eq:prop-cover-expansion-min-volume} for sets  admitting a proper cover.
Moreover, consider now $(x_v, w_v)\in (B_{z_i}\cap \CL)\times[1,\infty)$ and $u:=(x_u,w_u)\in B_{z_i}\times[\underline{w},\infty)$ with $B_{z_i}\subseteq \CK_n^\sss{(\mathrm{prop.})}$. Then $\|x_u-x_v\|\le 2\sqrt{d}$ by \eqref{eq:max-cell-dist}. Since we assumed
$\underline{w}\ge (2^d d^{d/2}/\beta\vee 1)$ above \eqref{eq:prop-cover-expansion-min-weight}, using \eqref{eq:connection-prob-gen} and \eqref{eq:kernels},
\begin{equation}\label{eq:p-for-proper}
 \mathrm{p}(u,v) \ge p\big(1\wedge(\beta\kappa_\sigma(\underline{w},1)/(2\sqrt{d})^{d})\big)^{\alpha}
 \ge p\Big(1\wedge \big((2^d d^{d/2}/\beta\vee 1)\beta/(2\sqrt{d})^{d}\big)\Big)^{\alpha}
 \ge p.
\end{equation}
This shows \eqref{eq:p-connection} for sets admitting a proper cover. The argument for $\alpha=\infty$ is similar.

In what follows we treat sets $\CL$ that do not admit a proper cover, i.e., when $\CL$ is contained in too few cells. We define an ``expanded'' cover, which we obtain after applying a suitable volume-increasing procedure ---the cover expansion algorithm--- to $\cup_{i\in[m']} B_{z_i}$ that we  explain at the end of the section.

\subsection{Cover expansion}
In this section we assume that $\CL$ does not admit a proper cover. By Claim~\ref{lemma:pigeonhole}, and re-indexing cells in Definition~\ref{def:cell-zi}, without loss of generality we may assume that $\CI=[m]\subseteq[m']$ satisfies \eqref{eq:pigeonhole-1} with $\nu=\re d^{d/2}2^{3d}$ and $\delta=1/2$.
We use $\Lambda(x, s)$ in \eqref{eq:xi-q-ab} here for the box of volume $s$ centered at $x\in \R^d$.
\begin{definition}[Cover expansion]\label{def:cover-expansion}
 Let $\CL$ be a set of vertex locations that does not admit a proper cover as in Definitions~\ref{def:cell-zi} and \ref{def:cover-proper}. Let $[m]:=\{j: \ell_j \ge \re d^{d/2}2^{3d}\}\subseteq [m']$ satisfy \eqref{eq:pigeonhole-1} with $\nu=\re d^{d/2}2^{3d}$ and $\delta=1/2$.
 The \emph{cover allocation} is defined as a subset of labels $\CJ^{\sss{(\star)}}\subseteq [m]$  and corresponding boxes $(B_j^\sss{(\star)})_{j\in \CJ^{\sss{(\star)}}}\subset\R^d$, centered at $(z_j)_{j\in \CJ^{\sss{(\star)}}}$, together with an allocation ${\buildrel \star \over \mapsto}$ of the cells $B_{z_i}: i\le m$ to these boxes, with
 \begin{equation}
  \mathrm{Cells}_j^\sss{(\star)}:=\bigcup_{i\le m}\big\{i: B_{z_i}\ {\buildrel \star \over \mapsto}\ B_j^\sss{(\star)}\big\}, \label{eq:assigned-cells-star}
 \end{equation}
 satisfying the following properties:
 \begin{itemize}
  \item[(disj.)] the boxes $(B_j^\sss{(\star)})_{j\in \CJ^{\sss{(\star)}}}$ are pairwise disjoint sets in $\R^d$;
  \item[(vol.)]   for all $j\in \CJ^{\sss{(\star)}}$; we have
        \begin{equation}
         B_j^\sss{(\star)}=\Lambda\Big(z_j, \frac{1}{\re d^{d/2}2^{3d}}\sum_{i\in\mathrm{Cells}_j^\sss{(\star)}} \ell_i\Big); \quad  \mathrm{Vol}(B_j^\sss{(\star)})=\frac{1}{\re d^{d/2}2^{3d}}\sum_{i\in\mathrm{Cells}_j^\sss{(\star)}} \ell_i\label{eq:merging-volume};
        \end{equation}
  \item[(near)] for each $i\in[m]$ ; $i\in\mathrm{Cells}_j^\sss{(\star)}$
        \begin{equation}
         \|z_i - z_j\|^d \le
         d^{d/2}\mathrm{Vol}(B_j^\sss{(\star)}).\label{eq:merging-distance}
        \end{equation}
 \end{itemize}
 We  call $B_j^\sss{(\star)}$ the expanded boxes, and define  the \emph{cover expansion} of $\CL$ as
 \begin{equation}\label{eq:expanded-cover-def}
  \CK_n^\sss{(\mathrm{exp})}(\CL):=\Lambda_n\cap \Big(\bigcup_{j\in\CJ^{\sss{(\star)}}}B_j^\sss{(\star)} \Big).
 \end{equation}
\end{definition}
We make a few comments about Definition~\ref{def:cover-expansion}: (disj) and (vol) together ensure that the total volume of the expanded cover is proportional to $|\CL|$. Further, (vol) ensures that $\mathrm{Vol}(B_j^\sss{(\star)})$ is proportional to the number of vertices that are in cells allocated to $B_j^\sss{(\star)}$.
Finally, (near) ensures that the center $z_i$ of each cell $B_{z_i}$ is relatively close to the center of the box to which it is allocated. The distance between the center of the allocated cells and the center of $B_j^{\sss{(\star)}}$ is at most $\sqrt{d}$ times the side-length of the box $B_j^{\sss{(\star)}}$. In particular if  $B_j^{\sss{(\star)}}$ contains many vertices of $\CL$ and it is thus large, this distance can be also large.

\begin{proposition}[Every set has either a proper cover or a cover expansion]\label{prop:cover-expansion-correctness}
 Assume $\CL$ does not admit a proper cover defined in Definition~\ref{def:cover-proper}. Then there exists a cover expansion of $\CL$ in the sense of Definition~
  \ref{def:cover-expansion}.
 Further, if $\CL$ is $n$-expandable then the total volume of the cover expansion of $\CL$ is linear in $|\CL|$, i.e.,
 \begin{equation}\label{eq:linear-volume-expansion}
  \mathrm{Vol}(\CK_n^\sss{(\mathrm{exp})}(\CL))
  \ge
  \frac{1}{2^{4d+1}\re d^{d/2}}|\CL|.
 \end{equation}
\end{proposition} 
We defer the proof of existence of $\CK_n^\sss{(\mathrm{exp})}(\CL)$ to the end of the section. Assuming that a cover expansion exists, we show now a few important properties. After that, we show how Proposition~\ref{prop:cover-expansion-deterministic} follows from Proposition~\ref{prop:cover-expansion-correctness}.

\begin{observation}[Cover-expansion properties]\label{obs:cover-expansion}
 Consider the cover expansion of a set $\CL$  that does not admit a proper cover according to Definition~\ref{def:cover-proper}.
 \begin{itemize}
  \setlength\itemsep{-0.3em}
  \item[(i)] Every expanded box has volume at least $1$, i.e., for all $j\in \CJ^{\sss{(\star)}}$, $\mathrm{Vol}(B_j^\sss{(\star)})\ge1$.
  \item[(ii)] For any cell with  $B_{z_i}\ {\buildrel \star \over \mapsto}\  B_j^\sss{(\star)}$,
        \begin{equation}
         \sup \big\{\|x_u-x_v\|: x_u\in\CL_i, x_v\in B_j^\sss{(\star)}\big\} \le
         4\sqrt{d}\mathrm{Vol}(B_j^\sss{(\star)})^{1/d}.
         \nonumber
        \end{equation}
  \item[(iii)]
        For every box $B_j^\sss{(\star)}$, there exists a box $B_j'$ centered at $z_j$ such that
        \begin{equation}
         \mathrm{Vol}(B_j')=d^{d/2}2^{3d}\mathrm{Vol}(B_j^\sss{(\star)}),
         \qquad \text{ and }\qquad
         |\CL \cap B_j'|\ge \re\mathrm{Vol}(B_j').\label{eq:cover-expansion-larger-box}
        \end{equation}
  \item[(iv)] If $\CL$ is (additionally) $s$-expandable for some $s\le n$, then for all $j\in \CJ^{\sss{(\star)}}$
        \begin{equation}
         \mathrm{Vol}(B_j^\sss{(\star)})\le d^{-d/2}2^{-3d}s.\label{eq:cover-expansion-maxvol}
        \end{equation}
  \item[(v)]
        If $\CL$ is $n$-expandable, then the total volume of a cover expansion is linear in $|\CL|$, i.e., \eqref{eq:linear-volume-expansion} holds.
 \end{itemize}
 \begin{proof}
  Part (i) is a consequence of Definition~\ref{def:cover-expansion}: every cell with label at most $m$  has  $\ell_i\ge\re d^{d/2}2^{3d}$, so by (vol), i.e., \eqref{eq:merging-volume}, Observation (i) follows.

  For part (ii) we apply the triangle inequality: since $x_u\in \CL_i$, $x_u$ is in $B_{z_i}$, and so by \eqref{eq:max-cell-dist},
  $\|x_u-z_i\|\le 2\sqrt{d}$; and by  \eqref{eq:merging-distance}
  $\|z_i-z_j\|\le \sqrt{d}\mathrm{Vol}(B_j^\sss{(\star)})^{1/d}$; hence
  $
  \|x_u-z_j\| \le 2\sqrt{d} + \sqrt{d}\mathrm{Vol}(B_j^\sss{(\star)})^{1/d}.
  $
  Also, for any $x_v\in B_j^\sss{(\star)}$,
  $
  \|z_j-x_v\|\le (\sqrt{d}/2)\mathrm{Vol}(B_j^\sss{(\star)})^{1/d}
  $ by \eqref{eq:merging-volume}.
  Combining these bounds and using $\mathrm{Vol}(B_j^\sss{(\star)})^{1/d}\ge 1$ yields
  \[
  \|x_u-x_v\| \le 2\sqrt{d} + (3\sqrt{d}/2)\mathrm{Vol}(B_j^\sss{(\star)})^{1/d}\le (7\sqrt{d}/2)\mathrm{Vol}(B_j^\sss{(\star)})^{1/d} \le
  4\sqrt{d}\mathrm{Vol}(B_j^\sss{(\star)})^{1/d},
  \]
  and part (ii) follows.
  For part (iii), note that part (ii) applied to $u\in \CL_i \subset B_{z_i}$ and $z_j$, yields 
  \begin{equation}
  \sup_{i\in \mathrm{Cells}_j^\sss{(\star)}} \big\{\|x_u-z_j\|: x_u\in\CL_{i}\big\} \le
  4\sqrt{d}\mathrm{Vol}(B_j^\sss{(\star)})^{1/d}.\nonumber
  \end{equation}
  Consequently, the box $B_j'$ centered at $z_j$ of volume
  $
  \mathrm{Vol}(B_j')=d^{d/2}2^{3d}\mathrm{Vol}(B_j^\sss{(\star)})
  $
  contains all $u\in \CL_i$ with $i\in \mathrm{Cells}_j^\sss{(\star)}$.
  Hence, using \eqref{eq:merging-volume}, we obtain
  \begin{equation}\label{eq:Bj-points}
  |\CL \cap B_j'| \,\ge\, \sum_{i\in\mathrm{Cells}_j^\sss{(\star)}}\ell_i \,=\, \re d^{d/2}2^{3d} \mathrm{Vol}(B_j^\sss{(\star)}) \,=  \, \re\mathrm{Vol}(B_j'),
  \end{equation}
  and part (iii) follows. For part (iv), by combining \eqref{eq:Bj-points} with Definition~\ref{def:expandable} we see that $\CL$ can only be $s$-expandable if $\mathrm{Vol}(B'_j)\le  s$. Rearrangement of the first part of \eqref{eq:cover-expansion-larger-box} yields \eqref{eq:cover-expansion-maxvol}.

 Part (v). Since $\CL$ is $n$-expandable, Definition~\ref{def:expandable} implies that $|\CL|\le \re n$. By choice of the boxes in~\eqref{eq:merging-volume}, $B_j^{\sss{(\star)}}$ has volume at most $n$. 
 Therefore, by an argument similar to \eqref{eq:max-cell-dist}, $\mathrm{Vol}(B_j^{\sss{(\star)}} \cap \Lambda_n) \ge 2^{-d} \mathrm{Vol}(B_j^{\sss{(\star)}})$  for all $j\in \CJ^{\sss{(\star)}}$.
  Since all boxes of the cover expansion are disjoint, and each cell is allocated once,  \eqref{eq:merging-volume} and \eqref{eq:expanded-cover-def} imply that
  \begin{align}
  \mathrm{Vol}(\CK_n^\sss{(\mathrm{exp})}(\CL)) & =\sum_{j\in\CJ^{\sss{(\star)}}}\mathrm{Vol}(B_j^\sss{(\star)}\cap \Lambda_n) \ge 2^{-d}
  \sum_{j\le \CJ^{\sss{(\star)}}}\mathrm{Vol}(B_j^\sss{(\star)})                                                                          \nonumber \\
                                                 & =
  \frac{1}{\re2^{4d}d^{d/2}}\sum_{j\in \CJ^{\sss{(\star)}}}\sum_{i\in \mathrm{Cells}_j^\sss{(\star)}} \ell_i
  =
  \frac{1}{\re d^{d/2}2^{4d}}\sum_{i\le m}\ell_i
  \ge
  \frac{1}{\re d^{d/2}2^{4d+1}}|\CL|,\nonumber
  \end{align}
  where the last bound follows by the assumption in Definition~\ref{def:cover-expansion} that $\ell_i \ge \re d^{d/2} 2^{3d}$ for $i\le m$, and the initial assumption that \eqref{eq:pigeonhole-1} in Claim \ref{lemma:pigeonhole} holds for $[m]$ with $\delta=1/2$.
 \end{proof}
\end{observation}

\begin{proof}[Proof of Proposition~\ref{prop:cover-expansion-deterministic} assuming Proposition~\ref{prop:cover-expansion-correctness}]
 For sets $\CL$ that admit a proper cover, we recall the reasoning below Definition~\ref{def:cover-proper} (in particular \eqref{eq:p-for-proper}) which  implies both bounds \eqref{eq:prop-cover-expansion-min-volume} and \eqref{eq:p-connection} in Proposition~\ref{prop:cover-expansion-deterministic}.  Let $\CL$ be an $s$-expandable set that does not admit a proper cover. Let $\CK_n^\sss{(\mathrm{exp})}$ be a cover-expansion of $\CL$ given by the boxes $(B_j^\sss{(\star)})_{j\in \CJ^{\sss{(\star)}}}, \CJ^{\sss{(\star)}}\subseteq [m]$ and an allocation ${\buildrel \star \over \mapsto}$ of the initial cells $(B_{z_i})_{i\in[m]}$ to these boxes. The existence of this cover expansion is guaranteed by Proposition~\ref{prop:cover-expansion-correctness}. The volume bound \eqref{eq:prop-cover-expansion-min-volume} follows from \eqref{eq:linear-volume-expansion} in Proposition~\ref{prop:cover-expansion-correctness}.  Hence, it only remains to verify \eqref{eq:p-connection}.

 Let $u=(x_u,w_u)\in \CK_n^\sss{(\mathrm{exp})}\times[\underline{w},\infty)$. By (disj), and \eqref{eq:expanded-cover-def}, there exists $j\in\CJ^\sss{(\star)}$ such that $x_u\in B_j^\sss{(\star)}$. Recall from \eqref{eq:assigned-cells-star} that $\mathrm{Cells}_j^{\sss{(\star)}}$ are the cells allocated to $B_j^\sss{(\star)}$, and from Definition~\ref{def:cell-zi} that $\CL_i=\CL\cap B_{z_i}$. Let now
 $\CL_{j}^{\sss{(\star)}}:=\cup_{i\in \mathrm{Cells}_j^{\sss{(\star)}}} \CL_i$.
 Recall the formula of the connection probability from \eqref{eq:connection-prob-gen}.
 
 \emph{Case (1): $\alpha<\infty$}.
 By Observation \ref{obs:cover-expansion}(ii) for any $v=(x_v, w_v)\in \CL_j^{\sss{(\star)}}\times[1,\infty)$, and any $(x_u,w_u)\in B_j^\sss{(\star)}\times [\underline w, \infty)$, using the lower bounds for the marks, we obtain using \eqref{eq:connection-prob-gen} that
 \[
  \mathrm{p}\big(u,v\big) = p\Big(1\wedge\beta\frac{\kappa_\sigma(w_u,w_v)}{\|x_v-x_u\|^{d}}\Big)^\alpha \ge p\Big(1\wedge\frac{\beta\underline{w}}{ (4\sqrt{d})^{d}\mathrm{Vol}(B_j^\sss{(\star)})}\Big)^\alpha=:r.
 \]
 By \eqref{eq:merging-volume}, $|\CL_{j}^{\sss{(\star)}}|=\sum_{i\in \mathrm{Cells}_j^{\sss{(\star)}}} \ell_i=\re d^{d/2}2^{3d}\mathrm{Vol}(B_j^\sss{(\star)})$. Hence, we have
 \begin{align*}
  \Prob\Big((x_u, w_u)\nsim \CL_j^{\sss{(\star)}}& \,\big| \, \{(x_u, w_u): x_u\in\CL\}\cup \{v\}\subseteq \CV\Big)                                                                                                                                            \\
                                                        & \le (1-r)^{\re d^{d/2}2^{3d}\mathrm{Vol}(B_j^\sss{(\star)})}                                                                                                                              \\
                                                        & \le \exp\big(-p\re d^{d/2}2^{3d} \big(\mathrm{Vol}(B_j^\sss{(\star)})\wedge\beta^\alpha\underline{w}^\alpha (4\sqrt{d})^{-\alpha d}\mathrm{Vol}(B_j^\sss{(\star)})^{1-\alpha}\big)\big).
 \end{align*}
 Take now $s=s(\underline w)= (2^d \beta \underline w)^{1/(1-1/\alpha)}$. Since $\alpha \ge 1$, and $\CL$ is $s$-expandable, we can use the upper bound in  \eqref{eq:cover-expansion-maxvol} on $\mathrm{Vol}(B_j^\sss{(\star)})$ to bound the second term in the minimum on the right-hand side of the last row, and we use  $\mathrm{Vol}(B_j^\sss{(\star)})\ge 1$ by Observation \ref{obs:cover-expansion}(i) to bound the first term. We obtain
 \begin{align*}
  \Prob\Big((x_u, w_u)\nsim \CL_j^{\sss{(\star)}}\, &\big| \, \{(x_u, w_u): x_u\in\CL\}\cup \{v\}\subseteq \CV\Big)                                                               \\
                                                        & \le
  \exp\big(-p\re \big(d^{d/2}2^{3d}\wedge d^{d/2}2^{3d}\beta^\alpha\underline{w}^\alpha (4\sqrt{d})^{-\alpha d}s^{1-\alpha}(d^{-d/2}2^{-3d})^{1-\alpha}\big)\big) \\
                                            & =
  \exp\big(-p\re \cdot\big((d^{d/2}2^{3d
})\wedge ((2^d\beta )^\alpha \underline{w}^\alpha s^{1-\alpha})\big)\big) \le \exp(-\re p) \le 1-p/2,
 \end{align*}
 where we used in the last row that  $d^{d/2}2^{3d}>1$, the definition of $s$ and also that the bound on $\underline{w}$ in  \eqref{eq:prop-cover-expansion-min-weight}  ensures that  the second term inside the minimum is at least 1, and that $\exp(-\re p)\le 1-p/2$ for $p\in[0,1]$. This concludes the proposition for $\alpha<\infty$.

 \emph{Case (2): $\alpha=\infty$}. Using the same bounds as for $\alpha<\infty$ on the distance, mark and volume of boxes, but now \eqref{eq:connection-prob-gen} for $\alpha=\infty$, for any $u=(x_u, w_u)\in B_j^\sss{(\star)}\times[\underline{w},\infty)$ and any $v=(x_v, w_v)\in\CL_j^{\sss{(\star)}}\times[1,\infty)$ that
 \begin{align*}
  \mathrm{p}\big((x_u,w_u),(x_v, w_v)\big) & \ge p\mathbbm 1\{\beta \underline{w} \ge(4\sqrt{d})^d\mathrm{Vol}(B_j^{(\star)})\}
  \ge p\mathbbm 1\{\beta \underline{w} \ge(4\sqrt{d})^dd^{-d/2}2^{-3d}s\}                                                      \\
& =p\mathbbm 1\{\beta \underline{w} \ge 2^{-d}s\}=p > p/2,
 \end{align*}
 where in the one-but-last step we used \eqref{eq:prop-cover-expansion-min-weight}, finishing the proof of $\alpha=\infty$.
\end{proof}
The case $\alpha=\infty$ does not use of the size of $\CL_j^{\sss{(\star)}}$, and only requires a single vertex in it, which is intuitive considering the threshold nature of $\mathrm p$ in \eqref{eq:connection-prob-gen}.
It remains to prove Proposition~\ref{prop:cover-expansion-correctness}, which is the content of the following subsection.
\subsection{The cover-expansion algorithm}
Now we give the algorithm producing the cover expansion of a set $\CL$ without a proper cover, thence, proving Proposition~\ref{prop:cover-expansion-correctness}.

\subsubsection*{Setup for the algorithm} Recall the notation from Definitions~\ref{def:cell-zi} and \ref{def:cover-expansion}. Throughout, we will assume that $\CL$ does not admit a proper cover in Definition~\ref{def:cover-proper} and that $(B_{z_i}: i\in [m])$ are the cells satisfying \eqref{eq:pigeonhole-1}.
Contrary to Definition~\ref{def:cover-expansion}, which allocates the initial cells $B_{z_i}$ \emph{to boxes} $B_j^{\sss{(\star)}}$, the algorithm allocates the labels $i$, $i\le m$ of the initial cells $B_{z_i}$ \emph{towards each other} in discrete rounds $r\in \N$.
We write $i\ {\buildrel r \over \mapsto}\ j$ to indicate that label $i$ is allocated to label $j$ in the allocation of round $r$. We also write
\begin{equation}
 {\buildrel r\over\mapsto}
 :=
 \{(i, j): i\ {\buildrel r\over\mapsto}\ j\}_{i\le m}; \quad \mathrm{Cells}_j^\sss{(r)}:=\bigcup_{i\le m}\big\{i: i\ {\buildrel r \over \mapsto}\ j\big\}; \quad \CJ^\sss{(r)}:=\{j: \mathrm{Cells}_j^\sss{(r)}\neq \emptyset\}. \nonumber
\end{equation}
In each round $r\ge 0$, the boxes $\{B_j^{^{\sss{(r)}}}\}_{j\in\CJ^\sss{(r)}}$, and the centers of these boxes are completely determined by ${\buildrel r\over \mapsto}$ by the formula
\begin{equation}
 B_j^\sss{(r)}:= \Lambda\bigg(z_j, \frac{1}{\re d^{d/2}2^{3d}}\sum_{i\in\mathrm{Cells}_j^\sss{(r)}} \ell_i\bigg) \quad \emph{for} \quad j\in \CJ^{\sss{(r)}};
 \label{eq:or-bjr}
\end{equation}
where $\Lambda(x, s)$ is a box of volume $s$ centered at $x\in \R^d$, see  \eqref{eq:xi-q-ab}.
Since label $j$ corresponds to center $z_j$ across different rounds, by slightly abusing notation we also write $B_{z_i} \ {\buildrel r\over\mapsto}\ B_j^\sss{(r)}$ if and only if $i \ {\buildrel r\over\mapsto}\  j$.
We say that ${\buildrel r\over\mapsto}$ satisfies one (or more) conditions in Definition~\ref{def:cover-expansion} if  $(B_j^{\sss{(r)}})_{j\in \CJ^{\sss{(r)}}}$ with allocation ${\buildrel r\over\mapsto}$ satisfies the condition(s).

The algorithm starts with the identity as initial allocation ${\buildrel 0\over \mapsto}$ that induces possibly overlapping boxes $B_1^\sss{(0)},\dots,B_m^\sss{(0)}$; we will show that ${\buildrel 0\over \mapsto}$ already satisfies  (near)\ and (vol.)\ of Definition~\ref{def:cover-expansion}.
In each stage the algorithm attempts to remove an overlap --  a non-empty intersection -- between a pair of boxes by re-allocating a few cell labels, while maintaining properties (near)\ and (vol.); we achieve  (disj) in the last round $r^\star$.
The last round $r^\star<\infty$ corresponds to the final output, by setting $\CJ^{\sss{(\star)}}:= \CJ^{\sss{(r^\star)}}; B_j^\sss{(\star)}:= B_j^\sss{(r^\star)}$ and defining $B_{z_i}\ {\buildrel \star\over \mapsto}\  B_j^\sss{(\star)}$  iff $i\  {\buildrel r^\star\over \mapsto}\ j$.

\subsubsection*{The cover-expansion algorithm}

\begin{itemize}[align=right, labelwidth=100pt, leftmargin=37pt]
 \item [(input)] $(B_{z_i})_{i\in[m]}$ and $\CL_i=\CL\cap B_{z_i}$ satisfying \eqref{eq:pigeonhole-1} with $\nu=\re d^{d/2}2^{3d}$ and $\delta=1/2$.
 \item[(init.)] Set $r:=0$, and  allocate $j\ {\buildrel \sss{0} \over \mapsto}\ j$ for all $j\le m$.
 \item[(while)]  If $(B_j^{\sss{(r)}})_{j\in \CJ^{\sss{(r)}}}$ in \eqref{eq:or-bjr} are all pairwise disjoint, set $r^\star:=r$; and return $\CJ^{\sss{(\star)}}:= \CJ^{\sss{(r^\star)}}$; $B_j^\sss{(\star)}:= B_j^\sss{(r^\star)}$ and ${\buildrel \star \over \mapsto}:={\buildrel r^\star \over \mapsto}$.

      Otherwise,  let $j_1(r)\in\CJ^\sss{(r)}$ be the label corresponding to the largest box $B_{j_1(r)}^{\sss{(r)}}$ with an overlap with some other box in round $r$, and let $j_2(r)$ be the label of the largest box that overlaps with $B_{{j_1(r)}}^{\sss{(r)}}$ (using an arbitrary tie-breaking rule).
      Define
      \begin{equation}
      \begin{aligned}
        \mathcal{I}_1^{\sss{(r)}} &:= \mathrm{Cells}_{j_2(r)}^{(r)}\cap\big\{i: \|z_i-z_{j_1(r)}\|\le \sqrt{d}\mathrm{Vol}(B_{j_1(r)}^\sss{(r)})^{1/d}\big\}_{i\le m};\\
        \mathcal{I}_2^{\sss{(r)}} &:= \mathrm{Cells}_{j_2(r)}^{(r)} \setminus \mathcal{I}_1^{\sss{(r)}}.\label{eq:cover-alg-z1}
        \end{aligned}
      \end{equation}
      Then we define ${\buildrel r+1 \over \mapsto}$ by  only re-allocating labels in $\mathrm{Cells}_{j_2(r)}^{(r)}$ as follows:
      \begin{itemize}\setlength\itemsep{-0.2em}
        \item[(i)] for $i\in \mathcal{I}_1^{\sss{(r)}}$  we allocate $i\ {\buildrel r+1 \over \mapsto}\ j_1(r)$, i.e., the labels of cells that are sufficiently close to the center of $B_{j_1(r)}^{\sss{(r)}}$ in order to satisfy \eqref{eq:merging-distance} are re-allocated to $j_1(r)$;
        \item[(ii)] for $i\in \mathcal{I}_2^{\sss{(r)}}$ we allocate $i\ {\buildrel r+1 \over \mapsto}\ i$, i.e., the labels of cells in $\mathrm{Cells}_{j_2(r)}^{(r)}$ that are potentially too far away from the center of $B_{j_1(r)}^\sss{(r+1)}$ are re-allocated back to themselves;
        \item[(iii)]  for  $i\in [m]\setminus (\mathrm{Cells}_{j_2}^{\sss{(r)}})$, we set $i\ {\buildrel r+1 \over \mapsto}\ k$ if and only if $i\ {\buildrel r \over \mapsto}\ k$ (that is, ${\buildrel r+1 \over \mapsto}$ agrees with ${\buildrel r \over \mapsto}$ outside labels in $\mathrm{Cells}_{j_2(r)}^{\sss{(r)}}$).
      \end{itemize}
      Increase $r$ by one and repeat (while).
\end{itemize}
We make an immediate observation.
\begin{observation}\label{obs:nonempty}
 In each iteration of (while), $\mathcal{I}_1^{\sss{(r)}}$ in \eqref{eq:cover-alg-z1} is always non-empty. Moreover,
 \begin{equation}\label{eq:step-size}
  \mathrm{Vol}\big(B_{{j_1(r)}}^\sss{(r+1)}\big) -
  \mathrm{Vol}\big(B_{{j_1(r)}}^\sss{(r)}\big) \ge 1.
 \end{equation}
\end{observation}
\begin{proof}It can be shown inductively that $j\ {\buildrel r\over \mapsto}\ j$ holds for all $j\in\CJ^\sss{(r)}$.  Since the boxes $B_{j_1(r)}^\sss{(r)}$ and $B_{j_2(r)}^\sss{(r)}$ overlap, the distance of their centers $\|z_{j_2(r)}-z_{j_1(r)}\|$ is at most the diameter of $B_{j_1(r)}^\sss{(r)}$, which is $\sqrt{d}\mathrm{Vol}(B_{{j_1(r)}}^\sss{(r)})^{1/d}$. Hence, $j_2(r)\in \CI_{1}^{(r)}$ and so we re-allocate $j_2(r)$ to $j_1(r)$ in round $r+1$.
 Since each cell contains $\ell_i\ge \re d^{d/2}2^{3d}$ many vertices by the assumption in (input), we obtain by \eqref{eq:or-bjr}
 \begin{equation*}
  \mathrm{Vol}\big(B_{{j_1(r)}}^\sss{(r+1)}\big) -
  \mathrm{Vol}\big(B_{{j_1(r)}}^\sss{(r)}\big) \ge \frac{\ell_{j_2(r)}}{\re d^{d/2} 2^{3d}}\ge 1.\qedhere
 \end{equation*}
\end{proof}

\begin{proof}[Proof of Proposition~\ref{prop:cover-expansion-correctness}]
 Once having shown that a cover expansion of $\CL$ exists, the bound on its volume \eqref{eq:linear-volume-expansion} holds by Observation \ref{obs:cover-expansion}(v). So it remains to show that the algorithm produces in finitely many rounds an output satisfying all conditions of a cover expansion in Definition~\ref{def:cover-expansion}.

 \emph{The algorithm stops in finitely many rounds.}
 We argue using a monotonicity argument. We say that a vector $\bm{a}=(a_1, \dots, a_m)\in \R^m$ is non-increasing if $a_i\ge a_{i+1}$ for all $i\le m-1$. We use the lexicographic ordering for non-increasing vectors $\bm{a}, \bm{b}\in \R^m$: let $\bm{a}>_L \bm{b}$ if there exists a coordinate $j\le m$ such that $a_\ell= b_\ell$ for all $\ell< j$ and $a_\ell> b_\ell$ for $\ell=j$.

 For all $r\in \N$, $\CJ^{\sss{(r)}}\subseteq[m]$, and hence, $m^{\sss{(r)}}:=|\CJ^{\sss{(r)}}|\le m$. Let $\bm{a}^{\sss{(r)}}\in \R^m$ be the non-increasing vector of the re-ordered $(\mathrm{Vol}(B_j^{\sss{(r)}}))_{j\in \CJ^{\sss{(r)}}}$ appended with $(m-m^{\sss{(r)}})$-many zeroes.
 By Observation \ref{obs:nonempty},
 the entry corresponding to $\mathrm{Vol}\big(B_{{j_1(r)}}^\sss{(r)}\big)$ in $\bm{a}^{\sss{(r)}}$ increases in $\bm{a}^{\sss{(r+1)}}$ by at least $1$. Moreover, the entry corresponding to $\mathrm{Vol}\big(B_{{j_2(r)}}^\sss{(r)}\big)$ increases the entry $\mathrm{Vol}\big(B_{{j_1(r)}}^\sss{(r)}\big)$, and the rest of its volume ``crumbles'' into  smaller volumes, since labels in $\CI_2^{\sss{(r)}})$ will be re-allocated to themselves. Since by definition, $j_1(r)$ corresponds to the largest box among $(B_j^{\sss{(r)}})_{j\in \CJ^{\sss{(r)}}}$ that has an overlap with some other box, so also $\mathrm{Vol}(B_{j_2(r)}^{\sss{(r)}})\le \mathrm{Vol}(B_{j_1(r)}^{\sss{(r)}})$, and the allocation of labels except those in $\mathrm{Cells}_{j_2(r)}^{\sss{(r)}}$ remains unchanged, these together imply that $\bm{a}^\sss{(r+1)}>_L\bm{a}^\sss{(r)}$.
 Finally, for any $r$ and  any $j \in \CJ^{\sss{(r)}}$,  $\mathrm{Vol}(B_j^{\sss{(r)}})\le |\CL|/(\re d^{d/2}2^{3d})=:b$ by \eqref{eq:merging-volume}, implying that for all $r$, $(b,\dots,b) >_L \bm{a}^{\sss{(r)}}$.
 So, $(\bm{a}^\sss{(r)})_{r\ge 0}$ is an increasing bounded sequence with respect to $>_L$, with an increase of at least $1$ per step by \eqref{eq:step-size}. Hence, $(\bm{a}^{\sss{(r)}})_{r\ge 0}$ converges and attains its limit after finitely many rounds, i.e.,  $r^\star<\infty$.

 \emph{The output corresponds to a cover expansion.}
 We now prove that the output $\CJ^{\sss{(\star)}}, \ {\buildrel \star\over\mapsto}$ and the corresponding boxes in \eqref{eq:or-bjr} satisfy the conditions of Definition~\ref{def:cover-expansion}. By the stopping condition in step (while) of the algorithm, $(B_j^{\sss{(\star)}})_{j\in \CJ^{\sss{(\star)}}}$ satisfy (disj.), and by their definition in \eqref{eq:or-bjr}, also (vol.). We need to still verify (near). We  show this by induction:
 initially, for $(B_j^{\sss{(0)}})_{j\in \CJ^{\sss{(0)}}}, {\buildrel 0\over\mapsto}$, (near)\ holds, since in (init.) all labels are allocated to themselves, so $\mathrm{Cells}_j^{\sss{(0)}}=\{j\}$, and thus the left-hand side in~\eqref{eq:merging-distance} is $0$.
 Assume then $r>0$. We prove that (near)\ holds for ${\buildrel r+1\over\mapsto}$, assuming that it holds for ${\buildrel r\over\mapsto}$. Recall from (while) that $j_1(r)$ is the label of the largest box that has an overlap; $j_2(r)$ is the label of the largest box overlapping with $B_{j_1(r)}^{\sss{(r)}}$; by \eqref{eq:cover-alg-z1}, $\mathcal I_1^{\sss{(r)}}$ is the set of labels in $\mathrm{Cells}_{j_2(r)}^{\sss{(r)}}$ re-allocated to $j_1(r)$, and $\mathcal I_2^{\sss{(r)}}=\mathrm{Cells}_{j_2(r)}^{\sss{(r)}}\setminus \mathcal I_1^{\sss{(r)}}$ is the set labels allocated in round $r$ to $j_2(r)$, and in round $r+1$  to  themselves. We distinguish between four cases for the proof of the inductive step:
 \begin{itemize}
  \setlength\itemsep{0em}
  \item Assume $i\notin(\mathrm{Cells}_{j_1(r)}^\sss{\mathrm{(r)}}\cup\mathrm{Cells}_{j_2(r)}^\sss{\mathrm{(r)}})$ and let $k$ be such that $i\ {\buildrel r+1\over \mapsto} k$. By (while) part (iii), $\mathrm{Cells}_k^\sss{(r+1)}=\mathrm{Cells}_k^\sss{(r)}$, so by the induction hypothesis, \eqref{eq:merging-distance} holds for ${\buildrel r+1\over\mapsto}$.
  \item
        Assume $i\in \mathrm{Cells}_{j_1(r)}^\sss{(r)}$. By \eqref{eq:step-size}, $B_{j_1(r)}^\sss{(r)}\subsetneq B_{j_1(r)}^\sss{(r+1)}$ as the volume increases by at least one while the centers of the boxes agree. Since $\|z_i-z_{j_1(r)}\|^d\le d^{d/2}\mathrm{Vol}(B_{j_1(r)}^\sss{(r)})$ by the induction hypothesis, it follows that $\|z_i-z_{j_1(r+1)}\|^d\le d^{d/2}\mathrm{Vol}(B_{j_1(r)}^\sss{(r+1)})$, implying \eqref{eq:merging-distance} for ${\buildrel r+1\over\mapsto}$.
  \item Assume $i\in\mathrm{Cells}_{j_2(r)}^\sss{(r)}\cap\mathcal I_1^{\sss{(r)}}$. The definition of $\mathcal I_1^{\sss{(r)}}$ in \eqref{eq:cover-alg-z1} forces that $\|z_i-z_{j_1(r)}\|$ satisfies \eqref{eq:merging-distance}.
  \item Assume $i\in\mathcal{I}_2^{\sss{(r)}}=\mathrm{Cells}_{j_2}^\sss{(r)}\setminus \mathcal I_1^{\sss{(r)}}$: \eqref{eq:merging-distance} holds for the same reason as for the base case, i.e., since $i\ {\buildrel r+1\over\mapsto}\ i$, $\|z_i -z_i\|=0$ trivially satisfies \eqref{eq:merging-distance}.
 \end{itemize}
 Having all possible cases covered, this finishes the proof of the induction. Since $r^\star<\infty$, this finishes the proof of Proposition~\ref{prop:cover-expansion-correctness}.
\end{proof}
\subsection{Poisson point processes are expandable}
We end this section by showing that a Poisson point process is typically $s$-expandable for $s$ sufficiently large. Recall $\Lambda_n=[-n^{1/d}/2, n^{1/d}/2]^d$.
\begin{lemma}[PPPs are expandable]\label{lem:expandable-ppp}
 Let $\Gamma$ be a Poisson point process on $\R^d$ equipped with an absolutely continuous intensity measure $\mu$  such that $\mu(\mathrm dx)\le \mathrm{Leb}(\mathrm dx)$. Then there exists a constant $C_{\ref{lem:expandable-ppp}}> 0$ such that for any $s\ge 4/(\re-2)$,
 \begin{equation}
  \Prob\big(\Gamma\cap \Lambda_n \text{ is not } s\text{-expandable}\big) \le
  C_{\ref{lem:expandable-ppp}}n\exp(-s/3).\nonumber
 \end{equation}
\end{lemma}
\begin{proof}
 Using stochastic domination of point processes, without loss of generality we can assume that $\Gamma$ has intensity measure $\mathrm{Leb}(\rd x)$. Let us define
  $ \mathrm{R}(s):= \{\widetilde s\in\mathbb{N}: \widetilde s\ge s\}$. We first show that when $s\ge 4(\re-2)$,
 \begin{equation}\label{eq:poisson-cover-1}
     \{\Gamma\cap\Lambda_n\text{ is } s\text{-expandable}\}\subseteq \{\forall x\in\Z^d\cap\Lambda_n, \tilde s\in\mathrm{R}(s): |\Gamma\cap\Lambda_{\tilde s}(x)|\le 2\tilde s\}. 
 \end{equation}
Indeed, if the bound on the right-hand side holds for all $\tilde s\in\mathrm{R}(s)$, then for any $s'\in(\tilde s, \tilde s+1)$, 
\[
|\Gamma \cap \Lambda_{s'}(x)|\le |\Gamma\cap\Lambda_{\tilde s+1}(x)|\le 2(\tilde s+1)\le 2(s'+2)\le \re s'
\]
whenever $s'\ge 4/(\re-2)$.
 We consider the complements of the events in \eqref{eq:poisson-cover-1}. By a union bound over the at most $n$ possible centers of the boxes in $\Lambda_n$, and by translation invariance of $\mathrm{Leb}$, we thus obtain
 \begin{equation}
  \begin{aligned}
  \Prob\big(\Gamma\cap\Lambda_n \text{ is not } s\text{-expandable}\big) & =\Prob\big(\exists x\in \Z^d\cap \Lambda_n,\exists \tilde s\in\mathrm{R}(s): |\Gamma\cap \Lambda_{\tilde s}(x) |\ge 2 \tilde s\big) \\ & \le
  n\sum_{\tilde s\in\mathrm{R}(s)}\Prob\big(|\Gamma\cap\Lambda_{\tilde s}|\ge 2 \tilde s\big).
  \label{eq:expandable-proof-3}
  \end{aligned}
 \end{equation}
 Since the intensity of $\Gamma$ is equal to one,  each summand on the right-hand side is at most $\exp(-(2\log 2 - 1)\tilde s)\le \exp(-\tilde s/3)$ by Lemma~\ref{lemma:poisson-1}.
We obtain for the summation in \eqref{eq:expandable-proof-3} for  some constant $C_{\ref{lem:expandable-ppp}}>0$,
 \begin{equation*}
  \Prob  \big(\Gamma\cap\Lambda_n \text{ is not } s\text{-expandable}\big) \le
  n\sum_{\tilde s\in \N: \tilde s\ge s}\exp(-\tilde s/3) \le
  C_{\ref{lem:expandable-ppp}} n\exp(-s/3).\qedhere 
 \end{equation*}
\end{proof}
\section{Upper bound: second-largest component}\label{sec:upper-2nd}
The main goal of this section is to prove the following proposition for general values of $n$ and $k$, which readily implies Theorem~\ref{thm:second-largest}(ii-iii), i.e., \eqref{eq:second-largest-main-upper} and \eqref{eq:second-largest-main-upper-2}.
Recall $\zeta_\mathrm{hh}=1-\gamma_\mathrm{hh}(\tau-1)$ from \eqref{eq:zeta-hh}.
We restrict ourselves to the parameter setting of Theorem~\ref{thm:subexponential-decay}, which assumes $\zeta_\mathrm{hh}>0$, and corresponds to $\tau\in(2,2+\sigma)$. Moreover, below we will use independence properties of Poisson point processes, and therefore restrict to such vertex sets: we generate the marked vertex set $\CV=\{(x_v, w_v)\}_{v\in V}$ from Definition~\ref{def:ksrg} with iid marks following distribution $F_W$ in~\eqref{eq:power-law} in Assumption~\ref{assumption:main} as a \emph{marked} Poisson point process on $\R^d\times[1,\infty)$ with intensity measure 
\begin{equation}\label{eq:poisson-intensity}
    \mu_\tau(\rd x\times \rd w):=\mathrm{Leb}\otimes F_W(\rd w)=\rd x\times (\tau-1)w^{-\tau}\rd w.
\end{equation}
We use this construction throughout the paper and in particular in parts of this section.
Some subresults in this section also hold for KSRGs with vertex set on $\Z^d$ and can be obtained by replacing concentration inequalities for Poisson random variables by Chernoff bounds. We leave these adaptations to the reader but include them in the statements. 
\begin{proposition}\label{prop:2nd-upper-bound-hh}
 Consider a KSRG under the same assumptions as in Theorem~\ref{thm:subexponential-decay}, with vertex set formed by a homogeneous Poisson point process. 
  For $\tau\ge \sigma+1$, there exists a constant $c_{\ref{prop:2nd-upper-bound-hh}}>0$ such that for all $n\ge k\ge 1$ 
 \begin{equation}
\Prob\big(|\CC^\sss{(2)}_n| > k\big)\le n\exp\big(-c_{\ref{prop:2nd-upper-bound-hh}}k^{\zeta_{\mathrm{hh}}}\big).\label{eq:prop-2nd-upper-bound}
 \end{equation}
 For $\tau<\sigma+1$, 
 the inequality holds with exponent $1/(\sigma+1-(\tau-1)/\alpha)$ in place of $\zeta_{\mathrm{hh}}$.
\end{proposition}
We follow the steps of the methodology from Section~\ref{sec:outline-hh}. The bulk of the work is to establish Steps 1 and 3 there, since we already developed the cover expansion of Step 4 in Section~\ref{sec:cover-expansion}.
We first introduce some notation.
We aim to partition the box $\Lambda_n$ into disjoint subboxes of (roughly) volume $k$. Define
\begin{equation}
 n' := k\lfloor (n/k)^{1/d}\rfloor^d.\label{eq:nprime-boxes}
\end{equation}
The box $\Lambda_{n'}\subseteq\Lambda_{n}$ is the largest box inside $\Lambda_n$ that can be partitioned into $n'/k$ disjoint subboxes of volume  exactly $k$ (boundaries are allocated uniquely, as in Definition~\ref{def:cells-general}). Let the boxes of this partitioning of $\Lambda_{n'}$ be $\CQ_1,\dots,\CQ_{n'/k}$, labeled so that $\CQ_i$ shares a boundary (that is, a $(d-1)$-dimensional face) with $\CQ_{i+1}$ for all $i<n'/k$.
Define for each $u=(x_u,w_u)\in \CV_n\subseteq\Lambda_n$,
\begin{equation}
 \CQ(u):= \argmin_{\CQ_i} \|x_u-\CQ_i\|, \label{eq:qu}
\end{equation}
with the convention that $\|x_u-\CQ_i\|=0$ if $x_u\in \CQ_i$, and take the box with the smallest index if the minimum is non-unique.
Similarly to \eqref{eq:max-cell-dist}, we observe that for any point $u\in\CV_n\subset \Lambda_n$
\begin{equation}
 \sup_{y\in \CQ(u)} \|x_u-y\|\le 2\sqrt{d}k^{1/d}.\label{eq:upper-hh-outside-tesselation}
\end{equation}
\subsection{Step 1. Construction of the backbone}  Recall the definition of $\CG_n[a,b)$ from \eqref{eq:gn-ab}.
We first show that, for some  $w_\mathrm{hh}=w_\mathrm{hh}(k)$, the graph $\CG_{n,1}:=\CG_{n}[w_\mathrm{hh}, 2w_\mathrm{hh})$ contains a so-called  backbone, a connected component $\CC_{\mathrm{bb}}$ that contains at least $s_k=\Theta(k^{\zeta_{\mathrm{hh}}})$ vertices in every subbox.
For $\lambda>1$, let $\varrho_\lambda$ be the survival probability of a Bienaym\'{e}-Galton-Watson branching process with $\mathrm{Poi}(\lambda)$ offspring distribution. Then let $\lambda_\star(1/2)$ be the mean offspring when $\varrho_{\lambda_\star(1/2)}=1/2$.  
Using $\beta$ from Definition~\ref{def:ksrg} and $\alpha>1$ from Assumption~\ref{assumption:main}, define the (small) constant $C_1$ to be the solution of the equation
\begin{align}
 (p/16)\beta^\alpha 2^{-\alpha d}d^{-\alpha d/2}  C_1^{-((1+\sigma)\alpha-(\tau-1))/(\tau-1)}& =\max(\log 2, \lambda_\star(1/2)), & \mbox{if } \alpha & <\infty, \label{eq:upper-hh-c1} \\
 \beta C_1^{-(1+\sigma)/(\tau-1)}d^{-d/2} 2^{-d-2\sigma}   & =1,       & \mbox{if } \alpha & =\infty.\label{eq:c1-infty}
\end{align}
We set, with $\gamma_\mathrm{hh}$ from \eqref{eq:gamma-hh},
\begin{equation}
\begin{aligned}
 w_\mathrm{hh}&:= w_\mathrm{hh}(k):=C_1^{-1/(\tau-1)}k^{\gamma_\mathrm{hh}},\\
 s_k&:=(C_1/16)k^{1-\gamma_\mathrm{hh}(\tau-1)}=(C_1/16)k^{\zeta_\mathrm{hh}}=kw_\mathrm{hh}^{-(\tau-1)}/16.
 \end{aligned}\label{eq:w-gamma-hh}
\end{equation}
To avoid cumbersome notation, we often assume that $s_k\in\mathbb{N}$. 
Let us define $k_1$ as the smallest non-negative number satisfying 
 \begin{equation}
  (1-p)^{C_1 k_1^{\zeta_{\mathrm{hh}}}/16} = (1-p)^{s_k}  \le 1/2.
  \label{eq:upper-hh-k0}
 \end{equation}
Recall the notation $\CV_\CQ[a,b)$ from~\eqref{eq:xi-q-ab}. Let
\begin{equation}
 \CA_\mathrm{bb}:=\CA_\mathrm{bb}(n,k) :=
 \left\{
 \begin{aligned}
  & \CG_{n,1}\text{ contains a connected component }\CC_\mathrm{bb}(n,k)\text{ s.t. }                \\
  & \text{ for all } i\le (n'/k): |\CV_{\CQ_i}[w_\mathrm{hh},2w_\mathrm{hh})\cap \CC_\mathrm{bb}| \ge s_k
 \end{aligned}
 \right\}.\label{eq:upper-hh-bb-event}
\end{equation}
On $\CA_\mathrm{bb}$, let $\CC_\mathrm{bb}:=\CC_{\mathrm{bb}}(n,k)$, the backbone, be the largest component in $\CG_n[w_{\mathrm{hh}},2w_{\mathrm{hh}})$ that satisfies the event $\CA_\mathrm{bb}$.
In the following lemma we obtain a lower bound on the probability that there exists a backbone. 
\begin{lemma}[Backbone construction]\label{lemma:upper-hh-bb}
 Consider a KSRG under the same assumptions as in Theorem~\ref{thm:subexponential-decay}, in particular $\tau\in(2,2+\sigma)$, with vertex set either formed by a homogeneous Poisson point process or $\Z^d$.
 There exist constants $c_{\ref{lemma:upper-hh-bb}}=c_{\ref{lemma:upper-hh-bb}}(p,\beta,d,\alpha, \tau,\sigma)>0$, such that for $k\ge k_1$ and all $n$ satisfying $n \ge k$,
 \begin{equation}
  \Prob\big(\neg\CA_\mathrm{bb}(n,k)\big) \le 3(n/k)\exp\big(-c_{\ref{lemma:upper-hh-bb}} k^{\zeta_\mathrm{hh}} \big).\label{eq:lem-hh-nobb}
 \end{equation}
\end{lemma}
\begin{proof}
 Towards proving \eqref{eq:lem-hh-nobb}, we reveal $\CV_n[w_{\mathrm{hh}}, 2w_{\mathrm{hh}})$, i.e., \emph{only} the vertex set of $\CG_{n,1}$, and define
 \begin{equation}\label{eq:apoi}
  \CA_\mathrm{poi}:= \{\forall i\le n'/k: |\CV_{\CQ_i}[w_\mathrm{hh},2w_\mathrm{hh})|\ge 4s_k \}.
 \end{equation}
     On $\CA_\mathrm{poi}$, every box contains enough vertices in $\CG_{n,1}$. Reveal now the edges of $\CG_{n,1}$ \emph{only within the box} $\CQ_1$: let $\CH$ be the induced subgraph of $\CG_{n,1}$ on $\CV_{\CQ_1}[w_\mathrm{hh},2w_\mathrm{hh})$, and define
 \begin{equation}\label{eq:J-def}
  \CA_\mathrm{init}:=\left\{\CH \mbox{\ contains a connected component\ } \CC_{\mathrm{init}} \mbox{\ with\ } |\CC_\mathrm{init}|\ge s_k\right\}.
 \end{equation}
  Then
 \begin{equation}
  \Prob(\neg\CA_\mathrm{bb})\le
  \Prob(\neg\CA_\mathrm{poi})+\Prob(\neg\CA_\mathrm{init}\mid \CA_\mathrm{poi})+\Prob(\neg\CA_\mathrm{bb} \mid\CA_\mathrm{init}\cap \CA_\mathrm{poi}).\label{eq:upper-hh-bb1}
 \end{equation}
 We first bound $\Prob(\neg\CA_\mathrm{poi})$ from above.
 The distribution of  $|\CV_{\CQ_i}[w_\mathrm{hh},2w_\mathrm{hh})|$ is Poisson 
 with mean $kw_\mathrm{hh}^{-(\tau-1)}(1-2^{-(\tau-1)})=16(1-2^{-(\tau-1)})s_k\ge 8s_k$ by \eqref{eq:poisson-intensity}, \eqref{eq:w-gamma-hh} and since $\tau\ge 2$. Lemma~\ref{lemma:poisson-1} yields
 \begin{equation}
  \Prob\big(|\CV_{\CQ_i}[w_{\mathrm{hh}}, 2w_{\mathrm{hh}})| \le 4s_k\big)
  \le
  \Prob\big(\Poi(8s_k)< 4s_k\big)
  \le
  \exp\big(-4s_k (1-(\log 2))\big).\nonumber
 \end{equation}
 Since $1-(\log 2)\ge 1/4$, by a union bound over the at most $n'/k\le n/k$ subboxes we get
 \begin{equation}
  \Prob(\neg\CA_\mathrm{poi})
  \le
  (n/k)\exp(-s_k).\label{eq:upper-hh-bb-1st}
 \end{equation}
 We will next show an upper bound on the third term on the right-hand side in \eqref{eq:upper-hh-bb1}. For this, we iteratively `construct' a backbone. The subboxes $\CQ_1,\dots,\CQ_{n'/k}$ are ordered so that  $\CQ_i$ and $\CQ_{i+1}$ share a boundary for all $i$. On $\CA_\mathrm{init}$, we know that $\CH$ inside $\CQ_1$ contains a connected component $\CC_{\mathrm{init}}$ with at least $s_k$ many vertices. We now reveal edges between $\CQ_{1}$ and $\CQ_{2}$, and bound the probability that there are at least $s_k$ many vertices in $\CQ_{2}$ that are connected by an edge to $\CC_\mathrm{init}$: denote this set of vertices by $\widetilde\CV_{2}$. Next, we apply the same bound to show that at least $s_k$ many vertices in $\CQ_{3}$ connect by an edge to $\widetilde\CV_{2}$, and so on.
 Hence, for $i\ge 1$, we need to  analyze the probability that a vertex in $\CQ_{i+1}$ connects to a vertex in $\widetilde\CV_i$, conditionally on $|\widetilde\CV_i|\ge s_k$.
 Since by assumption $\tau<2+\sigma$, by definition of $\gamma_\mathrm{hh}$ in \eqref{eq:gamma-hh} for all $\tau<2+\sigma$ and $\alpha\le\infty$, 
 \begin{equation}
  1-(1+\sigma) \gamma_\mathrm{hh} \ge 0, \quad \mbox{and} \quad 2+\sigma-\tau>0.\label{eq:gamma-upper-sigma}
 \end{equation}
 The Euclidean distance between vertices in neighboring boxes is at most $2\sqrt{d}k^{1/d}$ (twice the diameter of a single box), and all considered vertices have mark at least $w_\mathrm{hh}$.
 When $\alpha=\infty$, we use that $\gamma_{\mathrm{hh}}=1/(1+\sigma)$, see \eqref{eq:gamma-hh}, and so $w_{\mathrm{hh}}^{1+\sigma}/k=C_1^{-(\sigma+1)/(\tau-1)}$ by \eqref{eq:w-gamma-hh}. We obtain using $C_1$ from \eqref{eq:c1-infty}, $\mathrm{p}$ in \eqref{eq:connection-prob-gen}, that for any $u=(x_u, w_u)\in \CV_{\CQ_{i+1}}[w_{\mathrm{hh}}, 2w_{\mathrm{hh}})$,
 \begin{equation}
 \begin{aligned}
  \Prob\big((x_u,w_u)\sim\widetilde\CV_i \,\big|\, |\widetilde\CV_i |\ge s_k, \CA_{\mathrm{poi}}\big)
  & \ge 1-\Big(1-p\mathbbm 1\Big\{\frac{\beta w_{\mathrm{hh}}^{1+\sigma}}{(2\sqrt{d})^dk}\ge 1 \Big\}\Big)^{s_k} \\&= 1-(1-p)^{s_k}\ge 1/2,\label{eq:1/2-infty}
 \end{aligned}
 \end{equation}
 for all $k\ge k_1$ by~\eqref{eq:upper-hh-k0}.
 When $\alpha<\infty$, using $\mathrm p$ in \eqref{eq:connection-prob-gen} for $u\in \CQ_{\ell+1}$ with $w_u \ge w_\mathrm{hh}$, either the minimum is at $1$ below in \eqref{eq:temp-est-1} (in which case the right-hand side of \eqref{eq:1/2-infty} remains valid) or, the minimum in $\mathrm{p}$ is attained at the second term below: then we substitute $s_k$ from~\eqref{eq:w-gamma-hh},
 \begin{align}
  \Prob\big((x_u,w_u)\sim\widetilde\CV_i \mid |\widetilde\CV_\ell |\ge s_k, \CA_{\mathrm{poi}}\big)
  & \ge
  1-\big(1-p\big(1\wedge\beta (2\sqrt{d})^{-d} w_\mathrm{hh}^{1+\sigma}k^{-1}\big)^\alpha\big)^{s_k}\nonumber              \\
  & =
  1-\big(1-p\beta^\alpha (2\sqrt{d})^{-\alpha d} w_\mathrm{hh}^{(1+\sigma)\alpha}k^{-\alpha}\big)^{kw_\mathrm{hh}^{-(\tau-1)}/16} \label{eq:temp-est-1} \\
  & \ge
  1-\exp\big(- (p/16)\beta^\alpha (2\sqrt{d})^{-\alpha d} w_\mathrm{hh}^{(1+\sigma)\alpha-(\tau-1)}k^{1-\alpha}\big).\nonumber
 \end{align}
 By choice of $w_\mathrm{hh}$, and $\gamma_\mathrm{hh}$  in \eqref{eq:w-gamma-hh}, and \eqref{eq:gamma-hh}, respectively, factors containing $k$ cancel, and using the formula for $C_1$ in \eqref{eq:upper-hh-c1} we arrive at
 \begin{equation}
 \begin{aligned}
  \Prob\big(u\sim\widetilde\CV_i  \mid |\widetilde\CV_i |\ge s_k, \CA_{\mathrm{poi}}\big) &\ge
  1-\exp\big(- (p/16)\beta^\alpha (2\sqrt{d})^{-\alpha d} C_1^{1-(1+\sigma)\alpha/(\tau-1)}\big)\\
  &\ge 1/2.\label{eq:upper-hh-bb-half}
 \end{aligned}
 \end{equation}
 Combining \eqref{eq:upper-hh-bb-half} with \eqref{eq:1/2-infty}, we obtain a lower bound of $1/2$ for all $\alpha>1$ for any $u\in \CV_{\CQ_{i+1}}[w_{\mathrm{hh}}, 2w_{\mathrm{hh}})$.
 On $\CA_{\mathrm{poi}}$ (see \eqref{eq:apoi}) there are at least $4s_k$ vertices in $\CV_{\CQ_{i+1}}[w_{\mathrm{hh}}, 2w_{\mathrm{hh}})$. Each of these vertices connects conditionally independently by an edge to vertices in $\widetilde \CV_{i}$ with probability at least $1/2$, so for all $i\ge 1$, 
 \begin{align}
  \Prob\big(|\widetilde\CV_{i+1}|\ge s_k\,\big|\, |\widetilde\CV_{i}|\ge s_k, \CA_\mathrm{poi}\big)\ge
  \Prob\big(\,\Bin(4s_k, 1/2\,)\ge s_k
  \big)\ge
  1-\exp(-s_k/4),\nonumber
 \end{align}
 where the last bound follows by Chernoff's bound, see e.g. \cite[Theorem~2.1]{JLR}.     
 By a union bound over the at most $n'/k$ subboxes, we obtain
 \begin{equation}\label{eq:third-term-in-abb}
  \Prob\big(\neg\CA_\mathrm{bb} \mid \CA_\mathrm{init} \cap \CA_\mathrm{poi}\big)
  \le
  (n'/k)\exp(-s_k/4) \le (n/k) \exp(-s_k/4).
 \end{equation}
We will use this for the last term in \eqref{eq:upper-hh-bb1}, and \eqref{eq:upper-hh-bb-1st} to bound the first term. It remains to bound the second term,  
$\Prob(\neg \CA_\mathrm{init}\mid \CA_{\mathrm{poi}})
 $, with $\CA_\mathrm{init}$ from \eqref{eq:J-def}. 
 For this we show that the graph $\CH_1$ induced on $\CV_{\CQ_1}[w_{\mathrm{hh}}, 2w_{\mathrm{hh}})$ stochastically dominates a supercritical Erd{\H o}s-R\'enyi random graph with mean degree at least $\lambda_\star(1/2)$. We write $\mathrm{ER}(m,q)$ for an Erd{\H o}s-R\'enyi random graph on $m$ vertices with connection probability $q$. 
 Indeed, on the event $\CA_{\mathrm{poi}}$ there are at least $4s_k$ vertices in $\CV_{\CQ_1}[w_{\mathrm{hh}}, 2w_{\mathrm{hh}})$. Arbitrarily pick  $4s_k$ of them. Any two of those vertices, say $(x_u,w_u)$ and $(x_v, w_v)$, are within distance $\sqrt{d} k^{1/d}$, the diameter of $Q_i$.  So when $\alpha=\infty$, the same calculation as in \eqref{eq:1/2-infty} shows that they are connected with probability $p$, so the graph on $\CV_{\CQ_i}[w_{\mathrm{hh}}, 2w_{\mathrm{hh}})$ dominates $\mathrm{ER}(4s_k, p)$. 
Using \eqref{eq:connection-prob-gen}, for $\alpha<\infty$,
 \begin{equation}
  \mathrm{p}\big((x_u,w_u),(x_v, w_v)) \ge p\big(1\wedge (\beta  w_\mathrm{hh}^{1+\sigma}/(d^{d/2}k))\big)^\alpha. \nonumber
 \end{equation}
 If the minimum is at the first term, then again the graph on $\CV_{\CQ_1}[w_{\mathrm{hh}}, 2w_{\mathrm{hh}})$ dominates  $\mathrm{ER}(4s_k, p)$. Otherwise, if the minimum is at the second term, we compute the mean degree using \eqref{eq:w-gamma-hh}:
 \[ 
 \begin{aligned}
 4 d^{-\alpha d/2}s_k\cdot p\beta^\alpha  w_\mathrm{hh}^{(1+\sigma)\alpha}k^{-\alpha} &=  (p/4) \beta^{\alpha} d^{-\alpha d/2}  k^{1-\alpha} w_\mathrm{hh}^{(1+\sigma)\alpha-(\tau-1)} \\
 &=(p/4) \beta^{\alpha} d^{-\alpha d/2} C_1^{-((1+\sigma)\alpha - (\tau-1))/(\tau-1)}\ge \lambda_\star(1/2),
 \end{aligned}
 \]
by the definition of $C_1$ in \eqref{eq:upper-hh-c1}, since the powers of $k$ cancelled each other.
Hence, $\CH\succcurlyeq \mathrm{ER}\big(4s_k), \lambda_\star(1/2)/(4s_k)\big)=\mathrm{ER}_\star$, and the size of the largest connected component in $\CH$, denoted by $C^\sss{(1)}$ below, stochastically dominates the size of the largest component in $\mathrm{ER}_\star$.
 We apply a large-deviation principle for the size of the giant component in ERRGs obtained by O'Connell~\cite{oconnel1998}, see also~\cite{andreis2021er}.
 Denote by $\CC^\sss{(1)}(m,\lambda/m)$ the largest component of $\mathrm{ER}(m,\lambda/m)$, and recall that $\rho_\lambda$ is the survival probability of a Bienaym\'{e}--Galton--Watson branching process with $\mathrm{Poi}(\lambda)$ offspring. By~\cite[Theorem~3.1]{oconnel1998}, for every $\lambda > 1$ and $\tilde \varepsilon > 0$, there exists a constant $c_\lambda>0$ such that for each $m\ge 1$,
\begin{equation}\nonumber
\Prob\big(|\CC^\sss{(1)}(m,\lambda/m)|<(1-\tilde \varepsilon)\varrho_\lambda m\big)\le \exp\big(-c_\lambda m\big).
\end{equation}
Now, recall that by definition of $\lambda_\star(1/2)$, the survival probability of the branching process is $1/2$.
Applying the previous inequality with $\tilde \varepsilon=1/2$ to $\CC^{\sss{(1)}}$, we obtain that 
\begin{equation}\nonumber
\begin{aligned}
\Prob\Big(|\CC^{\sss{(1)}}(4s_k, \lambda_\star(1/2))|\le \frac12  \varrho_{\lambda_\star(1/2)} (4s_k)\Big) &= \Prob\big(|\CC^{\sss{(1)}}(4s_k, \lambda_\star(1/2))|\le s_k\big) \\
&\le \exp( - c_{\lambda_\star(1/2)} 4s_k).
\end{aligned}
\end{equation}
 Since 
 the number of boxes is $n'/k=\lfloor(n/k)^{1/d}\rfloor ^d\ge 1$ whenever $n\ge k$, we get 
 \begin{equation*}
  \Prob(\neg \CA_\mathrm{init} \mid \CA_{\mathrm{poi}})  \le \Prob( \CC^{\sss{(1)}}(4s_k, \lambda_\star(1/2))\le s_k \mid\CA_\mathrm{poi}) \le \exp(-c_{\lambda_\star(1/2)} 4 s_k)         \end{equation*}
 When combined with \eqref{eq:upper-hh-bb1}, \eqref{eq:upper-hh-bb-1st},  and \eqref{eq:third-term-in-abb}, and that $s_k=k^{\zeta_{\mathrm{hh}}}(C_1/16)$ in \eqref{eq:w-gamma-hh}, this  yields the statement of the lemma in \eqref{eq:lem-hh-nobb}.
\end{proof}
We will end Step 1 with a claim that shows \eqref{eq:q-bound}. We start by introducing a notation for the construction of the graph $\CG_n$ that facilitates later steps.
We recall the definition of KSRG from Definition~\ref{def:ksrg}.
Given the vertex set $\CV$, it is standard practice to use independent uniform random variables to facilitate couplings with the edge set. This definition here is more general and allows for other auxiliary random variables as well, leading to different distributions on graphs. This will be useful later.
\begin{definition}[Graph encoding]\label{def:encoding}
 Let $\CV\subset \R^d\times[1,\infty)$ be a discrete set and assume that
 $\Psi_{\CV}=\big\{\varphi_{u,v}:  \varphi_{u,v}\in[0,1], \{u,v\} \in \binom{\CV}{2}\big\}$ is a collection of random variables given $\CV$.
 For a given connectivity function $\mathrm{p}: (\R^d\times[1,\infty))^2\to [0,1]$,
 we call  $\CG'=(\CV', \CE')$ the (sub)graph \emph{encoded by} $(\CV, \Psi_\CV,\mathrm{p})$  if $\CV'=\CV$ and for all $\{u,v\}\in \binom{\CV}{2}$, with $u=(x_u, w_u), v=(x_v, w_v)$,
 \begin{equation}\label{eq:phi-edge-rule}
  \big\{\{u,v\}\in \CE'\big\} \Longleftrightarrow \big\{\varphi_{u,v}\le \mathrm{p}\big((x_u, w_u),(x_v, w_v)\big)\big\}.
 \end{equation}
Given $\CV$ in \eqref{eq:poisson-intensity},
  and $\mathrm{p}$ from \eqref{eq:connection-prob-gen},  let $\Psi_\CV$ be a collection of independent  $\mathrm{Unif}[0,1]$ random variables given $\CV$. $\CG_\infty$ in Definition~\ref{def:ksrg} is then the graph \emph{encoded by} $(\CV,\Psi_\CV, \mathrm{p})$.
  Writing $\Psi_n[a,b):=\big\{\varphi_{u,v}\in \Psi_\CV:\{u,v\}\in \binom{\CV_n[a,b)}{2}\}$ and $\Psi_n:=\Psi_n[1,\infty)$, $\CG_n$ in \eqref{eq:gn-ab} is then the graph encoded by $(\CV_n, \Psi_n,\mathrm{p})$.
    \end{definition}
An immediate corollary is the following.
\begin{corollary}\label{cor:graph-equal-law}
 Assume  $\widetilde\CG,\widehat\CG $ are two random graphs, encoded respectively by $(\widetilde\CV, \widetilde\Psi,\mathrm{p})$, and $(\widehat\CV, \widehat\Psi,\mathrm{p})$ for respective point processes $\widetilde\CV, \widehat \CV$ on $\R^d\times[1,\infty)$ using the \emph{same} connectivity function $\mathrm{p}$. If $(\widetilde\CV, \widetilde\Psi)$ and $(\widehat\CV, \widehat\Psi)$ have the same law then the encoded graphs
 $\widetilde\CG$ and $\widehat\CG$ also
 have the same law.
\end{corollary}

The collection of (conditionally) independent \emph{uniform} variables $\Psi_n=\{\varphi_{u,v}: \{u,v\}\in \binom{\CV_n}{2}\}$ and the connectivity function $\mathrm{p}$ determine the presence of edges in $\CG_n$. By \eqref{eq:phi-edge-rule}, if $\varphi_{u,v}\le r\le \mathrm{p}(u, v)$ for some $r>0$, then $\{u\sim v\}$.
Writing $\CQ(u)$ for the box containing or closest to $u\in\CV_n$ (see \eqref{eq:qu}),  let $v_u(1), v_u(2),\dots, v_u(s_k), \dots$ denote the vertices in $\CQ(u)\cap\CC_\mathrm{bb}$, in decreasing order with respect to their marks. Let
\begin{equation}
 \CS(u):=\big\{v_u(1),\dots,v_u(s_k)\big\}.\label{eq:su}
\end{equation}
\begin{claim}[Connections to the backbone]\label{claim:connecting-to-bb}
 Consider a KSRG under the same assumptions as in Theorem~\ref{thm:subexponential-decay}, with vertex set either a homogeneous Poisson point process or $\Z^d$. 
 Fix $n\ge k$ for any $k\ge k_1$ in \eqref{eq:upper-hh-k0} and assume $\CG_{n,1}$ satisfies the event $\CA_\mathrm{bb}(n,k)$.
 Let $\Psi_n=\{\varphi_{u,v}: \{u,v\}\in \binom{\CV_n}{2}\}$ be a collection of iid $\mathrm{Unif}[0,1]$ random variables and $r_k:=1-2^{-1/s_k}$. Then, for all $u\in\CV_n[2w_\mathrm{hh}(k),\infty)$ and $v\in\CS(u)$, $\mathrm{p}(u, v) \ge r_k$ and
 \begin{equation}
  \Prob\big(\forall v\!\in\!\CS(u): \varphi_{u,v}\!>\! r_k\mid \CG_{n,1}, \CA_\mathrm{bb}\big)=\Prob\big(\exists v\!\in\!\CS(u): \varphi_{u,v}\!\le\! r_k \mid \CG_{n,1}, \CA_\mathrm{bb}\big)=1/2. \label{eq:lem-hh-conn-qk}
 \end{equation}
 \begin{proof}
  On the event $\CA_{\mathrm{bb}}$, $\CC_{\mathrm{bb}}\subseteq \CG_{n,1}$ satisfies \eqref{eq:upper-hh-bb-event} and in particular $\CS(u)$ in \eqref{eq:su} is well-defined and has size $s_k$. Since $\{\varphi_{u,v}\}$ is a collection of iid $\mathrm{Unif}[0,1]$ random variables, (cf.\ Definition~\ref{def:encoding}), one \emph{must} set $r_k:=1-2^{-1/s_k}$ for  \eqref{eq:lem-hh-conn-qk} to hold. Hence, it only remains to show $\mathrm{p}(u,v)\ge r_k$ in the statement. We show this somewhat implicitly, using calculations we did  around \eqref{eq:1/2-infty}--\eqref{eq:upper-hh-bb-half}.

  With $\CQ(u)$ and $\CS(u)$ from \eqref{eq:qu} and \eqref{eq:su}, respectively, by \eqref{eq:upper-hh-outside-tesselation}, every $u\in\CV_n[2w_\mathrm{hh},\infty)$ is at distance at most $2\sqrt{d}k^{1/d}$ from any vertex in $v\in \CS(u)$.
  Since $w_u\ge 2w_{\mathrm{hh}}\ge w_{\mathrm{hh}}$, and $|\CS(u)|=s_k$,  the computations \eqref{eq:1/2-infty}--\eqref{eq:upper-hh-bb-half}  carry word-by-word through with $\widetilde \CV_i$ replaced by $\CS(u)$, obtaining
  \[ \Prob( u \sim \CS(u)\mid \CG_{n, 1}, \CV_n, \CA_\mathrm{bb}) = 1-\prod_{v\in \CS(u)} (1-\mathrm{p}(u,v)) \ge 1-(1-z_k)^{s_k}\ge 1/2, \]
  with $z_k$ either equaling $p $ in the right-hand side of \eqref{eq:1/2-infty} or the appropriate expression in the right-hand side of \eqref{eq:temp-est-1}, that bounds individually each $\mathrm{p}(u, v)$ from below. Following now the calculations towards \eqref{eq:upper-hh-bb-half} ensures that in both cases
  $z_k\ge 1-2^{-1/s_k}$. The assumption $k\ge k_1$  in \eqref{eq:upper-hh-k0} is  needed when $z_k=p$, and it implies that $r_k \le p$, see around \eqref{eq:1/2-infty}. 
 \end{proof}
\end{claim}
\subsection{Step 2. Revealing low-mark vertices}
Having established that $\CG_{n,1}$ contains a backbone with the right error probability, we define $\CG_{n,2}:=\CG_{n}[1,2w_\mathrm{hh})\supseteq \CG_{n,1}$.

\subsection{Step 3. Presampling the vertices connecting to the backbone}
We make Step 3 of Section~\ref{sec:outline-hh} precise now. Step 3 ensures that during Step 4 below no \emph{small-to-large} merging occurs when revealing the connector vertices of $\CV_n[2 w_{\mathrm{hh}}, \infty)$. That is, components of size  smaller than $k$ do not merge into a larger component via edges to a vertex $v\in \CV_n[2w_{\mathrm{hh}}, \infty)$ that is not connected to the backbone $\CC_\mathrm{bb}$ ($\CC_{\mathrm{bb}}$ will be contained in the giant component of $\CG_n$). So, we partially pre-sample some randomness that encodes the presence of some edges.

For a pair $n,k$, we now present the alternative graph-encoding $\widehat\CG_n$ of KSRGs (cf.\ Definitions \ref{def:ksrg} and \ref{def:encoding}) and verify that $\widehat \CG_n$ and  $\CG_n$ in Definition~\ref{def:ksrg} have the same law. The difference between the encoding in Definition~\ref{def:encoding} and the construction of $\widehat\CG_n$ is that in the latter the edge-variables $\varphi_{u,v}$ are no longer independent $\Unif[0,1]$ random variables, but are  sampled from a suitable (conditional) joint distribution, whenever $u\in\CV_n[2w_\mathrm{hh}(k),\infty)$ and  $v\in \CS(u)$ from \eqref{eq:su}. Recall $r_k=1-2^{-1/s_k}$ from Claim \ref{claim:connecting-to-bb}, with $w_{\mathrm{hh}}(k):=w_{\mathrm{hh}}$ and $s_k$ defined in \eqref{eq:w-gamma-hh}.
\begin{definition}[Alternative graph construction]\label{def:ksrg-alt} Fix $n$ and $k$.
 Consider the subgraph $\CG_{n,2}\!=\!\CG_n[1,2w_\mathrm{hh}(k))$ of $\CG_n$ from Definition~\ref{def:ksrg}, on a vertex set formed by a homogeneous Poisson point process. Assume $\CG_{n,2}$ is 
 encoded by $\big(\CV_n[1,2w_\mathrm{hh}), \Psi_n[1,2w_\mathrm{hh}), \mathrm{p}\big)$. Let
 $ \widehat\CV_n^\sss{(\mathrm{unsure})}[2w_\mathrm{hh},\infty)$ and $\widehat\CV_n^\sss{(\mathrm{sure})}[2w_\mathrm{hh},\infty)$ be two \emph{independent} Poisson point processes
 on $\Lambda_n\times[2w_\mathrm{hh},\infty)$, each with intensity $(1/2)\mathrm{Leb}\otimes F_W(\rd w)$, with  $F_W$ as in \eqref{eq:poisson-intensity}. Define
 \begin{equation}
  \widehat\CV_n[2w_\mathrm{hh},\infty):=\widehat\CV_n^\sss{(\mathrm{unsure})}[2w_\mathrm{hh},\infty)\cup\widehat\CV_n^\sss{(\mathrm{sure})}[2w_\mathrm{hh},\infty).\label{eq:hat-xin}
 \end{equation}
 Let $\Sigma_n:=\{U_{u,v}: u\in \widehat\CV_n[2w_\mathrm{hh},\infty), v\in \CV_n[1, 2w_\mathrm{hh})\cup\widehat\CV_n[2w_\mathrm{hh},\infty)\}$ be a collection of iid $\mathrm{Unif}[0,1]$ random variables (conditionally on these PPPs).

 (i) If $\CG_{n,1}=\CG_n[w_{\mathrm{hh}},2w_\mathrm{hh})\subseteq \CG_{n,2}$ does not satisfy the event $\CA_{\mathrm{bb}}$ in \eqref{eq:upper-hh-bb-event}, then set $\widehat \Psi_n:=\Psi_n[1,2w_{\mathrm{hh}})\cup \Sigma_n$ in Definition~\ref{def:encoding} to construct $\widehat\CG_n\supseteq \CG_{n,2}$ on $\CV_n[1,2w_\mathrm{hh})\cup \widehat\CV_n[2w_\mathrm{hh},\infty)$, i.e.,
 \begin{equation}\widehat \CG_n:=(\CV_n[1,2w_\mathrm{hh})\cup \widehat\CV_n[2w_\mathrm{hh},\infty), \Psi_n[1,2w_{\mathrm{hh}})\cup \Sigma_n, \mathrm{p}).\nonumber
 \end{equation}

 (ii) If $\CG_{n,1}\subseteq \CG_{n,2}$  satisfies the event $\CA_{\mathrm{bb}}$, then we construct $\widehat\CG_n\supseteq \CG_{n,2}$ \emph{conditionally on} $\CG_{n,2}$ as follows.
 For each $u\in \widehat\CV_n[2w_\mathrm{hh},\infty)$ in \eqref{eq:hat-xin}, the set of vertices $\CS(u)\subseteq \CV_n[1,2w_{\mathrm{hh}})$ is a deterministic function of $\CG_{n,1}\subseteq \CG_{n,2}$, given by \eqref{eq:su}.
 Let
 \begin{equation}
  \begin{aligned}
  \widehat\Psi_n^{\sss{(\mathrm{iid}, \mathrm{unsure})}} & :=\big\{ U_{u,v}: u\in \widehat\CV_n^{\sss{\mathrm{(unsure)}}}[2w_{\mathrm{hh}}, \infty), v\in  \widehat\CV_n^{\sss{\mathrm{(unsure)}}}[2w_{\mathrm{hh}}, \infty)\cup \CV_n[1,2w_{\mathrm{hh}})\setminus \CS(u) \big\}, \\
  \widehat\Psi_n^{\sss{(\mathrm{iid}, \mathrm{sure})}}   & :=\big\{ U_{u,v}: u\in \widehat\CV_n^{\sss{\mathrm{(sure)}}}[2w_{\mathrm{hh}}, \infty), v\in  \widehat\CV_n^{\sss{\mathrm{(sure)}}}[2w_{\mathrm{hh}}, \infty)\cup \CV_n[1,2w_{\mathrm{hh}})\setminus \CS(u) \big\},     \\
  \widehat\Psi_n^{\sss{(\mathrm{iid}, \mathrm{both})}}   & :=\big\{ U_{u,v}: u\in \widehat\CV_n^{\sss{\mathrm{(sure)}}}[2w_{\mathrm{hh}}, \infty), v\in  \widehat\CV_n^{\sss{\mathrm{(unsure)}}}[2w_{\mathrm{hh}}, \infty)\big\}
  \end{aligned}\label{eq:psi-iid}
 \end{equation}
 be disjoint subsets of $\Sigma_n$, and write $\widehat\Psi_n^{\sss{(\mathrm{iid})}}:=\widehat\Psi_n^{\sss{(\mathrm{iid}, \mathrm{unsure})}}\cup \widehat\Psi_n^{\sss{(\mathrm{iid}, \mathrm{sure})}}\cup \widehat\Psi_n^{\sss{(\mathrm{iid}, \mathrm{both})}}$ for the union.
 Conditionally on $\widehat\CV_n^\sss{(\mathrm{unsure})}[2w_\mathrm{hh},\infty), \widehat\CV_n^\sss{(\mathrm{sure})}[2w_\mathrm{hh},\infty)$ and $\CG_n[1,2w_\mathrm{hh})$, define also the collections of random variables
 \begin{align}
  \widehat\Psi_n^{\sss{(\mathrm{cond, unsure})}} & :=\big\{\widehat\varphi_{u,v}: u\in \widehat\CV_n^\sss{(\mathrm{unsure})}[2w_\mathrm{hh},\infty), v\in \CS(u) \big\}, \label{eq:xi-cond-unsure} \\
  \widehat\Psi_n^{\sss{(\mathrm{cond, sure})}}   & :=\big\{\widehat\varphi_{u,v}: u\in \widehat\CV_n^\sss{(\mathrm{sure})}[2w_\mathrm{hh},\infty), v\in \CS(u) \big\}, \label{eq:xi-cond-sure}
 \end{align}
 so that for different vertices $u_1, u_2 \in\widehat\CV_n[2w_\mathrm{hh},\infty)$, the collections $\{\widehat\varphi_{u_1,v}\}_{v\in\CS(u_1)}$ and $\{\widehat\varphi_{u_2,v'}\}_{v' \in \CS(u_2)}$ are \emph{independent}. The joint distribution of $\{\widehat\varphi_{u,v}\}_{v\in\CS(u)}$ for a single $u\in\widehat\CV_n^\sss{(\mathrm{unsure})}[2w_\mathrm{hh},\infty)$ is as follows: for \emph{any} sequence $(z_{u,v})_{v \in \CS(u)}\in[0,1]^{s_k}$ of length $s_k$, and with $r_k=1-2^{-1/s_k}$, 
 \begin{equation}
  \begin{aligned}
  \Prob\big(\forall v\in\CS(u): \widehat\varphi_{u,v}\le z_{u,v} & \mid u\in\widehat\CV_n^\sss{(\mathrm{unsure})}[2w_\mathrm{hh},\infty)\big) \\
                                                                  & :=
  \Prob\big(\forall v\in\CS(u): U_{u,v}\le z_{u,v}\mid \forall v\in\CS(u): U_{u,v}>r_k\big).
  \end{aligned}\label{eq:unif-unsure}
 \end{equation}
 Similarly we define the joint distribution of $\{\widehat\varphi_{u,v}\}_{v\in\CS(u)}$ for a single $u\in\widehat\CV_n^\sss{(\mathrm{sure})}[2w_\mathrm{hh},\infty)$ as follows: for any sequence $(z_{u,v})_{v \in \CS(u)}\in[0,1]^{s_k}$ of length $s_k$,
 \begin{equation}
  \begin{aligned}
  \Prob\big(\forall v\in\CS(u): \widehat\varphi_{u,v}\le z_{u,v} & \mid u\in\widehat\CV_n^\sss{(\mathrm{sure})}[2w_\mathrm{hh},\infty)\big) \\
                                                                  & :=
  \Prob\big(\forall v\in\CS(u): U_{u,v}\le z_{u,v}\mid \exists v\in\CS(u): U_{u,v}\le r_k\big).
  \end{aligned}\label{eq:unif-sure}
 \end{equation}
 We define $\widehat\CG_{n}$ as the graph  encoded by $(\widehat\CV_n,\widehat\Psi_n, \mathrm{p})$, where
 \begin{equation}
  \begin{aligned}
  \widehat\CV_n  & :=\CV_n[1,2w_\mathrm{hh})\cup \widehat\CV_n^\sss{(\mathrm{unsure})}[2w_\mathrm{hh},\infty)\cup\widehat\CV_n^\sss{(\mathrm{sure})}[2w_\mathrm{hh},\infty),                  \\
  \widehat\Psi_n & :=\Psi_n[1,2w_\mathrm{hh})\cup \widehat\Psi_n^{ \sss{(\mathrm{iid})}}\cup \widehat\Psi_n^{\sss{(\mathrm{cond, unsure})}}\cup \widehat\Psi_n^{\sss{(\mathrm{cond, sure})}}.
  \end{aligned}\label{eq:gn-hat}
 \end{equation}
\end{definition}
An immediate corollary is the following statement.
\begin{corollary}\label{cor:sure-bb}
 Consider a KSRG $\widehat \CG_n$ from Definition~\ref{def:ksrg-alt} for some $n,k$.
 On the event $\CA_\mathrm{bb}(n,k)$, every vertex in $\widehat\CV_n^\sss{(\mathrm{sure})}[2w_\mathrm{hh},\infty)$ is connected by an edge to $\CC_\mathrm{bb}(n,k)$ in $\widehat\CG_n$.
\end{corollary}
\begin{proof}
 The conditioning in \eqref{eq:unif-sure} guarantees that for each $u\in \widehat\CV_n^\sss{(\mathrm{sure})}[2w_\mathrm{hh},\infty)$ at least one $\widehat \varphi_{u,v}\le r_k$ occurs among the edge-variables $\{\widehat \varphi_{u,v}: v\in \CS(u)\}$, where $\CS(u)\subseteq \CC_{\mathrm{bb}}$, see~\eqref{eq:su}.
 Then since $\widehat \varphi_{u,v}\le r_k\le \mathrm{p}(u,v)$ holds by Claim \ref{claim:connecting-to-bb}, this ensures that $\{u,v\}$ is in the edge set of $\widehat\CG_n$ by the graph-encoding in Definition~\ref{def:encoding}.
\end{proof}
\begin{proposition}
 Fix a connectivity function $\mathrm{p}$. The law of the random graph $\widehat\CG_n$ in Definition~\ref{def:ksrg-alt} is identical to the law of the random graph $\CG_n$ in Definition~\ref{def:ksrg}.
\end{proposition}
\begin{proof}
 By Corollary \ref{cor:graph-equal-law} it is sufficient to show that $(\widehat\CV_n,\widehat\Psi_n)$ defined in \eqref{eq:gn-hat} has the same distribution as $(\CV_n, \Psi_n)$ from Definitions \ref{def:ksrg} and \ref{def:encoding}. 
 By \eqref{eq:gn-hat} in Definition~\ref{def:ksrg-alt}, the graph $\CG_{n,2}$ spanned on $\CV_n[1,2w_{\mathrm{hh}})\subseteq \widehat\CV_n$ is determined by $\Psi_n[1,2w_{\mathrm{hh}})=\{\varphi_{u,v}: u,v \in \CV_n[1,2w_{\mathrm{hh}})\}$ in Definition~\ref{def:encoding}. Thus $\CG_{n,2}$ has the same distribution both in Definition~\ref{def:encoding} and in Definition~\ref{def:ksrg-alt}.

 (i) If now $\Psi_n[1,2w_{\mathrm{hh}})$ is such that the graph $\CG_{n,2}$ does \emph{not} satisfy the event $\CA_\mathrm{bb}$, by (i) of Definition~\ref{def:ksrg-alt}, the statement holds since both $\{\varphi_{u,v}\}$ and $\{U_{u,v}\}$ are iid uniforms whenever $u\in \widehat\CV_n[2w_{\mathrm{hh}}, \infty)$, i.e., $\widehat\Psi_n\setminus\Psi_n[1,2w_{\mathrm{hh}})= \Sigma_n$ and $\Psi_n\setminus\Psi_n[1,2w_{\mathrm{hh}})$ have the same distribution.

 (ii) If $\Psi_n[1,2w_{\mathrm{hh}})$ is such that the graph $\CG_{n,2}$ does satisfy the event $\CA_\mathrm{bb}$, then we work conditionally on a realization of the graph $\CG_{n,2}=(\CV_n[1,2w_{\mathrm{hh}}), \Psi_n[1,2w_{\mathrm{hh}}), \mathrm{p})$, and also on the coupled realization of the PPPs $\CV_n[2w_\mathrm{hh},\infty)=\widehat\CV_n[2w_\mathrm{hh},\infty)$. Let us define the conditional probability measure (of the edges) under the coupling by
 \begin{equation}\label{eq:pstar}
  \Prob^\star(\cdot):=\Prob(\,\cdot\mid \CG_{n,2}, \CV_n[2w_\mathrm{hh},\infty))=\Prob(\,\cdot\mid \CG_{n,2}, \mbox{unlabeled } \widehat\CV_n[2w_\mathrm{hh},\infty) ),
 \end{equation}
 where in the conditioning we do \emph{not} reveal to which sub-PPP (either $\widehat\CV_n^\sss{(\mathrm{sure})}[2w_\mathrm{hh},\infty)$ or $\widehat\CV_n^\sss{(\mathrm{unsure})}[2w_\mathrm{hh},\infty)$) a vertex in $\widehat\CV_n[2w_\mathrm{hh},\infty)$ belongs to.
 Using $\widehat\Psi_n$ from \eqref{eq:gn-hat} and  $\widehat\Psi_n^\sss{(\mathrm{iid})}\subseteq \Sigma_n$  from \eqref{eq:psi-iid} (containing \emph{independent} copies $U_{u,v}$ of $\mathrm{Unif}[0,1]$ random variables, like $\Psi$ in Definition~\ref{def:encoding}), we see that variables in $\Psi_n\setminus \Psi_n[1,2w_{\mathrm{hh}})$ and $\widehat \Psi_n\setminus \Psi_n[1,2w_{\mathrm{hh}})$ also share the same (joint) law of iid $\mathrm{Unif}[0,1]$ whenever  $u$ and $v$ are such $u\in \widehat \CV_n[2w_{\mathrm{hh}, \infty})$ and that $v\notin \CS(u)$.
 Moreover, in \eqref{eq:xi-cond-unsure}-\eqref{eq:xi-cond-sure}, the collections $\{\widehat\varphi_{u,v}\}_{v\in\CS(u)}$ are independent across $u$ for different vertices $u\in\widehat{\CV}_n[2w_\mathrm{hh},\infty)$.
 So for $\widehat\CG_n\ {\buildrel d \over=}\ \CG_n$ it remains to show that  for any  $u\in\widehat{\CV}_n[2w_\mathrm{hh},\infty)=\CV_n[2w_\mathrm{hh},\infty)$, under the measure $\Prob^\star$,
 \begin{equation}\label{eq:coupling_su}
  \big\{\varphi_{u,v}, v\in\CS(u)\big\} \overset{d}=\big\{\widehat\varphi_{u,v},  v\in\CS(u)\big\}.
 \end{equation}
 We first analyze the distribution of the left-hand side, i.e., $\varphi_{u,v}$ being iid from Definition~\ref{def:encoding}.
 Let $(z_{u,v})_{v\in\CS(u)}\in[0,1]^{s_k}$ be any sequence of length $s_k$.
 By Claim \ref{claim:connecting-to-bb}, and the law of total probability
  \begin{align}
  \Prob^\star\big(\forall v\in\CS(u): \varphi_{u, v}\!\le\! z_{u,v}\big)
    & =
  (1/2)\Prob^\star\big(\forall v\in\CS(u): \varphi_{u, v}\!\le\! z_{u,v}\mid  \forall v\in\CS(u): \varphi_{u, v}\!>\! r_k \big) \nonumber\\
    & \hspace{15pt}+
  (1/2)\Prob^\star\big(\forall v\in\CS(u): \varphi_{u, v}\!\le\! z_{u,v}\mid  \exists v\in\CS(u): \varphi_{u, v}\!\le\! r_k \big).\label{eq:dist-phi}
  \end{align}
 We now analyze the right-hand side in \eqref{eq:coupling_su}. By the construction in  \eqref{eq:hat-xin}, $\widehat\CV_n[2w_\mathrm{hh},\infty)$ is the union of two iid sub-PPPs. Under $\mathbb P^\star$ in \eqref{eq:pstar} we did not reveal to which sub-PPP vertices belong to. Hence, for each $u \in \widehat\CV_n[2w_\mathrm{hh},\infty)$, independently of each other
 \[
 \begin{aligned}\Prob^\star(u \in \widehat\CV_n^\sss{(\mathrm{unsure})}[2w_\mathrm{hh},\infty) &\mid u \in \widehat\CV_n[2w_\mathrm{hh},\infty))\\&=\Prob^\star(u \in \widehat\CV_n^\sss{(\mathrm{sure})}[2w_\mathrm{hh},\infty) \mid u \in \widehat\CV_n[2w_\mathrm{hh},\infty))=1/2.
 \end{aligned}
 \]
 Thus, by the law of total probability, and using the distributions of $(\widehat\varphi_{u,v})_{v\in \CS(u)}$ given by~\eqref{eq:unif-unsure}, ~\eqref{eq:unif-sure},
  \begin{align}
  \Prob^\star\big(\forall v\in\CS(u): \widehat\varphi_{u,v}\le z_{u,v}\big) & =
  (1/2) \Prob^\star\big(\forall v\in\CS(u): \widehat\varphi_{u,v}\le z_{u,v}\mid u \in \widehat\CV_n^\sss{(\mathrm{unsure})}[2w_\mathrm{hh},\infty)\big) \nonumber\\
                                                                             & \hspace{15pt}+(1/2)
  \Prob^\star\big(\forall v\in\CS(u): \widehat\varphi_{u,v}\le z_{u,v}\mid u \in \widehat\CV_n^\sss{(\mathrm{sure})}[2w_\mathrm{hh},\infty)\big)         \nonumber\\
                                                                             & =(1/2)
  \Prob^\star\big(\forall v\in\CS(u): U_{u,v}\le z_{u,v}\mid  \forall v\in\CS(u): U_{u,v}\!>\! r_k \big)                                                     \nonumber\\
                                                                             & \hspace{15pt}+(1/2)
  \Prob^\star\big(\forall v\in\CS(u): U_{u,v}\le z_{u,v}\mid  \exists v\in\CS(u): U_{u,v}\!\le\! r_k \big).\label{eq:dist-varphi}
  \end{align}
 Note that $\{U_{u,v}\}_{u,v}$ and $\{\varphi_{u,v}\}_{u,v}$ are both sets of independent $\mathrm{Unif}[0,1]$ random variables by  Definitions \ref{def:ksrg-alt} and \ref{def:encoding}, respectively. Hence, \eqref{eq:coupling_su} follows by combining \eqref{eq:dist-phi} and~\eqref{eq:dist-varphi}.
\end{proof}

For the remainder of this section, we  construct $\CG_n$ following Definition~\ref{def:ksrg-alt} and write
\begin{equation}
 \CV_n[2w_\mathrm{hh},\infty):=\CV_n^\sss{(\mathrm{unsure})}[2w_\mathrm{hh},\infty)\cup\CV_n^\sss{(\mathrm{sure})}[2w_\mathrm{hh},\infty)\nonumber
\end{equation}
as the union of two independent PPPs of equal intensity, such that if $\CG_{n,2}=\CG[1,2w_\mathrm{hh})$ satisfies $\CA_\mathrm{bb}$ in \eqref{eq:upper-hh-bb-event}, each vertex in $\CV_n^\sss{(\mathrm{sure})}[2w_\mathrm{hh},\infty)$ connects by an edge to $\CC_\mathrm{bb}$ by Corollary~\ref{cor:sure-bb}.
To finish Step 3, on the event $\CA_{\mathrm{bb}}$, we define $\CG_{n,3}:=(\CV_{n,3}, \Psi_{n,3}, \mathrm{p})
$, with
\begin{equation}\label{eq:gn3}
 \begin{aligned}
  \CV_{n,3}&:=\CV_n[1,2w_{\mathrm{hh}})\cup \CV_n^\sss{(\mathrm{unsure})}[2_\mathrm{hh},\infty), \\
  \Psi_{n,3}&:=\Psi_n[1,2w_{\mathrm{hh}}) \cup  \widehat\Psi_n^{\sss{(\mathrm{iid}, \mathrm{unsure})}} \cup  \widehat\Psi_n^{\sss{(\mathrm{cond}, \mathrm{unsure})}},
 \end{aligned}
\end{equation}
i.e., the graph spanned on $\CV_{n,3}$.
We call the vertices in $\CV_n^\sss{(\mathrm{sure})}[2w_\mathrm{hh},\infty)$ \emph{sure-connector} vertices. If the event $A_{\mathrm{bb}}$ does not hold then we say that the construction failed and we leave $\CG_{n,3}$ undefined.
\subsection{Step 4. Cover expansion}In this step, we ensure that all components of size at least $k$ of $\CG_{n, 3}$ merge with the giant component of $\CG_n$ via edges towards sure-connector vertices, with error probability $\mathrm{err}_{n,k}$ from \eqref{eq:outline-error-prob}.  The next lemma proves this using the cover-expansion technique of Section~\ref{sec:cover-expansion}. The notion of expandability is from Definition~\ref{def:expandable}, and recall $s(\underline w)$ from  \eqref{eq:prop-cover-expansion-min-weight} that describes the necessary ``expandability parameter'' in Proposition~\ref{prop:cover-expansion-deterministic}, and $w_{\mathrm{hh}}$ from \eqref{eq:w-gamma-hh}. 
Define $k_2$ as 
\begin{equation}\label{eq:k2}
k_2:=\min\{k\in \N: 2w_\mathrm{hh}(k)=2C_1^{-1/(\tau-1)}k^{\gamma_{\mathrm{hh}}}>(2^d d^{d/2}/\beta\vee 1) \}
\end{equation}
so that the function $s(\cdot)$ is defined at $2w_{\mathrm{hh}}$. Slightly abusing notation, we say that a vertex set $\CV$ is $s$-expandable if the set of locations $(x_u)_{u\in \CV}$ is $s$-expandable.  
Define 
\begin{equation}
 \CA_\mathrm{exp}:=\CA_\mathrm{exp}(n,k):=\big\{\CV_{n,3} \mbox{ is $s(2w_{\mathrm{hh}})$-expandable}\big\}.\label{eq:1c-expandable}
\end{equation}
Recall that $k \ge k_1$ in \eqref{eq:upper-hh-k0} is necessary to build the backbone in Lemma~\ref{lemma:upper-hh-bb}.
\begin{lemma}[Cover-expansion]\label{lemma:hh-expandable}
 Consider a KSRG under the same assumptions as in Theorem~\ref{thm:subexponential-decay}, with vertex set a homogeneous Poisson point process. 
If $k\ge k_2$, then with $s(\cdot)$ from \eqref{eq:prop-cover-expansion-min-weight}, for some constant $c_{\ref{lemma:hh-expandable}}>0$,
 \begin{equation}
\Prob\big(\neg\CA_\mathrm{exp})\big) \le C_{\ref{lem:expandable-ppp}}n\exp(-s(2w_{\mathrm{hh}})/3)\le C_{\ref{lem:expandable-ppp}} n \exp(-c_{\ref{lemma:hh-expandable}} k^{1/(\sigma+1-(\tau-1)/\alpha)}). \label{eq:hh-expandable-prob}
 \end{equation}
 Moreover, 
 conditionally on any realization of $\CG_{n,3}$ satisfying  $\CA_\mathrm{bb}\cap\CA_\mathrm{exp}$,
 for all $k\ge (k_1\vee k_2)$ in \eqref{eq:upper-hh-k0}, \eqref{eq:k2} and any connected component $\CC$ of $\CG_{n,3}$ with $|\CC|> k$, 
 \begin{equation}
\Prob\big(\CC\not\sim\CV_n^\sss{(\mathrm{sure})}[2w_\mathrm{hh},\infty)\mid \CG_{n,3},  \CA_\mathrm{bb}\cap\CA_\mathrm{exp}\big) \le \exp\big(-c_{\ref{lemma:hh-expandable}} k^{\zeta_{\mathrm{hh}}}\big).\label{eq:lem-hh-exp}
 \end{equation}
\end{lemma}
\begin{proof}
 The statement \eqref{eq:hh-expandable-prob} follows directly from Lemma~\ref{lem:expandable-ppp}, by computing $s(2w_{\mathrm{hh}})$ using \eqref{eq:prop-cover-expansion-min-weight} and \eqref{eq:w-gamma-hh}, and $\gamma_{\mathrm{hh}}$ from \eqref{eq:gamma-hh}:
 \begin{equation*}
\begin{aligned}
s(2w_{\mathrm{hh}})=\big(2^{d+1}\beta C_1^{-1/(\tau-1)}\big)^{1/(1-1/\alpha)}k^{\gamma_{\mathrm{hh}}/(1-1/\alpha)}\le c_{\ref{lemma:hh-expandable}} k^{1/(\sigma+1-(\tau-1)/\alpha)}.
 \end{aligned}
 \end{equation*}
 We proceed to the proof of \eqref{eq:lem-hh-exp}.
 In Proposition~\ref{prop:cover-expansion-deterministic}, for a given mark $\underline w$, the function $s(\underline w)$ in \eqref{eq:prop-cover-expansion-min-weight} describes the necessary ``expandability parameter'', such that all vertices with mark at least $\underline w$ in $\CK_n(\CL)$ connect to any $s(\underline w)$-expandable set $\CL$ of vertices with probability at least $p/2$. We shall take $\underline w:=2w_{\mathrm{hh}}(k)$, the lowest possible mark in $\CV_n^{\sss{(\mathrm{sure})}}$. If $k\ge k_2$ in \eqref{eq:k2}, $\underline w$ satisfies the required lower bound in the statement of Proposition~\ref{prop:cover-expansion-deterministic}.
 
 On the event $\CA_{\mathrm{exp}}$, $\CV_{n,3}$ is thus $s(2w_{\mathrm{hh}})$-expandable. Since expandability carries through for subsets of $\CV_{n,3}$ (see below Definition~\ref{def:expandable}), \emph{any} subset of $\CV_{n,3}$ is $s(2w_{\mathrm{hh}})$-expandable. Hence, Proposition~\ref{prop:cover-expansion-deterministic} is applicable for any set $\CL\subseteq\CV_{n,3}$ and $\underline w:=2w_{\mathrm{hh}}$, and guarantees the existence of a set $\CK_n(\CL)\subseteq\Lambda_n$ satisfying \eqref{eq:prop-cover-expansion-min-volume} and \eqref{eq:p-connection}.

 Consider an arbitrary connected component $\CC$ of $\CG_{n,3}$ that satisfies $|\CC|> k$. With $\CK_n(\CC)$ from Proposition~\ref{prop:cover-expansion-deterministic}, we define the set of sure-connector vertices with location in $\CK_n(\CC)$ connected by an edge to $\CC$ as
 \begin{equation}
  \CH_{\CC}:=  \{v\in\CK_n(\CC)\cap\CV_{n}^\sss{(\mathrm{sure})}[2w_\mathrm{hh},\infty)\big): v\sim\CC\}.\label{eq:hs}
 \end{equation}
 Since $\CV^\sss{(\mathrm{sure})}[2w_\mathrm{hh},\infty)$ is a Poisson process, its cardinality in $\CK_n(\CC)$ follows a Poisson distribution. Since each of these vertices connects by an edge \emph{independently} to $\CC$ with probability at least $p/2$ by \eqref{eq:p-connection}, and an independent thinning of a PPP is another PPP,
 we obtain using the intensity measure in Definition~\ref{def:ksrg-alt} and the volume bound \eqref{eq:prop-cover-expansion-min-volume} on $\CK_n(\CC)$ for $|\CC|> k$ and $c:=(p/4)2^{-(4d+1)}2^{-(\tau-1)}d^{-d/2}/\re>0$,
 \begin{equation}
  \begin{aligned}
  \Prob\big(|\CH_\CC| = 0
  \mid
  \CG_{n,3},\,  &\CA_\mathrm{bb} \cap \CA_\mathrm{exp}
  \big) \\
    & \le
  \Prob\big(\Poi\big((p/2)\cdot(1/2)\cdot \mathrm{Vol}(\CK_n(\CC))\cdot(2w_\mathrm{hh})^{-(\tau-1)}\big)=0\big) \\
    & \le
  \exp\big(-(p/4)\mathrm{Vol}(\CK_n(\CC))(2w_\mathrm{hh})^{-(\tau-1)}\big)\big)                                 \\
    & \le
  \exp\big(-c\cdot kw_\mathrm{hh}^{-(\tau-1)}\big)=\exp\big(-16c\cdot s_k\big),\nonumber
  \end{aligned}
 \end{equation}
 where we used $s_k$ from \eqref{eq:w-gamma-hh} in the last step. Since $\{|\CH_{\CC}|>0\}$ in  \eqref{eq:hs} implies that $\{\CC \sim \CV_n^\sss{(\mathrm{sure})}[2w_\mathrm{hh},\infty)\}$, this finishes the proof of \eqref{eq:lem-hh-exp} for some constant $c_{\ref{lemma:hh-expandable}}>0$.
\end{proof}

\subsection*{Combining everything: preventing too large components}
\begin{proof}[Proof of Proposition~\ref{prop:2nd-upper-bound-hh}]
 Assume that    $k\ge k_1\vee k_2$  and $n\ge k$  holds. We construct $\CG_n\supseteq \CG_{n,3}$ following Definition~\ref{def:ksrg-alt}, where $\CG_{n,3}$ from \eqref{eq:gn3} is the subgraph of $\CG_n$ induced on $\CV_{n,3}=\CV_n[1,2w_{\mathrm{hh}})\cup \CV_n^\sss{(\mathrm{unsure})}[2w_\mathrm{hh},\infty) $.
 The events $\CA_{\mathrm{bb}}$  in \eqref{eq:upper-hh-bb-event} and $\CA_{\mathrm{exp}}$ in \eqref{eq:1c-expandable} are measurable with respect to $\CG_{n,3}$.  
  So,
 by the law of total probability (taking expectation over realizations of $\CG_{n,3}$), we obtain
 \begin{equation}
 \begin{aligned}
  \Prob\big(|\CC_{n}^\sss{(2)}| > k\big)
  & \le
  \E\big[\ind{\CA_{\mathrm{bb}}\cap\CA_\mathrm{exp}} \Prob\big(|\CC_{n}^\sss{(2)}| > k \mid \CG_{n,3}, \CA_\mathrm{exp}\cap \CA_\mathrm{bb}\big)\big]\\&\hspace{15pt} +
  \Prob\big(\neg\CA_\mathrm{bb}\big) + \Prob\big(\neg\CA_\mathrm{exp}\big).\label{eq:upper-hh-pr1}
 \end{aligned}
 \end{equation}
 Lemma~\ref{lemma:upper-hh-bb} applies since $k\ge k_1, n\ge k$, so $\Prob\big(\neg\CA_\mathrm{bb}\big) \le 3(n/k)\exp(-c_{\ref{lemma:upper-hh-bb}} k^{\zeta_{\mathrm{hh}}})$. The bound \eqref{eq:hh-expandable-prob} in Lemma~\ref{lemma:hh-expandable} applies to the third term  since $k\ge k_1\vee k_2$.  
Using  \eqref{eq:zeta-hh}, one  may verify that $1/(\sigma+1-(\tau-1)/\alpha)\ge \zeta_{\mathrm{hh}}$ if and only if $\tau\ge\sigma+1$.
 Thus,
 for some $c_{\mathrm{exp}}>0$, 
\begin{equation}\label{eq:sigma-issue}
 \Prob\big(\neg\CA_\mathrm{exp}\big) \le C_{\ref{lem:expandable-ppp}} n \exp\Big(- c_{\mathrm{exp}} \big(k^{\zeta_{\mathrm{hh}}} \ind{\tau\ge \sigma+1} + \ind{\tau<\sigma+1} k^{1/(\sigma+1-(\tau-1)/\alpha)}\big)\Big).
  \end{equation}
   We proceed to bounding the first term in \eqref{eq:upper-hh-pr1}.
 The not-yet-revealed vertices after Step 3 are $\CV_n\setminus \CV_{n,3}=\CV^\sss{(\mathrm{sure})}_n[2w_\mathrm{hh},\infty)$, and by Corollary \ref{cor:sure-bb} each vertex in $\CV^\sss{(\mathrm{sure})}_n[2w_\mathrm{hh},\infty)$ connects by an edge to $\CC_{\mathrm{bb}}$. Thus each component $\CC\nsupseteq \CC_{\mathrm{bb}}$ of $\CG_{n,3}$ either remains the same in $\CG_n$ or it merges with the component containing $\CC_{\mathrm{bb}}$ by connecting to a vertex in $\CV\setminus \CV_{n,3}$. 
 If all components of size at least $k$ in $\CG_{n,3}$ merge with the backbone, then there is at most one component above size $k$, and so the second-largest component is not larger than $k$. 
 Hence, if the second-largest component has size larger than $k$, there must be at least one connected component $\CC$ of size larger than $k$ in $\CG_{n,3}$ that does not connect by an edge to $\CV^\sss{(\mathrm{sure})}_n[2w_\mathrm{hh},\infty)$.
Formally, conditionally on $\CA_\mathrm{bb}$ and $\CG_{n,3}$, we have
 \begin{equation}
  \big\{|\CC^\sss{(2)}_n|> k\big\}\subseteq
  \{ \exists \mbox{ a component } \CC \mbox{ of }\CG_{n,3} \mbox{ with } |\CC|> k: \CC\not\sim\CV^\sss{(\mathrm{sure})}_n[2w_\mathrm{hh},\infty)\}.\label{eq:sure-connection}
 \end{equation}
 By a union bound over the at most $|\CV_{n,3}|/k$ components of size at least $k$, \eqref{eq:lem-hh-exp} of  Lemma~\ref{lemma:hh-expandable} yields
\begin{equation}
 \begin{aligned}
 \E\big[\ind{\CA_{\mathrm{bb}}\cap\CA_\mathrm{exp}} &\Prob\big(|\CC_{n}^\sss{(2)}| > k \mid \CG_{n,3}, \CA_\mathrm{exp}\cap \CA_\mathrm{bb}\big)\big]
  \\&\le \E\big[\ind{\CA_{\mathrm{bb}}\cap\CA_\mathrm{exp}} (|\CV_{n,3}|/k)\exp\big(-c_{\ref{lemma:hh-expandable}}k^{\zeta_{\mathrm{hh}}}\big)\big] \le (n/k) \exp(-c_{\ref{lemma:hh-expandable}} k^{\zeta_{\mathrm{hh}}}),\nonumber
  \end{aligned}
 \end{equation}
 since $\CV_{n,3}\subseteq\CV_n$ by construction, and $\E[|\CV_n|]=n$ by \eqref{eq:poisson-intensity}.
 Substituting this bound into \eqref{eq:upper-hh-pr1}, and using Lemma~\ref{lemma:upper-hh-bb} and \eqref{eq:sigma-issue} to bound the second and the third term yields that for $k\ge k_1\vee k_2$ and $n\ge k$ when $\tau\ge\sigma+1$,
 \begin{equation}\label{eq:cn2-final-bound}
  \Prob\big(|\CC_{n}^\sss{(2)}| > k\big)
  \le 
(C_{\ref{lem:expandable-ppp}}n + 3n/k + n/k ) \exp\big(-\min( c_{\ref{lemma:hh-expandable}}, c_{\ref{lemma:upper-hh-bb}}, c_{\mathrm{exp}})k^{\zeta_{\mathrm{hh}}}\big).
 \end{equation}
This finishes the proof of
 Proposition~\ref{prop:2nd-upper-bound-hh} for $\tau\ge \sigma+1$, $k\ge k_1\vee k_2$ and $n\ge k$.  For $k<k_1\vee k_2$, \eqref{eq:prop-2nd-upper-bound} is trivially satisfied for $c_{\ref{prop:2nd-upper-bound-hh}}>0$ sufficiently small. Finally, for $\tau<\sigma+1$, the only change is that the bound on $\Prob(\neg \CA_{\mathrm{exp}})$ in \eqref{eq:sigma-issue} becomes the leading order error term in \eqref{eq:upper-hh-pr1}, which is of order $n\exp(-\Theta(k^{1/(\sigma+1+(\tau-1)/\alpha)}))$. 
\end{proof}
\subsection*{The backbone: intermediate results}
We state two corollaries of the proof of Proposition~\ref{prop:2nd-upper-bound-hh}, and two propositions  based on the backbone constructions for later use.  We start with a corollary of the proof of Proposition~\ref{prop:2nd-upper-bound-hh}.
\begin{corollary}[Backbone becoming part of the giant]\label{cor:bb-in-giant}

 Consider a KSRG under the same assumptions as in Theorem~\ref{thm:subexponential-decay}, with vertex set formed by a homogeneous Poisson point process. Assume that $n>Ak^{2-\zeta_{\mathrm{hh}}}$ for some constant $A=A(\sigma, \tau, \alpha, d, \beta).$
 Then conditionally on the graph $\CG_{n,2}=\CG_n[1,2w_\mathrm{hh})$ satisfying $\CA_\mathrm{bb}(n,k)$ in \eqref{eq:upper-hh-bb-event}, if $\tau \ge \sigma+1$,
 \begin{equation}
\Prob\big(\CC_\mathrm{bb}(n,k)\nsubseteq\CC_n^\sss{(1)}\mid \CG_{n,2}, \CA_\mathrm{bb}(n,k)\big)\le (n/k)\exp\big(-c_{\ref{prop:2nd-upper-bound-hh}}k^{\zeta_\mathrm{hh}}\big). \label{eq:lem-hh-bb-giant}
 \end{equation}
  For $\tau<\sigma+1$, 
  the inequality holds with exponent $1/(\sigma+1-(\tau-1)/\alpha)$ in place of $\zeta_{\mathrm{hh}}$.
 \begin{proof}[Proof]Lemma~\ref{lemma:upper-hh-bb} constructs the backbone $\CC_{\mathrm{bb}}$, with size at least $s_kn/(2k) \ge k$ by definition of $s_k=\Theta(k^{\zeta_{\mathrm{hh}}})$ in \eqref{eq:w-gamma-hh} and by the lower bound  $n>Ak^{2-\zeta_\mathrm{hh}}$.  Using the complement of the event on the right-hand side of \eqref{eq:sure-connection}, if all components of size above $k$ of $\CG_{n,3}$ merge with the backbone, then there is at most one component above size $k$, which is the component containing the backbone. The right-hand side of \eqref{eq:cn2-final-bound} exactly bounds this event.
 \end{proof}
\end{corollary}
The next corollary follows from Lemma~\ref{lemma:upper-hh-bb}.
It is not sharp but it yields a useful estimate.
\begin{corollary}[Lower bound on largest component]\label{cor:lower-c1}
 Consider a KSRG under the same assumptions as in Theorem~\ref{thm:subexponential-decay}, with vertex set either formed by a homogeneous Poisson point process or $\Z^d$. For each $\delta>0$, there exists a constant $A>0$ such that for all $n$ sufficiently large
 \begin{equation}
     \Prob\big(|\CC_n^{\sss{(1)}}|\le n(A\log n)^{1-1/\zeta_\mathrm{hh}}\big)\le n^{-\delta}.\nonumber
 \end{equation}
 \begin{proof}
 If $\CA_\mathrm{bb}(n,k_n)$ holds for some $k_n$, then the largest component $\CC_n^\sss{(1)}$ must have at least the size of the backbone $\CC_{\mathrm{bb}}(n,k_n)$. Setting $k=k_n=(A\log n)^{1/\zeta_\mathrm{hh}}$ in \eqref{eq:lem-hh-nobb}, the backbone exists with probability at least $1-n^{-\delta}$ for $A=A(\delta)$ sufficiently large, since $s_{k}=s_{k_n}=(C_1/16)A\log n$ by \eqref{eq:w-gamma-hh}. Then its size is at least $(n'/k) s_k=\Theta(n(A\log n)^{1-1/\zeta_\mathrm{hh}})$ by definition of $n'$ in \eqref{eq:nprime-boxes}, finishing the proof.
 \end{proof}
\end{corollary}

The next proposition identifies the mark-threshold $\overline w$ so that (with polynomially small error probability) \emph{all} vertices with mark above $\overline w$ belong to the largest component $\CC_{n}^{\sss{(1)}}$. \begin{proposition}[Controlling marks of non-giant vertices]\label{prop:min-giant-weight}
 Consider a KSRG under the same assumptions as in Theorem~\ref{thm:subexponential-decay}, in particular $\tau\in(2,2+\sigma)$, with vertex set formed by a homogeneous Poisson point process. When  $\tau\ge \sigma+1$, for \emph{all} $\delta>0$, there exists $M_\delta>0$ such that for  $\overline{w}(n, \delta)=(M_\delta\log n)^{(1-\sigma\gamma_\mathrm{hh})/\zeta_\mathrm{hh}}$
 \begin{equation}
  \Prob\big(\exists u\in\CV_{n}[\overline{w}(n,\delta),\infty): u\notin\CC_{n}^\sss{(1)}\big) \le n^{-\delta}.\label{eq:min-giant-weight}
 \end{equation}
 When $\tau<\sigma+1$, 
  the same bound holds with $\overline{w}(n, \delta)=(M_\delta\log n)^{(1-\sigma\gamma_\mathrm{hh})(\sigma+1-(\tau-1)/\alpha)}$.
\begin{proof}[Proof sketch] We give the detailed proof in Appendix~\ref{app:aux} on page \pageref{proof:min-giant-weight}, and here a sketch when $\tau\ge\sigma+1$. We consider $k$ as a free parameter, so using Lemma~\ref{lemma:upper-hh-bb} with  $k=k_n=\Theta\big((\log n)^{1/\zeta_\mathrm{hh}}\big)$,
  a backbone $\CC_{\mathrm{bb}}(n,k_n)$ exists and satisfies $\CC_{\mathrm{bb}}(n,k_n)\subseteq \CC_{n}^{\sss{(1)}}$, with probability at least $1-n^{-\delta}$ by Corollary \ref{cor:bb-in-giant} and Lemma~\ref{lemma:upper-hh-bb} (the same calculation as the proof of Corollary \ref{cor:lower-c1}). We choose $\overline{w}=\overline{w}(n,\delta)$ to be the lowest possible value so that a vertex $u$ with mark $w_u\ge \overline{w}(n,\delta)$ connects by an edge to each backbone-vertex in its own subbox with probability at least $p$ in \eqref{eq:connection-prob-gen}. Recall also $s_k=(C_1/16)k^{1-\gamma_\mathrm{hh}(\tau-1)}$. For $u$ to not be contained in $\CC_{n}^{\sss{(1)}}$, these $s_{k_n}=\Theta(\log n)$ many edges must be all absent, which happens with probability $(1-p)^{s_{k_n}}=o(n^{-\delta-1})$. A union bound over the $O(n)$ such vertices finishes the proof.
 \end{proof}
\end{proposition}
\begin{remark}
 Combined with the proof of the lower bound of Theorem~\ref{thm:second-largest} below in Section~\ref{sec:lower}, one may show that Proposition~\ref{prop:min-giant-weight} is sharp up to a constant factor when $\tau\ge \sigma+1$, i.e., there exist constants $\delta, m_w>0$ such that for all $n$ sufficiently large \[\Prob\big(\exists v\in\CV_n[(m_w\log n)^{(1-\sigma\gamma_\mathrm{hh})/\zeta_\mathrm{hh}}, \infty): v\notin\CC_{n}^\sss{(1)}\big)\ge 1-n^{-\delta}.
 \]
\end{remark}
Let $\CC_n(0)[1,w)$ be the component containing $0$ in  $\CG_n[1, w)\subseteq \CG_n$, by setting $\CC_{n}(0)[1,w)$ to be the empty set if $w_0\ge w$. Then $\CC_{n}(0)[1, 2w_{\mathrm{hh}}(k))$ is the component of $0$ in $\CG_{n,2}$. In the next proposition, we show that this component has linear size with strictly positive probability when the truncation is at $2w_{\mathrm{hh}}(k_n)=\Theta((\log n)^{\gamma_{\mathrm{hh}}/\zeta_{\mathrm{hh}}})$, equivalently, when $k_n=\Theta((\log n)^{1/\zeta_{\mathrm{hh}}})$. 

\begin{proposition}[Existence of a large component]\label{proposition:existence-large}
  Consider a KSRG under the same assumptions as in Theorem~\ref{thm:subexponential-decay}, with vertex set either formed by a homogeneous Poisson point process or $\Z^d$. Then there exists a unique infinite component in $\CG$. Moreover, there exist constants $\rho, m>0$ such that for all $n$ sufficiently large, when $k_n=m(\log n)^{1/\zeta_\mathrm{hh}}$,
 \begin{align}                  \Prob^\sss{0}\big(|\CC_{n}(0)[1, 2w_{\mathrm{hh}}(k_n))|\ge \rho n\big)
  \ge
  \rho,
  \qquad
  \mbox{and}\qquad
  \Prob^\sss{0}(0\leftrightarrow\infty)\ge \rho.
    \label{eq:prop-subexp-lower-comp}
 \end{align} 

 \begin{proof}[Proof sketch]
 We build a connected backbone in $\Lambda_n$ on vertices with mark in the interval $[w_\mathrm{hh}(k_n), 2w_\mathrm{hh}(k_n))$ using Lemma~\ref{lemma:upper-hh-bb}. Then we use a second-moment method to show that the origin and linearly many other vertices are connected to this backbone via paths along which the vertex marks are increasing, giving the first inequality in \eqref{eq:prop-subexp-lower-comp}. The second inequality follows similarly, forming an infinite path along which the marks are increasing.
 The detailed proof can be found in Appendix~\ref{app:linear-sized}. 
 \end{proof}
\end{proposition}

\section{Upper bound: cluster-size decay}\label{sec:upper-sub}
In this section we prove Theorem~\ref{thm:subexponential-decay}(ii)--(iii).
We carry out the plan in Section~\ref{sec:second-sub} in detail. Instead of restricting to  KSRGs with parameters described in Theorem~\ref{thm:subexponential-decay}(ii--iii), we derive general conditions that give subexponential decay. Then we  show that Propositions \ref{prop:2nd-upper-bound-hh} and \ref{prop:min-giant-weight} imply these conditions. 
 Recall from Definition~\ref{def:ksrg} that $\Prob^\sss{x}$ denotes the conditional measure that $\CV$ contains a vertex at location $x$, with an unknown mark from distribution $F_W$. All results of this section hold for KSRGs on $\Z^d$.
\begin{proposition}[Prerequisites for cluster-size decay]\label{prop:condition-subexponential}
 Consider a KSRG satisfying Assumption~\ref{assumption:main} with parameters $\alpha>1, \tau>2$, $\sigma\ge 0$, and $d\in\N$.
 Assume that there exist $c_1\ge 0$ and $\zeta, \eta, c_2, c_3, M>0$, and a function $n_0(k)=O(k^{1+c_1})$, such that for all $k$ sufficiently large, and whenever $n \in [n_0(k), \infty)$, with
 $\overline{w}(n):=M(\log n)^\eta$,
 \begin{align}
  \Prob^\sss{x}\big(|\CC_{n}^\sss{(2)}| > k\big)
                                           & \le
  n^{c_2}\exp\big(-c_3k^{\zeta}\big),\label{eq:prop-2nd}                                                                                  \\ \Prob^\sss{x}\big(|\CC_{n}^\sss{(1)}|\le n^{c_3}\big)                  & \le  n^{-1-c_3},\label{eq:prop-outgiant}\\
  \Prob^\sss{x}\big(\exists v\in\CV_{n}[\overline{w}(n),\infty): v\notin\CC_{n}^\sss{(1)}\big) & \le n^{-c_3}.\label{eq:prop-marksgiant}
 \end{align}
 Then there exists a constant $A>0$ such that for all $k$ sufficiently large constant and $n$ satisfying $n \in [n_0(k), \infty]$,
 \begin{equation}\label{eq:exp-decay-71}
  \Prob^\sss{0}\big(|\CC_n(0)|> k, 0\notin\CC_n^\sss{(1)}\big)\le \exp\big(-(1/A)k^{\zeta}\big).
 \end{equation}
 Further, the Weak Law of Large Numbers holds:
 \begin{equation}
|\CC_n^\sss{(1)}|\,\big/\,n\overset\Prob\longrightarrow \Prob^\sss{0}(0\leftrightarrow\infty),\qquad \mbox{as $n\to\infty$.}\label{eq:lln-section}
 \end{equation}
 
\end{proposition}
Observe that \eqref{eq:exp-decay-71} does not follow from a naive application of \eqref{eq:prop-2nd}, since the polynomial prefactor on the right-hand side of \eqref{eq:prop-2nd} vanished in \eqref{eq:exp-decay-71}, and $n=\infty$ is also allowed in \eqref{eq:exp-decay-71}. The inequalities \eqref{eq:prop-2nd}--\eqref{eq:prop-marksgiant} are satisfied when $\tau\in(2, 2+\sigma)$ and $\tau>\sigma+1$ by Propositions \ref{prop:2nd-upper-bound-hh} and \ref{prop:min-giant-weight} and Corollary \ref{cor:lower-c1} (we leave it to the reader to verify that the results hold also for the Palm-version $\mathbb{P}^\sss{x}$ of $\Prob$). Thus, Theorem~\ref{thm:subexponential-decay}(ii)--(iii) follow immediately after we prove Proposition~\ref{prop:condition-subexponential}. We prove an intermediate claim that we need for Proposition~\ref{prop:condition-subexponential}.  
We work under the Palm measure, i.e., $\CV$ contains a vertex $u$ at location $x$ with unknown mark. We write $\CC_{\CQ}^\sss{(1)}$ for the largest component in the graph induced on vertices in a set $\CQ\subseteq\R^d$.
\begin{claim}[Leaving the giant]\label{claim:leaving-giant}
 Consider a KSRG satisfying Assumption~\ref{assumption:main} with parameters $\alpha>1, \tau>2$, $\sigma\ge 0$, and $d\in\N$.  Assume that \eqref{eq:prop-2nd}--\eqref{eq:prop-marksgiant} hold. Then there exists $\delta>0$ such that for 
$n$ sufficiently large, and all $N\in[n,\infty]$, it holds for $u:=(x, w_u)$ and any box $\CQ_n$ of volume $n$ inside $\Lambda_N$ that \begin{equation}
     \Prob^\sss{x}\big(u\in\CC_{\CQ_n}^\sss{(1)}, u\notin\CC_{N}^\sss{(1)}\big) \le n^{-\delta}.\label{eq:claim-remain-giant-2}
 \end{equation}
 \end{claim}
 \begin{proof}
 We will first prove the following bound that holds generally for a sequence of increasing (nested) graphs $G_n\subseteq G_{n+1}\subseteq \dots$, whose largest and second-largest components we denote by $C_n^\sss{(1)}$ and $C_n^\sss{(2)}$, respectively. 
Let $(k_n)_{n\ge 0}$ and $(K_n)_{n\ge 0}$ be two non-negative sequences such that $k_{n+1}<K_n$ for all $n\ge 0$. Then, for all $0<n\le N\le \infty$,
 \begin{equation}
     \Prob\big(u\in C_n^\sss{(1)}, u\notin C_{N}^\sss{(1)}\big)\le \sum_{\tilde n=n}^N \Big(\Prob\big(|C_{\tilde n}^\sss{(1)}|<K_{\tilde n}\big) + \Prob\big(|C_{\tilde n}^\sss{(2)}|>k_{\tilde n}\big)\Big).\label{eq:claim-remain-giant-1}
 \end{equation} 
 We verify the bound using an inductive argument.
 We define for $\tilde n\ge n$ the events
 \[
 \CA(\tilde n):=\big\{|C_{\tilde{n}}^\sss{(1)}|\ge K_{\tilde n}\big\}\cap \big\{|C_{\tilde{n}+1}^\sss{(2)}|\le k_{\tilde n+1}\big\}.
  \]
 Since by assumption $k_{\tilde n+1}< K_{\tilde n}$, the event $\CA(\tilde n)$ ensures that $|C_{\tilde n}^\sss{(1)}|$ is already larger than $|C_{\tilde{n}+1}^\sss{(2)}|$. Thus, $\CA(\tilde n)$
 implies that $C_{\tilde n}^\sss{(1)}\subseteq C_{\tilde n+1}^\sss{(1)}$. 
  Iteratively applying this argument yields that $\cap_{\tilde n \in [n, N-1]} \CA(\tilde n)$ implies that $\{C_{ n}^\sss{(1)}\subseteq C_{N}^\sss{(1)}\}$. We combine this with the observation that $u\in \CQ_n$ implies that $\{u\in C_n^\sss{(1)}, u\notin C_{N}^\sss{(1)}\}\subseteq\{C_{n}^\sss{(1)}\nsubseteq C_N^\sss{(1)}\}$. This yields
 \begin{align}
 \Prob\big(u\in C_n^\sss{(1)}, u\notin C_{N}^\sss{(1)}\big)\le
 \Prob\big(C_{n}^\sss{(1)}\nsubseteq C_N^\sss{(1)}\big) &\le
 \Prob\bigg(\{C_{n}^\sss{(1)}\nsubseteq C_N^\sss{(1)}\}\cap\bigcap_{\tilde n=n}^{N-1}\CA(\tilde n)\bigg) + \sum_{\tilde n=n}^{N-1} \Prob\big(\neg\CA(\tilde n)\big) \nonumber\\&\le 0 + \sum_{\tilde n=n}^N \Big(\Prob\big(|C_{\tilde n}^\sss{(1)}|<K_{\tilde n}\big) + \Prob\big(|C_{\tilde n}^\sss{(2)}|>k_{\tilde n}\big)\Big),\label{eq:leaving-giant-pr}
 \end{align}
 showing \eqref{eq:claim-remain-giant-1}. We move on to \eqref{eq:claim-remain-giant-2} for which we have to define the increasing sequence of graphs. Consider any sequence of boxes $(\CQ_{\tilde{n}})_{\tilde n\ge n}$ such that $x\in \CQ_n\subseteq\CQ_{n+1}\subseteq\dots\subseteq\CQ_N:=\Lambda_N$ and $\mathrm{Vol}(\CQ_{\tilde {n}})=\tilde n$, and let $G_{\tilde n}$ denote the induced subgraph of $\CG$ on $\CQ_{\tilde n}$ for $\tilde n\in [n, N]$. 
 We use the translation invariance of KSRGs, and the assumed lower bound on $|\mathcal{C}_{\tilde n}^\sss{(1)}|$ in \eqref{eq:prop-outgiant}, and we set $K_{\tilde n}:=\tilde n^{c_3}$. We also use the assumed upper bound on $|\CC_{\tilde n}^\sss{(2)}|$ in \eqref{eq:prop-2nd}. Then  if we set $
     k_{\tilde n}:=(A\log \tilde n)^{1/\zeta}$ for a sufficiently large $A$, then for all sufficiently large $n$ and $\tilde n\ge n$,
 \begin{equation}
     \Prob^\sss{x}\big(|\CC_{\CQ_{\tilde n}}^\sss{(1)}|<K_{\tilde n}\big)\le \tilde{n}^{-c-1}, \qquad \Prob^\sss{x}\big(|\CC_{\CQ_{\tilde n}}^\sss{(2)}|>k_{\tilde n}\big) \le \tilde n^{-c-1}.\nonumber
 \end{equation}
 Clearly $k_{\tilde n+1}<K_{\tilde n}$ for all $\tilde n\ge n$, so that substituting the bounds into \eqref{eq:leaving-giant-pr} and summing over $\tilde n\ge n$ yields the assertion \eqref{eq:claim-remain-giant-2} for any $\delta<c$ and $n$ sufficiently large.
 \end{proof}
We continue to prove Proposition~\ref{prop:condition-subexponential}, starting with some notation.
For some $\varepsilon\in (0, \min(1, \alpha-1, \tau-2))$, and using $c_2,c_3, \eta, \zeta, M$, and $\overline{w}(n)$ from  Proposition~\ref{prop:condition-subexponential}, let
\begin{equation}
\begin{aligned}
N_k&:= \exp(k^{\zeta}c_3/(2c_2)), & 
 n_k&:=N_k^\varepsilon, \\
 \overline{w}_{N_k} &:= \overline{w}(N_k)=M(k^\zeta)^\eta (\tfrac{c_3}{2c_2})^\eta, &
 t_k&:=n_k^{1/d}/(2k).
 \label{eq:nk-Nk}
 \end{aligned}
\end{equation}
Note that $N_k, n_k=\exp(\Theta(k^\zeta))$.
For $n\le N_k$, the statement \eqref{eq:exp-decay-71} follows directly from \eqref{eq:prop-2nd}, since when $n=N_k$, the right-hand side  of \eqref{eq:prop-2nd} becomes $\exp(- k^\zeta c_3/2)$, so we may choose any $A$ such that $1/A\le c_3/2$ in \eqref{eq:exp-decay-71}. 
In the remainder of the section we focus on $n>N_k$.
We write $\CC_{\Lambda(x, n)}(u)$ for the component of vertex $u:=(x, w_u)\in \CV$ in the graph $\CG_\infty$ restricted to $\Lambda(x, n)$.  Define for $x\in\R^d$ the two events
\begin{align}
 \CA_\mathrm{low\textnormal{-}edge}(x, n_k, N_k, \overline{w}_{N_k})  & :=\!\left\{
 \begin{aligned}
  & |\CC_{\Lambda(x, n_k)}(u)|\le k, \exists v_1, v_2\in\CV_{\Lambda(x, N_k)}[1, \overline{w}_{N_k}) \mbox{ s.t.} \\
  & v_1\!\in\!\CC_{\Lambda(x, n_k)}(u), \|x_{v_1}-x_{v_2}\|\ge t_k,\mbox{ and }v_1\sim v_2\end{aligned}
 \right\}, \label{eq:event-low-edge}                              \\
 \!\CA_\mathrm{long\textnormal{-}edge}(x, n_k, N_k, \overline{w}_{N_k}) \!& :=\!
 \big\{\exists v_1\!\in\!\CV_{\Lambda(x, n_k)}[1,\overline{w}_{N_k}), \exists v_2\!\in\!\CV\!\setminus\!\CV_{\Lambda(x, N_k)}, v_1\!\sim\! v_2\big\}.\label{eq:event-long-edge}
\end{align}
The next lemma relates the probability of the event $\{|\CC_n(u)|> k, u\notin\CC_n^\sss{(1)}\}$ to the events $\CA_\mathrm{low\textnormal{-}edge}$ and $\CA_\mathrm{long\textnormal{-}edge}$ using the assumed bounds in Proposition~\ref{prop:condition-subexponential}.
\begin{lemma}[Extending the box-sizes]\label{prop:refinedevents}
 Consider a KSRG satisfying Assumption~\ref{assumption:main} with parameters $\alpha>1, \tau>2$, $\sigma\ge 0$, and $d\in\N$.  Assume that \eqref{eq:prop-2nd}--\eqref{eq:prop-marksgiant} hold. Consider any $x\in \Lambda_n$ with $\|x-\partial\Lambda_n\|\ge N_k^{1/d}/2$, and $u=(x, w_u)$.
  Then there exists $A'>0$ such that for all $n$ with $n\in[N_k,\infty]$, 
  \begin{align}
  \Prob^\sss{x}\big(|\CC_n(u)| > k  , u\notin\CC_n^\sss{(1)}\big)  &\le
  \exp\big(-(1/A')k^{\zeta}\big)
  +\Prob^\sss{0}\big(\CA_\mathrm{low\textnormal{-}edge}(0, n_k, N_k, \overline{w}_{N_k})\big) \nonumber\\
  &\hspace{15pt}+
  \Prob^\sss{0}\big(\CA_\mathrm{long\textnormal{-}edge}(0, n_k, N_k, \overline{w}_{N_k})\big).\label{eq:prop-refined}
  \end{align}
\end{lemma}
\begin{proof}
 Let $\tilde n\in [1, n]$, and denote by $\CC_{\Lambda(x, \tilde n)}^\sss{(1)}$ the largest connected component in the induced subgraph of $\CG_\infty$ inside the box $\Lambda(x, \tilde n)\subseteq\Lambda_n$. For a vertex $u=(x, w_u)\in \CV$ define
  \begin{align}
  \CA_\mathrm{leave\textnormal{-}giant}(x, \tilde n)&:=
  \{u\in\CC_{\Lambda(x,\tilde n)}^\sss{(1)}, u\notin\CC_{n}^\sss{(1)}\}, \label{eq:leave-giant-1}\\  
  \CA_\mathrm{mark\textnormal{-}giant}(x, N_k,\overline{w}_{N_k})&:=\big\{\forall v\in\CV_{\Lambda(x, N_k)}[\overline{w}_{N_k},\infty): v\in\CC_{\Lambda(x, N_k)}^\sss{(1)}\big\}.  \label{eq:event-giant}
 \end{align}
 The first event relates to \eqref{eq:claim-remain-giant-2} in Claim \ref{claim:leaving-giant}, while the second one to \eqref{eq:prop-marksgiant} of Proposition~\ref{prop:condition-subexponential}. 
 The values of  $n_k<N_k\le n$ from \eqref{eq:nk-Nk} and the assumption  $\|x-\partial\Lambda_n\|\ge N_k^{1/d}/2$ ensure that $\Lambda(x, n_k)\!\subseteq\!\Lambda(x, N_k)\!\subseteq \!\Lambda_n$. Then we bound
  \begin{align}
  \big\{|\CC_n(u)|> \,&k  , u\notin\CC_n^\sss{(1)}\big\} \nonumber                                                     \\
                        & \subseteq
  \big\{|\CC_n(u)|> k, u\notin\CC_n^\sss{(1)}, u\notin \CC_{\Lambda(x, n_k)}^\sss{(1)}, |\CC_{\Lambda(x, n_k)}^\sss{(2)}| > k\big\} \cup  \big\{u\in\CC_{\Lambda(x, n_k)}^\sss{(1)}, u\notin\CC_n^\sss{(1)}\big\} \nonumber
  \\
                        & \hspace{15pt}\cup
  \big\{|\CC_n(u)|> k, u\notin\CC_n^\sss{(1)}, u\notin \CC_{\Lambda(x, n_k)}^\sss{(1)}, |\CC_{\Lambda(x, n_k)}^\sss{(2)}| \le k\big\} \nonumber \\
                        &                         \begin{aligned}
  &\subseteq\big\{|\CC_{\Lambda(x, n_k)}^\sss{(2)}| > k\big\} \, \cup \, \CA_\mathrm{leave\textnormal{-}giant}(x,n_k) \\&\hspace{15pt} \cup\,
  \big\{|\CC_n(u)|> k, u\notin\CC_n^\sss{(1)}, u\notin \CC_{\Lambda(x, n_k)}^\sss{(1)}, |\CC_{\Lambda(x, n_k)}(u)| \le k\big\}.
  \end{aligned}\label{eq:upper-sub-an-1}
 \end{align}
 Applying probabilities on both sides we obtain the inequality stated in~\eqref{eq:outline-upper-second-sub-last} for $x\!=\!0$. 
 We introduce a shorthand notation for the third event on the right-hand side of  \eqref{eq:upper-sub-an-1}, i.e.,
 \begin{equation}\nonumber
  \CA_{\mathrm{goal}}:= \big\{|\CC_n(u)|> k, u\notin\CC_n^\sss{(1)}, u\notin \CC_{\Lambda(x, n_k)}^\sss{(1)}, |\CC_{\Lambda(x, n_k)}(u)|\le k\big\}.
 \end{equation}
 Define the auxiliary events
 \begin{equation}
 \begin{aligned}
  \CA_\mathrm{becomes\textnormal{-}large}&:=\{|\CC_{\Lambda(x, n_k)}(u)|\le k, |\CC_n(u)|> k\},\\
  \CA_\mathrm{out\textnormal{-}of\textnormal{-}giant}(n_k,n)&:=\{u\notin\CC_{n}^\sss{(1)}, u\notin\CC_{\Lambda(x, n_k)}^\sss{(1)}\}, \label{eq:events-edge-out-giant}
  \end{aligned}
 \end{equation}
 and observe that $\CA_{\mathrm{goal}}  =
  \CA_\mathrm{becomes\textnormal{-}large}\cap \CA_\mathrm{out\textnormal{-}of\textnormal{-}giant}(n_k,n).$
 In order to bound $\Prob(\CA_{\mathrm{goal}})$, we distinguish whether $u$ enters the giant at the intermediate box  of size $N_k\in(n_k, n)$ or not:
   \begin{align}
  \CA_{\mathrm{goal}}& \subseteq	\{u\!\in\! \CC_{\Lambda(x, N_k)}^\sss{(1)}, u\!\notin\! \CC_{n}^\sss{(1)}  \}
  \!\cup\! \{|\CC_n(u)|\!>\! k, u\!\notin\!\CC_{\Lambda(x, N_k)}^\sss{(1)},
  u\!\notin\!\CC_{\Lambda(x, n_k)}^\sss{(1)}, |\CC_{\Lambda(x, n_k)}(u)|\!\le \! k\}       \nonumber                                                                                \\
                      & =  \CA_\mathrm{leave\textnormal{-}giant}(x, N_k)\cup \big(\CA_\mathrm{becomes\textnormal{-}large}\cap \CA_\mathrm{out\textnormal{-}of\textnormal{-}giant}(n_k,N_k)\big),
  \label{eq:upper-sub-an-3}
  \end{align}
 with $\CA_\mathrm{leave\textnormal{-}giant}(x, \tilde n)$ defined in \eqref{eq:event-giant}.
 We observe that $\CA_\mathrm{becomes\textnormal{-}large}$
 in \eqref{eq:events-edge-out-giant} implies that at least one of the at most $k$ vertices in $\CC_{\Lambda(x, n_k)}(u)$ has an incident edge crossing the boundary of $\Lambda(x, n_k)$, and so there must exist a ``fairly'' long edge either inside the cluster $\CC_{\Lambda(x, n_k)}(u)$ or between a vertex in $\CC_{\Lambda(x, n_k)}(u)$ and a vertex in $\CC_{n}(u)\setminus\CC_{\Lambda(x, n_k)}(u)$. More precisely, recalling $t_k=n_k^{1/d}/(2k)$, define
 
 \begin{equation}
  \CA_{\mathrm{edge}} := \big\{|\CC_{\Lambda(x, n_k)}(u)|\le k, \exists v_1\in\CC_{\Lambda(x, n_k)}(u), v_2\in\CV: v_1\sim v_2, \|v_1-v_2\|\ge t_k \big\}.\label{eq:edge-out-tk}
 \end{equation}
 We argue that $\CA_{\mathrm{becomes\textnormal{-}large}}\subseteq \CA_{\mathrm{edge}}$. Arguing by contradiction, if all edges incident to all vertices in $\CC_{\Lambda(x, n_k)}(u)$ were shorter than $t_k$, the furthest point that could be reached from $x$ with at most $k-1$ edges has Euclidean norm at most $(k-1)t_k < n_k^{1/d}/2$, and thus its location would be inside $\Lambda(x, n_k)$, contradicting the definition of $\CA_\mathrm{becomes\textnormal{-}large}$ in \eqref{eq:events-edge-out-giant}. Returning to \eqref{eq:upper-sub-an-3}, we obtain that
 \begin{equation}
  \begin{aligned} \CA_{\mathrm{goal}}
    & \subseteq \CA_\mathrm{leave\textnormal{-}giant}(x, N_k)\cup \big(\CA_\mathrm{edge}\cap \CA_\mathrm{out\textnormal{-}of\textnormal{-}giant}(n_k,N_k) \big) \\
  &\subseteq  \CA_\mathrm{leave\textnormal{-}giant}(x, N_k)\cup\big(\CA_\mathrm{edge}\cap \{u\notin\CC^\sss{(1)}_{\Lambda(x, N_k)}\}\big).
  \end{aligned}\nonumber
 \end{equation}
 In order to bound the probability of the existence of long edges, we  put restrictions on the marks: we distinguish whether all vertices in $\CV_{\Lambda(x, N_k)}\setminus\CC_{\Lambda(x, N_k)}^\sss{(1)}$ have mark at most $\overline{w}_{N_k}$ or not ---this is the event $\CA_{\mathrm{mark\textnormal{-}giant}}(x, N_k,\overline{w}_{N_k})$ in \eqref{eq:event-giant}.
 We obtain
 \begin{equation}\label{eq:goal-2}
 \begin{aligned}
  \CA_{\mathrm{goal}}
  &\subseteq \CA_\mathrm{leave\textnormal{-}giant}(x, N_k)\cup \big(\neg \CA_{\mathrm{mark\textnormal{-}giant}}(x, N_k,\overline{w}_{N_k})\big) \\
  &\hspace{10pt}\cup \big(\CA_\mathrm{edge}\cap  \{u\notin\CC^\sss{(1)}_{\Lambda(x, N_k)}\}\cap \CA_{\mathrm{mark\textnormal{-}giant}}(x, N_k,\overline{w}_{N_k}) \big).
  \end{aligned}
 \end{equation}
 The intersection with $\{u\notin\CC^\sss{(1)}_{\Lambda(x, N_k)}\} \cap  \CA_{\mathrm{mark\textnormal{-}giant}}(x, N_k,\overline{w}_{N_k})$ in the last event ensures that \emph{all} vertices in the cluster of $u$ with location in $\Lambda(x, N_k)\supseteq\Lambda(x, n_k)$ have mark at most $\overline w_{N_k}$. We make another case distinction, with respect to the locations of the vertices of the  edge $e_{\ge t_k}$ of length at least $t_k$ that exists on the event $\CA_\mathrm{edge}$ in \eqref{eq:edge-out-tk}. Namely, $e_{\ge t_k}$  either has both endpoints in $\Lambda_{N_k}$ or it has one endpoint inside $\Lambda_{n_k}$ and the other one outside $\Lambda_{N_k}$. For the first event, we obtain the event $\CA_{\mathrm{low\textnormal{-}edge}}(x, n_k, N_k, \overline w_{N_k})$, and for the latter $\CA_{\mathrm{long\textnormal{-}edge}}(x, n_k, N_k, \overline w_{N_k})$, respectively (defined in \eqref{eq:event-low-edge}---\eqref{eq:event-long-edge}).
 Hence,
 \begin{equation}
 \begin{aligned}
  \CA_\mathrm{edge}\cap  \{u\notin\CC^\sss{(1)}_{\Lambda(x, N_k)}\} &\cap  \CA_{\mathrm{mark\textnormal{-}giant}}(x, N_k,\overline{w}_{N_k}) \\&\hspace{10pt}\subseteq \CA_{\mathrm{low\textnormal{-}edge}}(x,n_k, N_k, \overline w_{N_k}) \cup \CA_{\mathrm{long\textnormal{-}edge}}(x, n_k, N_k, \overline w_{N_k}).
  \end{aligned}\nonumber
 \end{equation}
 Using this in \eqref{eq:goal-2}, then substituting \eqref{eq:goal-2}  back into \eqref{eq:upper-sub-an-1}, and then taking probabilities yields
 \begin{align}
  \Prob^\sss{x}\big(&|\CC_n(u)| > k, u\notin\CC_n^\sss{(1)}\big) \nonumber\\
  & \le \Prob^\sss{x}\big(|\CC_{\Lambda(x, n_k)}^\sss{(2)}| > k\big)+
  \Prob^\sss{x}\big(\neg\CA_\mathrm{mark\textnormal{-}giant}(x, N_k, \overline{w}_{N_k})\big)
  +\!\!
  \sum_{\tilde n\in\{n_k, N_k\}}\!\!
  \Prob^\sss{x}\big(\CA_\mathrm{leave\textnormal{-}giant}(x, \tilde n)\big)\nonumber
  \\
  & \hspace{15pt}
  +\Prob^\sss{x}\big(\CA_\mathrm{low\textnormal{-}edge}(x, n_k, N_k, \overline{w}_{N_k})\big)
  +
  \Prob^\sss{x}\big(\CA_\mathrm{long\textnormal{-}edge}(x, n_k, N_k, \overline{w}_{N_k})\big).\nonumber
 \end{align}
 The event $\CA_\mathrm{leave\textnormal{-}giant}(x, \tilde n)\!=\!\{u\!\in\!\CC_{\Lambda(x,\tilde n)}^\sss{(1)}, u\!\notin\!\CC_{n}^\sss{(1)}\}$ considers the graph in the box $\Lambda_n$ (which is centered at the origin)  and therefore does not necessarily have the same probability for all $x\in\Lambda_n$. The four other events consider the graph in boxes centered at $x$. Hence, we translate those events (and the Palm measure $\Prob^\sss{x}$) by $-x$ to obtain
 \begin{align}
  \Prob^\sss{x}\big(&|\CC_n(u)| \ge k, u\notin\CC_n^\sss{(1)}\big) \nonumber\\
  & \le \Prob^\sss{0}\big(|\CC_{n_k}^\sss{(2)}| \ge k\big)+
  \Prob^\sss{0}\big(\neg\CA_\mathrm{mark\textnormal{-}giant}(0, N_k, \overline{w}_{N_k})\big)
  +\!\!
  \sum_{\tilde n\in\{n_k, N_k\}}\!\!
  \Prob^\sss{x}\big(u\in\CC_{\Lambda(x,\tilde n)}^\sss{(1)}, u\notin\CC_{n}^\sss{(1)}\big)\nonumber
  \\
  & \hspace{15pt}
  +\Prob^\sss{0}\big(\CA_\mathrm{low\textnormal{-}edge}(0, n_k, N_k, \overline{w}_{N_k})\big)
  +
  \Prob^\sss{0}\big(\CA_\mathrm{long\textnormal{-}edge}(0, n_k, N_k, \overline{w}_{N_k})\big),\nonumber
 \end{align}
   The first two terms can be bounded  by substituting the definitions $n_k, N_k=\exp\big(\Theta(k^\zeta)\big)$ in~\eqref{eq:nk-Nk} into the assumed bounds on
 the probabilities
 in Proposition~\ref{prop:condition-subexponential}. 
 The sum  is bounded from above by $2n_k^{-\delta}=\exp\big(-\Theta(k^\zeta)\big)$ by Claim \ref{claim:leaving-giant}. 
This finishes the proof of \eqref{eq:prop-refined}.
\end{proof}

We move on to bounding $\Prob^\sss{0}(\CA_\mathrm{low\textnormal{-}edge})$ on the right-hand side of \eqref{eq:prop-refined} in Lemma~\ref{prop:refinedevents}, with $\CA_\mathrm{low\textnormal{-}edge}$ from \eqref{eq:event-low-edge}. To do so, we need an auxiliary claim that controls the probability that for every point in $\CV_{N_k}[1, \overline{w}_{N_k})$ there are not ``too many'' vertices at distance at least $t_k$, with $t_k=n_k^{1/d}/(2k)$.
We define first for $i\ge 1$,
and $u=(x_u, w_u)\in\CV_{n_k}[0, \overline{w})$ the annuli
\begin{equation}
 \CR_i(x_u) := \big(\Lambda\big(x_u, (2^{i}t_k)^d\big) \setminus \Lambda\big(x_u, (2^{i-1}t_k)^d\big)\big)\times[1,\infty).
 \label{eq:upper-sub-short-annuli}
\end{equation}
With the measure $\mu_\tau$ from \eqref{eq:poisson-intensity}, we then define the \emph{bad} events
\begin{align}
 \CA_\mathrm{dense}(n_k) & := \big\{\exists i\ge 1, u\in \CV_{n_k}[1,\overline{w}_{N_k}): |\CV_{N_k}\cap \CR_i(x_u)| > 2 \cdot\mu_\tau\big(\CR_i(x_u)\big)\big\}.
 \label{eq:upper-sub-short-not-dense}
\end{align}
In the following auxiliary claim we give an upper bound on $\Prob\big(\CA_\mathrm{dense}\big)$. Its proof is standard, based on Palm theory and Chernoff bounds, see page~\pageref{proof:claim-dense} of Appendix~\ref{app:integral-proofs}.
\begin{claim}\label{claim:dense}
Consider a KSRG with a homogeneous Poisson point process as vertex set.
 For all $c, \delta> 0$ there exists $n_0$ such that $\Prob^\sss{0}\big(\CA_\mathrm{dense}(n_k)\big) \le n_k^{-c}$ for all $n_k\ge (n_0\vee k^{d+\delta})$.
\end{claim}

We can now analyze $\Prob^\sss{0}(\CA_{\mathrm{low\textnormal{-}edge}})$  in Lemma~\ref{prop:refinedevents}. The next  claim also applies for KSRGs on $\Z^d$. Recall $n_k, N_k, \overline{w}_{N_k}$ and $t_k$ from \eqref{eq:nk-Nk}.
\begin{claim}[No \emph{low}-mark edge from a small component]\label{claim:edge-short}
 Consider a KSRG satisfying Assumption~\ref{assumption:main} with parameters $\alpha>1, \tau>2$, $\sigma\ge 0$, and $d\in\N$. For any $\varepsilon\in(0, \min(1, \alpha-1, \tau-2))$ in \eqref{eq:nk-Nk}, there exists  a constant $A'>0$, such that for $k$ sufficiently large
 \begin{equation}
  \Prob^\sss{0}\big(\CA_\mathrm{low\textnormal{-}edge}(0, n_k, N_k, \overline{w})\mid \neg\CA_\mathrm{dense}(n_k)\big) \le
  \begin{dcases}
  \exp(-(1/A')k^\zeta), & \text{if }\alpha<\infty, \\
  0,                      & \text{if }\alpha=\infty.
  \end{dcases}\label{eq:event-low-edge-lemma}
 \end{equation}
\end{claim}
\begin{proof}
 Assume first $\alpha=\infty$. The event $\CA_\mathrm{low\textnormal{-}edge}(0, n_k, N_k, \overline{w}_{N_k})$ is by definition in \eqref{eq:event-low-edge} restricted to vertices of mark at most $\overline{w}_{N_k}=\overline{w}(N_k)$ in \eqref{eq:nk-Nk}.  By definition of $t_k$ and $\overline{w}_{N_k}$ in~\eqref{eq:nk-Nk}, for $k$ sufficiently large  \[t_k=\exp((\varepsilon/d)(c_3/(2c_2))k^\zeta)/(2k)\ge \beta\overline{w}_{N_k}^{1+\sigma}=\beta M^{1+\sigma}((c_3/(2c_2))k^\zeta)^{\eta(\sigma+1)}.\]
  Hence, the indicator in  $\mathrm p(u,v)$ is then $0$ by \eqref{eq:connection-prob-gen}, so a connection between $u,v$ can not occur.

 Assume then $\alpha<\infty$.
 To obtain an upper bound on the left-hand side of \eqref{eq:event-low-edge-lemma}, we condition on the full realization $\CV$ containing $0$ and satisfying the event $\neg\CA_\mathrm{dense}$:
 \begin{equation}
 \begin{aligned}
  \Prob^\sss{0}\big(\CA_\mathrm{low\textnormal{-}edge}(0, &n_k, N_k, \overline{w})\mid \neg\CA_\mathrm{dense}\big) \\&=
  \E^\sss{0}\big[
  \Prob^\sss{0}\big(\CA_\mathrm{low\textnormal{-}edge}(0, n_k, N_k, \overline{w}_{N_k})\mid \CV,  \neg\CA_\mathrm{dense}\big)
  \big]. \end{aligned}\label{eq:upper-short-notdense-0}
 \end{equation}
 We denote the subgraph of $\CG_{n_k}$ with all edges of length at most $t_k=n_k^{1/d}/(2k)$ by $\CG_{n_k}(\le t_k)$ and write $\CC_{n_k}(0, \le t_k)$ for the component in this graph containing the origin.
 Clearly,
 \begin{equation*}
 \begin{aligned}
  &\left\{
  \begin{aligned}\hspace{-4pt}
    & |\CC_{n_k}(0)|\le k, \exists v_1, v_2\in\CV_{N_k}[1, \overline{w}_{N_k}) \mbox{ s.t. } \\
    & v_1\in\CC_{n_k}(0), \|x_{v_1}-x_{v_2}\|\ge t_k,\mbox{ and }v_1\sim v_2\end{aligned}
  \right\}\\
  &\hspace{30pt}\subseteq
  \left\{
  \begin{aligned}
    & |\CC_{n_k}(0, \le t_k)|\le k, \exists v_1, v_2\in\CV_{N_k}[1, \overline{w}_{N_k}) \mbox{ s.t. } \\
    & v_1\in\CC_{n_k}(0, \le t_k), \|x_{v_1}-x_{v_2}\|\ge t_k,\mbox{ and }v_1\sim v_2\end{aligned}\hspace{-2pt}
  \right\}
  ,
  \end{aligned}
 \end{equation*}
 where the first event is the definition of $\CA_\mathrm{low\textnormal{-}edge}$ in \eqref{eq:event-low-edge}.
 Conditionally on $\CV$, all edges of length at least $t_k$ are present independently of  edges shorter than $t_k$. We obtain
 by a union bound over all vertices in $\CV_{n_k}[1, \overline{w}_{N_k})\subseteq \CV$,
 \begin{align}
  \Prob^\sss{0}  \big(&\CA_\mathrm{low\textnormal{-}edge}(0, n_k, N_k, \overline{w}_{N_k})\mid \CV, \neg\CA_\mathrm{dense}\big)                                                                                        \nonumber      \\
        & \le
  \hspace{-12pt}\sum_{v_1\in\CV_{n_k}[1, \overline{w}_{N_k})}\hspace{-12pt}\Prob^\sss{0}\big(v_1\!\in\!\CC_{n_k}(0, \le\! t_k),|\CC_{n_k}(0, \le \!t_k)|\!\le\!k\mid \CV, \neg\CA_\mathrm{dense}\big)\hspace{-15pt}\underbrace{\sum_{\substack{v_2\in\CV_{N_k}[1, \overline{w}_{N_k}): \\ \|x_{v_1}-x_{v_2}\|\ge t_k}}\hspace{-15pt}\mathrm{p}(v_1, v_2)}_{:=T(v_1)}.\label{eq:upper-short-notdense-1}
 \end{align}
 Using the definitions of $\mathrm{p}$ in \eqref{eq:connection-prob-gen},  $\kappa_\sigma$ from \eqref{eq:kernels},  the upper bound on $|\CR_i(x_u)|$ in $\CA_\mathrm{dense}$  in~\eqref{eq:upper-sub-short-not-dense}, mark bounds   $w_{v_1}, w_{v_2}\le \overline{w}_{N_k}$, and the distance bound $\|x_{v_1}-x_{v_2}\|\ge 2^{(i-1)}t_k$ when $x_{v_2}\in \CR_i(x_{v_1})$ in \eqref{eq:upper-sub-short-annuli},
 the following bound holds uniformly for all
 $v_1\in \CV_{N_k}[1, \overline w_{N_k})$:
 \begin{align}
  T(v_1) & \le \sum_{i\ge1}\sum_{\substack{v_2\in \CV_{N_k}[1,\overline w_{N_k})\\v_2\in \CR_i(v_1)}} p\big(\beta\kappa_\sigma(\overline{w}_{N_k},\overline{w}_{N_k})2^{-(i-1)d}t_k^{- d}\big)^\alpha\nonumber \\
      & \le  2\sum_{i\ge 1}t_k^d2^{id}p\big(\beta\kappa_\sigma(\overline{w}_{N_k},\overline{w}_{N_k})2^{-(i-1) d}t_k^{- d}\big)^\alpha\nonumber \\
      & =
  2^{d+1}p\beta^\alpha\overline{w}_{N_k}^{\alpha(\sigma+1)}t_k^{(1-\alpha) d}\sum_{i\ge 1}2^{-(\alpha-1)(i-1)d}.\label{eq:upper-short-notdense-2}
 \end{align}
 Since $\alpha>1$ by assumption in Theorem~\ref{thm:subexponential-decay}, the sum on the right-hand side is finite. This gives a bound on $T(v_1)$ in \eqref{eq:upper-short-notdense-1} that does not depend on $v_1$.
Hence, returning to \eqref{eq:upper-short-notdense-1},
 \begin{align}
  \sum_{v_1\in\CV_{n_k}[1, \overline{w}_{N_k})}\Prob^\sss{0}\big(v_1 & \in\CC_{n_k}(0, \le t_k), |\CC_{n_k}(0, \le t_k)|\le k\mid \CV, \neg\CA_\mathrm{dense}\big)\nonumber \\
                                                         & =
  \E^\sss{0}\Bigg[\ind{|\CC_{n_k}(0, \le t_k)|\le k}\sum_{v_1\in\CV_{n_k}[1, \overline{w}_{N_k})} \ind{v_1\in\CC_{n_k}(0, \le t_k)}\mid \CV, \neg\CA_\mathrm{dense}\Bigg]
  \le k,\nonumber
 \end{align}
 since on realizations of the graph satisfying $\{\CC_{n_k}(0, \le t_k)\le k\}$, the sum that follows is at most $k$. We substitute this with \eqref{eq:upper-short-notdense-2} into  \eqref{eq:upper-short-notdense-1} and then into
 \eqref{eq:upper-short-notdense-0}.
 Thus, for some constant $C>0$,
 \begin{equation}
  \Prob^\sss{0}\big(\CA_\mathrm{low\textnormal{-}edge}(0, n_k, N_k, \overline{w}_{N_k})\mid \neg \CA_\mathrm{dense}\big)
  \le
  C\,k\,\overline{w}_{N_k}^{\alpha(\sigma+1)}\,t_k^{(1-\alpha)d}.\nonumber
 \end{equation}
 We substitute the definitions $t_k=n_k^{1/d}/(2k)$, and $n_k=\exp(\varepsilon(c_3/(2c_2))k^\zeta)$ from \eqref{eq:nk-Nk}, which yields \eqref{eq:event-low-edge-lemma} for any $\varepsilon>0$
 (using that $\overline{w}_{N_k}$ is polynomial in $k$).
\end{proof}
The last claim  bounds $\CA_\mathrm{long\textnormal{-}edge}$ in~\eqref{eq:prop-refined} in Lemma~\ref{prop:refinedevents}.
Recall $\CA_\mathrm{long\textnormal{-}edge}$ from \eqref{eq:event-long-edge}.
\begin{claim}
 [No \emph{long} edge from a small component]\label{claim:edge-long}
 Consider a KSRG satisfying Assumption~\ref{assumption:main} with parameters $\alpha>1, \tau>2$, $\sigma\ge 0$, and $d\in\N$.
 Assume $N\ge n\ge 1$, and $\overline{w}\ge 1$ such that
 \begin{equation}
  (N^{1/d}-n^{1/d})/2 \ge \big(\sqrt{d} n^{1/d}\vee (\beta\overline{w}^{1+\sigma})^{1/d}\vee N^{1/d}/4\big).\label{eq:sub-upper-tk-tilde-bound}
 \end{equation} There exists a constant  $C_{\ref{claim:edge-long}}>0$ such that
\begin{equation}
\begin{aligned}
  \Prob^\sss{0}\big(\exists v_1&\in\CV_{n}[1,\overline{w}), \exists v_2\in\CV\setminus\CV_{N}: v_1\sim v_2\big)  \\ 
  &\le C_{\ref{claim:edge-long}}\overline{w}^{c_{\ref{claim:edge-long}}}n N^{-\min(\alpha-1, \tau-2)}(1+\ind{\alpha-1=\tau-2}\log N).
  \end{aligned}\label{eq:edge-long-lemma}
 \end{equation}
 In particular, for $n_k$, $N_k$, $\overline{w}_{N_k}$ as in \eqref{eq:nk-Nk}, if $\delta=\delta(M, \eta)\in(0, \min(\alpha-1, \tau-2, 1))$  is sufficiently small, then for all $k\ge 1$, with $\CA_\mathrm{long\textnormal{-}edge}$ defined in \eqref{eq:event-long-edge},
 \begin{equation}
  \Prob^\sss{0}\big(\CA_\mathrm{long\textnormal{-}edge}(0, n_k, N_k, \overline{w}_{N_k})\big) \le \exp\big(-\delta k^{\zeta_\mathrm{hh}}\big).\label{eq:edge-long-lemma2}
 \end{equation}
\end{claim}
\begin{proof}
 We defer the proof of \eqref{eq:edge-long-lemma} (based on a first-moment method) to Appendix~\ref{app:integral-proofs} on page~\pageref{proof:edge-long}. 
 The bound \eqref{eq:edge-long-lemma2} follows directly from \eqref{eq:edge-long-lemma} by substituting $n_k$, $N_k$ and $\overline{w}_{N_k}$ from \eqref{eq:nk-Nk} to  \eqref{eq:edge-long-lemma}, then using that $\overline{w}_{N_k}$ and $\log N_k$ are polynomial in $k$ and of much smaller order than $n_k$ and $N_k$.
\end{proof}

Having bounded all terms on the right-hand side in \eqref{eq:prop-refined}, we prove Proposition~\ref{prop:condition-subexponential}.
\begin{proof}[Proof of Proposition~\ref{prop:condition-subexponential}]
For $n\le N_k$, using that  $N_k=\exp\big((c_3/(2c_2))k^{\zeta}\big)$,
 \eqref{eq:exp-decay-71} in Proposition~\ref{prop:condition-subexponential} follows directly from \eqref{eq:prop-2nd}, since
 \begin{align}
  \Prob^\sss{0}\big(|\CC_{n}(0)|> k, 0\notin\CC_{n}^\sss{(1)}\big) \le
  \Prob^\sss{0}\big(|\CC_{n}^\sss{(2)}| > k\big)
  & \le N_k^{c_2}\exp\big( - c_3k^{\zeta}\big) =
  \exp\big(- (c_3/2)k^{\zeta}\big).\nonumber
 \end{align}
 We now consider $n>N_k$.
 Recall the values of $n_k, \overline{w}_{N_k}$, and $t_k$ from \eqref{eq:nk-Nk}. Lemma~\ref{prop:refinedevents} and Claims \ref{claim:dense}--\ref{claim:edge-long} directly imply \eqref{eq:exp-decay-71} in Proposition~\ref{prop:condition-subexponential}. 
 
We  now prove the law of large numbers \eqref{eq:lln-section}. In \cite{maitra2021locallim} it is shown that finite KSRGs $\CG_n=(\CV_n, \CE_n)$ rooted at a vertex at the origin (see  Definition~\ref{def:ksrg}) converge locally to their infinite rooted version $(\CG_\infty, 0)$ as $n\to\infty$. We refer to~\cite{Hofbook2} and its references for an introduction to local limits. We  use the concept of local limits as a black box and verify a necessary and sufficient condition for the law of large numbers for the size of the giant component for graphs that have a local limit by Van der Hofstad~\cite[Theorem~2.2]{hofstad2021giantlocal} of which we state an adaptation. 
 Let $(\CG_n, o_n)_{n\ge 1}$ be a sequence of rooted graphs that converges locally in probability to $(\CG_\infty, \varnothing)$ (Theorem~2.2 in~\cite{hofstad2021giantlocal} demands additionally $|\CV_n|=n$, but its proof extends to cases in which $|\CV_n|\sim\mathrm{Poi}(n)$; we omit details here). Define
\begin{equation}\nonumber 
T_{n,k}:=\E\Big[\frac{1}{|\CV_n|^2}\sum_{u,v\in \CV_n}\ind{|\CC_n(u)|\ge k,\, |\CC_n(v)|\ge k,\, \CC(u)\neq \CC(v)}\Big].
\end{equation}
Then, 
 \begin{equation}\label{eq:lln-condition}
 \lim_{k\to\infty}\limsup_{n\to\infty}T_{n,k}=0 \qquad 
 \Longrightarrow \qquad
|\CC_n^\sss{(1)}|/|\CV_n|\overset\Prob\longrightarrow \Prob\big(\varnothing\leftrightarrow\infty\big),\quad\text{as }n\to\infty.\nonumber
\end{equation}
  For a pair of vertices $u,v$, the indicator in $T_{n,k}$ can only occur if at least one of the vertices is not in the largest component. More precisely, we can bound
 \[ \begin{aligned}
 T_{n,k}&\le \E\Big[\frac{1}{|\CV_n|^2} \sum_{u,v\in \CV_n} \ind{|\CC_n(u)|\ge k,u\notin\CC_n^\sss{(1)}} + \ind{|\CC_n(v)|\ge k,v\notin\CC_n^\sss{(1)}}  \Big] \\&= \E\Big[\frac{1}{|\CV_n|} \sum_{u\in \CV_n} 2\ind{|\CC_n(u)|\ge k,u\notin\CC_n^\sss{(1)}}  \Big] = 2\E\Big[\Prob\big(|\CC_n(U_n)|\ge k,U_n\notin\CC_n^\sss{(1)}\big| |\CV_n|\big)\Big],
 \end{aligned}\]
where $U_n$ is a uniformly selected vertex in $\CV_n$.
 We now restrict to the setting of KSRGs in the box $\Lambda_n$. Given the size $|\CV_n|$, the location $X_n$ of $U_n$  is uniform in $\Lambda_n$, while its mark $W_{U_n}$ is random, sampled from $W$.   
Integrating over the location $X_n=x\in \Lambda_n$ having density $\mathrm{Leb}(\cdot)/n$, we obtain that
\[
\begin{aligned}
\E\big[\Prob\big(&|\CC_n(U_n)|\ge k,U_n\notin\CC_n^\sss{(1)}\,\big| \, |\CV_n|\big)\big]\le \Prob\big(\|X_n-\partial\Lambda_n\|\!<\! N_k^{1/d}/2\big) \\
&+\frac{1}{n}\E\Big[\hspace{-20pt}\int\limits_{x\in\Lambda_n: \|x-\partial\Lambda_n\|\ge N_k^{1/d}/2}\hspace{-40pt}\Prob^x\big(|\CC_n((x,W_{U_n}))|\!\ge\! k,(x, W_{U_n}  )\notin\CC_n^\sss{(1)}\,\big| \, |\CV_n|, (x, W_{U_n})\in\CV_n\big)\rd x \Big]\end{aligned}
\]
 We apply Fubini's theorem to the second term:
 \[
 \begin{aligned}
\E\big[\Prob\big(&|\CC_n(U_n)|\ge k,U_n\notin\CC_n^\sss{(1)}\,\big|\, |\CV_n|\big)\big] \le \Prob\big(\|X_n-\partial\Lambda_n\|\!<\! N_k^{1/d}/2\big) \\
& +\frac{1}{n}\hspace{-20pt}\int\limits_{x\in\Lambda_n: \|x-\partial\Lambda_n\|\ge N_k^{1/d}/2}\hspace{-40pt}\Prob^x\big(|\CC_n((x,W_{U_n}))|\!\ge\! k,(x, W_{U_n}  )\notin\CC_n^\sss{(1)}\,\big| \, (x, W_{U_n})\in\CV_n\big)\rd x \end{aligned}
 \]
The first term tends to zero as $n\to\infty$.
 We recognise that we may apply Lemma~\ref{prop:refinedevents} and Claims~\ref{claim:dense}--\ref{claim:edge-long} to the probability in the integral inside the second term, which is  $\exp\big(-\Omega(k^\zeta)\big)$ uniformly over the domain of the integration. Thus, 
 \[
 \lim_{k\to\infty}\limsup_{n\to\infty}T_{n,k}=\lim_{k\to\infty}\exp\big(-\Omega(k^\zeta)\big)=0.
 \]
 This proves the condition on the left-hand side in~\eqref{eq:lln-condition}. The law of large numbers \eqref{eq:lln-section} follows as $|\CV_n|/n$ tends to 1 in probability.
\end{proof}

\section{Lower bounds}\label{sec:lower}
The main goal of this section is to prove a Proposition~\ref{prop:condition-lower} below that implies the lower bounds in Theorems ~\ref{thm:subexponential-decay}--\ref{thm:second-largest}. Informally, we show that if the graph $\CG_n[1, \mathrm{polylog}(n))$ induced on vertices with at most poly-logarithmically large marks in $n$ contains a linear-sized component with constant probability, then lower bounds as in Theorems ~\ref{thm:subexponential-decay}--\ref{thm:second-largest} follow. This general phrasing allows to derive lower bounds on $|\CC_n^{\sss{(2)}}|$ and on the cluster-size decay for KSRGs more generally, i.e., also without the assumption $\zeta_{\mathrm{hh}}>0$ of  Theorems~\ref{thm:subexponential-decay}--\ref{thm:second-largest}. We re-use this proposition in both~\cite{clusterII,clusterIII} after having established there its condition via renormalization techniques. 

In our proof below, we formalize the variational problem described in Section~\ref{sec:four-regimes}, and relate its solution to the size of the downward vertex boundary defined above~\eqref{eq:zeta-star}. In particular, Lemma~\ref{lemma:lower-vertex-boundary2} below implies Claim~\ref{lemma:lower-vertex-boundary} which states that $\zeta_\star=\max(\zeta_\mathrm{ll}, \zeta_\mathrm{hl}, \zeta_\mathrm{hh}, (d-1)/d)$.
At the end of the section we also prove Theorem~\ref{prop:lower-dev-giant} on the lower tail of large deviations of the largest component, which relies on the same methods as Proposition~\ref{prop:condition-lower}.

We introduce some notation to state Proposition~\ref{prop:condition-lower}.
Recall $\mathfrak{m}_\star$ from~\eqref{eq:m-star}, counting the multiplicity of the maximum in $\{\zeta_\mathrm{ll}$, $\zeta_\mathrm{hl}$, $\zeta_\mathrm{hh}$, $(d-1)/d\}$. If the values $\zeta_\mathrm{ll}, \zeta_\mathrm{hl}, \zeta_\mathrm{hh}$ are all negative and the dimension $d=1$, implying $\mathfrak{m}_\star=1$ and $\zeta_\star=(d-1)/d=0$  by Claim~\ref{lemma:lower-vertex-boundary}, then the model is always subcritical as shown by Gracar, L\"uchtrath, and M\"onch~\cite{gracar2022finiteness}. For all other parameter settings, we define  for some small $\eps>0$ to be specified later 
\begin{equation}
 \underline{k}_{n, \varepsilon} :=
 \begin{dcases}
  \big(\varepsilon(\log n)/(\log\log n)^{\mathfrak{m}_\star-1}\big)^{1/\zeta_\star}, & \text{if }\zeta_\star>0,                                                           \\
  \exp\big((\varepsilon\log n)^{1/(\mathfrak{m}_\star-1)}\big),                 & \text{if }\zeta_\star=0, \mbox{ and }\mathfrak{m}_\star>1.\label{eq:kn-second-largest}
 \end{dcases}
\end{equation}
For dimension $d\ge2$, $\zeta_\star\ge(d-1)/d$ is positive.
Thus, only in dimension $d=1$, $\underline{k}_{n,\varepsilon}$ can increase significantly faster than a polylog of $n$.
More precisely, for $d=1$,  $\underline{k}_{n,\varepsilon}$ equals $n^\eps$  if exactly one out of $\{\zeta_\mathrm{ll}, \zeta_\mathrm{hl}, \zeta_\mathrm{hh}\}$ is zero, and the others are negative (so $\mathfrak{m}_\star=2$ and $1/(\mathfrak{m}_\star-1)=1$); it increases stretched exponential in the logarithm if at least two elements out of $\{\zeta_\mathrm{ll}, \zeta_\mathrm{hl}, \zeta_\mathrm{hh}\}$ are zero, and none of them is positive. 
We recall from above Proposition~\ref{proposition:existence-large} that $\CC_n(0)[1, w)$ is the component of the origin in the induced subgraph $\CG_n[1, w)$ if $w_0<w$, and is the empty set if $w_0\ge w $.

\begin{proposition}[Lower bound holds when linear-sized giant on truncated marks exists]\label{prop:condition-lower}
Consider a KSRG satisfying Assumption~\ref{assumption:main} with parameters $\alpha>1, \tau>2$, $\sigma\ge 0$, and $d\in\N$. Assume that there exist constants $\eta, \rho>0$ such that for all $n$ sufficiently large, 
 \begin{align}
\Prob^\sss{0}\big(|\CC_n(0)[1, \log^\eta n)|\ge \rho n\big)
  \ge
  \rho.
  \label{eq:lower-sub-cond}
 \end{align}
 Then there exists $A>0$ such that for all $n\in [Ak,\infty]$, with $\zeta_\star, m_\star$ from~\eqref{eq:zeta-star} and \eqref{eq:m-star},
 \begin{equation}
  \Prob^\sss{0}\big(|\CC_n(0)|> k, 0\notin\CC_n^\sss{(1)}\big) \ge
  \exp\big(-Ak^{\zeta_\star}(\log k)^{\mathfrak{m}_\star-1}\big).\label{eq:lower-sub-iso-bound}
 \end{equation}
 Moreover, there exist $ \delta, \varepsilon>0$, such that for all $n$ sufficiently large, with $k_{n,\varepsilon}$ from \eqref{eq:kn-second-largest},
 \begin{equation}
  \Prob\big(|\CC_{n}^\sss{(2)}|\ge \underline{k}_{n,\varepsilon}\big)\ge 1-n^{-\delta}.\label{eq:lower-second-bound}
 \end{equation}
\end{proposition}
By Proposition~\ref{proposition:existence-large}, condition \eqref{eq:lower-sub-cond} is satisfied when $\zeta_\mathrm{hh}>0$, implying Theorem~\ref{thm:subexponential-decay}(i). 
We give a detailed proof of Proposition~\ref{prop:condition-lower} for KSRGs with vertex set given by a PPP. We leave adaptations of proofs of most subresults to vertex set $\Z^d$ to the reader (replacing concentration bounds for Poisson random variables to concentration bounds on sums of independent Bernoulli random variables). At the end of the section we explain the non-trivial adaptations.

\subsection{Strategy to find a localized component}\label{sec:lower-strat} 
To bound $\Prob(|\CC_{n}(0)|> k, 0\notin\CC_{n}^\sss{(1)})$ from below, we find a subevent that we can write as the intersection of ``almost independent'' events, for which we introduce some notation now. See Figure~\ref{fig:lower-bound} for a visualization.\smallskip 

\noindent\emph{Two components.\ \ }   We aim to find an isolated and localized component of at least $k$ vertices that is not the giant. For this, we take $\rho$ from \eqref{eq:lower-sub-cond}, and we encompass the box $\Lambda_{k/\rho}$ in a larger ball so that the distance of the ball from the box is half the radius of the ball. Formally, define 
\begin{equation}
r_k:=(k/\rho)^{1/d} \sqrt{d}, \qquad\CB_{\mathrm{in}} := \{x\in\R^d: \|x\|\le r_k\}, \qquad M_\mathrm{in}:=d^{d/2}/\rho, \label{eq:rk}
\end{equation}
so $r_k^d=M_\mathrm{in}k$.
These definitions imply that $\Lambda_{k/\rho}\subseteq \CB_{\mathrm{in}}$. We now constrain $\CC_n(0)$ to the ball  $\CB_{\mathrm{in}}$, and aim to find a component outside $\CB_{\mathrm{in}}$ that is larger than $|\CC_n(0)|$.
We `construct' these two components on vertices in two (hyper)rectangles. Recall that  $\Lambda(x,s)=\Lambda_s(x)$ denotes a box of volume $s$ centered at $x$, see \eqref{eq:xi-q-ab}. Let $M_\mathrm{out}:=2^{d+2}M_\mathrm{in}$ and define for $\eta>0$
\begin{equation}\label{eq:lower-giant-boxes}
\begin{aligned}
 \Lambda_\mathrm{in}&:=\Lambda(0, k/\rho), &\CR_\mathrm{in}&:=\Lambda_{\mathrm{in}} \times \big[1, 2^{\lceil\log_2 \log^\eta(k/\rho)\rceil}\big), \\
 \Lambda_\mathrm{out}&:=\Lambda(x_\mathrm{out}, (kM_\mathrm{out}/\rho)), &\CR_\mathrm{out}&:=\Lambda_\mathrm{out}\times\big[1, 2^{\lceil \log_2 \log^\eta(kM_\mathrm{out}/\rho)\rceil}\big),
 \end{aligned}
\end{equation}
where $x_\mathrm{out}:=(x^{\mathrm{out}}_1, 0, \dots, 0)\in \R^d$ is defined as any solution of $\|\partial\Lambda_\mathrm{out}-\partial \CB_{\mathrm{in}}\|:=r_k/2$ satisfying $\Lambda_\mathrm{out}\cap\CB_{\mathrm{in}}=\emptyset$. 
We assume that the constant $A$ in Proposition~\ref{prop:condition-lower} is sufficiently large so that $\Lambda_\mathrm{in}\cup\Lambda_\mathrm{out}\subseteq\Lambda_{n}$.
We abbreviate $\CC_\mathrm{in}(0):=\CC_{k/\rho}(0)[1, \log^\eta(k/\rho))$. 
Since $\CR_{\mathrm{in}}\subseteq \CB_{\mathrm{in}}\times[1,\infty)$, it is immediate that $\CC_\mathrm{in}(0)\subseteq \CV_{\CB_{\mathrm{in}}}$. 
Let $\CC_\mathrm{out}^\sss{(1)}$ be the largest component in the subgraph of $\CG_n$ induced on vertices in $\CR_\mathrm{out}$.
Define the events
\begin{equation}
 \begin{aligned}\label{eq:lower-large-events}
\CA^\sss{(k)}_\mathrm{giant\textnormal{-}in} &:=\{|\CC_\mathrm{in}(0)|> k\}, &\CA^\sss{(k)}_\mathrm{giant\textnormal{-}out}&:=\{|\CC_\mathrm{out}^\sss{(1)}| > k M_\mathrm{out}-1\},\\
    \CA^\sss{(k)}_\mathrm{small\textnormal{-}in} &:=\{|\CV_{\CB_{\mathrm{in}}}|\le k M_\mathrm{out}/2\}, \qquad &\CA^\sss{(k)}_\mathrm{components}&:=   \CA^\sss{(k)}_\mathrm{giant\textnormal{-}in}   \cap    \CA^\sss{(k)}_\mathrm{giant\textnormal{-}out}.
    \end{aligned}
\end{equation}
\smallskip 

\begin{figure}[b]
 \includegraphics[width=244pt]{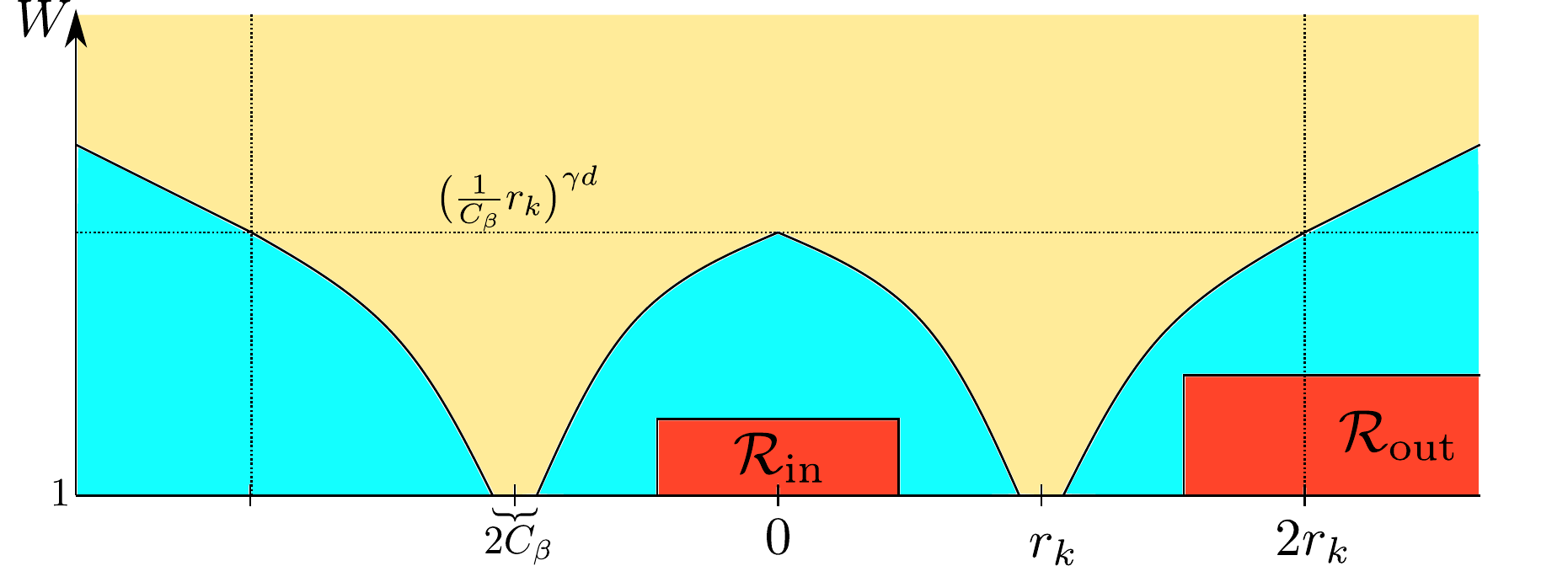}
   \captionsetup{skip=0pt}
 \caption{A visualization of the $\gamma$-suppressed profile $\CM_\gamma$.
 The horizontal axis represents space, the vertical axis represents marks. $\{\CV\le \CM_\gamma\}$ demands no vertices in the yellow region. $\CA_{\mathrm{no\textnormal{-}edge}}(\gamma)$ demands that there is no edge between vertices in the inner blue and vertices in the outer blue regions, $\CA_{\mathrm{components}}$ requires that the two red areas contain large components;  $\CA_\mathrm{regular}(\eta)$ ensures that $\CV$ is `close to typical' in the red areas $\CR_\mathrm{in}$ and $\CR_\mathrm{out}$.}\label{fig:lower-bound}
\end{figure}
\noindent\emph{Isolation.\ \ }
On $\CA^\sss{(k)}_\mathrm{components}$, $\CC_n(0)$ or $ \CC_{\mathrm{in}}(0)$ could still be part of the largest/infinite component.
To prevent this, we will ban edges that cross the boundary of $\CB_{\mathrm{in}}$.
 We first define a suppressed mark-profile  that is parametrized by $\gamma\ge 0$. Below, we optimize its shape to obtain the ``optimally-suppressed mark-profile''. Set $C_\beta:=(2\beta)^{1/d}$, and
define for $x\in\R^d$ with $\|x-\partial\CB_{\mathrm{in}}\|=|\|x\|-r_k|=:z$ the $\gamma$-suppressed profile by
\begin{align}
 \label{eq:suppressed-curve} f_\gamma(z) &:=
 \begin{dcases}
  1                                           & \text{if }z\le C_\beta,       \\
  (z/C_\beta)^{\gamma d}                      & \text{if }z\in(C_\beta, r_k], \\
  (z/C_\beta)^{d}(r_k/C_\beta)^{-d(1-\gamma)} & \text{if }z>r_k,
 \end{dcases}\\
 \CM_\gamma&:=\{(x_v, f_\gamma(|\|x_v\|-r_k|)): x_v\in\R^d\}.\label{eq:M-gamma}
\end{align}
We say that $v$ is \emph{below, on, or above} $\CM_\gamma$ if $w_v$ is at most, equal to, or strictly larger than $f_\gamma(|\|x_v\|-r_k|)$, respectively. 
We split the PPP $\CV$ into four independent PPPs, depending on whether points fall below or above $\CM_\gamma$, and inside or outside $\CB_{\mathrm{in}}$:
\begin{equation}
 \begin{aligned}
  \CV_{\le\CM_\gamma}^\sss{\mathrm{in}}  & := \{(x_u,w_u)\in\CV: x_u\in\CB_{\mathrm{in}}, w_u\le f_\gamma(|\|x_u\|-r_k|)\},    \\
  \CV_{\le\CM_\gamma}^\sss{\mathrm{out}} & := \{(x_v,w_v)\in\CV: x_v\notin\CB_{\mathrm{in}}, w_v\le f_\gamma(|\|x_v\|-r_k|)\}, \\
  \CV_{>\CM_\gamma}^\sss{\mathrm{in}}    & := \{(x_u,w_u)\in\CV: x_u\in\CB_{\mathrm{in}}, w_u> f_\gamma(|\|x_u\|-r_k|)\},      \\
  \CV_{>\CM_\gamma}^\sss{\mathrm{out}}   & := \{(x_v,w_v)\in\CV: x_v\notin\CB_{\mathrm{in}}, w_v> f_\gamma(|\|x_v\|-r_k|)\}.
 \end{aligned}\label{eq:xi-in-out}
\end{equation}
For $A,B\subseteq\CV$ we denote by $|\CE(A,B)|$ the number of edges between vertices in $A$ and $B$. Define
\begin{equation}
 \{\CV\le \CM_{\gamma}\}:=\{|\CV_{>\CM_\gamma}^\sss{\mathrm{in}}\cup\CV_{>\CM_\gamma}^\sss{\mathrm{out}}| = 0\}, \hspace{7pt}
 \CA^\sss{(k)}_{\mathrm{no\textnormal{-}edge}}(\gamma):=\{|\CE(\CV_{\le\CM_\gamma}^\sss{\mathrm{in}},\CV_{\le\CM_\gamma}^\sss{\mathrm{out}})|=0\},
 \label{eq:lower-events}
\end{equation}
On $\{\CV\le \CM_\gamma\}\cap\CA^\sss{(k)}_{\mathrm{no\textnormal{-}edge}}(\gamma)$, the vertices in $\CB_{\mathrm{in}}$ are not connected to the unique infinite component when $n=\infty$, and are isolated from the rest of $\CG_n$ when $n<\infty$. Combined with the events from~\eqref{eq:lower-large-events} and using that $|\CC_n(0)|\ge |\CC_\mathrm{in}(0)|$ we obtain 
\begin{equation}\label{eq:before-tech}
\begin{aligned}
    \{|\CC_n(0)|>k, 0\notin\CC_n^\sss{(1)}\} \supseteq \big(&\{\CV\le \CM_{\gamma}\}\cap  \CA^\sss{(k)}_{\mathrm{no\textnormal{-}edge}}(\gamma)\cap\{|\CC_\mathrm{in}(0)|>k\}\\&\cap\{|\CC_\mathrm{out}^\sss{(1)}|>kM_\mathrm{out}-1\}\cap \{|\CV_{\CB_{\mathrm{in}}}|\le kM_\mathrm{out}/2\}\big).
    \end{aligned}
\end{equation}
We comment on the profile function $f_\gamma$ in \eqref{eq:suppressed-curve}: the event $\{\CV\le \CM_\gamma\}$ demands no vertices within distance $C_\beta$ from $\partial \CB_{\mathrm{in}}$, since $f_\gamma(|\|x\|-r_k|)= 1$ for $\|x-\partial \CB_{\mathrm{in}}\|\le C_\beta$, and vertex marks are above $1$. The function $f_\gamma$ is continuous and increasing in $z$: the closer a point is to the boundary of $\CB_{\mathrm{in}}$, the stronger the mark restriction. This is natural since vertices with higher mark close to $\partial \CB_{\mathrm{in}}$ are more likely to have an edge crossing this boundary, which we want to prevent. While $\{\CV\le \CM_\gamma\}$ becomes less likely when $\gamma$ is small, $\CA^\sss{(k)}_{\mathrm{no\textnormal{-}edge}}(\gamma)$ becomes more likely. This leads to a variational problem, that we set up after a technicality. \smallskip 

\noindent\emph{Ensuring almost independence.\ \ }
The events $ \CA^\sss{(k)}_{\mathrm{no\textnormal{-}edge}}(\gamma)$ and $\{|\CC_\mathrm{in}(0)|> k\}$ in \eqref{eq:before-tech} are negatively correlated. Indeed, $\{|\CC_\mathrm{in}(0)|> k\}$ from~\eqref{eq:lower-large-events} may push up the number of high-mark vertices in $\CR_{\mathrm{in}}$, making $\CA^\sss{(k)}_{\mathrm{no\textnormal{-}edge}}(\gamma)$ less likely.
To overcome the dependence, we introduce two auxiliary events that ensure regularity of the vertex marks in the hyperrectangles $\CR_{\mathrm{in}}, \CR_{\mathrm{out}}$ from \eqref{eq:lower-giant-boxes}.
Let $c_{\mathrm{in}}:=1/\rho$, $c_{\mathrm{out}}:=M_{\mathrm{out}}/\rho$, and  define for $\mathrm{loc}\in\{\mathrm{in}, \mathrm{out}\}$, $\eta>0$,
\begin{equation}
 j^\star_{\mathrm{loc}}  :=\lceil\log_2 \log^\eta(kc_{\mathrm{loc}})\rceil,\qquad 
 I_j^{\sss{\mathrm{loc}}}:=[ 2^{j-1}, 2^j) \quad \text{for }  1\le j\le j^\star_{\mathrm{loc}},
 \label{eq:lower-weight-intervals}
\end{equation}
so that the upper bounds of the largest weight intervals agree with the upper boundaries of the hyperrectangles $\CR_\mathrm{in}$ and $\CR_\mathrm{out}$ defined in \eqref{eq:lower-giant-boxes}.
Using $\Lambda_{\mathrm{in}}, \Lambda_{\mathrm{out}}$ in \eqref{eq:lower-giant-boxes}, the intensity measure $\mu_\tau$ of $\CV$ in \eqref{eq:poisson-intensity}, and $\CV_{\mathrm{loc}}(I_j^{\sss{\mathrm{loc}}})$ for the vertices in $\CV\cap(\Lambda_{\mathrm{loc}}\times I_j^{\sss{\mathrm{loc}}})$, consider the following events for $\mathrm{loc}\in\{\mathrm{in}, \mathrm{out}\}$:
\begin{equation}
 \begin{aligned}
  \CA_\mathrm{regular}^\sss{(k, \mathrm{loc})}(\eta)  &:=\{\forall j\le j^\star_{\mathrm{loc}}: |\CV_\mathrm{loc}(I_j^{\sss{\mathrm{loc}}})|\le 2\mu_\tau(\Lambda_\mathrm{loc}\times I_j^{\sss{\mathrm{loc}}})\}, \\
  \CA^\sss{(k)}_\mathrm{regular}({\eta})             & := \CA_\mathrm{regular}^\sss{(k, \mathrm{in})}({\eta})\cap\CA_\mathrm{regular}^\sss{(k, \mathrm{out})}({\eta}).
  \label{eq:lower-events-global}
 \end{aligned}
\end{equation}
Finally, fix a realization of the induced subgraphs $\CG_{\CR_\mathrm{in}}\cup\CG_{\CR_\mathrm{out}}=(\CV_{\CR_\mathrm{in}}, \CE(\CG_{\CR_\mathrm{in}}))\cup(\CV_{\CR_\mathrm{out}}, \CE(\CG_{\CR_\mathrm{out}}))$  so that the vertex set $\CV_{\CR_\mathrm{in}}\cup \CV_{\CR_\mathrm{out}}$ satisfies the event $\CA^\sss{(k)}_\mathrm{regular}(\eta)$ for some $\eta>0$, and the two induced subgraphs on vertices in $\CR_\mathrm{in}$ and on $\CR_\mathrm{out}$ satisfy $\CA_\mathrm{components}^\sss{(k)}$ defined in \eqref{eq:lower-large-events}. We define the conditional probability measure  and expectation by
\begin{equation}
\begin{aligned}
\widetilde\Prob\big(\,\cdot\,\big) &:= \Prob\big(\,\cdot\mid \CG_{\CR_\mathrm{in}}\cup\CG_{\CR_\mathrm{out}}, \CA^\sss{(k)}_\mathrm{regular}(\eta), \CA_\mathrm{components}^\sss{(k)}\big), \\ \widetilde\E[\,\cdot\,]&:= \E\big[\,\cdot \mid \CG_{\CR_\mathrm{in}}\cup\CG_{\CR_\mathrm{out}}, \CA^\sss{(k)}_\mathrm{regular}(\eta), \CA_\mathrm{components}^\sss{(k)} \big].\end{aligned}\label{eq:cond-prob-lower-1}
\end{equation}
 In the conditioning we reveal both the vertex and edge sets within the disjoint boxes $\CR_{\mathrm{in}}, \CR_{\mathrm{out}}$.  The event $\CA^\sss{(k)}_\mathrm{regular}(\eta)$ checks the number of vertices in hyperrectangles inside $\CR_{\mathrm{in}}, \CR_{\mathrm{out}}$ while $\CA_\mathrm{components}^\sss{(k)}$ depends on the edges spanned on $\CR_{\mathrm{in}}$ and spanned on $\CR_{\mathrm{out}}$, hence both $\CA^\sss{(k)}_\mathrm{regular}(\eta), \CA_\mathrm{components}^\sss{(k)}$ are measurable with respect to $\CG_{\CR_\mathrm{in}}, \CG_{\CR_\mathrm{out}}$.

 \subsection{Isolation via a variational problem}\label{sec:variational}
In this section we analyze the events $\{\CV\!\le\!\CM_\gamma\}$ and $\CA_\mathrm{no\textnormal{-}edge}^\sss{(k)}$  in \eqref{eq:lower-events} under the conditional probability measure in~\eqref{eq:cond-prob-lower-1}.

\begin{lemma}[Lower bound for isolation]\label{prop:subexponential-lower}
 Consider a KSRG satisfying Assumption~\ref{assumption:main} with parameters $\alpha>1, \tau>2$, $\sigma\ge 0$, and $d\in\N$.  There exists $\gamma_\star\in(0, 1/(\sigma+1)]$ such that
for any constant $\eta>0$ in \eqref{eq:lower-events-global} there exists $A>0$ such that for any  realization of $\CG_{\CR_\mathrm{in}}\cup\CG_{\CR_\mathrm{out}}$ satisfying $\CA^\sss{(k)}_\mathrm{regular}(\eta)$,  
        \begin{align}\label{eq:io-lower-zeta}
         \widetilde\Prob\big(\{\CV\le \CM_{\gamma_\star}\}\cap\CA^\sss{(k)}_{\mathrm{no\textnormal{-}edge}}(\gamma_\star)\big)
         \ge
         \exp\big(-Ak^{\zeta_\star}(\log k)^{\mathfrak{m}_\star-1}\big).
        \end{align}
The same bound holds for the Palm-version $\widetilde{\Prob}^\sss{0}$ of $\widetilde\Prob$.
\end{lemma}
The events $\{\CV\!\le\!\CM_\gamma\}$ and $\CA_\mathrm{no\textnormal{-}edge}^\sss{(k)}$  are independent of each other under $ \widetilde\Prob$ in \eqref{eq:cond-prob-lower-1}, since having no points above $\CM_\gamma$ is independent of the conditioning in $\widetilde\Prob$ (since each point in $\CR_\mathrm{in}\cup\CR_\mathrm{out}$ is below $\CM_\gamma$ if $k$ is sufficiently large), and $\CA_\mathrm{no\textnormal{-}edge}^\sss{(k)}$ only depends on points of $\CV$ below $\CM_\gamma$ with endpoints on different sides of $\partial \CB_\mathrm{in}$. Hence, for any $\gamma\ge 0$,
 \begin{equation}
  \begin{aligned}
  \widetilde\Prob\big(\{\CV\le\CM_\gamma\}\cap\CA_\mathrm{no\textnormal{-}edge}^\sss{(k)}(\gamma)\big)          =\widetilde\Prob\big(\CV\le\CM_\gamma\big)
  \cdot\widetilde\Prob\big(\CA_\mathrm{no\textnormal{-}edge}^\sss{(k)}(\gamma)\big)         .\label{eq:lower-sub-independent}
  \end{aligned}
 \end{equation}
 We show below that the two factors decay exponentially fast respectively in the expected number of vertices above $\CM_\gamma$ (which is non-increasing in $\gamma$), and the expected number of edges between vertices below $\CM_\gamma$ crossing $\partial\CB_{\mathrm{in}}$ (which is non-decreasing in $\gamma$). We compute these in the following two lemmas, then balance them to get the optimal $\gamma$.
  Recall $f_{\gamma},M_\gamma$ from \eqref{eq:suppressed-curve}, \eqref{eq:M-gamma}, and the PPPs in \eqref{eq:xi-in-out}. Let $\CV_{>\CM_{\gamma}}:=\CV_{>\CM_\gamma}^\sss{\mathrm{in}}\cup \CV_{>\CM_\gamma}^\sss{\mathrm{out}}$.
\begin{lemma}[Vertices above $\CM_\gamma$]\label{lemma:lower-vertices-above}
 Consider a KSRG satisfying Assumption~\ref{assumption:main} with parameters $\alpha>1, \tau>2$, $\sigma\ge 0$, and $d\in\N$. For each $\gamma\ge 0$,
 there exists a constant $C_{\ref{lemma:lower-vertices-above}}>0$ such that for all $k \ge 1$
  \begin{equation}
  \widetilde\E\big[|\CV_{>\CM_{\gamma}}|\big] \le C_{\ref{lemma:lower-vertices-above}}k^{\max\big(1-\gamma(\tau-1),\tfrac{d-1}{d}\big)}\cdot(\log k)^{\Ind{1-\gamma(\tau-1)=\tfrac{d-1}{d}}}.
  \label{eq:expected-above-curve}
 \end{equation}
\end{lemma}
For readability, we need to introduce a few more `exponents', then we state the other lemma that bounds the expected number of edges between $\CV_{\le\CM_\gamma}^\sss{\mathrm{in}}$ and $\CV_{\le\CM_\gamma}^\sss{\mathrm{out}}$. Let
\begin{equation}\label{eq:xi-long}
\begin{aligned}
\xi_\mathrm{ll}&:=0, & \xi_\mathrm{hl}&:=\alpha-(\tau-1), & \xi_\mathrm{hh}&:=(\sigma+1)\alpha-2(\tau-1),\\
 \Xi&:=\{\xi_\mathrm{ll}, \xi_\mathrm{hl}, \xi_\mathrm{hh}\},& \xi_\star&:=\max(\Xi), &\mathfrak{m}_\mathrm{long}&:=\sum_{\xi\in\Xi}\ind{\xi_\star=\xi}.
 \end{aligned}
\end{equation}
\begin{lemma}[Edges crossing $\partial \CB_{\mathrm{in}}$ below $\CM_\gamma$]\label{lemma:lower-edges-below} Consider a KSRG under the conditions of Lemma~\ref{prop:subexponential-lower} with $\alpha<\infty$. 
 For each $\gamma\ge 0$ there exists a constant $C_{\ref{lemma:lower-edges-below}}=C_{\ref{lemma:lower-edges-below}}(\rho)>0$ such that for all $k\ge 1$ and any realization of $\CV_{\CR_\mathrm{in}}\cup\CV_{\CR_\mathrm{out}}$ that satisfies $\CA_\mathrm{regular}(\eta)$ in \eqref{eq:lower-events-global} for some $\eta>0$, 
  \begin{equation}
  \begin{aligned}
\widetilde\E\big[|\CE\big(\CV_{\le\CM_\gamma}^\sss{\mathrm{in}}, \CV_{\le\CM_\gamma}^\sss{\mathrm{out}}\big)| \big] &\le C_{\ref{lemma:lower-edges-below}}k^{\max\big(2-\alpha + \gamma\xi_\star,\frac{d-1}{d}\big)}\\&\hspace{11pt}\cdot(\log k)^{(\mathfrak{m}_\mathrm{long}-1)\Ind{2-\alpha + \gamma\xi_\star > \frac{d-1}{d}}+\mathfrak{m}_\mathrm{long}\Ind{2-\alpha + \gamma\xi_\star = \frac{d-1}{d}}}.
  \end{aligned}
  \label{eq:expected-edges-below}
 \end{equation}
     Assume now $\gamma\in[0, 1/(\sigma+1)]$. For any KSRG under the conditions of Lemma~\ref{prop:subexponential-lower} with vertex set formed by a homogeneous Poisson point process, for any $\alpha\in(1,\infty]$, we have for any $u\in\CV_{\le\CM_\gamma}^\sss{\mathrm{in}}, v\in \CV_{\le\CM_\gamma}^\sss{\mathrm{out}}$
 \begin{equation}
  \mathrm{p}(u, v) \le
  \begin{dcases}
  2^{-\alpha},&\text{if }\alpha<\infty,\\
  0,&\text{if }\alpha=\infty.
  \end{dcases}
  \label{eq:claim-long-edge}
 \end{equation}
\end{lemma}
\begin{remark}
One can prove that the right-hand side of \eqref{eq:expected-edges-below} is the correct order for the expectation for all $\gamma\in[0,1]$ whenever $\xi_{\mathrm{hh}}<\xi^\star$ in \eqref{eq:xi-long}, by computing a matching lower bound up to constant factor. When $\xi_{\mathrm{hh}}=\xi_\star$, then the right-hand side of \eqref{eq:expected-edges-below} is the correct order when $\gamma\in[0,1/(\sigma+1)]$. When $\xi_{\mathrm{hh}}=\xi^\star$ and $\gamma>1/(\sigma+1)$, the right-hand side of \eqref{eq:expected-edges-below} is not a sharp upper bound, but it suffices for the purposes of the proofs below.
\end{remark}

\begin{proof}[Proof sketch of Lemmas \ref{lemma:lower-vertices-above} and  \ref{lemma:lower-edges-below}]
Since the quantities we compute are  functions of Poisson variables, the proof is an integration and case-distinction exercise over the domains of the underlying Poisson processes and connection probability.
 We defer the (lengthy) integrals to the appendix on page~\pageref{proof:lower-vertices-above},
  and give intuition. We omit among others technicalities caused by the conditioning in $\widetilde\Prob$ in~\eqref{eq:cond-prob-lower-1}. Define the hyperrectangle $\CR^\uparrow:=[-2r_k, 2r_k]^d\times [(1\vee(r_k/C_\beta)^{\gamma d}), \infty)$ and $\mathrm A_\beta:=\{x\in \R^d, \|x\|\in [r_k-C_\beta, r_k+C_\beta]\}\times[1,\infty)$, an annulus in $\R^d$ times all mark-coordinates. Then by definition of $f_\gamma$ in \eqref{eq:suppressed-curve},  the set $(\CR^\uparrow\cup \mathrm A_\beta)$ is above $\CM_\gamma$, and  $\mu_\tau( \CR^\uparrow\cup \mathrm{A}_\beta )=\Theta(k^{1-\gamma(\tau-1)} + k^{(d-1)/d})$ by the definition of the Poisson intensity $\mu_\tau$ in~\eqref{eq:poisson-intensity}.
Integration shows that the Poisson intensity $\mu_{\tau}$ of the larger space-mark area above $\CM_{\gamma}$ (the left-hand side of \eqref{eq:expected-edges-below}) is of the same order if $1-\gamma(\tau-1)\neq(d-1)/d$. When $1-\gamma(\tau-1)=(d-1)/d$ we get an extra $\log k$ factor.
\smallskip

We explain now the exponents of $k$ in \eqref{eq:expected-edges-below} in  Lemma~\ref{lemma:lower-edges-below}.
 The expected number of edges between vertices of constant mark within constant distance of $\partial \CB_{\mathrm{in}}$ is $\Theta(k^{(d-1)/d})$.  Let $0\le\gamma_\wedge\le\gamma_\vee\le \gamma$. Using $\mu_{\tau}$ in \eqref{eq:poisson-intensity}, the expected number of vertex pairs $\CV_{\le\CM_\gamma}^\sss{\mathrm{in}}$ and  $\CV_{\le\CM_\gamma}^\sss{\mathrm{out}}$ within distance $\Theta(r_k)$ from $\partial\CB_{\mathrm{in}}$, and marks $w_\vee=\Theta(k^{\gamma_\vee }), w_\wedge=\Theta(k^{\gamma_\wedge})$ is
 \[ \mathbb E[\mathrm{Pairs}(\gamma_\vee,\gamma_\wedge)]:=\Theta\big(k^{1-\gamma_\vee(\tau-1)}\cdot k^{ 1-\gamma_\wedge(\tau-1)}\big)=\Theta\big(k^{2-(\gamma_\vee+\gamma_\wedge)(\tau-1)}\big).\] The typical Euclidean distance between such vertices is $\Theta(r_k)=\Theta(k^{1/d})$. Therefore, by the connection probability $\mathrm{p}$ in \eqref{eq:connection-prob-gen}, a pair of such vertices are connected with probability roughly $\Theta(k^{\alpha(\gamma_\vee+\sigma\gamma_\wedge-1)})$  when $\gamma_\vee+\sigma\gamma_\wedge\le 1$ and $\alpha<\infty$. Thus, there are \begin{equation}
 \begin{aligned}
  \mathbb E[\mathrm{Edges}(\gamma_\vee, \gamma_\wedge)]&:=\mathbb E[\mathrm{Pairs}(\gamma_\vee,\gamma_\wedge)]\cdot \Theta(k^{\alpha(\gamma_\vee+\sigma\gamma_\wedge-1)})\\&=\Theta\big(k^{2-\alpha+\gamma_\vee(\alpha-(\tau-1))+\gamma_\wedge(\sigma\alpha-(\tau-1))}\big).\end{aligned}\label{eq:intuition-edges}
 \end{equation}
 such edges in expectation. 
 The proof below on page~\pageref{proof:lower-edges-below} reveals that the expectation of $|\CE( \CV_{\le\CM_\gamma}^\sss{\mathrm{in}} , \CV_{\le\CM_\gamma}^\sss{\mathrm{out}}\big)|$ is either $\Theta(k^{(d-1)/d})$ (coming from the constant-distance edges) or its order is the maximal value of the right-hand side  in \eqref{eq:intuition-edges}, when maximized with respect to $0\le \gamma_{\wedge}\le \gamma_{\vee}\le\gamma$. Logarithmic factors  arise when there are multiple maximizers.  
 The exponent of $k$ is linear in both $\gamma_\vee$ and $\gamma_\wedge$. When computing the maximizing pair in the interval $[0,\gamma]$, with $\xi_\mathrm{ll}, \xi_\mathrm{hl}, \xi_\mathrm{hh}$ from~\eqref{eq:xi-long}, we arrive at
 \begin{equation}\nonumber
     (\gamma_\vee^\ast, \gamma_\wedge^\ast)=\begin{dcases}
         (\gamma, 0),&\text{if both }\alpha>\tau-1, \sigma\alpha<\tau-1 \quad \big(\Longleftrightarrow\ \xi_\mathrm{hl}>\max(\xi_\mathrm{ll}, \xi_\mathrm{hh})\big),\\
         (\gamma, \gamma),&\text{if both }\alpha>\tau-1, \sigma\alpha>\tau-1  \hspace{3pt} \big(\Longleftrightarrow \xi_\mathrm{hh}>\max(\xi_\mathrm{hl}, \xi_\mathrm{hh})\big),\\(0,0),&\text{if }0 = \xi_\mathrm{ll}>\max(\xi_\mathrm{hl}, \xi_\mathrm{hh}).
     \end{dcases}
 \end{equation}
 The last case summarizes the outcome of the cases remaining after the first two rows. 
 The maximum of $\{\xi_\mathrm{ll}, \xi_\mathrm{hl}, \xi_\mathrm{hh}\}$ is non-unique if at least one of $\alpha=\tau-1$ and $\sigma\alpha=
 \tau-1
 $ holds. In this case any convex combination of the maximizing vectors among $\{(0,0), (\gamma, 0), (\gamma, \gamma)\}$ gives the maximal value on the right hand-side of \eqref{eq:intuition-edges}. This leads to a polylogarithmic correction factor, where the exponent is the dimension of the simplex formed by the maximizers, i.e., $\mathfrak{m}_\mathrm{long}-1$. When the exponent of the maximum equals $(d-1)/d$, edges of all lengths between constant order and $\Theta(k^{1/d})$ contribute to the number of edges, leading to an extra factor $\log k$ in \eqref{eq:expected-edges-below}. We obtain~\eqref{eq:expected-edges-below} by substituting $(\gamma_\vee^\ast, \gamma_\wedge^\ast)$ into~\eqref{eq:intuition-edges} and combining this with the $\Theta(k^{(d-1)/d}) $ many short edges crossing the boundary. 
 The maximizer(s) tell(s) us if the dominant contribution of long edges comes  from  edges between vertices with constant mark when $(\gamma_\vee^\ast, \gamma_\wedge^\ast)=(0, 0)$, from edges between one high-mark vertex and one low-mark vertex when $(\gamma_\vee^\ast, \gamma_\wedge^\ast)=(\gamma, 0)$, or from edges between two high-mark vertices when $(\gamma_\vee^\ast, \gamma_\wedge^\ast)=(\gamma, \gamma)$. 
These edge types are  the \emph{dominant types of connectivity} described in Section~\ref{sec:four-regimes}. 

We prove~\eqref{eq:claim-long-edge} in Lemma~\ref{lemma:lower-edges-below} by showing that $\beta \kappa_\sigma(w_u,w_v)/$ $\|x_u-x_v\|^d \le 1/2$ whenever $u, v$ are below $\CM_\gamma$ and on different sides of $\partial \CB_{\mathrm{in}}$.
 \end{proof}
We aim to balance the expectations in~\eqref{eq:expected-above-curve} and~\eqref{eq:expected-edges-below}. Thus, we say that $\gamma$ is optimal if the exponents of $k$ in the first two cases of ~\eqref{eq:expected-above-curve} (non-increasing in $\gamma$) and~\eqref{eq:expected-edges-below} (non-decreasing in $\gamma$) are equal. Define when $\alpha<\infty$
\begin{equation}\label{eq:gamma-opt-long}
\begin{aligned}
  \gamma_\mathrm{long}  &:= \min\Big\{\gamma: 1-\gamma(\tau-1)\le 2-\alpha+\gamma\xi_\star\Big\}
  =\frac{\alpha-1}{\max(\xi_\mathrm{ll}, \xi_\mathrm{hl}, \xi_\mathrm{hh})+\tau-1}.
  \end{aligned}
\end{equation}
Setting  $\gamma_\mathrm{long}$ as the smallest exponent $\gamma$ such that the expected number of vertices with mark $\Omega(k^\gamma)$ is at most the expected number of edges between lower-mark vertices, supports the definition of $\gamma_\mathrm{high}$ in~\eqref{eq:gamma-long} as the smallest exponent $\gamma$ such that a vertex of mark $\Theta(k^\gamma)$ is incident to constantly many edges of length $\Omega(k^{1/d})$ in expectation. The values $\gamma_\mathrm{long}$ and $\gamma_\mathrm{high}$ agree when high-high or high-low connections are dominant. To use~\eqref{eq:claim-long-edge} below when bounding $\widetilde\Prob\big(\CA_\mathrm{no\textnormal{-}edge}^\sss{(k)}(\gamma)\big)$ from below, we truncate $\gamma_\mathrm{long}$ and set
\begin{equation}\label{eq:gamma-opt}
    \gamma_\star:=
    \begin{dcases}
        \min\big(\gamma_\mathrm{long}, 1/(\sigma+1)\big),&\text{if }\alpha<\infty, \\
        1/(\sigma+1),&\text{if }\alpha=\infty
    \end{dcases}
\end{equation}
for the optimally suppressed mark profile.
The following two lemmas relate the exponents of $k$ and $\log k$ in~\eqref{eq:expected-above-curve} and~\eqref{eq:expected-edges-below} to the exponent $\zeta_\star$  defined in~\eqref{eq:zeta-star}, which appears in the lower bound on $\widetilde\Prob\big(\CA_\mathrm{no\textnormal{-}edge}^\sss{(k)}(\gamma_\star)\big)$ in~\eqref{eq:io-lower-zeta}. Recall $\zeta_\mathrm{ll}$, $\zeta_\mathrm{hl}$, and $\zeta_\mathrm{hh}$ from~\eqref{eq:zeta-ll}, \eqref{eq:gamma-lh}, and~\eqref{eq:zeta-hh}, respectively, and $\mathfrak{m_\star}$, $\mathfrak{m}_\mathrm{long}$, and $\xi_\star$ from~\eqref{eq:m-star} and~\eqref{eq:xi-long}.
\begin{lemma}[Exponents of the optimally-suppressed mark-profile]\label{lemma:zeta-opt-other}
 Consider a KSRG under the conditions of Lemma~\ref{prop:subexponential-lower}. 
 When $\alpha<\infty$, 
  \begin{align}
  2-\alpha+\xi_\star\gamma_\star\le \max\big(1-\gamma_\star(\tau-1), (d-1)/d\big)
 &=\max\big(\zeta_\mathrm{ll}, \zeta_\mathrm{hl}, \zeta_\mathrm{hh}, (d-1)/d\big),\label{eq:zeta-opt-1}
 \end{align}
 and 
 \begin{equation}
 \begin{aligned}
 \mathfrak{m}_\star-1&= (\mathfrak{m}_\mathrm{long}-1)\ind{2-\alpha+\xi_\star\gamma_\star>\tfrac{d-1}{d}}+\mathfrak{m}_\mathrm{long}\ind{2-\alpha+\xi_\star\gamma_\star=\tfrac{d-1}{d}}\\
 &\ge\ind{1-\gamma_\star(\tau-1)=\tfrac{d-1}{d}}.\label{eq:zeta-opt-m2}
 \end{aligned}
 \end{equation}
 When $\alpha=\infty$, $\max(1-\gamma_\star(\tau-1), (d-1)/d)=\max(\zeta_\mathrm{ll}, \zeta_\mathrm{hl}, \zeta_\mathrm{hh}, (d-1)/d)$, and $\mathfrak{m}_\star-1=\ind{1-\gamma_\star(\tau-1)=(d-1)/d}$.
 \end{lemma}

  The proof is based on rearrangements of the formulas of $\zeta_\mathrm{ll}, \zeta_\mathrm{hl}, \zeta_\mathrm{hh},$ and $\zeta_\mathrm{hh}$, and we postpone it to the appendix on page \pageref{proof:zeta-opt-other}. The following lemma connects~\eqref{eq:zeta-opt-1} to $\zeta_\star$ defined in~\eqref{eq:zeta-star} and implies Claim~\ref{lemma:lower-vertex-boundary}. We recall that we write $u\searrow\Lambda_k^\complement$ if the vertex $u=(x_u, w_u)\in\Lambda_k\times[1,\infty)$ has an edge to a vertex $v=(x_v, w_v)\in\Lambda_k^\complement\times[1,w_u]$.  
\begin{lemma}[Exponents of the downward vertex boundary]\label{lemma:lower-vertex-boundary2}
 Consider a KSRG under the conditions of Lemma~\ref{prop:subexponential-lower}. There exists a constant $C>0$ such that for all $k$,
 \begin{equation}
  \zeta_\star=\lim_{k\to\infty}\frac{\log \E\big[\big|\big\{u\in\Lambda_k:  u\searrow\Lambda_k^\complement\big\}\big|\big]}{\log k} =\max\big(\zeta_\mathrm{hh}, \zeta_\mathrm{hl}, \zeta_\mathrm{ll}, (d-1)/d\big)<1.\label{eq:lemma-zeta-star}
  \end{equation}
  Moreover, if $\max(\zeta_\mathrm{hh}, \zeta_\mathrm{hl}, \zeta_\mathrm{ll})\ge 0$, then 
  \begin{equation}
  \zeta_\mathrm{long}=\lim_{k\to\infty}\frac{\log \E\big[\big|\big\{u\in\Lambda_{k/2}:  u\searrow\Lambda_k^\complement\big\}\big|\big]}{\log k} =\max(\zeta_\mathrm{hh}, \zeta_\mathrm{hl}, \zeta_\mathrm{ll}).\label{eq:lemma-zeta-long}
  \end{equation}
If $\max(\zeta_\mathrm{hh}, \zeta_\mathrm{hl}, \zeta_\mathrm{ll})<0$, then the $\mathrm{\limsup}$ of the expression on the left-hand side is negative.
\end{lemma} 
The proof is similar to the proof of Lemma~\ref{lemma:lower-edges-below}, so we give a sketch in the appendix on page~\pageref{proof:vertex-boundary}.
We state an immediate corollary of Lemmas~\ref{lemma:lower-vertices-above}--\ref{lemma:lower-vertex-boundary2}.
\begin{corollary}[Optimized expectations]\label{cor:optimized-expectations}
    Consider a KSRG under the conditions of Lemma~\ref{prop:subexponential-lower}. 
 There exists a constant $C_{\ref{cor:optimized-expectations}}>0$ such that for any realization of $\CV_{\CR_\mathrm{in}}\cup\CV_{\CR_\mathrm{out}}$ that satisfies $\CA_\mathrm{regular}(\eta)$ in \eqref{eq:lower-events-global} for some $\eta>0$, 
 \[
 \left. 
 \begin{aligned}
     \widetilde\E\big[|\CV_{>\CM_{\gamma_\star}}|\big]& \\
     \widetilde\E\big[|\CE\big(\CV_{\le\CM_{\gamma_\star}}^\sss{\mathrm{in}}, \CV_{\le\CM_{\gamma_\star}}^\sss{\mathrm{out}}\big)| \big] &
 \end{aligned}
 \right\}\le C_{\ref{cor:optimized-expectations}}k^{\zeta_\star}(\log k)^{\mathfrak{m_\star}-1}.
 \]
 \begin{proof}
     When $\alpha<\infty$, the exponents of $r_k$ and $(\log r_k)$ in Lemmas~\ref{lemma:lower-vertices-above} and~\ref{lemma:lower-edges-below} are at most  $\zeta_\star$ and $\mathfrak{m}_\star-1$ by Lemmas~\ref{lemma:zeta-opt-other} and~\ref{lemma:lower-vertex-boundary2} when $\gamma=\gamma_\star$. When $\alpha=\infty$, the bound on the expected number of vertices above $\CM_{\gamma_\star}$ follows analogously. The expected number of edges below $\CM_{\gamma_\star}$ is 0 by~\eqref{eq:claim-long-edge} while the right-hand site is non-negative. 
 \end{proof}
\end{corollary}
\begin{remark}\label{remark:polylog}When the maximum in $\{\zeta_\mathrm{hh}, \zeta_\mathrm{hl}, \zeta_\mathrm{ll}, (d-1)/d\}$ is non-unique, the log-correction factors in the expectations in~\eqref{eq:lemma-zeta-star} and~\eqref{eq:lemma-zeta-long} might differ from those in Corollary~\ref{cor:optimized-expectations}, but these disappear in the limit of the logarithms in~\eqref{eq:lemma-zeta-star} and~\eqref{eq:lemma-zeta-long}. These different polylog factors are due to the fact that on phase-transition boundaries the expected number of downward edges from high-mark vertices is no longer of constant order. 
\end{remark}
We are ready to prove Lemma~\ref{prop:subexponential-lower}. We first assume that the vertex set is formed by a Poisson point process, and then explain the adaptations when the vertex set is $\Z^d$.
\begin{proof}[Proof of Lemma~\ref{prop:subexponential-lower} on Poisson point process]
 We set $\gamma=\gamma_\star\le 1/(\sigma+1)$ defined in~\eqref{eq:gamma-opt}. We recall from~\eqref{eq:lower-sub-independent} that
 \begin{equation}
  \begin{aligned}
  \widetilde\Prob\big(\{\CV\le\CM_{\gamma_\star}\}\cap\CA_\mathrm{no\textnormal{-}edge}^\sss{(k)}(\gamma_\star)\big)          =\widetilde\Prob\big(\CV\le\CM_{\gamma_\star}\big)\cdot
  \widetilde\Prob\big(\CA_\mathrm{no\textnormal{-}edge}^\sss{(k)}(\gamma_\star)\big)         .\label{eq:lower-sub-pr1}
  \end{aligned}
 \end{equation}
 We analyze the two probabilities separately. For the first factor we use the above independence and that the vertex set is formed by a Poisson point process. By Corollary~\ref{cor:optimized-expectations},
 \begin{align}
  \widetilde\Prob\big(\CV\le\CM_{\gamma_\star}\big)
  &=\widetilde\Prob\big(|\CV_{> \CM_{\gamma_\star}}|=0\big)
  =
  \exp\big(-\widetilde\E\big[|\CV_{>\CM_{\gamma_\star}}|\big]\big)
  \nonumber \\
  &\ge \exp\Big(-C_{\ref{cor:optimized-expectations}}k^{\zeta_\star}(\log k)^{\mathfrak{m}_\ast-1}\Big).\label{eq:lower-sub-pr2} 
 \end{align}
 We now turn to the second factor in \eqref{eq:lower-sub-pr1}.
  By definition of $\CA_{\mathrm{no\textnormal{-}edge}}^\sss{(k)}$ in \eqref{eq:lower-events}, and using the conditional independence of edges,
 \begin{equation}
\widetilde\Prob\big(\CA_\mathrm{no\textnormal{-}edge}^\sss{(k)}(\gamma_\star)\big)
  =
  \widetilde\E\Bigg[\prod_{u\in \CV_{\le\CM_{\gamma_\star}}^\sss{\mathrm{in}}, v\in \CV_{\le\CM_{\gamma_\star}}^\sss{\mathrm{out}}}(1-\mathrm{p}(u,v))\Bigg].\label{eq:product-no-edges}
 \end{equation}
 We will now use that $\gamma_\star\le1/(\sigma+1)$ by~\eqref{eq:gamma-opt}, which enables us to use~\eqref{eq:claim-long-edge}. When $\alpha=\infty$, $\mathrm{p}(u,v)=0$ for each factor. So,
 $
  \widetilde\Prob\big(\CA_\mathrm{no\textnormal{-}edge}^\sss{(k)}(\gamma_\star)\big) = 1
 $, which finishes the proof of \eqref{eq:io-lower-zeta} when $\alpha=\infty$ when 
combining \eqref{eq:lower-sub-pr1} with \eqref{eq:lower-sub-pr2}.
 Assume now $\alpha<\infty$. By~\eqref{eq:claim-long-edge}, $1-\mathrm{p}(u,v)\ge 1- 2^{-\alpha}$ for all $(u,v)\in\CV_{\le\CM_{\gamma_\star}}^\sss{\mathrm{in}}\times\CV_{\le\CM_{\gamma_\star}}^\sss{\mathrm{out}}$. Hence, there exists a constant $c>0$, such that $1-\mathrm{p}(u,v)\ge\exp(-c\cdot\mathrm{p}(u,v))$ for all such $(u,v)$.
 Using this in \eqref{eq:product-no-edges} and that $s\mapsto \exp(-s)$ is a convex function,  Jensen's inequality gives a lower bound in terms of the  expected number of  edges between vertices below $\CM_{\gamma_\star}$, i.e.,
 \begin{align}
  \widetilde\Prob\big(\CA_\mathrm{no\textnormal{-}edge}^\sss{(k)}(\gamma)\big)
    &\ge
  \widetilde\E\Bigg[\exp\Bigg(-\!c\!\!\sum_{\substack{u\in \CV_{\le\CM_{\gamma_\star}}^\sss{\mathrm{in}},                                                                                                                             \\ v\in \CV_{\le\CM_{\gamma_\star}}^\sss{\mathrm{out}}}}\!\!\mathrm{p}(u,v)\Bigg)\Bigg] \ge
  \exp\Bigg(-c\,\widetilde\E\Bigg[\!\sum_{\substack{u\in \CV_{\le\CM_{\gamma_\star}}^\sss{\mathrm{in}},\\ 
  v\in \CV_{\le\CM_{\gamma_\star}}^\sss{\mathrm{out}}}}\!\!\mathrm{p}(u,v)\Bigg]\Bigg)     \nonumber  \\
  & = \exp\big(-c\,\widetilde\E\big[|\CE\big(\CV_{\le\CM_{\gamma_\star}}^\sss{\mathrm{in}}, \CV_{\le\CM_{\gamma_\star}}^\sss{\mathrm{out}}\big)|\big]\big).\label{eq:lower-sub-2-pr2}
 \end{align}
 We invoke Corollary~\ref{cor:optimized-expectations} and obtain combined with~\eqref{eq:lower-sub-pr2} and~\eqref{eq:lower-sub-pr1} that
 \begin{equation}\nonumber
 \widetilde\Prob\big(\{\CV\le\CM_{\gamma_\star}\}\cap\CA_\mathrm{no\textnormal{-}edge}^\sss{(k)}(\gamma_\star)\big) \ge
  \exp\Big(-C_{\ref{cor:optimized-expectations}}(c+1)k^{\zeta_\star}(\log k)^{\mathfrak{m}_\ast-1}\Big),  
 \end{equation}
 proving Lemma~\ref{prop:subexponential-lower} when the vertex set is formed by a Poisson point process.
\end{proof}
\begin{proof}[Proof of Lemma~\ref{prop:subexponential-lower} for KSRGs on $\Z^d$]
We explain how to adjust the proof to KSRGs on $\Z^d$ using the assumption $(p\wedge\beta)<1$ in Lemma~\ref{prop:subexponential-lower} by Assumption~\ref{assumption:main}.
Since the vertex locations are given by $\Z^d$, the event $\{\CV\le \CM_\gamma\}$ as defined in \eqref{eq:M-gamma} would never hold, since $\Z^d$ does have points within distance $C_\beta$ from $\partial\CB_{\mathrm{in}}$ in case $C_\beta=(2\beta)^{1/d} \ge 1$ (see $f_\gamma$ in \eqref{eq:suppressed-curve} and the reasoning below \eqref{eq:lower-events}). Thus, if $C_\beta\ge1$, we must adjust the definition of $f_\gamma$ within distance $C_\beta$ of $\partial \CB_{\mathrm{in}}$ to be a constant $c=c(p,\beta,\alpha, \sigma)>1$ close to 1 to restrict vertex marks of vertices that are present close to $\partial \CB_{\mathrm{in}}$.
With that change, the upper bound $2^{-\alpha}$ on $\mathrm{p}(u, v)$ in~\eqref{eq:claim-long-edge} for vertices within distance $C_\beta$ from $\partial\CB_{\mathrm{in}}$ should be replaced by another constant $c'=c'(p, \beta, \alpha, \sigma)$ smaller than $1$, as these nearby vertices are connected with an edge with positive probability strictly bounded away from one. This affects constant prefactors in \eqref{eq:lower-sub-2-pr2} when $\alpha<\infty$. When $\alpha=\infty$, the expected number of potential edges between vertices below $\CM_{\gamma_\star}$ is $\Theta(k^{(d-1)/d})$ by similar calculations as in Lemma~\ref{lemma:lower-edges-below}. To bound $\Prob(\CA_{\mathrm{no\textnormal{-}edge}}(\gamma_\star))$ from below the same reasoning applies as in \eqref{eq:lower-sub-2-pr2} when $\alpha<\infty$.

The proofs of Lemmas \ref{lemma:lower-vertices-above} and \ref{lemma:lower-edges-below} remain valid by replacing concentration bounds for Poisson random variables by concentration bounds for sums of independent Bernoulli random variables, and replacing integrals over $\R^d$ by summations over $\Z^d$. The proofs of Lemmas~\ref{lemma:zeta-opt-other}--\ref{lemma:lower-vertex-boundary2} remain verbatim valid.
\end{proof}

\subsection{Second-largest component and cluster-size decay}\label{sec:lower-second-csd}
We are ready to prove Proposition~\ref{prop:condition-lower}. Recall $\widetilde{\Prob}$ from ~\eqref{eq:cond-prob-lower-1}, and the intersection of events in~\eqref{eq:before-tech}.
\begin{proof}[Proof of Proposition~\ref{prop:condition-lower}]
 We first show \eqref{eq:lower-sub-iso-bound}. Recall the events $\CA^\sss{(k)}_\mathrm{components}$,  $\CA^\sss{(k)}_\mathrm{small\textnormal{-}in}$ from~\eqref{eq:lower-large-events}, and $\CA^\sss{(k)}_\mathrm{regular}({\eta}) $ from \eqref{eq:lower-events-global}. 
 Set 
 \begin{equation}\label{eq:isolation-event}
     \CA^\sss{(k)}_{\mathrm{isolation}}:=\{\CV\le \CM_{\gamma_\star}\}\cap \CA_\mathrm{no\textnormal{-}edge}(\gamma_\star).
 \end{equation}
 The intersection of all these four events implies the event $\{|\CC_n(0) |> k, 0\notin\CC_n^\sss{(1)}\}$, since $|\CC_n(0)|\ge |\CC_{\mathrm{in}}(0)|> k$, and $\CA^\sss{(k)}_{\mathrm{isolation}}$ ensures that $\CC_n(0)$ is fully contained in $\CB_{\mathrm{in}}$. Hence, the events $\{|\CC_n(0)|\le |\CV_{\CB_{\mathrm{in}}}|\le kM_{\mathrm{out}}/2\}$ and $\{|\CC_{\mathrm{out}}|> kM_{\mathrm{out}}-1\}$ ensure that $\CC_n(0)$ is not the largest component of $\CG_n$.
  So, by the law of total probability
  \begin{align}
  \Prob^\sss{0}\big(|\CC_n&(0)  |> k, 0\notin\CC_n^\sss{(1)}\big)                                                                                                                     \nonumber\\
                      & \ge
  \Prob^\sss{0}\big(\CA^\sss{(k)}_\mathrm{components} \cap \CA^\sss{(k)}_{\mathrm{isolation}}\cap\CA^\sss{(k)}_\mathrm{small\textnormal{-}in}\cap\CA^\sss{(k)}_\mathrm{regular}(\eta)\big) \nonumber\\
                      & \ge
                      \Prob^\sss{0}\big(\CA^\sss{(k)}_\mathrm{components} \cap \CA^\sss{(k)}_{\mathrm{isolation}}\cap\CA^\sss{(k)}_\mathrm{regular}(\eta)\big) - \Prob^\sss{0}\big(\neg\CA^\sss{(k)}_\mathrm{small\textnormal{-}in}\big)\nonumber\\
                      &=
                      \Prob^\sss{0}\big(\CA^\sss{(k)}_\mathrm{components} \cap \CA^\sss{(k)}_\mathrm{regular}(\eta)\big)\Prob^\sss{0}\big(\CA^\sss{(k)}_{\mathrm{isolation}}\mid\CA^\sss{(k)}_\mathrm{components} \cap \CA^\sss{(k)}_\mathrm{regular}(\eta)\big) \label{eq:lower-cond-pr1} \\&\hspace{30pt}- \Prob^\sss{0}\big(\neg\CA^\sss{(k)}_\mathrm{small\textnormal{-}in}\big)\nonumber   .
  \end{align}
  Recall $\CA_{\mathrm{small\textnormal{-}in}}^{\sss{(k)}}=\{|\CV_{\CB_{\mathrm{in}}}|\le k M_\mathrm{out}/2\}$ from \eqref{eq:lower-large-events}, and $M_\mathrm{out}=2^{d+2}M_\mathrm{in}$ above \eqref{eq:lower-giant-boxes}. The box with side-length $2r_k=2(kM_\mathrm{in})^{1/d}$ (by definition in \eqref{eq:rk}) centered at the origin is the smallest box that contains $\CB_{\mathrm{in}}$. Using the intensity measure $\mu_\tau$ from \eqref{eq:poisson-intensity}, and writing $\CB_{\mathrm{in}}^\le:=\{(x, w_x)\in \R^{d+1}: x\in \CB_{\mathrm{in}}; (x, w_x) \le \CM_\gamma\}$, we have
\begin{align}
  \mu_\tau(\CB_{\mathrm{in}}^\le) \le 2^{d}kM_\mathrm{in}=2^{d+2}kM_\mathrm{in}/4=kM_\mathrm{out}/4.\nonumber
\end{align}
By a standard concentration inequality for Poisson random variables (see Lemma~\ref{lemma:poisson-1} for $x=2$), there exist $c',c>0$ such that, since $r_k=\Theta(k^{1/d})$,
\begin{align}
     \Prob^\sss{0}\big(\neg\CA^\sss{(k)}_\mathrm{small\textnormal{-}in}\big)\le \exp(-c' r_k^{d})=\exp(-c k).\label{eq:bgiant}
 \end{align}
 Returning to \eqref{eq:lower-cond-pr1}, the event $\CA_\mathrm{regular}^{\sss{(k)}}(\eta)=\CA_{\mathrm{regular}}^\sss{(k, \mathrm{in})}(\eta)\cap\CA_{\mathrm{regular}}^\sss{(k, \mathrm{out})}(\eta)$, defined in \eqref{eq:lower-events-global}, holds with probability tending to $1$ as $k\to \infty$, again by concentration inequalities for Poisson random variables (see Lemma~\ref{lemma:poisson-1} for $x=2$).
 Hence,
\begin{equation}\label{eq:ak-inbetween}
 \begin{aligned}
  \Prob^\sss{0}\big(\CA^\sss{(k)}_\mathrm{components}\cap \CA^\sss{(k)}_\mathrm{regular}(\eta)\big) &\ge \Prob^\sss{0}\big(\CA^\sss{(k)}_\mathrm{components}\big) - \Prob^\sss{0}\big(\neg(\CA^\sss{(k)}_\mathrm{regular}(\eta)\big)\\&=\Prob^\sss{0}\big(\CA^\sss{(k)}_\mathrm{components}\big)-o_k(1).\end{aligned}
 \end{equation}
 We recall from \eqref{eq:lower-large-events} that
 $\CA^\sss{(k)}_\mathrm{components}=\{|\CC_\mathrm{in}(0)|> k\}\cap\{|\CC_\mathrm{out}^\sss{(1)}| > kM_\mathrm{out}-1\}$.   Translate the hyperrectangle $\CR_\mathrm{out}$ in \eqref{eq:lower-giant-boxes} containing $\CC_\mathrm{out}^\sss{(1)}$ to the origin of $\R^d$:
 \begin{equation}\label{eq:zout-centered}
 \CR_\mathrm{out}':=\Lambda(0, kM_\mathrm{out}/\rho)\times[1,\log^\eta(kM_\mathrm{out}/\rho)), 
 \end{equation}
  and write $\CC_{\mathrm{out'}}^\sss{(1)}, \CC_{\mathrm{out'}}(0)$ for the largest component and for the component containing $(0,w_0)$ in the subgraph of $\CG_n$ induced by vertices in $\CR'_\mathrm{out}$. As before, we may ignore the conditioning $(0, w_0)\in\CV$ in Definition~\ref{def:ksrg} in our computations. We use translation invariance of the probability measure and that the events $\{|\CC_\mathrm{in}(0)|> k\}$ and $\{|\CC_\mathrm{out}^\sss{(1)}| > M_\mathrm{out}k-1\}$ are independent because they are induced subgraphs of the disjoint hyperrectangles $\CR_\mathrm{in}$ and $\CR_\mathrm{out}$  in \eqref{eq:lower-giant-boxes}. Hence,
  \[
  \begin{aligned}
  \Prob^\sss{0}\big(\CA^\sss{(k)}_\mathrm{components}\big)&=\Prob^\sss{0}\big(|\CC_\mathrm{in}(0)|> k\big)\Prob\big(|\CC_{\mathrm{out}'}^\sss{(1)}|> M_\mathrm{out}k-1\big)\\&\ge
  \Prob^\sss{0}\big(|\CC_\mathrm{in}(0)|> k\big)\Prob^\sss{0}\big(|\CC_{\mathrm{out}'}(0)|> M_\mathrm{out}k\big).
  \end{aligned}
 \]
 The bound $\Prob^\sss{0}\big(|\CC_{n}(0)[1,\log^\eta n)|\ge \rho n\big) \ge \rho$ in \eqref{eq:lower-sub-cond} in Proposition~\ref{prop:condition-lower} holds for all $n$ sufficiently large by assumption.
 In particular, since $\CC_\mathrm{in}(0)=\CC_{k/\rho}(0)[1, \log^\eta(k/\rho))$ by definition of $\CR_\mathrm{in}$ in \eqref{eq:lower-giant-boxes}, and $\CC_\mathrm{out'}(0)=\CC_{k M_{\mathrm{out}}/\rho, \eta}(0)$ by definition of $\CR_\mathrm{out}'$ in \eqref{eq:zout-centered}, we obtain for $k$ sufficiently large
 \begin{equation}\label{eq:ak-comp-intermed}
  \begin{aligned}\Prob^\sss{0}\big(\CA^\sss{(k)}_\mathrm{components}\big)&\ge\Prob^{\sss{0}}\big(|\CC_{k/\rho}(0)[1,\log^\eta(k/\rho))|> k\big)\\
  &\hspace{20pt}\cdot\Prob^\sss{0}\big(|\CC_{kM_{\mathrm{out}}/\rho}(0)[1,\log^\eta(kM_{\mathrm{out}}/\rho))|> M_\mathrm{out}k\big)
  \ge \rho^2,
  \end{aligned}
 \end{equation}
 implying that $
  \Prob^\sss{0}\big(\CA^\sss{(k)}_\mathrm{components}\cap \CA^\sss{(k)}_\mathrm{regular}(\eta)\big) \ge \rho^2-o_k(1)>3\rho^2/4$ in \eqref{eq:ak-inbetween}.
 Since the event $\CA^\sss{(k)}_\mathrm{components}
  \cap \CA^\sss{(k)}_\mathrm{regular}(\eta)$ is measurable with respect to the $\sigma$-algebra generated by the subgraph $\CG_{\CR_\mathrm{in}}\cup\CG_{\CR_\mathrm{out}}$, we take expectation over all possible realizations of the latter satisfying $\CA^\sss{(k)}_\mathrm{regular}(\eta)$, and recalling the definition of the measure $\widetilde\Prob$ from \eqref{eq:cond-prob-lower-1}, we obtain by the definition of $\widetilde{\Prob}$ in \eqref{eq:cond-prob-lower-1}
  \[\begin{aligned}
  &\Prob^\sss{0}\big(\CA^\sss{(k)}_{\mathrm{isolation}}\mid \CA^\sss{(k)}_\mathrm{components}
  \cap \CA^\sss{(k)}_\mathrm{regular}(\eta)\big) \\
  &= \mathbb E^\sss{0}\big[ \Prob^\sss{0}\big( \CA^\sss{(k)}_{\mathrm{isolation}}\mid \CG_{\CR_\mathrm{in}}\cup\CG_{\CR_\mathrm{out}}, \CA^\sss{(k)}_\mathrm{components}, \CA_\mathrm{regular}^\sss{(k)}(\eta)\big)\big]\\&=\ \mathbb E^\sss{0}\big[ \widetilde\Prob^\sss{0}(\CA^\sss{(k)}_{\mathrm{isolation}})\big].\nonumber
\end{aligned}\]
We apply Lemma~\ref{prop:subexponential-lower} on the right-hand side, and substitute the bound $3\rho^2/4$ below~\eqref{eq:ak-comp-intermed} into \eqref{eq:ak-inbetween} and then in turn into \eqref{eq:lower-cond-pr1} and \eqref{eq:bgiant}, to obtain for $k$ sufficiently large
\begin{align}
  \Prob^\sss{0}\big(&\CA^\sss{(k)}_\mathrm{components} \cap \CA^\sss{(k)}_{\mathrm{isolation}}\cap\CA^\sss{(k)}_\mathrm{small\textnormal{-}in}\cap\CA^\sss{(k)}_\mathrm{regular}(\eta)\big) \nonumber\\
  &\ge
  (\rho^2\!/2)
  \exp\!\big(\!\!-\!\!Ak^{\zeta_\star}(\log k)^{\mathfrak{m}_\star-1}\big)-\exp(-ck) \label{eq:lower-final-proof-sub-pre}\\
  &\ge (\rho^2\!/2)
  \exp\!\big(\!\!-\!\!A'k^{\zeta_\star}(\log k)^{\mathfrak{m}_\star-1}\big).\label{eq:lower-final-proof-sub}
  \end{align}
 We obtained the second row by substituting $r_k=(kM_\mathrm{in})^{1/d}$ in \eqref{eq:rk} and setting $A':=AM_\mathrm{in}^{\zeta_\star}/2$ that also compensates for the constants from the $\log$-correction term. Since $\zeta_\star<1$ by Lemma~\ref{lemma:lower-vertex-boundary2}, the second term in~\eqref{eq:lower-final-proof-sub-pre} is of smaller order than the first term in~\eqref{eq:lower-final-proof-sub-pre}.
 By~\eqref{eq:lower-cond-pr1}, this finishes the proof of \eqref{eq:lower-sub-iso-bound}. We turn to the proof of \eqref{eq:lower-second-bound}.
\vskip0.5em
 \emph{Lower bound on second-largest component.}
 We generalize an argument from \cite{kiwimitsche2ndlargest}.
 We have to bound $\Prob\big(|\CC_n^\sss{(2)}|< \underline{k}_{n,\varepsilon}\big)$ from above for a suitably chosen $\varepsilon$ in the definition of $\underline{k}_{n,\varepsilon}$ in \eqref{eq:kn-second-largest}.
 To do so, we fix $\vartheta\in(0,1)$ to be specified later, and assume for simplicity that $n^{(1-\vartheta)/d}\in\N$. We then partition $\Lambda_n$ into $m_n:=n^{1-\vartheta}$ many subboxes $\Lambda_{\tilde n}^\sss{(1)},\dots, \Lambda_{\tilde n}^\sss{(m_n)}$, centered respectively at $x_1,\dots,x_{m_n}$, each of volume $\tilde n:=n^\vartheta$.
 By disjointness, the induced subgraphs $\CG_{\tilde n}^\sss{(1)}, \dots, \CG_{\tilde n}^\sss{(m_n)}$ in these boxes are independent realizations of $\CG_{\tilde n}$, translated to $x_1,\dots, x_{m_n}$.
 We write $\widetilde\CG_{\tilde n}^\sss{(1)}, \dots, \widetilde\CG_{\tilde n}^\sss{(m_n)}$ for the induced subgraphs, translated back to the origin; $\widetilde\CV_{\le\CM_{\gamma_\star}}^\sss{(\mathrm{in}, i)}$ for the vertex set in $\widetilde\CG_{\tilde n}^\sss{(i)}$ that is below $\CM_{\gamma_\star}$ and inside $\CB_{\mathrm{in}}$ after the translation, see \eqref{eq:xi-in-out}; and write $\CV_{\le\CM_{\gamma_\star}}^\sss{(\mathrm{in}, i)}$ for the same vertex set before the translation.
 For the translated subgraphs $\big(\widetilde\CG_{\tilde n}^\sss{(i)})_{i\le m_n}$, we define for  $k=\underline{k}_{n,\varepsilon}$ the same events as in \eqref{eq:lower-large-events}, \eqref{eq:isolation-event}, \eqref{eq:lower-events-global},
 \[
  \CA^\sss{(i)}_{\mathrm{good}} := \CA^\sss{(\underline{k}_{n, \varepsilon}),i}_\mathrm{components}\cap \CA^\sss{(\underline{k}_{n, \varepsilon}),i}_\mathrm{isolation}\cap\CA^\sss{(\underline{k}_{n, \varepsilon}),i}_\mathrm{small\textnormal{-}in}\cap \CA^\sss{(\underline{k}_{n, \varepsilon}),i}_\mathrm{regular}(\eta),
 \]
where now in the definition of these events we replace $\CC_{\Box}(0)$ with the component containing the point of $\CV$ closest to the origin $0\in \R^d$ for $\Box\in\{\mathrm{in}, \widetilde n\}$.
 We also assume that $\widetilde n=n^\vartheta$ is sufficiently large compared to $\underline{k}_{n,\varepsilon}$ in \eqref{eq:kn-second-largest} so that the spatial projection of the box $\CR_{\mathrm{out}}$ still fits within $\Lambda_{\tilde n}$.
This can be ensured even if $\underline k_{n, \varepsilon}=\Theta(n^\varepsilon)$ is maximal in \eqref{eq:kn-second-largest} by choosing $\varepsilon<\vartheta$.
 If   $\CA^\sss{(i)}_{\mathrm{good}}$ holds for some $i\le m_n$,  then the induced graph $\CG_{\tilde n}^\sss{(i)}$ in subbox $\Lambda_{\tilde n}^\sss{(i)}$ contains a component $\CC_{\tilde n}^{\sss{(i)}}$ in $\CV_{\le\CM_{\gamma_\star}}^\sss{(\mathrm{in},i)}$ (which we call a `candidate' second-largest component of $\CG_n$) with size at least $\underline{k}_{n,\varepsilon}$ that is not the largest component in its own box, and all vertices in $\Lambda_{\tilde n}^\sss{(i)}$ are below $\CM_{\gamma_\star}(x_i)$, i.e., $\CM_{\gamma_\star}$ shifted to $x_i$. 
 
 Since the event $\CA_\mathrm{good}^\sss{(i)}$ is restricted to the induced subgraph $\CG_{\tilde n}^\sss{(i)}$, on $\CA_\mathrm{good}^\sss{(i)}$ there might still be an edge from a candidate second-largest component $\CC_{\tilde n}^\sss{(i)}$ to a vertex in a different box $\Lambda_{\tilde n}^\sss{(j)}$. We exclude such edges in another event: we demand that the whole vertex set  $\CV_{\le\CM_{\gamma_\star}}^\sss{(\mathrm{in},i)}$ has no edge to any other box, so that the component $\CC_{\tilde n}^\sss{(i)}$ is isolated also in $\CG_n$ and has size at least $\underline k_{n,\varepsilon}$.
Taking complements we obtain that \begin{equation}
  \big\{|\CC_n^\sss{(2)}|< \underline{k}_{n,\varepsilon}\}\subseteq
  \{\exists i\le m_n: \CV_{\le\CM_{\gamma_\star}}^\sss{(\mathrm{in}, i)} \sim \CV_n\setminus\CV_{\tilde n}^\sss{(i)}\}
  \cup
  \Bigg(\bigcap_{i\le m_n}\big(\neg\CA_\mathrm{good}^\sss{(i)}\big)\Bigg).\nonumber
 \end{equation}
By translation invariance, a union bound, and the independence of $(\CG_{\tilde n}^\sss{(i)})_{i\le m_n}$,
 \begin{align}\label{eq:two-terms-second}
  \Prob\big(|\CC_n^\sss{(2)}|< \underline{k}_{n,\varepsilon}\big) & \le
  m_n\Prob\big(\CV_{\le\CM_{\gamma_\star}}^\sss{\mathrm{in}} \sim \CV_n\setminus\CV_{\tilde n}\big)
  +
  \big(1-\Prob\big(\CA_\mathrm{good}^\sss{(1)}\big)\big)^{m_n}=:T_1+T_2.
 \end{align}
By the definitions in \eqref{eq:xi-in-out} and \eqref{eq:rk}, each $u\in\CV_{\le\CM_{\gamma_\star}}^\sss{\mathrm{in}}$ has $\|x_u\|\le r_k$ with $k=\underline{k}_{n,\varepsilon}$, and mark $w_u\le f_\star (r_{\underline{k}_{n,\varepsilon}})$, ($f_\star =f_{\gamma_\star}$ is from below \eqref{eq:gamma-opt}). As a result, 
$\CV_{\le\CM_{\gamma_\star}}^\sss{\mathrm{in}}\subseteq  \CV_{(2r_{\underline{k}_{n,\varepsilon}})^d}[1,f_\star (r_{\underline{k}_{n,\varepsilon}}))\subseteq  \CV_{\tilde n}$, whenever $(2r_{\underline{k}_{n,\varepsilon}})^d=\underline k_{n,\varepsilon}M_{\mathrm{in}}2^d<n^\vartheta$, which holds whenever $\varepsilon<\vartheta$ by \eqref{eq:kn-second-largest}. Hence we can bound $T_1$ as
\[T_1\le m_n\Prob\big(\CV_{(2r_{\underline{k}_{n,\varepsilon}})^d}[1,f_\star (r_{\underline{k}_{n,\varepsilon}})) \sim \CV_n\setminus\CV_{\tilde n}\big).\]
We can directly apply Claim \ref{claim:edge-long} to  the right-hand side, i.e., setting there $N:=\tilde n=n^\vartheta$ and  $n:=(2r_{\underline{k}_{n,\varepsilon}})^d=\underline{k}_{n,\varepsilon} M_{\mathrm{in}}2^d$ (by \eqref{eq:rk})  and $\overline w:=f_\star (r_{\underline{k}_{n,\varepsilon}})$.
The profile $f_{\star}=f_{\gamma_\star}$ is defined below  \eqref{eq:gamma-opt}, using \eqref{eq:suppressed-curve} with exponent $\gamma_\star$ and $r_k=(M_{\mathrm{in}}k)^{1/d}$ in \eqref{eq:rk}, and finally $\underline k_{n, \varepsilon}$ from \eqref{eq:kn-second-largest} we obtain
\[\overline w:=f_\star (r_{\underline{k}_{n,\varepsilon}})=
C_\beta^{-\gamma_{\star} d} r_{\underline{k}_{n,\varepsilon}}^{\gamma_{\star} d}=C_\beta^{-\gamma_{\star} d} M_\mathrm{in}^{\gamma_{\star}}\underline{k}_{n,\varepsilon}^{\gamma_{\star}}\]
Condition \eqref{eq:sub-upper-tk-tilde-bound} holds whenever $\varepsilon<\vartheta$, since $N=\Theta(n^\vartheta)$ while $\underline k_{n,\varepsilon} =O(n^\varepsilon)$ and we truncated $\gamma_\star$ in \eqref{eq:gamma-opt} at $1/(\sigma+1)$, so also $\overline w^{\sigma +1}=\Theta(k_{n,\varepsilon}^{\gamma_\star(\sigma+1) }) =O(n^\varepsilon)$.
Then Claim \ref{claim:edge-long}  yields for some $C>0$
 \begin{align*}
  \hspace{-5pt}T_1& \le
  m_n C_{\ref{claim:edge-long}} f_\star (r_{\underline{k}_{n,\varepsilon}})^{c_{\ref{claim:edge-long}}}
  (\underline{k}_{n,\varepsilon}M_{\mathrm{in}}2^{d})n^{-\vartheta\min(\alpha-1, \tau-2)}\big(1+\mathbbm{1}_{\{\alpha=\tau-1\}}\log(n^\vartheta)\big) \\
  & \le
  C (\log n)
 \cdot  \underline{k}_{n,\varepsilon}^{1 + \gamma_{\star} c_{\ref{claim:edge-long}}}
  \cdot n^{1-\vartheta\min(\alpha, \tau-1)}.
 \end{align*}
 Since $\underline{k}_{n,\varepsilon}$ in \eqref{eq:kn-second-largest} is at most $n^\varepsilon$, as long as $1-\vartheta\min(\alpha, \tau-1)<0$,
 we can choose $\varepsilon>0$ in \eqref{eq:kn-second-largest} small such that  for any $\delta\in(0,\vartheta\min(\alpha, \tau-1)-1)$, for all $n$ sufficiently large,
 \begin{equation}\label{eq:t1-estimate}
 T_1\le n^{-\delta}.
 \end{equation} We turn to bound $T_2$ in \eqref{eq:two-terms-second} using $(1-x)^{m_n}\le \exp(-m_n x)$, where we apply \eqref{eq:lower-final-proof-sub} on $x=\Prob(\CA_\mathrm{good}^\sss{(1)})$ to obtain a lower bound on the exponent
 \begin{align*}
  m_n\Prob\big(\CA_\mathrm{good}^\sss{(1)}\big)
  & \ge
  (\rho^2/2)n^{1-\vartheta}\exp\big(-A'\underline{k}_{n,\varepsilon}^{\zeta_\star}(\log \underline{k}_{n,\varepsilon})^{\mathfrak{m}_\star-1}\big) \\
  & =(\rho^2/2)\exp\big((1-\vartheta)(\log n)-
  A'\underline{k}_{n,\varepsilon}^{\zeta_\star}(\log \underline{k}_{n,\varepsilon})^{\mathfrak{m}_\star-1}\big).
 \end{align*}
 In order to show $T_2\le n^{-\delta}$ in \eqref{eq:two-terms-second}, it is much stronger to show that with $\mathfrak{m}_\star-1=\mathfrak{m}'$,
 \begin{equation}
  \mbox{$\forall\varepsilon'>0$, there exists $\varepsilon_1>0$ s.t.\ for all $\varepsilon<\varepsilon_1$: }
  A'\underline{k}_{n,\varepsilon}^{\zeta_\star}(\log\underline{k}_{n,\varepsilon})^{\mathfrak{m}'}<\varepsilon'\log n.
  \label{eq:to-verify-kn}
 \end{equation}
 We recall the definition of $\underline{k}_{n,\varepsilon}$ in \eqref{eq:kn-second-largest} and formally check the two cases.

 \emph{Case 1. $\zeta_\star>0$.} We substitute $\underline{k}_{n,\varepsilon}$ in the first row of \eqref{eq:kn-second-largest} to \eqref{eq:to-verify-kn}
 \begin{align}
A'  \underline{k}_{n,\varepsilon}^{\zeta_\star}(\log\underline{k}_{n,\varepsilon})^{\mathfrak{m}'}
  =
  A'\frac{\varepsilon\log n}{(\log\log n)^{\mathfrak{m}'}}\cdot \Big(\log\Big(\frac{\varepsilon\log n}{(\log\log n)^{\mathfrak{m}'}}\Big)^{1/\zeta_\star}\Big)^{\mathfrak{m}'}.\nonumber
 \end{align}
 The last factor is at most $\zeta_\star^{-\mathfrak m'}(\log\log n)^{\mathfrak{m}'}$, and \eqref{eq:to-verify-kn} follows whenever $\varepsilon<\varepsilon' \zeta_\star^{\mathfrak m'}/A'.$

 \noindent\emph{Case 2. $\zeta_\star=0$.} We substitute $\underline{k}_{n,\varepsilon}$ in the second row of \eqref{eq:kn-second-largest} to \eqref{eq:to-verify-kn} \begin{align}
  A'\underline{k}_{n,\varepsilon}^{\zeta_\star}(\log \underline{k}_{n,\varepsilon})^{\mathfrak{m}'}
  =A' \big(\log\big(\exp\big[(\varepsilon\log n)^{1/\mathfrak{m}'}\big]\big)\big)^{\mathfrak{m}'}=A'\varepsilon\log n,\nonumber
 \end{align}
 and \eqref{eq:to-verify-kn} again follows. Choose now any  $\vartheta\in(1/\min(\alpha,\tau-1),1)$ --- which is possible since $\alpha>1, \tau>2$ ---and then combine \eqref{eq:t1-estimate}
with $T_2\le n^{-\delta}$ to bound \eqref{eq:two-terms-second}. This finishes the proof of \eqref{eq:lower-second-bound} and hence Proposition~\ref{prop:condition-lower} subject to Lemma~\ref{prop:subexponential-lower}.
\end{proof}

\subsection{Lower tail of large deviations}\label{sec:lower-tail}
We finish this section with the proof of Theorem~\ref{prop:lower-dev-giant}, which is based on Lemma~\ref{prop:subexponential-lower}. 
\begin{proof}[Proof of Theorem~\ref{prop:lower-dev-giant}]
For $\rho\ge1$ the statement is trivial.
 There exists a constant $C>0$ such that for any $\rho\in(0, 1)$ and $n\ge 1$ a box of volume $n$ is contained in the union of $\lceil C/\rho\rceil$ (partially overlapping) balls of volume $n\rho/2$. We use balls instead of boxes to reuse the optimally-suppressed mark profile from~\eqref{eq:suppressed-curve} which is defined for a ball; this is a minor technical detail. Fix $\rho\in(0,1)$, and write $\CV^\sss{(i)}$ for the vertices in the $i$-th ball of such a cover of balls of volume $n\rho/2$.
 Recall that $|\CE(A, B)|$ denotes the number of edges between the sets $A,B$. Then 
 \begin{equation}
 \{|\CC_n^\sss{(1)}|<\rho n\} \supseteq \bigcap_{i\le\lceil C/\rho\rceil}\{|\CE(\CV^\sss{(i)}, \CV\setminus\CV^\sss{(i)})|=0\}\cap \{|\CV^\sss{(i)}|<\rho n\}\label{eq:decreasing-events}
 \end{equation}
 Indeed, on the event on the right-hand side, each connected component of $\CG_n$ is fully contained in some ball (or the intersection of some balls) with at most $\rho n$ vertices.
 We apply an FKG inequality to bound the probability of intersection from below.
 
  We give a (natural) definition of increasing events, using the collection $\Psi$ from Definition~\ref{def:encoding} that encodes the presence of edges using a set of uniform random variables $\Psi_\CV=\{\varphi_{u,v}: u,v\in\CV\}$. We say that a function $f(\CV, \Psi_\CV)$ defined on the marked vertex set $\CV$ and edge-variable set $\Psi_\CV$ is \emph{increasing} if it is non-decreasing in $\CV$ with respect to set inclusion (formally, if $\CV'\supseteq \CV, \Psi_{\CV'}\supseteq \Psi_\CV$, then $f(\CV', \Psi_{\CV'})\ge f(\CV, \Psi_{\CV})$ holds), 
  as well as coordinate-wise non-increasing with respect to the edge variables (formally, if $\Psi'_\CV$ satisfies $\varphi'_{u,v}\le \varphi_{u,v}$ for all $u,v\in {\binom \CV 2}$, then $f(\CV, \Psi'_\CV) \ge f(\CV, \Psi_\CV)$ holds). Intuitively this means that more vertices and edges increase the value of $f$.
 Similarly to \cite{dickson2022triangle}, we obtain that for two increasing functions $f_1$, $f_2$,
 \begin{equation}
 \begin{aligned}
 \E[f_1(\CV,\Psi_\CV)\cdot f_2(\CV,\Psi_\CV)] &= \E\big[\E[f_1(\CV,\Psi_\CV)\cdot f_2(\CV,\Psi_\CV)\mid \CV]\big]\\&\ge \E\big[\E[f_1(\CV,\Psi_\CV)\mid \CV]\cdot \E[f_2(\CV,\Psi_\CV)\mid \CV]\big] \\
 &\ge 
 \E[f_1(\CV,\Psi_\CV)\big]\cdot \E[f_2(\CV,\Psi_\CV)]
 ,\nonumber
 \end{aligned}
 \end{equation}
 by applying FKG to the random graph conditioned to have $\CV$ as its vertex set for the first inequality using that $f_1$ and $f_2$ are increasing in the edge-set, and then FKG for point processes for the second inequality \cite[Theorem~20.4]{last2017lectures}.
 We say that an event $\CA$ is decreasing iff the function $-\mathbbm{1}_\CA$ is increasing. It follows that for decreasing events $\CA, \CA'$
 \begin{equation}
 \begin{aligned}
 \Prob\big(\CA\cap\CA'\big)
 =
 \E[(-\mathbbm{1}_{\CA}(\CV,\Psi_\CV))\cdot (-\mathbbm{1}_{\CA'}(\CV,\Psi_\CV))]
 &\ge
 \E[\mathbbm{1}_{\CA}(\CV,\Psi_\CV)]\cdot\E[\mathbbm{1}_{\CA'}(\CV,\Psi_\CV)]\\
 &= \Prob\big(\CA\big)\cdot\Prob\big(\CA'\big).
 \end{aligned}\label{eq:fkg}
 \end{equation}
 Observe that the events on the right-hand side in \eqref{eq:decreasing-events} are all decreasing (adding vertices/edges make the events less likely to occur) so that \eqref{eq:fkg} applies. Hence,
 \begin{equation}\label{eq:after-FKG}
 \Prob\big(|\CC_n^\sss{(1)}|<\rho n\big) \ge \prod_{i\le\lceil C/\rho\rceil}\Prob\big(|\CE(\CV^\sss{(i)}, \CV\setminus\CV^\sss{(i)})|=0\big)\cdot\Prob\big(|\CV^\sss{(i)}|<\rho n\big).
 \end{equation}
Since each ball has volume $n\rho/2$, the event $\{|\CV^\sss{(i)}|<\rho n\}$ holds with probability at least $1/2$ by concentration inequalities for Poisson random variables (Lemma~\ref{lemma:poisson-1} for $x=2$). To bound $\Prob\big(|\CE(\CV^\sss{(i)}, \CV\setminus\CV^\sss{(i)})|=0\big)$, we consider the optimally-suppressed mark profile translated to the center of the $i$-th ball, with $kM_\mathrm{in}$ replaced by $n\rho/2$. We restrict $\CV$ to be below the mark profile, and to have no edges between $\CV^\sss{(i)}$ and $\CV\setminus \CV^\sss{(i)}$.
 We apply Lemma~\ref{prop:subexponential-lower}, integrate over all realizations  of $\CG_{\mathrm{in}}, \CG_{\mathrm{out}}$ satisfying $\CA_\mathrm{regular}$, and use that
the event $\CA_\mathrm{regular}$ in the conditioning in $\widetilde\Prob$ in \eqref{eq:cond-prob-lower-1} holds with high probability by Poisson concentration (Lemma~\ref{lemma:poisson-1} for $x=2$), see the argument below \eqref{eq:bgiant}.
 We obtain that for all $i\le\lceil C/\rho\rceil$,
 \[
 \Prob\big(|\CE(\CV^\sss{(i)}, \CV\setminus\CV^\sss{(i)})|=0\big)\cdot\Prob\big(|\CV^\sss{(i)}|<\rho n\big)
 \ge  \exp\big(-\Theta(n^{\zeta_\star}(\log n)^{\mathfrak{m}_\star-1})\big)/2,
 \]
 which proves \eqref{eq:lower-dev-giant} when taking the product over $\lceil C/\rho\rceil$ balls in~\eqref{eq:after-FKG}.
\end{proof}

\section{Proofs of main results}\label{sec:main-proofs}
We conclude the paper by formally verifying the statements in Sections~\ref{sec:intro} and~\ref{sec:main-results}, starting with the main results.
\begin{proof}[Proofs of Theorems \ref{thm:subexponential-decay}--\ref{thm:second-largest}(i)]
Proposition~\ref{prop:condition-lower} proves the lower bounds in Theorems~\ref{thm:subexponential-decay}--\ref{thm:second-largest}: its condition \eqref{eq:lower-sub-cond} on having a large enough component on restricted marks occurs with positive probability by Proposition~\ref{proposition:existence-large} when $\zeta_\mathrm{hh}>0$. The assumption $\tau>2$ is necessary to have a locally finite graph with multiple components.
\end{proof}
\begin{proof}[Proofs of Theorems \ref{thm:subexponential-decay}--\ref{thm:second-largest}(ii-iii), and Corollary \ref{cor:lln}]
Proposition~\ref{prop:2nd-upper-bound-hh} proves the upper bounds (part ii-iii) in Theorem~\ref{thm:second-largest}.  Substituting $k=k_n=(A\log n)^{1/\zeta_\mathrm{hh}}$ for a sufficiently large constant $A=A(\delta)$ yields part (ii), which uses $\tau\ge \sigma+1$. For part (iii), i.e., when $\tau<\sigma+1$, we substitute $k=(A\log n)^{\sigma+1-(\tau-1)/\alpha}$ instead. The condition $\zeta_\mathrm{hh}>0$ is required to construct a backbone of high-mark vertices (Lemma~\ref{lemma:upper-hh-bb}), and to merge components of size at least $k$ with the backbone via a high-mark vertex in \eqref{eq:sure-connection}. The distinction between $\tau\ge \sigma+1$ and $\tau<\sigma+1$ arises from the cover-expansion step in Lemma~\ref{lemma:hh-expandable}.

For the proof of Theorem~\ref{thm:subexponential-decay}(ii-iii) and Corollary~\ref{cor:lln} it suffices to verify  prerequisites~\eqref{eq:prop-2nd}--\eqref{eq:prop-marksgiant} of Proposition~\ref{prop:condition-subexponential}.
Let $\zeta=\zeta_\mathrm{hh}$ when $\tau\ge \sigma+1$ and $\zeta=1/(\sigma+1-(\tau-1)/\alpha)$ when $\tau<\sigma+1$. Set $c_1=0$, $c_2=1$, and let $c_3>0$ be a sufficiently small constant. Then~\eqref{eq:prop-2nd} is implied by Proposition~\ref{prop:2nd-upper-bound-hh}, \eqref{eq:prop-outgiant} by Corollary~\ref{cor:lower-c1}, and~\eqref{eq:prop-marksgiant} by Proposition~\ref{prop:min-giant-weight} (we leave it to the reader to verify that these statements hold for the Palm-version $\Prob^\sss{0}$ of $\Prob$ as well). 
\end{proof}
We continue with the statements in Section~\ref{sec:intro}.
\begin{proof}[Proof of Theorem~\ref{thm:informal}]
    For continuum scale-free percolation, geometric inhomogeneous random graphs, and hyperbolic random graphs we have $\sigma=1$. When $\tau\in(2,3)$, then $\tau\ge \sigma+1$, and $\zeta_\mathrm{GIRG}=(3-\tau)/(2-(\tau-1)/\alpha)$ agrees with $\zeta_\mathrm{hh}$ from~\eqref{eq:zeta-hh}. The statement assumes that $\zeta_\mathrm{GIRG}>\max\big(2-\alpha, (d-1)/d\big)=\max(\zeta_\mathrm{ll}, \zeta_\mathrm{short})$ by~\eqref{eq:zeta-ll} and~\eqref{eq:zeta-nn}. This implies that $\alpha>\tau-1$, and also that  $\zeta_\mathrm{GIRG}>\zeta_\mathrm{hl}=(\tau-1)/\alpha-(\tau-2)$ when $\alpha>\tau-1$. As a result, $\mathfrak{m}_\star$ in~\eqref{eq:m-star} is equal to $1$ and there are no polylog factors in Theorems \ref{thm:subexponential-decay}--\ref{thm:second-largest}(i), and $\zeta_\star=\zeta_\mathrm{GIRG}$ by Lemma~\ref{lemma:lower-vertex-boundary2}.    
    The statements in~\eqref{eq:product-kernel} now follow from Theorems \ref{thm:subexponential-decay}--\ref{thm:second-largest}(i-ii), \ref{prop:lower-dev-giant}, and Corollary~\ref{cor:lln}.
        We mention that hyperbolic random graphs are generally defined with exactly $n$ vertices on an $n$-dependent hyperbolic space, giving an $n$-dependent vertex-mark distribution and an $n$-dependent connection probability function. However, these converge (fast) to their limiting distribution and connection probabilities, and can be bounded from above and from below by connection probabilities satisfying  Assumption \ref{assumption:main} respectively, see \cite[below Equation (9.8); and Equations (9.16) (9.17)]{KomLod20}. So, one can build the same structures as we did here and use these upper bounding connection probabilities in upper bounds and the lower bounding connection probabilities in lower bound estimates to arrive to the same result for HRGs. The results generally extend to models where  the number of vertices is exactly $n$, and where vertex locations are independent uniform random variables on $\Lambda_n$, by conditioning on a Poisson$(n)$ variable to be exactly $n$. We leave this technical extension to the reader: one needs to replace concentration bounds for Poisson random variables with Chernoff bounds, and one also needs to add extra events that control the number of vertices in certain space-mark areas.
 \end{proof}
\begin{proof}[Proof of Claim~\ref{lemma:lower-vertex-boundary}]
    The statement is implied by Lemma~\ref{lemma:lower-vertex-boundary2}.
\end{proof}

\subsection*{Acknowledgements}
We thank the two anonymous referees for their careful reading of the manuscript which led to significant improvement of the paper.
During the preparation of the manuscript, JJ was employed at Eindhoven University of Technology and CWI Amsterdam.
 JJ thanks Johannes Lengler for stimulating discussions during a three-month visit to ETH Z\"urich, which was supported by Swiss National Science Foundation (SNF) grant 192079, and the Netherlands Organisation for Scientific Research (NWO) Gravitation-grant NETWORKS-024.002.00. The work of JJ and JK is partly supported through grant NWO 613.009.122.  The work of DM is partially supported by grant Fondecyt grant 1220174 and by grant GrHyDy ANR-20-CE40-0002.

\begin{appendix}

\section{Proofs based on backbone construction}
We present the proofs of the propositions at the end of Section~\ref{sec:upper-2nd}.
\begin{proof}[Proof of Proposition~\ref{prop:min-giant-weight}]\label{proof:min-giant-weight}
We give the detailed proof for $\tau\ge \sigma+1$. At the end of the proof we explain the adjustments for $\tau<\sigma+1$. 
 We will first derive a bound on $\Prob\big(\exists v\in\CV_{n}[\overline{w}, \infty): v\notin\CC_{n}^\sss{(1)}\big)$ for arbitrary $\overline{w}\ge 1$.
 We make use of the backbone construction from Section~\ref{sec:upper-2nd}: we will show that vertices with mark at least $\overline{w}$ are likely to connect by an edge to the backbone $\CC_{\mathrm{bb}}$, which will be a subset of the giant component. Observe that the event in \eqref{eq:min-giant-weight} allows us to choose the size of the boxes when we build the backbone, i.e., the value of $k$ is not yet defined with respect to $\overline w$. We define $k=k(\overline w)$ implicitly  by $\overline{w}=:A_1k^{1-\sigma\gamma_\mathrm{hh}}$, where $A_1$ is a large enough constant to be determined later. We aim to show that for some $A_2>0$, and $n\ge k$,
 \begin{equation}
  \Prob\big(\neg\CA_\mathrm{mark\textnormal{-}giant}(n,\overline{w})\big)\!:=\!\Prob\big(\exists v\!\in\!\CV_{n}[\overline w,\infty)\!:\! v\!\notin\!\CC_{n}^\sss{(1)}\big)\!\le\! n\exp\!\big(\!-\!A_2 k(\overline w)^{\zeta_\mathrm{hh}}\!\big).\label{eq:min-weight-bb}
 \end{equation}
 If this bound holds, then substituting $\overline{w}=\overline{w}_{n}=(M_w \log n)^{(1-\sigma\gamma_\mathrm{hh})/\zeta_\mathrm{hh}}$ yields the value 
 \[ k(\overline w_n)=A_1^{-1/(1-\sigma \gamma_\mathrm{hh})} \overline{w}_{n}^{1/(1-\sigma \gamma_\mathrm{hh})} =A_1^{-1/(1-\sigma \gamma_\mathrm{hh})} (M_w\log n)^{1/\zeta_{\mathrm{hh}}}.\] When we substitute this back to  \eqref{eq:min-weight-bb} we obtain that for $M_w$ sufficiently large the right-hand side is at most $n^{-\delta}$, as required in \eqref{eq:min-giant-weight}. We now prove \eqref{eq:min-weight-bb}.

 Recall $\CA_\mathrm{bb}(n,k)$ and $\CC_\mathrm{bb}(n,k)$ from \eqref{eq:upper-hh-bb-event}.
 Distinguishing two cases depending on whether $\CA_{\mathrm{bb}}(n,k)$ holds for $\CG_{n,2}=\CG_n[1,2w_\mathrm{hh})$ or not (with $w_\mathrm{hh}(k)$ in \eqref{eq:w-gamma-hh}), by Lemma~\ref{lemma:upper-hh-bb},
 \begin{align}
  \Prob\big(\neg\CA_\mathrm{mark\textnormal{-}giant}(n,\overline{w})\big) & \le \Prob\big(\neg\CA_\mathrm{bb}\big)+
  \E\big[\ind{\CA_\mathrm{bb}}\Prob\big(\neg\CA_\mathrm{mark\textnormal{-}giant}(n,\overline{w})\mid \CG_{n,2}, \CA_\mathrm{bb}\big)\big] \nonumber\\ &\le 3n\exp(-c_{\ref{lemma:upper-hh-bb}}k^{\zeta_{\mathrm{hh}}})\!+\! \E\big[\ind{\CA_\mathrm{bb}}\Prob\big(\neg\CA_\mathrm{mark\textnormal{-}giant}(n,\overline{w})\!\mid\! \CG_{n,2}, \CA_\mathrm{bb}\big)\big].\label{eq:min-weight-pr1}
 \end{align}
 On the event $\CA_\mathrm{bb}$, there is a backbone $\CC_\mathrm{bb}$. This backbone is either not part of the giant component, or if it is, then a vertex with mark at least $\overline{w}$ outside the giant has  \emph{no connection} to any of the vertices in the backbone. Hence, conditionally on the event $\CA_\mathrm{bb}$,
 \begin{align}
  \neg\CA_\mathrm{mark\textnormal{-}giant}(n,\overline{w})
  \subseteq
  \{\CC_\mathrm{bb}\nsubseteq \CC_n^\sss{(1)}\}\cup
  \{\exists v\in\CV_n[\overline{w},\infty): v\not\sim \CC_\mathrm{bb} , \CC_\mathrm{bb}\subseteq \CC_n^\sss{(1)}\}.\nonumber
 \end{align}
 By a union bound and Corollary \ref{cor:bb-in-giant}, this implies that
 \begin{align}
  \Prob\big(\neg\CA_\mathrm{mark\textnormal{-}giant}&(n,\overline{w})  \mid \CG_{n,2}, \CA_\mathrm{bb}\big) \nonumber\\
                                                        & \le
  (n/k)\exp\big(-c_{\ref{prop:2nd-upper-bound-hh}}k^{\zeta_\mathrm{hh}}\big)
  +
  \Prob\big(\exists v\in\CV_n[\overline{w},\infty): v\not\sim \CC_\mathrm{bb}\mid \CG_{n,2}, \CA_\mathrm{bb}\big)\label{eq:min-weight-pr2}.
 \end{align}
 Recall that $\CG_{n,2}$ is the graph spanned on vertices with mark in $[1, 2w_{\mathrm{hh}})$, see Definition~\ref{def:ksrg-alt}.
 With $C_1$ from \eqref{eq:upper-hh-c1}--\eqref{eq:c1-infty}, we may assume $A_1\ge 2C_1^{-1/(\tau-1)}$. Since  $w_\mathrm{hh}=C_1^{-1/(\tau-1)}k^{\gamma_{\mathrm{hh}}}$ defined in \eqref{eq:w-gamma-hh}, and since $1-(1+\sigma)\gamma_{\mathrm{hh}}\ge 0$ (see \eqref{eq:gamma-upper-sigma}), this implies that
 \begin{equation}\label{eq:overline-w-k}
  \overline w=A_1k^{1-\sigma\gamma_\mathrm{hh}}\ge 2w_\mathrm{hh}=2 C_1^{-1/(\tau-1)}k^{\gamma_{\mathrm{hh}}}\nonumber.
 \end{equation}
 Hence, vertices of mark at least $\overline w$ are part of $\CV_n[2w_{\mathrm{hh}},\infty)$ and are not revealed in $\CG_{n,2}$.
 Conditioning on  the number of vertices $|\CV_n[\overline w,\infty)|$, the location of each vertex is independent and uniform in $\Lambda_n$. Taking a union bound over these vertices in $\CV_n[\overline w,\infty)$, yields
 \begin{align}
  \Prob\big(\exists v\in\CV_{n}[\overline w,\infty) & : v\not\sim \CC_{\mathrm{bb}} \mid \CG_{n,2}, \CA_\mathrm{bb}\big) \nonumber\\ &\le
  \E[|\CV_n[\overline w,\infty)|]\cdot
  \sup_{v\in\CV_n[\overline w,\infty)}\Prob\big( v\not\sim \CC_{\mathrm{bb}} \mid \CG_{n,2}, \CA_\mathrm{bb}\big)           \nonumber\\
&\le
  n\,
  \sup_{v\in\CV_n[\overline w,\infty)}\Prob\big( v\not\sim \CC_{\mathrm{bb}} \mid \CG_{n,2}, \CA_\mathrm{bb}\big).\label{eq:overlinew-connect}
 \end{align}
 Now we use that the backbone is spatially `everywhere'. Let $\CQ(v)$ be the box of $v$ as in \eqref{eq:qu}.
 Conditionally on $\CA_\mathrm{bb}$, $\CQ(v)$ contains at least $s_k=\Theta(k^{\zeta_\mathrm{hh}})$ vertices in $\CC_\mathrm{bb}$ with mark in $[w_\mathrm{hh}, 2w_\mathrm{hh})$, where $w_\mathrm{hh}$ is defined in \eqref{eq:w-gamma-hh}, yielding the set of vertices $\CS(v)$ in \eqref{eq:su}. We use the distance bound in \eqref{eq:upper-hh-outside-tesselation},
 and $\mathrm{p}$, $\kappa_\sigma$ defined in \eqref{eq:connection-prob-gen}, and \eqref{eq:kernels}, respectively, and the value $\overline w$ in \eqref{eq:overline-w-k}, to obtain that for any $v\in \CV_n[\overline w, \infty)$ and  $u\in\CS(v)$, when $\alpha<\infty$,
 \begin{align*}
  \mathrm{p}(u, v) & \ge
  p\Big(1\wedge \big(\beta\kappa_\sigma(w_\mathrm{hh}, A_1k^{1-\sigma\gamma_\mathrm{hh}})(2\sqrt{d})^{-d}k^{-1}\big)\Big)^\alpha \\
                  & =
  p\Big(1\wedge \big(\beta C_1^{-\sigma/(\tau-1)}k^{\sigma\gamma_\mathrm{hh}}A_1k^{1-\sigma\gamma_\mathrm{hh}} (2\sqrt{d})^{-d}k^{-1}\big)^\alpha\Big) = p,
 \end{align*} whenever $A_1\ge (2\sqrt{d})^d C_1^{\sigma/(\tau-1)}/\beta$,
 since the exponent of $k$ in the second term of the minimum is $0$.
 The same bound holds when $\alpha=\infty$.
 Since $v$ connects by an edge to each of the $s_k=\Theta(k^{\zeta_\mathrm{hh}})$ many backbone vertices in $\CS(v)$ with probability at least $p$, conditionally independently of each other, we bound \eqref{eq:overlinew-connect} by
 \[
  \begin{aligned}
  \Prob\big(\exists v\in\CV_{n}[\overline w,\infty): v\not\sim \CC_{\mathrm{bb}} \mid \CG_{n,2}, \CA_\mathrm{bb}\big) & \le n(1-p)^{s_k}.
  \end{aligned}
 \]
 Since $s_k=\Theta(k^{\zeta_{\mathrm{hh}}})$ in \eqref{eq:w-gamma-hh}, combining this with \eqref{eq:min-weight-pr1} and \eqref{eq:min-weight-pr2}  yields \eqref{eq:min-weight-bb} for $A_2$  sufficiently small. As argued below \eqref{eq:min-weight-bb}, this yields \eqref{eq:min-giant-weight} when $\tau\ge\sigma+1$. When $\tau<\sigma+1$, the exponent $\zeta_\mathrm{hh}$ in the exponential on the right-hand side in~\eqref{eq:min-weight-bb} and in the first term on the right-hand side in~\eqref{eq:min-weight-pr1} and~\eqref{eq:min-weight-pr2} change to $1/(\sigma+1-(\tau-1)/\alpha)$ due to Corollary~\ref{cor:bb-in-giant}. Setting $\overline w_n=(M_w \log n)^{(1-\sigma\gamma_\mathrm{hh})(\sigma+1-(\tau-1))}$ proves~\eqref{eq:min-giant-weight} when $\tau<\sigma+1$.
\end{proof}
\subsection{Construction of a linear-sized component}\label{app:linear-sized}
We will prove the first statement of Proposition \ref{proposition:existence-large}. At the end of the section, we comment how the proof can be adjusted to obtain the second statement considering the infinite model. Throughout the proof, we will consider the Palm version of $\mathbb{P}$, conditioning $\CV$ to contain a vertex at location 0. We will leave it out in the notation. We will show using a second-moment method that linearly many vertices connect to the backbone
$\CC_\mathrm{bb}(n, k)$ for a properly chosen $k=k_n$ that we define now
Recall $C_1$ from \eqref{eq:upper-hh-c1}-\eqref{eq:c1-infty}. Implicitly define $k=k_n$ as the solution of the equation
\begin{equation}
 (C_1/16)k_n^{\zeta_\mathrm{hh}}:=(2/c_{\ref{prop:2nd-upper-bound-hh}})\log n,\label{eq:kn-paths-to-bb}
\end{equation}
yielding $m=\big(32/(C_1c_{\ref{prop:2nd-upper-bound-hh}})\big)^{1/\zeta_{\mathrm{hh}}}$ in the statement of Proposition \ref{proposition:existence-large}, and the mark-truncation value in the definition of $\CG_{n,2}$ at  $2w_{\mathrm{hh}}(k_n)=2C_1^{-1/(\tau-1)}k_n^{\gamma_\mathrm{hh}}$ from \eqref{eq:w-gamma-hh}, with  $w_\mathrm{hh}(k_n) = (M_w\log(n))^{\gamma_\mathrm{hh}/\zeta_{\mathrm{hh}}}$ for some constant $M_w$.
We reveal the realization of the graph $\CG_{n,1}=\CG_n[w_\mathrm{hh}(k_n), 2w_\mathrm{hh}(k_n))$ (defined above \eqref{eq:upper-hh-c1}), conditioned to satisfy the event $\CA_\mathrm{bb}(n,k_n)$ in \eqref{eq:upper-hh-bb-event}.
Recalling the intensity measure of the Poisson vertex set $\mu_{\tau}$ from \eqref{eq:poisson-intensity},  for any constant $m_w\ge1$ we define the event
\begin{equation}
 \CA_\mathrm{reg}':= \big\{|\CV_n[m_w, 2m_w)|\big/\mu_\tau\big(\Lambda_n\times[m_w, 2m_w)\big)\, \in[1/2, 2]\big\},\label{eq:global-prime}
\end{equation}
that is, that the PPP in $\Lambda_n$ is regular in the sense that the number of constant-mark vertices is roughly as expected. Writing $w_0$ for the mark of $0$,  we define the conditional probability measure
\begin{equation}
 \widetilde\Prob_\mathrm{bb}(\,\cdot\,):= \Prob\big(\,\cdot\mid \CA_\mathrm{reg}', \CA_\mathrm{bb}, \CG_{n}[1, w_\mathrm{hh}), w_0\in[m_w, 2m_w)\big),\label{eq:lower-cond-prob-2}
\end{equation}
with corresponding expectation $\widetilde\E_\mathrm{bb}$. We state  a lemma that implies Proposition \ref{proposition:existence-large}.
\begin{lemma}[Constructing a component]\label{lemma:construction}
 Consider a KSRG under the same conditions as Proposition \ref{proposition:existence-large}.
 Take $k_n$ as in \eqref{eq:kn-paths-to-bb}. For any $m_w\ge1$, for all sufficiently large $n$,
 \begin{equation}
  \Prob\big(\CA_\mathrm{reg}', \CA_\mathrm{bb}, w_0\in[m_w, 2m_w)\big) \ge m_w^{-(\tau-1)}/2.\label{eq:construction-bb}
 \end{equation}
 Moreover, there exist constants $m_w\ge1, \rho_{\ref{lemma:construction}}>0$, such that for all sufficiently large $n$,
 \begin{equation}
  \widetilde\Prob_\mathrm{bb}\big(|\CC_{n}(0)[1, 2w_{\mathrm{hh}}(k_n))|\ge \rho_{\ref{lemma:construction}} n\big)\ge \rho_{\ref{lemma:construction}}.\label{eq:constructing-1}
 \end{equation}
\end{lemma}
To prove the lemma (in particular the second statement), we need to define auxiliary notation and an auxiliary claim: we define for $u:=(x_u, w_u)\in\CV_n[m_w, 2m_w)$  the event
\begin{equation}
 \{u\overset{\pi}\longrightarrow\CC_\mathrm{bb}\}:= \left\{
 \begin{aligned} & \mbox{$\exists$  a path  in $\CG_{n}[1,2w_\mathrm{hh})$   from $u$ to some vertex in $\CC_\mathrm{bb}$, } \\
  & \mbox{with all vertices (except $u$) having mark in  $\R^+\setminus[m_w, 2m_w)$}\end{aligned}\right\}.\label{eq:u-pi-bb}
 \end{equation}
Recall $\Lambda(0, s)$  from \eqref{eq:xi-q-ab} the box centered at $0$ of volume $s$, and that $u=(x_u,w_u)$ denotes the location and mark of a vertex $u\in \CV$. The next claim states that if $u$ and $0$ are both vertices with mark in $[m_w, 2m_w)$, falling into different subboxes $\CQ(0)\neq \CQ(u)$ (where $\CQ(u)$ denotes the box of the tessellation that contains the vertex $u$, or is closest to $u$), then the event that both $0$ and $u$ connect to the backbone $\CC_{\mathrm{bb}}$ happens with constant probability.
\begin{claim}[Paths to the backbone]\label{lemma:lower-conn-bb-2}
 Consider a KSRG under the conditions of Proposition \ref{proposition:existence-large}. There exist positive constants $m_w, q_{\ref{lemma:lower-conn-bb-2}}, C_{\ref{lemma:lower-conn-bb-2}}> 0$ such that for all   $u\in\CV_n[m_w, 2m_w)$ with $x_u\notin\Lambda(0, C_{\ref{lemma:lower-conn-bb-2}}k)$, and $n$ sufficiently large
 \begin{align}
  \widetilde\Prob_\mathrm{bb}\big(\{0\!\overset{\pi}\longrightarrow\!\CC_\mathrm{bb})\}\!\cap\! \{ u\!\overset{\pi}\longrightarrow\!\CC_\mathrm{bb}\}\!\mid\! u\!\in\!\CV_n[m_w, 2m_w),  x_u\!\notin\!\Lambda(0, C_{\ref{lemma:lower-conn-bb-2}}k)\big)\ge
  q_{\ref{lemma:lower-conn-bb-2}}.\nonumber
 \end{align}
 \begin{proof}
 We will build what we call "mark-increasing paths". Recall $k=k_n=m(\log(n))^{1/\zeta_{\mathrm{hh}}}$ in Proposition~\ref{proposition:existence-large} and  that $w_\mathrm{hh}=w_{\mathrm{hh}}(k_n)=(M_w\log(n))^{\gamma_\mathrm{hh}/\zeta_\mathrm{hh}}$ by~\eqref{eq:kn-paths-to-bb} and~\eqref{eq:w-gamma-hh} for some constant $M_w>0$. Define
 \begin{equation}\label{eq:j-star}
  j_\ast:=\max\{j: 2^{j+1}m_w<w_\mathrm{hh}(k_n)\},
 \end{equation}
  and define for $0\le j\le j_\ast$ and $x\in\Lambda_n$ the following boxes and disjoint mark intervals:
 \begin{align}
  Q_j(x) := \Lambda\big(x, \beta 2^{-\sigma}d^{-d/2}(2^{j}m_w)^{\sigma+1}\big)\cap\Lambda_n, \qquad
  I_j := [2^jm_w, 2^{j+1}m_w),
  \label{eq:lower-bj}
 \end{align}
 and write $Q_{j_\ast+1}:=\CQ(x)$, i.e., the volume-$k_n$ subbox   containing $x$ in the partitioning of $\Lambda_n$ for the backbone construction given in Section~\ref{sec:upper-2nd}, and define also $I_{j_\ast+1}:=[w_\mathrm{hh}, 2w_\mathrm{hh})$.
 Even with the truncation by $\Lambda_n$ in \eqref{eq:lower-bj}, the volume bound  
 \begin{equation}
 \beta 2^{-\sigma-d}d^{-d/2}(2^{j}m_w)^{\sigma+1}\le \mathrm{Vol}(Q_j(x))\le \beta 2^{-\sigma}d^{-d/2}(2^{j}m_w)^{\sigma+1}\label{eq:vol-qj}
 \end{equation}
 holds for all $x\in\Lambda_n$ and $j\le j_\ast$.
 By \eqref{eq:poisson-intensity}, the number of vertices of $\CV$ in $Q_j(x_u)\times I_j$ has Poisson distribution with mean
 \begin{equation}
  \mu_\tau\big(Q_j(x_u)\times I_j\big) \ge 2^{-1}\beta(2^j m_w)^{-(\tau-1)}2^{-\sigma-d}d^{-d/2}(2^{j}m_w)^{\sigma+1},
  \label{eq:qj-measure}
 \end{equation}
 where we used that $1-2^{-(\tau-1)}\ge 2^{-1}$ for $\tau>2$.
 Moreover,  \eqref{eq:j-star} implies $2^jm_w\le w_{\mathrm{hh}}$ for all $j\le j_\star$ and substituting this into \eqref{eq:lower-bj} with $w_\mathrm{hh}=C_1^{-1/(\tau-1)}k_n^{\gamma_\mathrm{hh}}$  from \eqref{eq:w-gamma-hh} yields
 \begin{equation}
  \begin{aligned}
  Q_j(x) & \subseteq
  \Lambda\big(x, \beta 2^{-\sigma}d^{-d/2}w_\mathrm{hh}^{\sigma+1}\big)
  =
  \Lambda\big(x, \beta 2^{-\sigma}d^{-d/2}C_1^{-(\sigma+1)/(\tau-1)}k_n^{\gamma_\mathrm{hh}(\sigma+1)}\big) \\
          & \subseteq
  \Lambda\big(x, \beta 2^{-\sigma}d^{-d/2}C_1^{-(\sigma+1)/(\tau-1)}k_n\big)=:Q^\star(x),\nonumber
  \end{aligned}
 \end{equation}
  where the last inclusion follows from  $\gamma_\mathrm{hh}\le1/(\sigma+1)$ by \eqref{eq:gamma-hh} for all parameters $\alpha,\tau$ such that $\zeta_\mathrm{hh}>0$. Let $\mathrm{diam}(Q^\star(x))=:(Ck_n)^{1/d}$ denote the diameter of $Q^\star(x)$, and let $Q^\diamond:=\Lambda(0, 2^{d}Ck_n)$ denote the box centered at $0$, such that  $\mathrm{diam}(Q^\diamond)=2\mathrm{diam}(Q^\star)$. For any $x_u\notin\CQ^\diamond$, and all pairs $j,j'\le j_\ast$, it holds that
 \begin{equation}\label{eq:key-to-indep}
  Q_j(x_u)\cap Q_{j'}(0)=\emptyset,
 \end{equation}
  and thus the PPPs restricted to $Q_j(x_u)\times I_j$ are independent for $j \neq j'$, and the PPP restricted to $Q_j(x_u)\times I_j$ is also independent of the PPP in $Q_{j'}(0)\times I_{j'}$ for all $j,j'\le j_\ast$.  On the conditional measure $\widetilde \Prob_\mathrm{bb}$, defined in \eqref{eq:lower-cond-prob-2}, we fixed (revealed) the realization of $\CV_n[w_\mathrm{hh}, 2w_\mathrm{hh})$. Edges among $\CV_n[w_\mathrm{hh}, 2w_\mathrm{hh})$ and $\CV\cap (Q_{j_*}(x_u)\times I_{j_*})$ are thus also present conditionally independently.
 We define for $u=(x_u, w_u)=:u_0$ the event of having a "mark-increasing path" (a subevent of $\{u\overset{\pi}\longrightarrow\CC_\mathrm{bb}\}$ defined in \eqref{eq:u-pi-bb}):
 \begin{equation}
  \{u\leadsto\CC_\mathrm{bb}\}:=\left\{\begin{aligned} &\exists (u_1,\dots, u_{j_\ast+1}), u_{j_*+1}\in \CC_{\mathrm{bb}}, \\ &\forall j \in [j_{\ast}+1]: u_j\in Q_j(x_u)\times I_j, u_{j-1}\leftrightarrow u_j\end{aligned}\right\}.\label{eq:u-leadsto-bb}
 \end{equation}
  on which there is a path from $u$ to the backbone, where the $j$th vertex on the path is in  $Q_j(x_u)\times I_j$ (that are disjoint across $j$). The mark of $u_{j_*+1}$ is in the right range by definition $I_{j_\ast+1}$ of  below \eqref{eq:lower-bj}.
 By this disjointness and \eqref{eq:key-to-indep}, the events $\{0\leadsto\CC_\mathrm{bb}\}$ and $\{u\leadsto\CC_\mathrm{bb}\}$ are  independent conditionally on $\CV_n[w_\mathrm{hh}, 2w_\mathrm{hh})$.

 To bound $\widetilde \Prob_{\mathrm{bb}}(u\leadsto\CC_\mathrm{bb})$ from below,  we greedily `construct' a path from $u=u_0$ to the backbone. By assumption, $w_u\in[m_w, 2m_w)$ hence $u_0\in Q_0(x_u)\times I_0$. We first bound the probability that $u_0$ connects by an edge to a vertex $u_1\in Q_1(x_u)\times I_1$. Then, if there is such a connection, we choose $u_1$ to be an arbitrary vertex connected to $u_0$, and give a uniform lower bound (over the possible $u_1$) on the probability that it connects by an edge to a vertex $u_2\in Q_2(x_u)\times I_2$. We continue this process until we reach $u_{j_\ast}$ that has mark just smaller than the minimal mark of vertices in the backbone, by definition of $j_\ast$ in \eqref{eq:j-star}.
 Then we find a connection from $u_{j_\ast}$ to the backbone.
 We now bound the probability that two vertices $u_{j-1}$ and $u_{j}$ are connected by an edge.

 By construction, $w_{u_{j-1}}\ge 2^{j-1}m_w$, $w_{u_{j}}\ge 2^{j}m_w$. Hence,
 with the kernel $\kappa_\sigma$ from \eqref{eq:kernels}, and volume bound \eqref{eq:vol-qj}, we obtain 
 \[
  \beta\kappa_\sigma(w_{u_{j-1}}, w_{u_{j}}) = \beta w_{u_{j-1}}^\sigma w_{u_{j}} \ge \beta m_w^{\sigma+1}2^{(\sigma + 1)j-\sigma}\ge\beta2^{-\sigma}(2^{j}m_w)^{\sigma+1} 
 \]
Further, their distance $\|x_{u_{j-1}}-x_{u_{j}}\|^d\le d^{d/2}\mathrm{Vol}(Q_{j}(x_u))\le \beta2^{-\sigma}(2^{j}m_w)^{\sigma+1}$ by \eqref{eq:vol-qj}, hence
 \[
  \mathrm{p}(u_{j-1}, u_{j}) \ge p\Big(1\wedge \frac{\beta\kappa_\sigma(w_{u_j}, w_{u_{j+1}})}{\|x_{u_j}-x_{u_{j+1}}\|^{ d}}\Big)^\alpha = p.
 \]
 The same computation (without $\alpha$ in the exponent) is valid for $\alpha=\infty$. Having already chosen $u_{j-1}$ on the path, each  $v\in \CV_n \cap (Q_j(x_u)\times I_j)$ connects independently by an edge to $u_{j-1}$ with probability $p$. The conditioning in the measure $\widetilde \Prob_{\mathrm{bb}}$ in \eqref{eq:lower-cond-prob-2} only affects the number of vertices with mark  in $[m_w, 2m_w)$ and in $[w_{\mathrm{hh}}, 2w_{\mathrm{hh}})$. Due to independence of the number of points of PPPs in disjoint sets, for $j\le j_\ast$, the number of candidate vertices for the role of $u_j$ is thus stochastically dominated from below by a $\mathrm{Poi}\big(p\mu_{\tau}(Q_j(x_u)\times I_j)\big)$ random variable. The mean is at least $p\beta2^{-\sigma-d-1}d^{-d/2}(2^{j}m_w)^{\sigma+2-\tau}$ by \eqref{eq:poisson-intensity} and \eqref{eq:qj-measure}. For $j=j_\ast+1$ we use that the vertices in the backbone in $Q_{j_\ast + 1}\times I_{j_\ast + 1}$ are exactly those in $\CQ(u)\times [w_{\mathrm{hh}},2w_{\mathrm{hh}})$, that we denote by $\CS(u)$:
 \begin{align}
  \widetilde \Prob_{\mathrm{bb}}\big(\neg\{u\leadsto\CC_\mathrm{bb})\big) & \le \widetilde \Prob_{\mathrm{bb}}\big(u_{j_\ast}\not\leftrightarrow \CS(u)\big)+ \sum_{j=1}^{j_\ast}\Prob\big(\Poi(p\beta2^{-\sigma-d-1}d^{-d/2}(2^{j}m_w)^{\sigma+2-\tau})=0\big)\nonumber \\
                                                                          & \le
  \widetilde \Prob_{\mathrm{bb}}\big(u_{j_\ast}\not\leftrightarrow \CS(u)\big)+\sum_{j=1}^\infty\exp\big(-p\beta2^{-\sigma-d-1}d^{-d/2}(2^{j}m_w)^{\sigma+2-\tau}\big).\label{eq:no-conn-bb}
 \end{align}
 Since $\tau<2+\sigma$ by assumption, the sum can be made arbitrarily small by choosing $m_w$ sufficiently large.
 By definition of $j_\ast$ and $I_{j_\ast}$ in \eqref{eq:j-star}, $w_{u_{j_\ast}}\ge 2^{j_\ast}m_w\ge w_\mathrm{hh}/4$.
 We recall that $\CQ(x_u)$ is the subbox of volume $k_n$ in the partitioning for the backbone containing $x_u$, that contains at least $s_{k_n}$ backbone vertices of mark at least $w_\mathrm{hh}$, both defined in \eqref{eq:w-gamma-hh}. When $\alpha<\infty$, we follow the computations in~\eqref{eq:temp-est-1} (that are also valid under the conditional measure $\widetilde \Prob_{\mathrm{bb}}$ from \eqref{eq:lower-cond-prob-2}, ensuring that $\CC_{\mathrm{bb}}$ exists and $|\CS(u)|\ge s_{k_n}$ by $\CA_{\mathrm{bb}}$ in  \eqref{eq:upper-hh-bb-event}), and use that $u_{j_\ast}$ is at distance at most $2\sqrt{d}k_n^{1/d}$ from any vertex in $\CS(u)$, to obtain 
 \begin{align}
  \widetilde\Prob_\mathrm{bb}\big(u_{j_\ast}\leftrightarrow\CS(u)\big)
  & \ge
  1-\big(1-p\big(1\wedge \beta d^{-d/2}2^{-(2\sigma+d)}w_\mathrm{hh}^{1+\sigma}k_n^{-1}\big)^\alpha\big)^{s_{k_n}}                                            \nonumber \\
  & \ge1-\exp\big(p\big(s_{k_n}\wedge \beta^\alpha d^{-\alpha d/2}2^{-(2\sigma+d)\alpha}w_\mathrm{hh}^{(1+\sigma)\alpha}k_n^{-\alpha}s_{k_n}\big)\big).\label{eq:lower-alpha-finite-bb}
 \end{align}
 Using that $s_{k_n}=k_n w_\mathrm{hh}^{-(\tau-1)}/16$, $w_\mathrm{hh}=C_1^{-1/(\tau-1)}k_n^{\gamma_\mathrm{hh}}$, and $\gamma_\mathrm{hh}=(\alpha-1)/((\sigma+1)\alpha-(\tau-1))$, we have
 \[
  w_\mathrm{hh}^{(1+\sigma)\alpha}k_n^{-\alpha}s_{k_n}
  \ge
  2^{-4}w_\mathrm{hh}^{(1+\sigma)\alpha - (\tau-1)}k_n^{1-\alpha}=2^{-4}C_1^{1-(1+\sigma)\alpha/(\tau-1)}.
 \]
 Using this bound on the right-hand side in \eqref{eq:lower-alpha-finite-bb}, this yields combined with \eqref{eq:no-conn-bb} that there exists a constant $q>0$ such that if $m_w$ is sufficiently large,
 $\widetilde \Prob_{\mathrm{bb}}\big(u\leadsto\CC_\mathrm{bb}\big)
  \ge
  q$
 establishing Claim \ref{lemma:lower-conn-bb-2} when $\alpha<\infty$, by the reasoning about independence below \eqref{eq:u-leadsto-bb}.

 When $\alpha=\infty$, the choice of $C_1$ in \eqref{eq:c1-infty} ensures that for any vertex $u_\mathrm{bb}\in\CC_\mathrm{bb}\cap \CQ(u)$ that
 \begin{align}
  \mathrm{p}(u_{j_\ast}, u_\mathrm{bb}) \ge p\mathbbm{1}\Big\{\beta\frac{\kappa_\sigma(w_\mathrm{hh}/4, w_\mathrm{hh})}{2^dd^{d/2}k_n}\ge 1\Big\} = p,
  \nonumber
 \end{align}
 establishing Lemma \ref{lemma:lower-conn-bb-2} for $\alpha=\infty$ when combined with \eqref{eq:no-conn-bb} for $m_w$ sufficiently large.
\end{proof}
\end{claim}

We are now ready to prove Lemma \ref{lemma:construction}.
\begin{proof}[Proof of Lemma \ref{lemma:construction}]
 We first show \eqref{eq:construction-bb}. By a union bound, concentration inequalities for Poisson random variables (Lemma \ref{lemma:poisson-1} for $x\in\{1/2, 2\}$) and $F_W(\rd w_0)=(\tau-1)w_0^{-\tau}\rd w$ in Definition \ref{def:ksrg}, it follows that $\Prob(\neg\CA_\mathrm{reg}')=\exp(-\Theta(n m_w^{-(\tau-1)}))=o(1)$, and thus
 \begin{equation}
  \begin{aligned}
  \Prob\big(\neg  \big(\CA_\mathrm{reg}'\cap\CA_\mathrm{bb}\cap\{ w_0\in[m_w, 2m_w)\}\big)\big) 
                  & \le \Prob\big(\neg\CA_\mathrm{bb}\big)
  +
  \Prob\big(w_0\notin[m_w, 2m_w)\big) + o(1)\\
  &\le
  \Prob\big(\neg\CA_\mathrm{bb}\big)
  +
  1-(1-2^{-(\tau-1)})m_w^{-(\tau-1)}
  +o(1)
  .
  \end{aligned}\nonumber
 \end{equation}
 By the choice of $k=k_n$ in \eqref{eq:kn-paths-to-bb}, the first term tends to zero by Lemma~\ref{lemma:upper-hh-bb} as $n$ tends to infinity. 
 Since $1-2^{-(\tau-1)}> 1/2$ for $\tau>2$,  for $n$ sufficiently large (depending also on the constant $m_w$) it follows that
 \[
 \begin{aligned}
  \Prob\big(\neg\big(\CA_\mathrm{reg}'\cap\CA_\mathrm{bb}\cap\{ w_0\in[m_w, 2m_w)\}\big)\big)&\le 1-m_w^{-(\tau-1)}(1-2^{-(\tau-1)})+o(1)\\ &\le 1-m_w^{-(\tau-1)}/2,
  \end{aligned}
 \]
 and  \eqref{eq:construction-bb} follows. We proceed to \eqref{eq:constructing-1}.
 Conditionally on the realization of  $\CG_{n}[w_\mathrm{hh}, 2w_\mathrm{hh})$ satisfying $\CA_\mathrm{bb}$ (present in the conditioning in $\widetilde\Prob_\mathrm{bb}$ in \eqref{eq:lower-cond-prob-2}), we define the following set and random variable:
 \begin{align}
  \CU:=\{u\in\CV_n[m_w, 2m_w): u\overset\pi\longrightarrow\CC_\mathrm{bb}\},\quad
  X :=     \ind{0\overset{\pi}\longrightarrow \CC_\mathrm{bb}}\hspace{-10pt}\sum_{\substack{u\in\CV_n[m_w, 2m_w): \\ x_u\notin\Lambda(0,C_{\ref{lemma:lower-conn-bb-2}}k_n)}}\hspace{-10pt}\ind{u\overset{\pi}\longrightarrow \CC_\mathrm{bb}}\label{eq:lower-x},
 \end{align}
 with $\overset{\pi}\longrightarrow$ from \eqref{eq:u-pi-bb}.
 The measure $\widetilde\Prob_\mathrm{bb}$ is a conditional measure where $\CA_\mathrm{reg}'$ (defined in \eqref{eq:global-prime}) holds and so
 $|\CV_{n}[m_w, 2m_w)|\le 2\mu_\tau(\Lambda_n\times [m_w, 2m_w))$. Using $\mu_\tau$ in \eqref{eq:poisson-intensity},
 we obtain that deterministically under $\widetilde\Prob_\mathrm{bb}$:
 \[X^2\le 4\big(\mu_\tau(\Lambda_n\times[m_w, 2m_w)\big)^2 \le 4(m_w^{-(\tau-1)}n)^2.
 \]
When $0\in \CU$ holds, then $|\CC_{n}(0)[1, 2w_\mathrm{hh})|\ge |\CU|\ge X$, and so we apply Paley-Zygmund's inequality to $X$ under the measure $\widetilde \Prob_{\mathrm{bb}}$, which yields for $\rho':=\widetilde\E_\mathrm{bb}[X]/(2n)$ that
 \begin{equation}
  \begin{aligned}
  \widetilde\Prob_\mathrm{bb}\big(|\CC_{n}(0)[1, 2w_\mathrm{hh})|\ge \rho' n\big) & \ge\widetilde\Prob_\mathrm{bb}\big(|\CU|\ge \rho' n,\  0\in\CU\big)\ge
  \widetilde\Prob_\mathrm{bb}\big(X\ge \widetilde\E_\mathrm{bb}[X]/2 \big)
  \\&\ge
  (1/4)\frac{\widetilde\E_\mathrm{bb}[X]^2}{\widetilde\E_\mathrm{bb}[X^2]}
  \ge \frac{\widetilde\E_\mathrm{bb}[X]^2}{ 16 n^2m_w^{-2(\tau-1)}}.\label{eq:paley-zygmund}
  \end{aligned}
 \end{equation}
 We now bound the numerator on the right-hand side from below.
 Conditionally on $|\CV_n[m_w, 2m_w)|$, the vertices have a uniform location in $\Lambda_{n}$, so
 \[
  \widetilde\Prob_\mathrm{bb}\big(x_u\notin\Lambda(0,C_{\ref{lemma:lower-conn-bb-2}}k_n) \mid  u\in\CV_{n}[m_w, 2m_w), |\CV_{n}[m_w, 2m_w)|\big) = \big(n-C_{\ref{lemma:lower-conn-bb-2}}k_n\big)/n\ge 1/2,
 \]
 where the last inequality follows from assuming that $n$ is sufficiently large (recall $k_n=\Theta\big(\log^{1/\zeta_\mathrm{hh}}(n)\big)$ by \eqref{eq:kn-paths-to-bb}).
 The conditioning on $\CA_\mathrm{reg}'$ implies that $|\CV_n[m_w, 2m_w)|\ge \mu_\tau(\Lambda_n\times [m_w, 2m_w))/2$. Using linearity of expectation of $X$ in \eqref{eq:lower-x}, and the tower rule, (by first conditioning on $|\CV_n[m_w, 2m_w)|$) we obtain for $n$ sufficiently large
   \begin{align}
  \widetilde\E_\mathrm{bb}\big[X
    \big]
    & \ge (\mu_\tau(\Lambda_n\times [m_w, 2m_w))/2)\nonumber\\
  &\hspace{15pt}\cdot \widetilde \E_{\mathrm{bb}}\Big[\widetilde\Prob_\mathrm{bb}\big(x_u\notin\Lambda(0,C_{\ref{lemma:lower-conn-bb-2}}k_n) \mid u\in\CV_n[m_w, 2m_w), |\CV_n[m_w, 2m_w)|\big)\Big] \nonumber\\
    & \hspace{15pt}\cdot
  \widetilde\Prob_\mathrm{bb}\big(\{0\overset{\pi}\longrightarrow\CC_\mathrm{bb}\}\cap \{ u\overset{\pi}\longrightarrow\CC_\mathrm{bb}\}\mid u\in\CV_{n}[m_w, 2m_w), x_u\notin\Lambda(0,C_{\ref{lemma:lower-conn-bb-2}}k_n)\big)\nonumber\\
    & \ge
  n m_w^{-(\tau-1)}(1-2^{-(\tau-1)})(1/2)\cdot (1/2)\cdot q_{\ref{lemma:lower-conn-bb-2}}
  \ge nq_{\ref{lemma:lower-conn-bb-2}}m_w^{-(\tau-1)}2^{-3},\label{eq:paley-2}
  \end{align}
 where the second bound follows if $m_w$ is chosen as in Claim \ref{lemma:lower-conn-bb-2}, and from the definition of $\mu_\tau$ in \eqref{eq:poisson-intensity};  the last bound holds since $2^{-(\tau-1)}\le 1/2$ for $\tau>2$.
 Substituting the last bound  \eqref{eq:paley-2} into the numerator on the right-hand side of~\eqref{eq:paley-zygmund}, we  obtain that
 \[ \widetilde\Prob_\mathrm{bb}\big(|\CC_{n}(0)[1, 2w_\mathrm{hh})|\ge \rho'n)\ge 2^{-10}q_{\ref{lemma:lower-conn-bb-2}}^2 \]
 holds with
 $\rho'=\widetilde\E_\mathrm{bb}[X]/(2n)\ge q_{\ref{lemma:lower-conn-bb-2}}m_w^{-(\tau-1)}2^{-4}$,
 which yields the statement of Lemma \ref{lemma:construction} for $\rho_{\ref{lemma:construction}}=\min\{q_{\ref{lemma:lower-conn-bb-2}}m_w^{-(\tau-1)}2^{-4}, 2^{-10}q_{\ref{lemma:lower-conn-bb-2}}^2\}$.
\end{proof}
\begin{proof}[Proof of Proposition \ref{proposition:existence-large}]
We start with the first inequality in \eqref{eq:prop-subexp-lower-comp}.
 Using $\widetilde\Prob_\mathrm{bb}$ in \eqref{eq:lower-cond-prob-2}, we observe that the bound in Lemma \ref{lemma:construction} holds uniformly over all realizations of $\CG_{n}[w_\mathrm{hh}, 2w_\mathrm{hh})$ satisfying $\CA_{\mathrm{bb}}$. Hence, by first taking expectation over these possible realizations, we obtain that
 \[ \Prob(|\CC_{n}(0)[1, 2w_\mathrm{hh})|\ge\rho_{\ref{lemma:construction}}n \mid \CA_{\mathrm{reg}}', \CA_{\mathrm{bb}}, w_0\in[m_w, 2m_w]) \ge \rho_{\ref{lemma:construction}}\] also holds. The statement now follows with $\rho:=\rho_{\ref{lemma:construction}}m_w^{-(\tau-1)}/2$ by the law of total probability combining \eqref{eq:construction-bb} and \eqref{eq:constructing-1} of Lemma \ref{lemma:construction}.
 The second inequality  in \eqref{eq:prop-subexp-lower-comp} for $n=\infty$ follows from the same construction as the greedy path in the proof of Claim \ref{lemma:lower-conn-bb-2} below \eqref{eq:u-leadsto-bb} can be made infinitely long. We leave it to the reader to fill in the details. Uniqueness of the infinite component follows from the classical Burton--Keane argument~\cite{BurtonKeane}.
  \end{proof}

\section{Proofs using first-moment method}\label{app:integral-proofs}
We start with the proof of Claim \ref{claim:dense}.
\begin{proof}[Proof]\label{proof:claim-dense}
We first condition on the realization of the spatial coordinates of the Poisson point process $\CV_{n_k}$ while we leave the marks unrevealed, i.e., random. We emphasize this in notation by using $W_0, W_u$ for the random mark of the vertices involved. We use Mecke's formula on the spatial coordinate to obtain that
 \begin{equation*}
 \begin{aligned}
  \Prob^\sss{0}\big(\CA_\mathrm{dense}\big) & =
  \Prob^\sss{0}\big( \exists u\in\CV_{n_k}[1,\overline{w}), i\ge 1: |\CV_{N_k}\cap \CR_i(x_u)| > 2 \cdot\mu_\tau\big(\CR_i(x_u)\big)\big) \\
  &\le \E^\sss{0}\Big[\hspace{-3pt}\sum_{u \in \CV_{n_k}}\sum_{i\ge 1}\Prob\big( |\CV_{N_k}\cap \CR_i(x_u)| > 2 \cdot\mu_\tau\big(\CR_i(x_u)\big) \mid (x_u)_{u\in \CV_{n_k}}\big)\Big] 
  \\
  &=\sum_{i\ge 1}\int_{x_u\in\Lambda_{n_k}}\hspace{-12pt}\Prob\big( |\CV_{N_k}\cap \CR_i(x_u)| > 2 \cdot\mu_\tau\big(\CR_i(x_u)\big)  \mid (0, W_0), (x_u, W_u)\!\in\!\CV_{n_k}\big)\rd x_u \\
  &\hspace{15pt}
  +\sum_{i\ge 1}\Prob\big( |\CV_{N_k}\cap \CR_i(0)| > 2 \cdot\mu_\tau\big(\CR_i(x_u)\big)  \mid (0, W_0)\!\in\!\CV_{n_k}\big).
 \end{aligned}
 \end{equation*}
 The last row contains the term coming from the Palm measure in the first two lines.
 We omit the conditioning $(0, W_0)\in\CV_{n_k}$ at the expense of replacing $2\mu_\tau\big(\CR_i(x_u)\big)$ at the right-hand side between brackets by $2\mu_\tau\big(\CR_i(x_u)\big)-1$. By translation invariance of the PPP we obtain that 
 \[
  \Prob^\sss{0}\big(\CA_\mathrm{dense}\big)
  \le (n_k+1)\sum_{i\ge 1}\Prob^\sss{0}\big( |\CV_{N_k}\cap \CR_i(0)| > 2 \cdot\mu_\tau\big(\CR_i(x_u)\big) - 1\big).
 \]
  For each $i\ge1$,
  $|\CV_{N_k}\cap \CR_i(0)|$ is distributed as $\Poi(\lambda_i)$ with
 \[\lambda_i:=\mu_{\tau}(\CR_i(0))=(2^{i} t_k)^d-(2^{i-1} t_k)^d=(2^{i-1}t_k)^d (2^d-1)\ge2^{d(i-2)}n_kk^{-d}\ge2^{d(i-2)}n_k^{\frac{\delta}{d+\delta}},\]
  having used $t_k=n_k^{1/d}/(2k)$, and $1/k \ge n_k^{-1/(d+\delta)}$.
 By concentration inequalities for Poisson random variables (see Lemma~\ref{lemma:poisson-1} applied with $x=2$, using $1+2\log(2)-2> 1/4$), we obtain for $n_k$ sufficiently large
 \[
 \Prob^\sss{0}\big(|\CV_{N_k}\cap \CR_i(0)| > 2 \cdot\mu_\tau\big(\CR_i(0)\big)-1\big) 
 \le\exp\big(-2^{d(i-2)-2}n_k^{\frac{\delta}{d+\delta}}\big),
 \]
 and the statement follows for every $c > 0$, i.e., 
 \begin{equation*}
  \Prob^\sss{0}\big(\CA_\mathrm{dense}\big) \le (n_k+1)\sum_{i\ge 1} \exp\big(-2^{d(i-2)-2}n_k^{\frac{\delta}{d+\delta}}\big)
  =o( n_k^{-c}).\qedhere
 \end{equation*}
\end{proof}
Now we  prove Claim \ref{claim:edge-long} used for the upper bound of subexponential decay.
\begin{proof}[Proof of Claim \ref{claim:edge-long}]\phantomsection\label{proof:edge-long}
 For compact sets $\CK_1, \CK_2\subseteq\R^d$, let $\|\CK_1-\CK_2\|:=\min\{\|x-y\|, x\in\CK_1, y\in\CK_2\}$.
 We define
 \begin{equation}
  t_{n,N}:=\|\partial \Lambda_{N}- \partial \Lambda_{n}\| = (N^{1/d}-n^{1/d})/2.    \nonumber
 \end{equation}
 The definition of $\CA_\mathrm{long\textnormal{-}edge}$ in \eqref{eq:event-long-edge} implies that
 \begin{equation}
  \CA_\mathrm{long\textnormal{-}edge}(0,n, N, \overline{w})\subseteq
  \big\{\exists u\in\CV_{n}[1,\overline{w}), v\in\CV: \|x_u-x_v\|\ge t_{n,N}, u\leftrightarrow v\big\},\nonumber
 \end{equation}
 so that after conditioning on $\CV_{n}[1,\overline{w})$ it follows by a union bound that
 \begin{align}
  \Prob^\sss{0}\big(\CA_\mathrm{long\textnormal{-}edge}(0,n, N, \overline{w})\big) & \le
  \E^\sss{0}\Bigg[\sum_{u\in\CV_{n}[1,\overline{w})}\hspace{-10pt}\Prob\big(\exists v \in \CV: \|x_u-x_v\|\!\ge\! t_{n,N},  u\!\leftrightarrow \!v \mid u\!\in\!\CV_{n}[1,\overline{w})\big)\Bigg] \nonumber\\
                                                            & =:\E^\sss{0}\Bigg[\sum_{u\in\CV_{n}[1,\overline{w})}q(u)\Bigg].\label{eq:edge-long-proof-union}
 \end{align}
 \emph{Assume $\alpha<\infty$.} Since the diameter of $\Lambda_{n}$ is $\sqrt{d}n^{1/d}$, the lower bound on $N$ in the statement of Claim \ref{claim:edge-long} implies that $\|x_u-x_v\|\le t_{n,N}$ for all $x_u, x_v\in\Lambda_{n}$. Hence, $v\notin\CV_n$ whenever $\|x_u-x_v\|> t_{n,N}$. This implies by Markov's bound, using the connection probability in \eqref{eq:connection-prob-gen}, and that
 $w_u\le \overline{w}$ and the intensity $\mu_\tau$ and the translation invariance of the intensity of $\CV$ in \eqref{eq:poisson-intensity}, that for all $u\in\CV_{n}$,  
 \begin{align}
  q(u)
  & \le
  \E^\sss{0}\Bigg[\sum_{\substack{v\in\CV: \|x_u-x_v\|> t_{n,N}}}p\Big(1\wedge \beta\frac{\kappa_\sigma(\overline{w}, w_v)}{\|x_u-x_v\|^d}\Big)^\alpha \Bigg] \nonumber\\
  & =
  p(\tau-1)
  \int_{w_v=1}^\infty
  \int_{x_v: \|x_u-x_v\|\ge t_{n,N}}
  \Big(1\wedge \beta\frac{\kappa_\sigma(\overline{w}, w_v)}{\|x_u-x_v\|^d}\Big)^\alpha w_v^{-\tau}\rd w_v\rd x_v                            \label{eq:qu-min-for-alpha}                 \\
  & =
  p(\tau-1)
  \int_{w_v=1}^{\infty}
  w_v^{-\tau}
  \int_{x_v: \|x_v\|\ge t_{n,N}, \|x_v\|^d\le \beta\kappa_\sigma(\overline{w},w_v)}
  \rd w_v\rd x_v                                                                                                                                                           \nonumber   \\
  & \hspace{15pt}+
  p(\tau-1)
  \int_{w_v=1}^{\infty}(\beta\kappa_\sigma(\overline{w}, w_v))^\alpha w_v^{-\tau}
  \int_{x_v: \|x_v\|\ge t_{n,N}, \|x_v\|^d\ge \beta\kappa_\sigma(\overline{w},w_v)}
  \hspace{-23pt}\|x_v\|^{-\alpha d} \rd w_v\rd x_v                                     \nonumber\\
  & =: T_1 + T_2.\label{eq:upper-sub-long-pr1}
 \end{align}
 We analyze separately $T_1$ and $T_2$. Analyzing $T_1$, the integration with respect to $x_v$ gives the Lebesgue measure of the set $\{x_v: t_{n,N}^d \le \|x_v^d\|\le \beta \kappa_\sigma(\overline w, w_v)\}$, which is nonzero only if this set is nonempty, and then can be bounded from above by $c_d\beta \kappa_\sigma(\overline w, w_v)$ for some constant depending only on $d$. So we obtain 
 \begin{align*}
  T_1 & \le c_dp\beta(\tau-1)\int_{w_v=1}^\infty\Ind{t_{n,N}^d \le \beta\kappa_\sigma(\overline{w},w_v)}\kappa_\sigma(\overline{w},w_v)w_v^{-\tau}\rd w_v \\
      & =
  c_dp\beta(\tau-1)\Big(\int_{w_v=1}^{\overline{w}}\hspace{-10pt}\Ind{t_{n,N}^d \le \beta\overline{w}w_v^\sigma}\overline{w} w_v^{\sigma-\tau}\rd w_v
  + \int_{w_v=\overline{w}}^\infty\hspace{-10pt}\Ind{t_{n,N}^d \le \beta\overline{w}^\sigma w_v}\overline{w}^\sigma w_v^{1-\tau}\rd w_v\Big),
 \end{align*}
 where in the last step we used the definition of $\kappa_\sigma$ in \eqref{eq:kernels} and cut the integration into two based on the minimum of $\overline w$ and $w_v$.
 Since we assume $t_{n,N}^d\ge \beta\overline{w}^{1+\sigma}$ in the statement of the lemma, and  $w_v^\sigma\le \overline{w}^\sigma$, the indicator in the first integral is $0$. Moving the indicator in the second integral into the integration boundary yields 
 \begin{equation}
  T_1 \le
  c_dp\beta\tfrac{\tau-1}{\tau-2}\overline{w}^\sigma(t_{n,N}^d/(\beta\overline{w}^\sigma))^{-(\tau-2)} = c_dp\beta^{\tau-1}\tfrac{\tau-1}{\tau-2}\overline{w}^{\sigma(\tau-1)}t_{n,N}^{-d(\tau-2)}.\label{eq:upper-sub-long-pr2}
 \end{equation}
 We turn to $T_2$ in \eqref{eq:upper-sub-long-pr1}. For some $d$-dependent constant $c_d'$, using again $\kappa_\sigma$ in \eqref{eq:kernels},
 \begin{align}
  T_2 & \le pc_d'(\tau-1)
  \int_{w_v=1}^{\infty}(\beta\kappa_\sigma(\overline{w}, w_v))^\alpha w_v^{-\tau}
  \max\{t_{n,N}^d, \beta\kappa_\sigma(\overline{w},w_v)\}^{1-\alpha}\rd w_v \nonumber\\
      & =
  pc_d'(\tau-1)
  \int_{w_v=1}^{\overline{w}}(\beta \overline{w}w_v^\sigma)^\alpha w_v^{-\tau}
  \max\{t_{n,N}^d, \beta\overline{w}w_v^\sigma\}^{1-\alpha}\rd w_v         \label{eq:t2-cut-1}            \\
      & \hspace{15pt}+
  pc_d'(\tau-1)
  \int_{w_v=\overline{w}}^\infty(\beta\overline{w}^\sigma w_v)^\alpha w_v^{-\tau}
  \max\{t_{n,N}^d, \beta\overline{w}^\sigma w_v\}^{1-\alpha}\rd w_v. \label{eq:t2-cut-2}
 \end{align}
 For the first integral, we bound $w_v^\sigma\le \overline{w}^\sigma$, and observe that by the assumption $t_{n,N}^d\ge \beta\overline{w}^{1+\sigma}$ in the statement of the lemma, the maximum is always attained at $t_{n,N}^d$. Hence, 
 \begin{align}
  \eqref{eq:t2-cut-1}  \le
  pc_d'(\beta\overline{w}^{\sigma+1})^\alpha t_{n,N}^{-(\alpha-1)d}.\label{eq:t2-cut-1-bound}
 \end{align}
 We split the integral in \eqref{eq:t2-cut-2} according to where the maximum is attained, i.e.,
 \begin{align}
  \eqref{eq:t2-cut-2} & \le
  pc_d'\beta\overline{w}^\sigma(\tau-1)
  \int_{w_v=\max\{\overline{w}, t_{n,N}^d/(\overline{w}^\sigma \beta)\}}^{\infty} w_v^{-(\tau-1)}
  \rd w_v\nonumber
  \\
      & \hspace{15pt}+
  \ind{t_{n,N}^d/(\overline{w}^\sigma\beta) \ge \overline{w}}
  pc_d'(\tau-1)
  t_{n,N}^{-(\alpha-1)d}
  \int_{w_v=\overline{w}}^{t_{n,N}^d/(\overline{w}^\sigma\beta)}(\beta\overline{w}^\sigma w_v)^\alpha w_v^{-\tau}\rd w_v\nonumber \\
      & =pc_d'\beta\overline{w}^\sigma(\tau-1)
  \int_{w_v=t_{n,N}^d/(\overline{w}^\sigma \beta)}^{\infty} \hspace{-8pt}w_v^{-(\tau-1)}
  \rd w_v  \nonumber\\
  &\hspace{15pt}+
  pc_d'(\tau-1)
  t_{n,N}^{-(\alpha-1)d}(\beta\overline{w}^\sigma )^\alpha
  \int_{w_v=\overline{w}}^{t_{n,N}^d/(\overline{w}^\sigma\beta)} \hspace{-8pt}w_v^{\alpha-\tau}
  \rd w_v        \nonumber                                                                                                        \\
      & =
  \frac{pc_d'(\beta\overline{w}^\sigma)^{\tau-1}(\tau\!-\!1)}{\tau-2}t_{n,N}^{-d(\tau-2)}
  \!+\!
  pc_d'(\tau-1)
  t_{n,N}^{-(\alpha-1)d}(\beta\overline{w}^\sigma )^\alpha
  \hspace{-5pt}\int_{w_v=\overline{w}}^{t_{n,N}^d/(\overline{w}^\sigma\beta)} \hspace{-9pt}w_v^{\alpha-\tau}
  \rd w_v,\label{eq:upper-sub-long-pr4}
 \end{align}
 where  the second step follows from the assumption that $t_{n,N} \ge \beta\overline{w}^{1+\sigma}$ in the statement of the lemma.
 For the remaining term containing the integral on the right-hand side of \eqref{eq:upper-sub-long-pr4}, say $T_{22}$, we have three cases, i.e., for some $C>0$,
 \begin{align}
  T_{22}
  & \le
  \begin{dcases}
  Ct_{n,N}^{-(\alpha-1)d}(\beta\overline{w}^\sigma )^\alpha \big(t_{n,N}^d/(\overline{w}^\sigma\beta)\big)^{\alpha-(\tau-1)}
                                                                                                                             & \text{if }\alpha>\tau-1, \\
  Ct_{n,N}^{-(\alpha-1)d}(\beta\overline{w}^\sigma )^\alpha\log\big(t_{n,N}^d/(\overline{w}^\sigma\beta)\big) & \text{if }\alpha=\tau-1, \\
  C\overline{w}^{\alpha-(\tau-1)}t_{n,N}^{-(\alpha-1)d}(\beta\overline{w}^\sigma )^\alpha                            & \text{if }\alpha<\tau-1. \\
  \end{dcases}\nonumber
 \end{align}
 Elementary rewriting of the first and third case yields
 \begin{align}
  T_{22}
  & \le
  \begin{dcases}
  Ct_{n,N}^{-d(\tau-2)}(\beta\overline{w}^\sigma)^{\tau-1}                                                           & \text{if }\alpha>\tau-1, \\
  Ct_{n,N}^{-d(\alpha-1)}(\beta\overline{w}^\sigma )^\alpha\log\big(t_{n,N}^d/(\overline{w}^\sigma\beta)\big) & \text{if }\alpha=\tau-1, \\
  Ct_{n,N}^{-d(\alpha-1)}\beta^\alpha \overline{w}^{\alpha(\sigma+1)-(\tau-1)}                                       & \text{if }\alpha<\tau-1.
  \end{dcases}\nonumber
 \end{align}
 Combining this bound with \eqref{eq:upper-sub-long-pr4}, then \eqref{eq:t2-cut-1-bound} and  \eqref{eq:upper-sub-long-pr2}, gives in \eqref{eq:upper-sub-long-pr1} and \eqref{eq:edge-long-proof-union} that for some $C'>0, b>0$,
 \begin{align*}
  \Prob^\sss{0}\big(\CA_\mathrm{long\textnormal{-}edge}(0, n, N, \overline{w})\big) & \le
  C'\overline{w}^{b}\widetilde{t}_k^{-d \min\{\alpha-1,\tau-2\}}(1+\ind{\alpha=\tau-1}\log(t_{n,N}))\E[|\CV_{n}[1,\overline{w})]|] \\
                                                            & \le
  C\overline{w}^{b}N^{-\min\{\alpha-1,\tau-2\}}(1+\ind{\alpha=\tau-1})\log(N))n,
 \end{align*}
 where the last bound follows by the assumed bound in \eqref{eq:sub-upper-tk-tilde-bound}, since $t_{n,N}=(N^{1/d}-n^{1/d})/2$, and the intensity of $\CV_{n}[1,\overline{w})$ in \eqref{eq:poisson-intensity}. This finishes the proof of \eqref{eq:edge-long-lemma} when $\alpha<\infty$.

 It remains to show the bound for the case $\alpha=\infty$. In this case, the same calculations hold, with only $T_1$ in \eqref{eq:upper-sub-long-pr1} present, since the connection probability is $0$ when the minimum in \eqref{eq:qu-min-for-alpha} is not attained at $1$.

 Lastly, we verify~\eqref{eq:edge-long-lemma2}. Assume first that $k$ is at least a sufficiently large constant so that~\eqref{eq:sub-upper-tk-tilde-bound} holds with $N=N_k$, $n=n_k$ and $\overline w=\overline w_{N_k}$ defined in~\eqref{eq:nk-Nk}. Then~\eqref{eq:edge-long-lemma2} follows by substituting these sequences into~\eqref{eq:edge-long-lemma2}. For smaller values of $k$, one might adjust $\delta$ to be sufficiently small.
\end{proof}
We proceed with the proofs of two lemmas for the lower bound.

\begin{proof}[Proof of Lemma \ref{lemma:lower-vertices-above}]\label{proof:lower-vertices-above}
The expectation $\widetilde\E$ in~\eqref{eq:cond-prob-lower-1} is conditional on $\CV\cap(\CR_\mathrm{in}\cup\CR_\mathrm{out})$ where $\CR_\mathrm{in}$, $\CR_\mathrm{out}$ defined in~\eqref{eq:lower-giant-boxes} are hyperrectangles below $\CM_\gamma$ for all $\eta$ and any $n$ sufficiently large, since the upper mark thresholds in $\CR_\mathrm{in}$ and $\CR_\mathrm{out}$ are polylogarithmic in $k$. Hence, $\CV_{>\CM_\gamma}$ is independent of the conditioning in $\widetilde \E$, so 
\[
\widetilde\E\big[|\CV_{>\CM_\gamma}|]=\widetilde\E\big[|\CV^\sss{\mathrm{in}}_{>\CM_\gamma}\cup\CV^\sss{\mathrm{out}}_{>\CM_\gamma}|]=\E\big[|\CV^\sss{\mathrm{in}}_{>\CM_\gamma}\cup\CV^\sss{\mathrm{out}}_{>\CM_\gamma}|] = \E\big[\big|\CV_{>\CM_\gamma}^\sss{\mathrm{in}}\big|\big]\!+\! \E\big[\big|\CV_{>\CM_\gamma}^\sss{\mathrm{out}}\big|\big].
\]
 We introduce some notation: for two functions $g(k), h(k)$, we write $g\!\lesssim\! h$ if $g\!=\!O(h)$.
 Since $f_\gamma(x)$ is symmetric around the boundary of $\partial \CB_{\mathrm{in}}$ (see its definition in \eqref{eq:suppressed-curve}), it is easy to see that $\E\big[\big|\CV_{>\CM_\gamma}^\sss{\mathrm{in}}\big|\big]\!\le\! \E\big[\big|\CV_{>\CM_\gamma}^\sss{\mathrm{out}}\big|\big]$. It is sufficient to show that
 \begin{equation}
  \E\big[\big|\CV_{>\CM_\gamma}^\sss{\mathrm{out}}\big|\big]
  \lesssim r_k^{d(1-\gamma(\tau-1))} + (1+\ind{1-\gamma(\tau-1)=1-1/d}\log(r_k))r_k^{d-1}. \label{eq:points-above-sufficient-crit}
 \end{equation}
Using the intensity measure of $\CV$ in \eqref{eq:poisson-intensity}, switching to polar coordinates in the first $d$ directions, and integrating with respect to the mark-coordinate, we obtain using the exact form of  $f_\gamma$ in \eqref{eq:suppressed-curve} 
 \begin{align}
  \E\big[\big|\CV_{>\CM_\gamma}^\sss{\mathrm{out}}\big|\big]
  & \lesssim
  \int_{z=0}^\infty (z+r_k)^{d-1} \int_{ f_\gamma(z)}^{\infty} (\tau-1)w^{-\tau}\rd w \rd z
  \nonumber \\
  & \lesssim
  \int_{z=0}^{C_\beta}(z+r_k)^{d-1} \rd z
  +
  \int_{z=C_\beta}^{r_k}(z+r_k)^{d-1} z^{-d\gamma(\tau-1)} \rd z                      \nonumber\\
  & \hspace{15pt}+
  \int_{z=r_k}^{\infty}(z+r_k)^{d-1}(z^{d}r_k^{-d(1-\gamma)})^{-(\tau-1)} \rd z
  =: I_1 + I_2 + I_3.\label{eq:points-above-three-ints}
 \end{align}
 The integration length of $I_1$ is a constant, so $
  I_1 \lesssim r_k^{d-1}.$
 For $I_2$ we apply the binomial theorem, i.e.,
 \begin{align}
  I_2                                                       \lesssim\sum_{j=0}^{d-1}r_k^j\int_{C_\beta}^{r_k} z^{(1-\gamma(\tau-1))d-1-j}\rd z. \label{eq:points-above-pr1}
 \end{align}
 Analyzing the summands separately, we obtain for $j\le d-1$
 \begin{equation}
  r_k^j\int_{C_\beta}^{r_k}z^{(1-\gamma(\tau-1))d-1-j}\rd z
  \lesssim
  \begin{dcases}
  r_k^{d(1-\gamma(\tau-1))}, & \text{if }d(1-\gamma(\tau-1)) > j, \\
  \log(r_k)r_k^j,            & \text{if }d(1-\gamma(\tau-1)) = j, \\
  r_k^j,                     & \text{if }d(1-\gamma(\tau-1)) < j.
  \end{dcases}\nonumber
 \end{equation}
 Using these bounds, which are non-decreasing in $j\in[d-1]$, in \eqref{eq:points-above-pr1}, we obtain
 \begin{equation}
  I_2
  \lesssim
  (1+\ind{1-\gamma(\tau-1)=1-1/d}\log(r_k))r_k^{d-1} + r_k^{d(1-\gamma(\tau-1))}.\label{eq:points-above-pr2}
 \end{equation}
 It remains to bound $I_3$ in \eqref{eq:points-above-three-ints}. Using that $\tau>2$ by assumption, and $z+r_k\le 2z$,
 \begin{align}
  I_3 \lesssim
  r_k^{d(1-\gamma)(\tau-1)}\int_{z=r_k}^\infty z^{-d(\tau-2)-1}\rd z
  \lesssim
  r_k^{d((1-\gamma)(\tau-1) - (\tau-2))}=r_k^{d(1-\gamma(\tau-1))}.\nonumber
 \end{align}
 We use that $r_k=\Theta(k^{1/d})$ by definition in~\eqref{eq:rk}.
 Together with the bound on $I_1$ below the definitions of $I_1, I_2$, and $I_3$ in \eqref{eq:points-above-pr1}, and on $I_2$ in \eqref{eq:points-above-pr2}, this proves \eqref{eq:points-above-sufficient-crit} and also finishes the proof of \eqref{eq:expected-above-curve}.
\end{proof}

\begin{proof}[Proof of Lemma {\ref{lemma:lower-edges-below}}]\label{proof:lower-edges-below}
 We start with the proof of~\eqref{eq:expected-edges-below}. We split the expected number of edges depending on the locations of the endpoints of the vertices.
 \begin{align}
  \widetilde\E \!\Big[\! \big|\CE\big(\CV_{\le\CM_\gamma}^\sss{\mathrm{in}}, \CV_{\le\CM_\gamma}^\sss{\mathrm{out}}\big)\big|\! \Big]
          &\!=\!
          \widetilde\E \!\Big[\! \big|\CE\big(\CV_{\le\CM_\gamma\setminus\CR_\mathrm{in}}^\sss{\mathrm{in}}, \CV_{\le\CM_\gamma\setminus\CR_\mathrm{out}}^\sss{\mathrm{out}}\big)\big|\! \Big]\!
     +\!
                  \widetilde\E \Big[ \!\big|\CE\big(\CV_{\CR_\mathrm{in}}, \CV_{\le\CM_\gamma\setminus\CR_\mathrm{out}}^\sss{\mathrm{out}}\big)\big| \!\Big]\nonumber \\
          &\hspace{15pt}+
                  \widetilde\E \Big[ \big|\CE\big(\CV_{\CR_\mathrm{in}}, \CV_{\CR_\mathrm{out}}\big)\big| \Big]+
                  \widetilde\E \Big[ \big|\CE\big(\CV_{\le\CM_\gamma\setminus\CR_\mathrm{in}}^\sss{\mathrm{in}}, \CV_{\CR_\mathrm{out}}\big)\big| \Big]
  \label{eq:lower-edges-e3}
 \end{align}
 We analyze the first term on the right-hand side and at the end we sketch how the bounds could be adapted for the other three terms.
 Since $\CA_\mathrm{regular}(\eta)$ is measurable with respect to $\CV_{\CR_\mathrm{in}}\cup\CV_{\CR_\mathrm{out}}$, it can be left out of the conditioning. Further, points of $\CV$ in disjoint sets are independently present, hence
  \begin{align}
  \E\Big[\big|\CE\big(\CV_{\le\CM_\gamma\!\setminus\!\CR_\mathrm{in}}^\sss{\mathrm{in}}, \CV_{\le\CM_\gamma\!\setminus\!\CR_\mathrm{out}}^\sss{\mathrm{out}}\big)\big| \Big|\CV_{\CR_\mathrm{in}}\!\cup\!\CV_{\CR_\mathrm{out}}, \CA_\mathrm{regular}(\eta)\Big] &=
  \E\big[ \big|\CE\!\big(\CV_{\le\CM_\gamma\!\setminus\!\CR_\mathrm{in}}^\sss{\mathrm{in}}, \CV_{\le\CM_\gamma\!\setminus\!\CR_\mathrm{out}}^\sss{\mathrm{out}}\big)\!\big| \big]\nonumber \\
  &\le
  \E\big[\big|\CE\!\big(\CV_{\le\CM_\gamma}^\sss{\mathrm{in}}, \CV_{\le\CM_\gamma}^\sss{\mathrm{out}}\big)\!\big| \big].\label{eq:lower-edges-in-out-I}
  \end{align}
We use the notation $g\lesssim h$ if $g=O(h)$.
 We integrate over the locations and marks of the vertices in $\CV_{\le\CM_\gamma}^\sss{\mathrm{in}}\cup\CV_{\le\CM_\gamma}^\sss{\mathrm{out}}$ by writing $z_u=\|x_u-\partial \CB_{\mathrm{in}}\|$, and bounding from above the connectivity function $\mathrm{p}$ in \eqref{eq:connection-prob-gen} to obtain 
 \begin{equation}
 \begin{aligned}\label{eq:lower-edges-in-out-4}
  &\E\big[\big|\CE\big(\CV_{\le\CM_\gamma}^\sss{\mathrm{in}}, \CV_{\le\CM_\gamma}^\sss{\mathrm{out}}\big)\big| \big] \\
  &\hspace{-3pt}\lesssim\int_{x_u:\|x_u\|\le r_k}\hspace{-2pt}\int_{x_v: \|x_v\|\ge r_k} \hspace{-2pt}\int_{w_u=1}^{f_\gamma(z_u)}\hspace{-4pt}\int_{w_v=1}^{f_\gamma(z_v)}\hspace{-3pt} \frac{(\kappa_\sigma(w_u, w_v)\!)^\alpha}{\|x_u-x_v\|^{\alpha d}}\!(w_uw_v)^{-\tau}\!\rd w_v \rd w_u \rd x_v\rd x_u.
 \end{aligned}\end{equation}
 We analyze the double integral over the marks.
 Define
 \begin{equation}
  \begin{aligned}
  g_1(z_1, z_2) & := \int_{w_1=1}^{f_\gamma(z_1)}\int_{w_2=1}^{f_\gamma(z_2)}\big(\kappa_\sigma(w_1, w_2)\big)^\alpha(w_1w_2)^{-\tau}\rd w_1 \rd w_2.
  \end{aligned}\nonumber
 \end{equation}
 Using the definition and symmetry of $\kappa_\sigma$ in \eqref{eq:kernels}, we reparametrize by $w\le \tilde w$. We also use that $f_\gamma$ is increasing to obtain
 \begin{align*}
  g_1(z_1, z_2) & \lesssim
  \int_{w=1}^{f_\gamma(z_1\wedge z_2)}\int_{\widetilde w=w}^{f_\gamma(z_1\vee z_2)}\big(\kappa_\sigma(w, \widetilde w)\big)^\alpha(w\widetilde w)^{-\tau}\rd \widetilde w\rd w\\
  &=
  \int_{w=1}^{f_\gamma(z_1\wedge z_2)}w^{\sigma\alpha-\tau}\int_{\widetilde w=w}^{f_\gamma(z_1\vee z_2)}\widetilde w^{\alpha-\tau}\rd \widetilde w\rd w.
 \end{align*}
 When integrating, we have nine cases depending on whether the exponents are larger, equal or smaller than $-1$ each.
The definition $f_\gamma$ in \eqref{eq:suppressed-curve} undergoes a change at $z=r_k$.   This yields for $1\le z_2\le \min\{z_1, r_k\}$  (so $f_\gamma(z_2)=1\vee(z_2/C_\beta)^{\gamma d}$) that $g_1(z_1, z_2)\lesssim g_2(z_1\vee z_2, z_1\wedge z_2)$,
 with
 \begin{equation}
  g_2(z_\vee, z_\wedge) \!:=\!
  \begin{dcases}
  f_\gamma(z_\vee)^{\alpha\!-\!(\tau\!-\!1)}z_\wedge^{\gamma d(\sigma\alpha\!-\!(\tau\!-\!1))},            & \text{if }\alpha\!>\!\tau\!-\!1, \sigma\alpha\!>\!\tau\!-\!1,          \\
  f_\gamma(z_\vee)^{\alpha\!-\!(\tau\!-\!1)}\log(z_\wedge),                                        & \text{if }\alpha\!>\!\tau\!-\!1, \sigma\alpha\!=\!\tau\!-\!1,          \\
  f_\gamma(z_\vee)^{\alpha\!-\!(\tau\!-\!1)},                                                 & \text{if } \alpha\!>\!\tau\!-\!1, \sigma\alpha\!<\!\tau\!-\!1,         \\
  (1\!+\!\log(f_\gamma(z_\vee)/f_\gamma(z_\wedge)))z_\wedge^{\gamma d((\sigma\!+\!1)\alpha\!-\!2(\tau\!-\!1))}, & \text{if }\alpha\!=\!\tau\!-\!1, \sigma\alpha\!>\!\tau\!-\!1,          \\
  \log(f_\gamma(z_\vee))\log(z_\wedge),                                                    & \text{if }\alpha\!=\!\tau\!-\!1, \sigma\alpha\!=\!\tau\!-\!1,          \\
  \log(f_\gamma(z_\vee)),                                                             & \text{if }\alpha\!=\!\tau\!-\!1, \sigma\alpha\!<\!\tau\!-\!1,          \\
  z_\wedge^{\gamma d((\sigma\!+\!1)\alpha\!-\!2(\tau\!-\!1))},                                      & \text{if }\alpha\!<\!\tau\!-\!1, (\sigma\!+\!1)\alpha \!>\! 2(\tau\!-\!1), \\
  \log(z_\wedge),                                                                       & \text{if }\alpha\!<\!\tau\!-\!1, (\sigma\!+\!1)\alpha \!=\! 2(\tau\!-\!1), \\
  1,                                                                               & \text{if }\alpha\!<\!\tau\!-\!1, (\sigma\!+\!1)\alpha \!<\! 2(\tau\!-\!1).
  \end{dcases}
  \nonumber
 \end{equation}
 We define 
 \begin{equation}
     g(z_1, z_2):=g_2(z_1\vee z_2, z_1\wedge z_2).\label{eq:g-def}
 \end{equation}
 The function $g$ is non-decreasing and positive since all exponents are positive and $f_\gamma(\cdot)$, $z_\vee$, and $z_\wedge$ are all at least one.
 Returning to \eqref{eq:lower-edges-in-out-4}, we use $g(z_1, z_2)$ to bound the inner two integrals from above.
We make a case distinction on whether the vertex $u$ (inside) or $v$ (outside) is closer to $\partial \CB_{\mathrm{in}}$.   Then we obtain (since $f_\gamma(x)=1$ when $z(x)\le C_\beta$ by \eqref{eq:suppressed-curve}),
 \begin{align}
  \E\big[\big|\CE\big(\CV_{\le\CM_\gamma}^\sss{\mathrm{in}}, \CV_{\le\CM_\gamma}^\sss{\mathrm{out}}\big)\big| \big] & \lesssim
  \int_{x_u:\|x_u\|\le r_k-C_\beta}\int_{x_v: \|x_v\|> 2r_k}\|x_u-x_v\|^{-\alpha d} g(z_u,z_v) \rd x_v \rd x_u\nonumber                \\
    & \hspace{15pt}+
  \int_{x_u:\|x_u\|\le r_k-C_\beta}\int_{\substack{x_v: z_v\le z_u}}\|x_u-x_v\|^{-\alpha d} g(z_u, z_v) \rd x_v \rd x_u \nonumber \\
    & \hspace{15pt}+
  \int_{x_v:C_\beta\le\|x_v\|-r_k\le r_k}\int_{\substack{x_u: z_u\le z_v}}\|x_u-x_v\|^{-\alpha d} g(z_u,z_v) \rd x_v \rd x_u
  \nonumber                                                                                                                                    \\
    & =: I_1 + I_{2a} + I_{2b}=:I_1 + I_2.
  \label{eq:edges-below-g-int}
 \end{align}
 To evaluate the integrals we change variables, starting with $I_1$.  For $I_1$ we use $z_u\le r_k\le z_v$, and $g$ is  increasing in both its arguments, and that $\|x_u-x_v\|\ge | \|x_v\|-r_k| \ge z_v:=t$ and use polar coordinates in the second row below. Then $\|x_v\|=t+r_k$. 
 Thus, since there are $\Theta((t+r_k)^{d-1})$ points outside at distance $t$ from $\partial \CB_{\mathrm{in}}$,
 \begin{align}
  I_1 & \lesssim
  \int_{x_u:\|x_u\|\le r_k-C_\beta}\int_{x_v: \|x_v\|> 2r_k}|\|x_v\|-r_k|^{-\alpha d} g(z_v, r_k) \rd x_v \rd x_u \nonumber\\
      & \lesssim
  r_k^d \int_{t> r_k}(t+r_k)^{d-1}t^{-\alpha d} g(t, r_k) \rd t
  \lesssim r_k^d \int_{t\ge r_k}t^{-d(\alpha-1) -1} g(t, r_k) \rd t.\label{eq:i1-bound}
 \end{align}
 Before substituting the definition of $g$ into the bound, we recall that $f_\gamma(t)=C_\beta^{-\gamma d} t^dr_k^{-d(1-\gamma)}$ for $t>r_k$ by \eqref{eq:suppressed-curve}. The following elementary integration inequalities will be helpful soon:
 \begin{align}
  r_k^d \int_{t\ge r_k}\!t^{-d(\alpha-1) -1} f_\gamma(t)^{\alpha-(\tau-1)} \rd t
                                                      &\lesssim
  r_k^{d-d(1-\gamma)(\alpha-(\tau-1))} \int_{t\ge r_k}\!\!\!t^{-d(\tau-2) -1} \rd t\nonumber\\
                                                      &\lesssim r_k^{d(2-\alpha+\gamma(\alpha-(\tau-1)))}, \label{eq:integral-estimate-1} \\
  r_k^d \int_{t\ge r_k}t^{-d(\alpha-1) -1}\rd t  &\lesssim r_k^{d(2-\alpha)}.\label{eq:integral-estimate-2} \end{align}
Substituting the definition of $g$ in \eqref{eq:g-def} into \eqref{eq:i1-bound}, since $r_k<t$, we must set $z_2=r_k$ in $g$ in \eqref{eq:g-def}. We obtain then by elementary integration on \eqref{eq:i1-bound} and the bounds in \eqref{eq:integral-estimate-1}-\eqref{eq:integral-estimate-2} that:
 \begin{equation}
  I_1\lesssim
  \begin{dcases}
  r_k^{d(2-\alpha+\gamma((\sigma+1)\alpha-2(\tau-1)))},  & \text{if }\alpha>\tau-1, \sigma\alpha>\tau-1,          \\
  r_k^{d(2-\alpha+\gamma(\alpha-(\tau-1)))}\log(r_k),    & \text{if }\alpha>\tau-1, \sigma\alpha=\tau-1,          \\
  r_k^{d(2-\alpha+\gamma(\alpha-(\tau-1)))},             & \text{if } \alpha>\tau-1, \sigma\alpha<\tau-1,         \\
  r_k^{d(2-\alpha+\gamma((\sigma+1)\alpha-2(\tau-1))},   & \text{if }\alpha=\tau-1, \sigma\alpha>\tau-1,          \\
  r_k^{d(2-\alpha)}\log^2(r_k),                          & \text{if }\alpha=\tau-1, \sigma\alpha=\tau-1,          \\
  r_k^{d(2-\alpha)}\log(r_k),                            & \text{if }\alpha=\tau-1, \sigma\alpha<\tau-1,          \\
  r_k^{d(2-\alpha+\gamma ((\sigma+1)\alpha-2(\tau-1)))}, & \text{if }\alpha<\tau-1, (\sigma+1)\alpha > 2(\tau-1), \\
  r_k^{d(2-\alpha)}\log(r_k),                            & \text{if }\alpha<\tau-1, (\sigma+1)\alpha = 2(\tau-1), \\
  r_k^{d(2-\alpha)},                                     & \text{if }\alpha<\tau-1, (\sigma+1)\alpha < 2(\tau-1).
  \end{dcases}\label{eq:i1-final-bound}
 \end{equation}
 We turn to $I_{2a}$ in \eqref{eq:edges-below-g-int}, handling the case when the outside vertex $v$ is closer to the boundary $\partial \CB_{\mathrm{in}}$ than $u$, implying $z_v\le z_u$.  We reparametrize this integral based on the distance $z_u$ from $\partial \CB_{\mathrm{in}}$ of  the inside vertex $u$. Indeed, when $z_u\in[C_\beta, r_k]$ then $\|x_u\|=r_k-z_u$. Since $v$ is closer, we must also have that $z_v=\|x_v\|-r_k\in[C_\beta, z_u]$, and hence $t:=\|x_u-x_v\|\in[z_u+C_\beta, 2r_k]$.  Hence,
 \begin{align}
  I_{2a} & \lesssim
  \int_{z_u=C_\beta}^{r_k}
  \int_{x_u: r_k-\|x_u\|=z_u}
  \int_{t=z_u+C_\beta}^{2r_k}
  \int_{z_v=C_\beta}^{z_u}
  \int_{\substack{x_v: z(x_v)=z_v, \\\|x_u-x_v\|=t
  }}t^{-\alpha d}g(z_u, z_v)
  \rd x_v
  \rd z_v
  \rd t
  \rd x_u
  \rd z_u.\nonumber
 \end{align}
 The integrand does not depend on $x_v$ anymore, hence the most inside integral, over $x_v$, can be bounded from above by maximizing the Lebesgue measure of where $x_v$ may fall: $x_v$ has distance $t$ from $x_u$ and distance $z_v$ from the boundary. Some geometry shows that $x_v$ is then on the intersection of two spheres with radii $t$ and $r_k+z_v$, respectively, with Lebesgue measure then at most $\Theta(t^{d-2})$. We can also integrate over all the potential locations $x_u$, giving a factor $\Theta((r_k-z_u)^{d-1})$, so we obtain
 \begin{align}
  I_{2a} & \lesssim
  \int_{z_u=C_\beta}^{r_k}
  (r_k-z_u)^{d-1}
  \int_{t=z_u+C_\beta}^{2r_k}
  t^{d-2-\alpha d}
  \int_{z_v=C_\beta}^{z_u}
  g(z_u, z_v)
  \rd z_v
  \rd t
  \rd z_u          \nonumber \\
         & \lesssim
  \int_{z_u=C_\beta}^{r_k}
  (r_k-z_u)^{d-1}
  z_u^{-d(\alpha-1)-1}
  \int_{z_v=C_\beta}^{z_u}
  g(z_u, z_v)
  \rd z_v
  \rd z_u,\label{eq:i2a-bound}
 \end{align}
 where we integrated over $t$ to obtain the second row.
 Treating $I_{2b}$ in \eqref{eq:edges-below-g-int} is very similar, but now we reparametrize the integral based on the distance $z_v$ of $v$ from the boundary and the distance $t=\|x_u-x_v\|$.  We obtain
 \begin{align}
  I_{2b} & \lesssim
  \int_{z_v=C_\beta}^{r_k}
  \int_{x_v: \|x_v\|-r_k=z_v}
  \int_{t=z_v+C_\beta}^{3r_k}
  \int_{z_u=C_\beta}^{z_v}
  \int_{\substack{x_u: z_u=z(x_u),           \\\|x_u-x_v\|=t}} t^{-\alpha d}g(z_u, z_v)
  \rd x_u
  \rd z_u
  \rd t
  \rd x_v
  \rd z_v \nonumber                         \\
         & \lesssim\int_{z_v=C_\beta}^{r_k}
  (r_k+z_v)^{d-1}
  z_v^{-d(\alpha-1)-1}
  \int_{z_u=C_\beta}^{z_v}
  g(z_u, z_v)
  \rd z_u
  \rd z_v      .\nonumber
 \end{align}
 This bound dominates the bound on $I_{2a}$ in \eqref{eq:i2a-bound}. Applying the binomial theorem on $(r_k+z_v)^{d-1}$, we obtain
 \begin{align}
  I_2=I_{2a} + I_{2b} & \lesssim
  \sum_{j=0}^{d-1}r_k^j
  \int_{z_v=C_\beta}^{r_k}
  z_v^{d(2-\alpha)-2-j}
  \int_{z_u=C_\beta}^{z_v}
  g(z_u, z_v)
  \rd z_u
  \rd z_v.\label{eq:i2-bound}
 \end{align}
 We evaluate the inner integral using the definition of $g$ in \eqref{eq:g-def}, and since $z_u\le z_v$ we set $z_1=z_v, z_2=z_u$ in \eqref{eq:g-def}, and $f_\gamma(z)= (z/C_\beta)^{\gamma d}$, and obtain the nine cases:
 \begin{align}
  \int_{z_u=C_\beta}^{z_v}
  g(z_u, z_v)
  \rd z_u
  \lesssim
  \begin{dcases}
  z_v^{\gamma d((\sigma+1)\alpha-2(\tau-1))+1}, & \text{if }\alpha>\tau-1, \sigma\alpha>\tau-1,          \\
  z_v^{\gamma d(\alpha-(\tau-1))+1}\log(z_v),   & \text{if }\alpha>\tau-1, \sigma\alpha=\tau-1,          \\
  z_v^{\gamma d(\alpha-(\tau-1))+1},            & \text{if } \alpha>\tau-1, \sigma\alpha<\tau-1,         \\
  z_v^{\gamma d((\sigma+1)\alpha-2(\tau-1))+1}, & \text{if }\alpha=\tau-1, \sigma\alpha>\tau-1,          \\
  z_v\log^2(z_v),                               & \text{if }\alpha=\tau-1, \sigma\alpha=\tau-1,          \\
  z_v\log(z_v),                                 & \text{if }\alpha=\tau-1, \sigma\alpha<\tau-1,          \\
  z_v^{\gamma d((\sigma+1)\alpha-2(\tau-1))+1}, & \text{if }\alpha<\tau-1, (\sigma+1)\alpha > 2(\tau-1), \\
  z_v\log(z_v),                                 & \text{if }\alpha<\tau-1, (\sigma+1)\alpha = 2(\tau-1), \\
  z_v,                                          & \text{if }\alpha<\tau-1, (\sigma+1)\alpha < 2(\tau-1).
  \end{dcases}\label{eq:i2-bound-g}
 \end{align} 
 We substitute \eqref{eq:i2-bound-g} into \eqref{eq:i2-bound}.
 Using $\xi_\star$ and $\mathfrak{m}_\mathrm{long}$ from~\eqref{eq:xi-long}, the nine cases can be summarized as obtaining the integrand of $z_v^{d(2-\alpha + \gamma\xi_\star)-1-j}\log^{\mathfrak{m}_\mathrm{long}-1}(z_v)$. Following similar reasoning as from \eqref{eq:points-above-pr1} to \eqref{eq:points-above-pr2}, we obtain
 \begin{align}
  I_2 \lesssim
  \begin{dcases}
  r_k^{d(2-\alpha + \gamma\xi_\star)}\log^{\mathfrak{m}_\mathrm{long}-1}(r_k), & \text{if }d(2-\alpha + \gamma\xi_\star) > d-1, \\
  r_k^{d(2-\alpha + \gamma\xi_\star)}\log^{\mathfrak{m}_\mathrm{long}}(r_k),   & \text{if }d(2-\alpha + \gamma\xi_\star) = d-1, \\
  r_k^{d-1},                                                      & \text{if }d(2-\alpha + \gamma\xi_\star) < d-1,
  \end{dcases}\nonumber
 \end{align}
 where the second bound follows from similar reasoning as in \eqref{eq:points-above-pr1} leading to \eqref{eq:points-above-pr2}.
 The presence of a $(d-1)$ term and the additional $\log$-factors ensure that the bound on $I_2$ dominates the bound on $I_1$ in \eqref{eq:i1-final-bound}. Recalling that $I_1+I_2$ dominates the expected number of edges below the $\gamma$-suppressed profile from \eqref{eq:edges-below-g-int}, this yields by \eqref{eq:lower-edges-in-out-I} that, 
 \begin{align*}
  \E\Big[\big|\CE\big(\CV_{\le\CM_\gamma\setminus\CR_\mathrm{in}}^\sss{\mathrm{in}}, \CV_{\le\CM_\gamma\setminus\CR_\mathrm{out}}^\sss{\mathrm{out}}\big)\big| \,&\Big|\, \CV_{\CR_\mathrm{in}}\cup\CV_{\CR_\mathrm{out}}, \CA_\mathrm{regular}(\eta)\Big] \\
          & \lesssim
  \begin{dcases}
  k^{2-\alpha + \gamma\xi_\star}(\log k)^{\mathfrak{m}_\mathrm{long}-1}, & \text{if }2-\alpha + \gamma\xi_\star > \tfrac{d-1}{d}, \\
  k^{2-\alpha + \gamma\xi_\star}(\log k)^{\mathfrak{m}_\mathrm{long}},   & \text{if }2-\alpha + \gamma\xi_\star = \tfrac{d-1}{d}, \\
  k^{(d-1)/d},                                                      & \text{if }2-\alpha + \gamma\xi_\star < \tfrac{d-1}{d},
  \end{dcases}
 \end{align*}
 where we also used that $r_k=\Theta(k^{1/d})$ by~\eqref{eq:rk}.
 To obtain bounds on the other three expectations in \eqref{eq:lower-edges-e3}, one can
 replace the integrals over the marks in \eqref{eq:lower-edges-in-out-4}  by summing over the mark intervals $I_j$ defined in \eqref{eq:lower-weight-intervals}: the upper bounds on the number of vertices using the definitions of $\CA_\mathrm{regular}^\sss{(k,\mathrm{in})}(\eta)$ and $\CA_\mathrm{regular}^\sss{(k,\mathrm{out})}(\eta)$ in \eqref{eq:lower-events-global} ensure that the total number of points in each interval only differs from its expectation by a constant factor. Then one can use an upper bound on the mark of each vertex in $I_j^{\sss{\mathrm{loc}}}$ in \eqref{eq:lower-weight-intervals}, given by $I_j^{\sss{\mathrm{loc}}}=[2^{j-1}, 2^j)$, and thus also the mark is at most a factor two larger than the mark of a typical vertex in $I_j^{\sss{\mathrm{loc}}}$. Lastly, the distance between vertices in $\CR_\mathrm{in}$ and outside $\CB_{\mathrm{in}}$ (but within distance $r_k$ of $\CB_{\mathrm{in}}$) can be bounded from below by $r_k/2$ by Lemma \ref{lemma:construction}, and analogously we can bound the distance between vertices in $\CR_\mathrm{out}$ and vertices inside $\CB_{\mathrm{in}}$.
 We leave it to the reader to fill in the details.

It remains to show~\eqref{eq:claim-long-edge} in Lemma~\ref{lemma:lower-edges-below}.
  We show that whenever $u, v$ are below $\CM_\gamma$ and on different sides of $\partial \CB_{\mathrm{in}}$, then $\beta \kappa_\sigma(w_u,w_v)/$ $\|x_u-x_v\|^d \le 1/2$. By definition of $\mathrm{p}$ in \eqref{eq:connection-prob-gen}, this directly implies \eqref{eq:claim-long-edge}.
To see this bound, for $\sigma\ge 0$ the connection probability $\mathrm{p}$ is increasing in the marks.
  Therefore, without loss of generality we will assume that $u\in \CB_{\mathrm{in}}$ and $v\notin \CB_{\mathrm{in}}$ fall exactly on $\CM_\gamma$, and that $k$ is large enough that $r_k\ge C_\beta$. Since $f_\gamma$ in \eqref{eq:suppressed-curve} changes its definition at $r_k$ outside $\partial \CB_{\mathrm{in}}$, we distinguish two cases.\smallskip 

  \noindent\emph{Case 1. Assume $|\|x_v\|-r_k|\le r_k$.}
  Since the suppressed mark profile equals one (the minimal mark) for all points within distance $C_\beta$ by~\eqref{eq:suppressed-curve}, there are no vertex pairs within distance $C_\beta$ of $\CB_{\mathrm{in}}$.
  Thus, for any $(u, v)$, there exist $t\ge2C_\beta$, $\nu\in(0,1)$ such that $\|x_u-\partial \CB_{\mathrm{in}}\|=r_k-\|x_u\|=(1-\nu)t$, and $\|x_v-\partial \CB_{\mathrm{in}}\|=\|x_v\|-r_k=\nu t$. Since the line-segment $(x_u, x_v)$ must pass through $\partial \CB_{\mathrm{in}}$ it follows that $\|x_u-x_v\|\ge t$.
  Hence, $w_v=C_\beta^{-\gamma d} (\nu t)^{\gamma d}, w_u=C_\beta^{-\gamma d} ((1-\nu) t)^{\gamma d}$ by assuming $u,v$ being on $\CM_\gamma$ and $f_\gamma$ in \eqref{eq:suppressed-curve}. Using $\kappa_\sigma$ in \eqref{eq:kernels}, and $\nu\in(0,1)$,
  \begin{align}
  \beta\frac{\kappa_\sigma(w_u, w_v)}{\|x_u-x_v\|^d}
    & \le
  \beta C_\beta^{-\gamma d(\sigma+1)}\frac{\max\{(1-\nu)t, \nu t\}^{\gamma d}\min\{(1-\nu)t, \nu t\}^{\sigma \gamma d}}{t^d} \\
  &\le
  \beta C_\beta^{-\gamma d(\sigma+1)}t^{d(\gamma(\sigma + 1)-1)}.\nonumber
  \end{align}
  The right-hand side is non-increasing in $t$ \emph{whenever} $\gamma\le\frac{1}{\sigma+1}$. Since $t\ge 2C_\beta\ge C_\beta$, and $C_\beta=(2\beta)^{1/d}$,
  \begin{equation}
  \beta\frac{\kappa_\sigma(w_u, w_v)}{\|x_u-x_v\|^d}
  \le
  \beta C_\beta^{-d\gamma(\sigma+1)}C_\beta^{d(\gamma(\sigma+1)-1)}
  =
  \beta  C_\beta^{-d}
  = 1/2.    \nonumber
  \end{equation}
  \noindent
  \noindent\emph{Case 2. Assume  $|\|x_v\|-r_k|> r_k$.} In this case $\|x_v-\partial\CB_{\mathrm{in}}\|\ge\|x_u-\partial \CB_{\mathrm{in}}\|$, and $\|x_v-x_u\|\ge \|x_v-\partial\CB_{\mathrm{in}}\|=|\|x_v\|-r_k|$.
  Since $f_\gamma(z)$ is increasing, $w_v=f_\gamma(|\|x_v\|-r_k|)\ge f_\gamma(r_k)\ge f_\gamma(|\|x_u\|-r_k|)=w_u$. We obtain by definition of $\kappa_\sigma$ in  \eqref{eq:kernels} and \eqref{eq:suppressed-curve}
  \begin{align}
  \beta\frac{\kappa_\sigma(w_u, w_v)}{\|x_u-x_v\|^d}
  &\le
  \beta \frac{f_\gamma(|\|x_v\|-r_k|)f_\gamma(|\|x_u\|-r_k|)^{\sigma}}{|\|x_v\|-r_k|^d}\\ 
    & \le
  \beta C_\beta^{-\gamma d(\sigma+1)}  |\|x_v\|-r_k|^{d}r_k^{-d(1-\gamma)}r_k^{\sigma \gamma d}|\|x_v\|-r_k|^{-d}\nonumber \\
    & =\beta C_\beta^{-\gamma d(\sigma+1)}r_k^{-d(1-\gamma(\sigma+1))} \le 1/2,\nonumber
  \end{align}
  where to obtain the second inequality we used that $f_\gamma(|\|x_u\|-r_k|)^{\sigma}\le r_k^{\sigma \gamma d}$, and to obtain the second row the power of $|\|x_v\|-r_k|$ canceled. The bound $1/2$ follows because $r_k\ge C_\beta$, $\gamma \le 1/(\sigma+1)$ and $C_\beta=(2\beta)^{1/d}$.
  \end{proof}

To finish the auxiliary lower bound proofs, we provide a sketch of the proof of Lemma~\ref{lemma:lower-vertex-boundary2}.

\begin{proof}[Proof sketch of Lemma~\ref{lemma:lower-vertex-boundary2}]\phantomsection\label{proof:vertex-boundary}  Let $\CZ:=\{\zeta_\mathrm{hh}, \zeta_\mathrm{hl}, \zeta_\mathrm{ll}, \zeta_\mathrm{short}\}$, with $\zeta_\mathrm{short}=(d-1)/d$.
We first sketch the lower bounds, starting with the case $\max(\CZ)=\zeta_\mathrm{short}$. With probability tending to one as $k\to\infty$, there are $\Theta(k^{(d-1)/d})$ vertices within constant distance from the boundary of $\Lambda_k$. For each such vertex $u$, there is with constant probability a vertex $v$ outside $\Lambda_k$ within constant distance, to which $u$ connects by an edge with constant probability by~\eqref{eq:connection-prob-gen}.  When $\max(\CZ)\in\{\zeta_\mathrm{ll}, \zeta_\mathrm{hl}, \zeta_\mathrm{hh}\}$, the lower bounds of~\eqref{eq:lemma-zeta-star} and~\eqref{eq:lemma-zeta-long} can be shown by the following domination of a binomial random variable. If $\max(\CZ)=\zeta_\mathrm{ll}=2-\alpha$, and $\tau<\infty,$  there are $\Theta(k)$ vertices inside $\Lambda_{k/2}$ with mark in $[2,3)$, and $\Theta(k)$ vertices in $\Lambda_{2k}/\Lambda_k$ with mark in $[1,2)$. The number of vertices in $\Lambda_{k/2}$ with a downward edge to a vertex in $\Lambda_{2k} \setminus\Lambda_k$ dominates a binomial random variable with parameters $\Theta(k)$ and $p=\Theta(k^{1-\alpha})$, which has expectation $\Theta(k^{\zeta_\mathrm{ll}})$. When $\tau=\infty$, all marks are identical, and each edge is downward by definition. When $\max(\CZ)=\zeta_\mathrm{hl}$ (respectively $\zeta_\mathrm{hh}$), a similar reasoning works, but we consider vertices of mark at least $k^{\gamma_\mathrm{hl}}$ (resp.\ $k^{\gamma_\mathrm{hh}}$) inside $\Lambda_{k/2}$, and of mark in $[1,2)$ (resp. $[k^{\gamma_\mathrm{hh}}/2, k^{\gamma_\mathrm{hh}})$) in $\Lambda_{2k}\setminus\Lambda_k$.

For the upper bounds, one needs to slightly modify the optimally suppressed mark-profile in~\eqref{eq:suppressed-curve}, so that its minimum is around the boundary of the box $\Lambda_k$, rather than around the boundary of the ball $\CB_{\mathrm{in}}$. The suppression-profile outside $\Lambda_k^\complement$ is not required as we only consider downward edges. The number of vertices with a downward edge in $\Lambda_k$ is bounded from above by the total number of vertices with mark above the suppressed mark profile $|\CV^{\sss{\mathrm{in}}}_{>\CM_\star}|$, plus the number of downward edges to $\Lambda_k^\complement$ emanating from vertices below the profile, denoted by $|\CE(\CV^{\sss{\mathrm{in}}}_{\le \CM_\star} \searrow \Lambda_k^\complement)|$.  Lemma~\ref{lemma:lower-vertices-above} can be used to bound $|\CV^{\sss{\mathrm{in}}}_{>M_\gamma}|$.  The expectation of $|\CE(\CV^{\sss{\mathrm{in}}}_{\le \CM_\star} \searrow \Lambda_k^\complement)|$  can be bounded similarly to the proof of Lemma~\ref{lemma:lower-edges-below}. This yields the upper bound of~\eqref{eq:lemma-zeta-star}. The upper bound of~\eqref{eq:lemma-zeta-long}, under the assumption $\max(\zeta_\mathrm{hh}, \zeta_\mathrm{hl}, \zeta_\mathrm{ll})\ge 0$, follows from the same computations, restricting the integrals to vertices at distance at least $\Omega(k^{1/d})$ from the boundary of $\Lambda_k$. When $\max(\zeta_\mathrm{hh}, \zeta_\mathrm{hl}, \zeta_\mathrm{ll})< 0$ these integrals are of order $O(k^{-\eps})$ for some $\eps>0$. We leave the details to the reader.  

Assumption~\ref{assumption:main} ensures that $\alpha>1$. Combining this with the definitions of $\zeta_\mathrm{ll}, \zeta_\mathrm{hl}, \zeta_\mathrm{hh}$ in~\eqref{eq:zeta-ll},~\eqref{eq:gamma-lh}, and~\eqref{eq:gamma-hh}, implies that $\max\big(\zeta_\mathrm{hh}, \zeta_\mathrm{hl}, \zeta_\mathrm{ll}, (d-1)/d\big)<1$.  
\end{proof}

\section{Auxiliary proof}\label{app:aux}

It remains to prove Lemma~\ref{lemma:zeta-opt-other}.
\begin{proof}[Proof of Lemma~\ref{lemma:zeta-opt-other}]\label{proof:zeta-opt-other}
We start with three helping statements to prove the bounds for $\alpha<\infty$. First, we prove the implication
\begin{equation}\label{eq:imp1}\tag{$\Rightarrow_1$}
\max\left\{\begin{aligned}&1-\gamma_\mathrm{long}(\tau-1), \\&1-\gamma_\star(\tau-1),\\& 2-\alpha+\gamma_\star\xi_\star\end{aligned}\right\}\ge 0\quad \Longrightarrow\quad 
    \left(\begin{aligned}1-\gamma_\mathrm{long}&(\tau-1)\\&=1-\gamma_\star(\tau-1) \\&=2-\alpha+\gamma_\star\xi_\star\end{aligned}\right).
\end{equation}
Since $\gamma_\star=\min(\gamma_\mathrm{long}, 1/(\sigma+1))$  by definition in~\eqref{eq:gamma-opt}, the second term in the maximum is at least the first term. Since $\gamma_\mathrm{long}$ is the smallest $\gamma$ such that $1-\gamma(\tau-1)\le 2-\alpha+\gamma\xi_\star$ by~\eqref{eq:gamma-opt-long}, the second term in the maximum is at least the third term. Thus the left-hand side is equivalent to $1-\gamma_\star(\tau-1)\ge 0$.  By the same definitions, the right-hand side only fails to be true if $\gamma_\mathrm{long}\neq\gamma_\star$, which is when $\gamma_\mathrm{long}>\gamma_\star=1/(\sigma+1)$. 
Thus, \eqref{eq:imp1} is equivalent to showing
\begin{equation}\label{eq:opt-3-pr}
    1-\gamma_\star(\tau-1) \ge 0 \qquad \Longrightarrow\qquad \gamma_\mathrm{long}\le 1/(\sigma+1).
\end{equation}
If $\gamma_\mathrm{long}=\gamma_\star$, the implication holds since $\gamma_\star\le1/(\sigma+1)$ by definition.
If $\gamma_\mathrm{long}>\gamma_\star$, then $\gamma_\star=1/(\sigma+1)$ and the left-hand side is equivalent to $\tau-1\le \sigma+1$.
We substitute $\gamma_\mathrm{long}$ from~\eqref{eq:gamma-opt-long} with $\xi_\star$ from~\eqref{eq:xi-long} to see that $\gamma_\mathrm{long}\le1/(\sigma+1)$ is equivalent to \[\max\big(\tau-1, \alpha, (\sigma+1)\alpha-(\tau-1)\big)\ge (\alpha-1)(\sigma+1).\] The third term in the maximum is at least $(\alpha-1)(\sigma+1)$ if $\tau-1\le \sigma+1$, proving~\eqref{eq:opt-3-pr} which is equivalent to~\eqref{eq:imp1}.
\smallskip 

We now state and prove another, second implication. Recall $\zeta_\mathrm{hh}=1-\gamma_\mathrm{hh}(\tau-1)$, $\zeta_\mathrm{ll}=2-\alpha$, and $\zeta_\mathrm{hl}=1-(1-1/\alpha)(\tau-1)$ from~\eqref{eq:zeta-hh}, \eqref{eq:zeta-ll}, and \eqref{eq:gamma-lh}, and  $\xi_\mathrm{hh}=(\sigma+1)\alpha-2(\tau-1)$, $\xi_\mathrm{ll}=0$, and $\xi_\mathrm{hl}=\alpha-(\tau-1)$ from~\eqref{eq:xi-long}. We prove now
\begin{equation}\label{eq:imp2}\tag{$\Rightarrow_2$}
    \zeta_\mathrm{hh}<0\le \max(\zeta_\mathrm{ll}, \zeta_\mathrm{hl}) \qquad 
    \Longrightarrow \qquad \xi_\mathrm{hh}< \max(\xi_\mathrm{ll}, \xi_\mathrm{hl}).
\end{equation}
By~\eqref{eq:zeta-hh}, $\zeta_\mathrm{hh}<0$ if and only if $\sigma+1<\tau-1$, while $\max(\zeta_\mathrm{ll},\zeta_\mathrm{hl})\ge 0$ implies by elementary operations that $\alpha\le \max(2, (\tau-1)/(\tau-2))$. On the one hand, if $\alpha\le 2$ and $\sigma+1<\tau-1$, ~\eqref{eq:imp2} follows immediately since  \[\xi_\mathrm{hh}<(\tau-1)2-2(\tau-1)=0=\xi_\mathrm{ll}.\] 
If on the other hand $\alpha\le (\tau-1)/(\tau-2)$ and $\sigma+1<\tau-1$, then 
\[\xi_\mathrm{hh}=\xi_\mathrm{hl}+\sigma\alpha-(\tau-1)\le \xi_\mathrm{hl}+(\tau-1)\Big(\frac{\sigma}{\tau-2} - 1\Big)<\xi_\mathrm{hl},\]
which finishes the proof of \eqref{eq:imp2}. Next, we prove a third implication
\begin{equation}\label{eq:imp3}\tag{$\Rightarrow_3$}
\max(\zeta_\mathrm{ll}, \zeta_\mathrm{hl}, \zeta_\mathrm{hh})\ge 0 \quad \Longrightarrow\quad 
    1-\gamma_\mathrm{long}(\tau-1)= \max(\zeta_\mathrm{ll}, \zeta_\mathrm{hl}, \zeta_\mathrm{hh}) .
\end{equation}
Recalling the definitions of $\zeta_\mathrm{ll}=2-\alpha$, $\zeta_\mathrm{hl}=1-(\alpha-1)(\tau-1)/\alpha$, and $\zeta_\mathrm{hh}=1-\gamma_\mathrm{hh}(\tau-1)$ from~\eqref{eq:zeta-ll}, \eqref{eq:gamma-lh}, and \eqref{eq:zeta-hh}, as well as $\gamma_\mathrm{long}$ from~\eqref{eq:gamma-opt-long}, the right-hand side is equivalent to showing 
\begin{equation}\label{eq:opt-3}
    \frac{\alpha-1}{\max(\tau-1, \alpha, (\sigma+1)\alpha-(\tau-1))}=\min\Big((\alpha-1)/(\tau-1), (\alpha-1)/\alpha, \gamma_\mathrm{hh}\Big).
\end{equation}
If $\tau\le 2+\sigma$, the definition of $\gamma_\mathrm{hh}=(1-1/\alpha)/(\sigma+1-(\tau-1)/\alpha)$ in~\eqref{eq:gamma-hh} proves the equality in this case. If $\tau>2+\sigma$, then $\zeta_\mathrm{hh}<0$ by definition in~\eqref{eq:zeta-hh}, and we need to show $1-\gamma_\mathrm{long}(\tau-1)=\max(\zeta_\mathrm{ll}, \zeta_\mathrm{hl})$. 

If $\max(\zeta_\mathrm{ll}, \zeta_\mathrm{hl})\ge 0>  \zeta_{\mathrm{hh}}$, \eqref{eq:imp2} implies (by subtracting $\tau-1$ from each of the $\xi$ values) that the maximum in the denominator on the left-hand side in~\eqref{eq:opt-3} is never attained at the third term in~\eqref{eq:opt-3}. Hence, $1-\gamma_\mathrm{long}(\tau-1)=\max(\zeta_\mathrm{ll}, \zeta_\mathrm{hl})$ follows since formally clearly
\begin{equation}\nonumber 
    \frac{\alpha-1}{\max(\tau-1, \alpha)}=\min\Big((\alpha-1)/(\tau-1), (\alpha-1)/\alpha\Big)
\end{equation}
holds. Using similar rearrangements and the definitions, the reader may verify that 
\begin{equation}\label{eq:imp4}\tag{$\Rightarrow_4$}
    \max(\zeta_\mathrm{ll}, \zeta_\mathrm{hl},\zeta_\mathrm{hh})<0
    \quad \Longrightarrow\quad 
    1-\gamma_\mathrm{long}(\tau-1)<0.
\end{equation} 

We prove now~\eqref{eq:zeta-opt-1}. By definition of $\gamma_\mathrm{long}$ and $
\gamma_\star$ in~\eqref{eq:gamma-opt-long} and~\eqref{eq:gamma-opt}, $2-\alpha+\gamma_\star\xi_
\star\le1-\gamma_\star(\tau-1)$, leaving to verify the equality in~\eqref{eq:zeta-opt-1}.
First assume $\max(\zeta_\mathrm{ll}, \zeta_\mathrm{hl}, \zeta_\mathrm{hh})<0$, so that  $1-\gamma_\mathrm{long}(\tau-1)<0$ by~\eqref{eq:imp4}. Then~\eqref{eq:imp1} implies that also $1-\gamma_\star(\tau-1)<0$, otherwise they would be equal and all nonnegative. 
So,~\eqref{eq:zeta-opt-1} holds in this case since  $(d-1)/d\ge 0$ for all $d\ge 1$.  Finally we assume $\max(\zeta_\mathrm{ll}, \zeta_\mathrm{hl}, \zeta_\mathrm{hh})\ge 0$. Then, \eqref{eq:imp3} and~\eqref{eq:imp1} imply that then $\max(\zeta_\mathrm{ll}, \zeta_\mathrm{hl}, \zeta_\mathrm{hh})=1-\gamma_\star(\tau-1)$.
Thus,~\eqref{eq:zeta-opt-1} follows again.
\smallskip 

To prove \eqref{eq:zeta-opt-m2}, we introduce some general notation in which we  count the multiplicity of the maximum. Let for a list (with potentially repeated elements) $\CY=\{y_1,\ldots, y_\ell\}\subseteq\R$,
\begin{equation}
    \mathfrak{m}(\CY):=\mathfrak{m}(y_1,\ldots, y_\ell):= \sum_{i\in[\ell]}\ind{y_i=\max(\CY)}.
\end{equation}
Define $\mathrm{sign}:\R\mapsto\{-, 0, +\}$ as $\mathrm{sign}(x)\!=\!-$ for $x\!<\!0$, $\mathrm{sign}(x)\!=\!+$ for $x\!>\!0$, and $\mathrm{sign}(0)\!=\!0$. Consider now two lists of numbers $\{y_1, \dots, y_\ell\}$ and $\{z_1, \dots, z_\ell\}$ of length $\ell$.
We claim that
\begin{equation}\label{eq:imp-mult}
    \big(\mathrm{sign}(y_i-y_j)=\mathrm{sign}(z_i-z_j)\quad \forall i\neq j\big)\quad \Longrightarrow \quad \mathfrak{m}(y_1,\ldots, y_\ell)=\mathfrak{m}(z_1,\ldots, z_\ell).
\end{equation}
Indeed, the index of a maximal element in both lists can be identified in a list if all sign differences are equal to $0$ or $+$, and the multiplicity can be computed by counting how often the sign difference with the other elements equals $0$.
We will use this observation to prove
\begin{equation}\label{eq:opt-5}\tag{$\Rightarrow_5$}
    \max(\zeta_\mathrm{ll}, \zeta_\mathrm{hl}, \zeta_\mathrm{hh})\ge0\qquad \Longrightarrow\qquad \mathfrak{m}(\zeta_\mathrm{ll}, \zeta_\mathrm{hl}, \zeta_\mathrm{hh}) = \mathfrak{m}(\xi_\mathrm{ll}, \xi_\mathrm{hl}, \xi_\mathrm{hh})=\mathfrak{m}_\mathrm{long}.
\end{equation}
We claim that $\mathfrak{m}(\xi_\mathrm{ll},\xi_\mathrm{hl})=\mathfrak{m}(\zeta_\mathrm{ll}, \zeta_\mathrm{hl})$: using the definitions of  $\xi_\mathrm{ll}=0$ and $\xi_\mathrm{hl}=\alpha-(\tau-1)$ from~\eqref{eq:xi-long}, and $\zeta_\mathrm{ll}=2-\alpha$ and $\zeta_\mathrm{hl}=1-(1-1/\alpha)(\tau-1)$ in~\eqref{eq:zeta-ll} and~\eqref{eq:gamma-lh}, it is elementary to compute that $\mathrm{sign}(\xi_\mathrm{ll}-\xi_\mathrm{hl})=\mathrm{sign}(\zeta_\mathrm{ll}-\zeta_\mathrm{hl})$.

Assume now that $\zeta_\mathrm{hh}<0\le \max(\zeta_\mathrm{ll}, \zeta_\mathrm{hl})$, so $\mathfrak{m}(\zeta_\mathrm{ll}, \zeta_\mathrm{hl}, \zeta_\mathrm{hh})=\mathfrak{m}(\zeta_\mathrm{ll}, \zeta_\mathrm{hl})=\mathfrak{m}(\xi_\mathrm{ll}, \xi_\mathrm{hl})$. By~\eqref{eq:imp2}, also $\xi_\mathrm{hh}<\max(\xi_\mathrm{ll}, \xi_\mathrm{hl})$, so $\mathfrak{m}(\xi_\mathrm{ll}, \xi_\mathrm{hl})=\mathfrak{m}(\xi_\mathrm{ll}, \xi_\mathrm{hl}, \xi_\mathrm{hh})$. Thus,~\eqref{eq:opt-5} follows when $\zeta_\mathrm{hh}<0$ by definition of $\mathfrak{m}_\mathrm{long}$ in~\eqref{eq:xi-long}.

Assume next that $\zeta_\mathrm{hh}=1-\gamma_\mathrm{hh}(\tau-1)\ge 0$. Using $\gamma_\mathrm{hh}=(\alpha-1)/((\sigma+1)\alpha-(\tau-1))$ and $\xi_\mathrm{hh}=(\sigma+1)\alpha-2(\tau-1)$, we leave it to the reader to verify that also $\mathrm{sign}(\xi_\mathrm{ll}-\xi_\mathrm{hh})=\mathrm{sign}(\zeta_\mathrm{ll}-\zeta_\mathrm{hh})$ and $\mathrm{sign}(\xi_\mathrm{hl}-\xi_\mathrm{hh})=\mathrm{sign}(\zeta_\mathrm{hl}-\zeta_\mathrm{hh})$. This proves~\eqref{eq:opt-5} in all cases.\smallskip 

 We now analyze the left-hand side in~\eqref{eq:zeta-opt-m2}, and note that $\mathfrak{m}_\star=\mathfrak{m}(\zeta_\mathrm{ll}, \zeta_\mathrm{hl}, \zeta_\mathrm{hh}, (d-1)/d)$ by definition in~\eqref{eq:m-star}. Thus, 
 \[
 \mathfrak{m}_\star-1=\big(\mathfrak{m}(\zeta_\mathrm{ll}, \zeta_\mathrm{hl}, \zeta_\mathrm{hh})-1\big)\ind{\max(\zeta_\mathrm{ll}, \zeta_\mathrm{hl}, \zeta_\mathrm{hh})>\tfrac{d-1}{d}}
 +
 \mathfrak{m}(\zeta_\mathrm{ll}, \zeta_\mathrm{hl}, \zeta_\mathrm{hh})\ind{\max(\zeta_\mathrm{ll}, \zeta_\mathrm{hl}, \zeta_\mathrm{hh})=\tfrac{d-1}{d}}, 
 \]
 as $\mathfrak{m}(\zeta_\mathrm{ll}, \zeta_\mathrm{hl}, \zeta_\mathrm{hh}, (d-1)/d)-1=0$ if $\max(\zeta_\mathrm{ll}, \zeta_\mathrm{hl}, \zeta_\mathrm{hh})<(d-1)/d$. Since $(d-1)/d\ge 0$, we can replace the multiplicities on the right-hand side by $\mathfrak{m}_\mathrm{long}$ using~\eqref{eq:opt-5}. By~\eqref{eq:imp1} and~\eqref{eq:imp3} we can replace the maximum inside the indicators by $2-\alpha+\gamma_\star\xi_\star$. Thus,
 \[
 \mathfrak{m}_\star-1=\big(\mathfrak{m}_\mathrm{long}-1\big)\ind{2-\alpha+\gamma_\star\xi_\star>\tfrac{d-1}{d}}
 +
 \mathfrak{m}_\mathrm{long}\ind{2-\alpha+\gamma_\star\xi_\star=\tfrac{d-1}{d}}.
 \]
 This proves the equality in~\eqref{eq:zeta-opt-m2}. We turn to the inequality in~\eqref{eq:zeta-opt-m2}. If the right-hand side of~\eqref{eq:zeta-opt-m2} is zero, the bound holds trivially since $\mathfrak{m}_\star\ge 1$. If the right-hand side of~\eqref{eq:zeta-opt-m2} is one, i.e., $1-\gamma_\star(\tau-1)=\tfrac{d-1}{d}$, then by~\eqref{eq:imp1} and~\eqref{eq:imp3} also $\max(\zeta_\mathrm{ll}, \zeta_\mathrm{hl}, \zeta_\mathrm{hh})=(d-1)/d$, and $\mathfrak{m}_\star \ge 2$, proving~\eqref{eq:zeta-opt-m2}.\smallskip

 Finally, we prove the statements for $\alpha=\infty$. We compute $\lim_\mathrm{\alpha\to\infty}\zeta_\mathrm{ll}=2-\alpha=-\infty$ and $\lim_{\alpha\to\infty}\zeta_\mathrm{hl}=(\tau-1)/\alpha-(\tau-2)=-(\tau-2)$. Hence, $\max(\zeta_\mathrm{ll}, \zeta_\mathrm{hl})<(d-1)/d$. By~\eqref{eq:gamma-opt}, $\gamma_\star=1/(\sigma+1)$, and $\zeta_\mathrm{hh}=1-(\tau-1)/(\sigma+1)$ by~\eqref{eq:zeta-hh}. So, $1-\gamma_\star(\tau-1)=\zeta_\mathrm{hh}$, proving   $\max(1-\gamma_\star(\tau-1), (d-1)/d)=\max(\zeta_\mathrm{ll}, \zeta_\mathrm{hl}, \zeta_\mathrm{hh}, (d-1)/d)$. By the same argumentation 
 \begin{equation*}\mathfrak{m}_\star-1=\mathfrak{m}(\zeta_\mathrm{ll}, \zeta_\mathrm{hl}, \zeta_\mathrm{hh}, (d-1)/d)-1=\mathfrak{m}(\zeta_\mathrm{hh}, (d-1)/d) -1= \ind{1-\gamma_\star(\tau-1)=(d-1)/d}.\qedhere\end{equation*}
\end{proof}
\vspace{-5pt}
Lastly, we state a Poisson concentration bound (without proof) that we often use.
\vspace{-5pt}
\begin{lemma}[Poisson bound {\cite{mitzenmacher2017probability}}]\label{lemma:poisson-1}
 For $x>1$,
 \begin{equation}
  \Prob\big(\Poi(\lambda) \ge x\lambda\big)\le\exp(-\lambda(1+x(\log x)-x)),\nonumber
 \end{equation}
 and for $x<1$,
 \begin{equation}
  \Prob\big(\Poi(\lambda) \le x\lambda\big)\le
  \exp(-\lambda(1-x-x(\log1/x)).\nonumber
 \end{equation}
\end{lemma}
\end{appendix}

\end{document}